%% file: mongluing.tex
\documentclass[11pt,reqno]{amsart}
\usepackage{mongluing}
\usepackage{tikz,pgf}
\title{Partial compactification of monopoles and metric asymptotics}
\author{Chris Kottke}
\address{Northeastern University\\Department of Mathematics}
\email{c.kottke@neu.edu}
\author{Michael Singer}
\address{University College London\\Department of Mathematics}
\email{michael.singer@ucl.ac.uk}
\begin{document}
\maketitle

\begin{abstract}
We construct a partial compactification of the moduli
space, $\monM_k$, of $\SU(2)$ magnetic monopoles on $\bbR^3$, wherein monopoles of
charge $k$ decompose into widely separated `monopole clusters' of lower
charge going off to infinity at comparable rates. The hyperK\"ahler metric
on $\monM_k$ has a complete asymptotic expansion up to the boundary, the leading term of which
generalizes the asymptotic metric discovered by Bielawski, Gibbons and
Manton in the case that each lower charge is 1.
\end{abstract}

\tableofcontents
\section{Introduction} \label{S:intro}
\input{intro}

\section{The Bogomolny equations on a scattering manifold}
\label{S:bogo1}
\input{bogoscat}

\subsection{Dirac monopoles} \label{S:dirac}
\input{dirac}

\section{Formal 1-parameter families} \label{S:Mgl}
\input{single}

\section{Moduli of ideal monopoles}\label{S:ideal}
\input{ideal}


\section{Universal gluing space and parameterized gluing}
\label{S:global}
\input{global}


\section{The metric}
\label{S:metric}
\input{metric}

\appendix
\section{Sobolev spaces} \label{S:sobolev}
\input{sobolev}

\section{Coulomb gauge}\label{S:coulomb}
\input{coulomb}

\section{Linear analysis} \label{S:linear}
\input{linear}

\section{Pseudodifferential operators} \label{S:double}
\input{double}

\bibliographystyle{amsalpha}
\bibliography{mongluing}

\end{document}

%% file: intro.tex
This paper is the first in a series aimed at studying the asymptotic
regions of the monopole moduli spaces and the behaviour of the $L^2$
metric in these regions.

Recall that an $\SU(2)$ monopole is a (gauge equivalence class) of
solutions of the {\em Bogomolny equations}
\begin{equation}\label{e3.16.5.15}
*F_A = \nabla_A\Phi,
\end{equation}
where $(A,\Phi)$ is a pair consisting of a connection $A$ on a
principal $\SU(2)$-bundle $P$ over $\RR^3$ and $\Phi$, the {\em Higgs field}, is a section
of the associated adjoint bundle $\ad(P)$.  The Bogomolny equations are
supplemented by assuming that the {\em Yang--Mills--Higgs} action is finite,
\begin{equation}\label{e4.16.5.15}
\int_{\RR^3} ( |F_A|^2 + |\nabla_A\Phi|^2) < \infty,
\end{equation}
and that
\begin{equation}\label{e5.16.5.15}
  |\Phi(z)| \smallto 1\mbox{ as }|z| \smallto\infty\mbox{ in }\RR^3,
\end{equation}
see below for further discussion.  Then \eqref{e5.16.5.15} entails that
the degree of $\Phi$ 
over a large sphere in $\RR^3$ is a positive integer $k$, the {\em magnetic
charge} (or monopole number).  By the device of framing
$(A,\Phi)$ at infinity, and restricting to the gauge group $\gauG_0$ of elements
of $\Aut(P)$ that approach the identity at infinity, the
{\em moduli space} $\monM_k$ of {\em framed monopoles of charge $k$}
is defined as the set of solutions of \eqref{e3.16.5.15}
satisfying \eqref{e4.16.5.15} and \eqref{e5.16.5.15}, divided by the
action of $\cG_0$.    It is known that $\monM_k$ is a smooth manifold
of real dimension $4k$, non-compact, but carrying a complete
riemannian $L^2$ metric $G_k$
\cite{AH,taubes1983stability,taubes1985min}.  More precisely, if
$(A,\Phi)$ 
represents a point $m$ of $\monM_k$, then 
\begin{equation}\label{e3.21.9.15}
T_m\monM_k = L^2\mbox{-}\Ker(L_{A,\Phi})
\end{equation}
where $L_{A,\Phi}$ is the linear operator
\begin{equation}\label{e4.21.9.15}
L_{A,\Phi}
: C^\infty(\RR^3, \Lambda\otimes \ad(P))
\longrightarrow 
C^\infty(\RR^3, \Lambda\otimes \ad(P)),
\end{equation}
and
\begin{equation}
L_{A,\Phi} = 
\begin{bmatrix} *\rd_A & -\rd_A \\
                                                -\rd_A^* &
                                                0 \end{bmatrix} +
                                              \ad(\Phi)\otimes \Id
\label{e4a.21.9.15}
\end{equation}
and $\Lambda = \Lambda^1\oplus \Lambda^0$.  The `top row' of this
operator is the linearization at $(A,\Phi)$ of the Bogomolny
equations; the bottom row is the Coulomb gauge fixing condition that 
$(a,\phi)$ is $L^2$-orthogonal to the tangent space of the gauge orbit
containing $(A,\Phi)$.  The Riemannian metric, $G_k$, on the
moduli space is defined by the formula
\begin{equation}\label{e2.21.9.15}
\|(a,\phi)\|_{G_k}^2 = \int_{\RR^3} |a|^2 + |\phi|^2
\end{equation}
for $(a,\phi)$ in the tangent space \eqref{e3.21.9.15}.
Here we have chosen the $\SU(2)$-invariant metric
$-\frac{1}{2}\tr(A^2)$ on the Lie algebra $\su(2)$.

The metric $G_k$ is hyperK\"ahler \cite{hitchin1987hyperkahler,AH}.
It is known that $G_1$ is the flat metric on $\RR^3\times S^1$ and
that $G_2$ is essentially the riemannian product of $\RR^3\times S^1$
and the famous Atiyah--Hitchin metric \cite{AH}.  For $k\geq 3$ it is
not feasible to compute $G_k$ explicitly but one may hope to 
gain a partial understanding of it 
asymptotically in terms of the metrics $G_{k_j}$ on moduli spaces of
lower charge.  

In order to explain this idea more carefully, recall \cite[Prop.\
3.8]{AH} about the asymptotic structure of $\monM_k$.  The following
statement uses the fact that a monopole $m \in \monM_k$ has a well-defined
centre of mass (cf.\ \S\ref{s1.29.10.15}); denote by $\moncM_k$ the
moduli space of monopoles centred at the origin, a submanifold of
$\monM_k$ of dimension $4k-3$. 

\begin{thm}[\cite{AH},Proposition~3.8\footnote{We have
  slightly rephrased the statement of this Proposition.}]
Given an infinite sequence of points of $\monM_k$, there exists a
subsequence $(m_\nu)$, a partition $k = k_0 + \cdots + k_N$ with
$k_i>0$, $i \geq 1$ (we allow $k_0 = 0$) and sequences of points $(z^i_\nu) \in \RR^3$ ($i=0,\ldots,N$)
such that
\begin{enumerate}
\item[(i)] the sequence $m_\nu(\cdot - z^i_\nu)$ 
converges weakly
(i.e. on compact subsets of $\RR^3$) to a $k_i$-monopole $m^i$,
 centered at the origin,
\item[(ii)] as $\nu \smallto \infty$, the $|z_\nu^iz_\nu^j| \smallto \infty$ for $i\neq j$ and the
  unit vectors
$$
\frac{\overrightarrow{z_{\nu}^iz_{\nu}^j}}{|z_\nu^iz_{\nu}^j|}
$$
converge in $\bbS^2$. Here we assume that $z^0_\nu \smallto z^0$ converges in $\bbR^3$, so $\abs{z^i_\nu} \smallto \infty$
for $i \geq 1$.
\end{enumerate}
\label{t1.16.5.15}\end{thm}

In fact, Taubes \cite{taubes1985min} has proved more refined results, showing
that, along appropriate sequences, $\bbR^3$ can be divided into
`strong-field' regions in which most the energy \eqref{e4.16.5.15} is
concentrated, the centers of whose path components can be associated with the
sequences $(z_\nu^i)$ above, and a `weak-field' region in which the fields are
approximately abelian, and to high order satisfy the abelian, or `Dirac',
monopole equations (cf.\ \S\ref{S:dirac} below), with prescribed singularities approaching the strong-field
regions. (This dichotomy is reflected in our geometric construction below.)

Thus the non-compactness of $\monM_k$ is captured by sequences of monopoles of
lower charge escaping to $\infty$ in $\RR^3$.  Note that in the above theorem,
there is no control on the relative sizes of the $|z_\nu^iz_{\nu}^j|$ for
different pairs $ij$.  The simplest part of the asymptotic region of
$\monM_k$, the subject of this paper, corresponds to the case that all these
lengths are uniformly comparable as $\nu \smallto\infty$.

The asymptotic behaviour of $G_k$ has been studied in special cases by
various authors. 
In \cite{gibbons1995moduli}, Gibbons and Manton derived a model metric for the
asymptotic region corresponding to $k_0 = 0$ and $k_i = 1$, $i \geq 1$, and in
\cite{RB_GM} Bielawski proved that the $L^2$ metric is
exponentially close to the model in this region. Using the representation of
monopoles via spectral curves, Bielawski also proved the existence of
simplified hyperK\"ahler models for cluster regions of higher charge in
\cite{bielawski2008monopoles}, though a description of these models in terms of
the metrics on lower moduli spaces appears to be difficult to obtain
directly from this work.

To study these asymptotic regions directly, we construct a partial compactification
of $\monM_k$ by associating boundary hypersurfaces to these limits.
These boundary hypersurfaces are moduli spaces of {\em ideal
  monopoles}, objects which roughly speaking
consist of the following data (they are defined more precisely below): a list $\ul{k} = (k_0,k_1,\ldots,k_N)$, 
of integers with $k_0\geq 0$, $k_i\geq 1$ for $i=1,\ldots,N$, and
$\sum_{i=0}^N = k$; a collection of monopoles respectively of charges $k_0,\ldots,
k_N$; and an asymptotic configuration of distinct non-zero points
$(\zeta_1,\ldots,\zeta_N)$ up to scale. To be more specific, define
\begin{equation}\label{e1.19.11.15}
\cE_N^* = \Big\{\ul \zeta = (\zeta_1,\ldots,\zeta_N) \in (\bbR^3)^N :
\zeta_i \neq \zeta_j, \mbox{ for }i\neq j, \zeta_i \neq 0,\ \sum_i
\abs{\zeta_i}^2 = 1\Big\} 
\end{equation}
which we view as part of the boundary of the radial compactification
of $\RR^{3N}$.  The moduli space, $\cI_{\ul{k}}$, of ideal monopoles of
type $\ul{k}$ is then a non-trivial fibre bundle
\begin{equation}
	\cI_{\ul k} \to \cE_N^*
	\label{E:idmon_bundle}
\end{equation}
over the space of ideal configurations $\cE^*_N$, with fibre
\begin{equation}
	(\cI_{\ul k})_{\ul \zeta} = \monM_{k_0} \times \moncM_{k_1} \times \cdots \times \moncM_{k_N}.
	\label{E:idmon_fiber}
\end{equation}
The $\cI_{\ul k}$ (or more correctly, their quotients by a
symmetric group interchanging factors of equal charge) form the boundary
hypersurfaces of our partial compactification, with normal
coordinate given by a scaling paramter $\ve$, defined so that the $z^i = \zeta_i/\ve$.  Thus the paramters in the base,
$\cE_N^*$, represent the limiting locations of the monopole clusters of charges
$k_i$, $i = 1,\ldots, N$ which have gone off to infinity, with a cluster of
charge $k_0$ which remains behind, while the parameters in the fiber represent
the (possibly recentered) clusters themselves.

In \S\ref{S:ideal}, we shall describe the structure of the bundle
$\cI_{\ul{k}}$ over $\cE^*_N$. The type $\ul k$ determines a 
 rank $N+1$ {\em Gibbons-Manton
torus bundle} $\TGM \to \cE^*_N$ with respect to which
$\idmon_{\ul k}$ is the associated fiber bundle
\[
	\idmon_{\ul k} = \TGM \times_{\UU(1)^{N+1}} (\monM_{k_0}\times
	 \moncM_{k_1} \times \cdots \times \moncM_{k_N}) \to \cE^*_N
\]
given by the quotient by $\UU(1)^{N+1}$ acting on $\TGM$
on the left and by the circle actions on the framed moduli spaces on
the right.  This description generalizes the case studied by
Gibbons--Manton and Bielawski, where $k_0=0$, $k_i=1$, the centred
moduli spaces are reduced to circles, and so the asymptotic moduli
space is just the Gibbons--Manton torus bundle itself of type
$(0,1,\ldots,1)$ over $\cE^*_k$.

Our first main result is 
\begin{thm}
For each $\iota_0 \in \cI_{\ul k}$, there exists a neighborhood $\cU \ni \iota_0$,
$\ve_0 > 0$, and a smooth map
\begin{equation}
	\Psi : \cU \times (0,\ve_0) \to \monM_k
	\label{E:intro_gluing_map}
\end{equation}
which is a local diffeomorphism onto its image, and such that $\Psi(\iota,\ve)
\to \iota$ as $\ve \to 0$, with a complete asymptotic expansion in $\ve$.
\label{T:main_one}
\end{thm}
\noindent
The map itself depends on the choice of gauge representative for $\iota_0$, though two choices
will agree to leading order in $\ve$.

Monopole gluing theorems are not new; indeed, Taubes' gluing theorem in
\cite{JT} for widely separated charge 1 monopoles was the first existence
result for monopoles of higher charge. More recently, Oliveira \cite{Oliveira}
and Foscolo \cite{Foscolo} have obtained gluing results for monopoles on asymptotically
conic 3-manifolds, and on $\bbR^2\times \bbS^1$, respectively. In Donaldson's
representation \cite{donaldson1984nahm} of $\monM_k$ as a space of rational
maps, gluing corresponds simply to addition of rational maps, though metric
information is not readily available in this picture. The chief advantage of
our approach over more traditional techniques is that we obtain complete
asymptotic expansions in the gluing parameter, and in particular can
compute the metric to leading order in this parameter.

The tangent bundle of $\cU\times (0,\ve_0)$ can be identified with the product
\begin{equation}
	T\monM_{k_0} \oplus (\bigoplus_{i=1}^N T \moncM_{k_i} \oplus \bbR^3).
	\label{E:prod_tangt_space}
\end{equation}
Here the $N$ factors of $\bbR^3$ come from identifying $\cE_N^*\times
(0,\ve_0)$ with the configuration space $\cC^*_N$ of distinct non-zero
points, {\em not} up to scaling, 
which is an open subset of $(\bbR^3)^N$. 
Our second main result states
\begin{thm}
The pulled back metric $\Psi^*(G_k)$ has a complete asymptotic expansion as $\ve \smallto 0$, 
with leading order 
\begin{equation}
	\Psi^*(G_k) \cong G_{k_0} \oplus (\bigoplus_{i = 1}^N G^c_{k_i} \oplus 2\pi k_i \eta_i) + \cO(\ve)
	\label{E:intro_metric_leading_order}
\end{equation}
with respect to the identification with \eqref{E:prod_tangt_space}.
\label{T:main_two}
\end{thm}
\noindent
Further refinements (not proved here) give the next order term in the metric as
well, with the result that $\Psi^*(G_k)$ generalizes the asympototic metric of
Gibbons and Manton \cite{gibbons1995moduli} for the case where $k_0 =
0$, $k_i = 1$, $i = 1,\ldots, N$.  

\begin{rmk}  The structure of this compactification, and in particular
  the Gibbons--Manton bundles associated to a general $\ul k$, is also
  known to Bielawski.  More details will appear in \cite{B_future}.
\end{rmk}

\subsection{Overview of the construction} \label{S:intro_overview}
We give a brief overview of our construction, highlighting the advantages of
our approach, and explaining how ideal monopoles enter.

The first step is to pass to the radial compactification, $\fX = \ol{\bbR^3}$, of $\bbR^3$ as a convenient
way to deal with the non-compactness of $\bbR^3$. 
Having done so, the Euclidean metric becomes a smooth, bounded,
positive definite metric on the so-called scattering tangent bundle $\scT \fX$ of $\fX$, and 
monopoles may be defined in terms of data which are smooth up to $\pa \fX$ and are regarded as
sections of $\bigwedge^j \scT^\ast \fX\otimes \adP$, where we
introduce the notation $\adP = \ad(P)_{\CC}$. This point of view, as applied
to the classical theory of monopoles on $\bbR^3$, is summarized in \S\ref{S:bogo1}.

The next step is to incorporate the parameter $\ve$ into the geometry of the
problem. We begin with the product $\fZ_0 = \fX
\times [0,\infty)_\ve$, equipping each fiber of the projection $\fb : \fZ_0 \to
[0,\infty)$ with the Euclidean metric on $\fX$. 
Fixing a configuration $\ul
\zeta \in \cE^\ast_N$, the paths $(z(\ve),\ve) = (\zeta_j/\ve,\ve)$ approach the 
corner $\pa \fX \times \set 0 \subset \fZ_0$, and we let
\[
	\fZ_1 = [\fZ_0; \pa \fX \times \set 0]
\]
be the real blow-up of this corner in $\fZ_0$. The new boundary face obtained
by this blow-up, denoted by $\fD$, is diffeomorphic to a cylinder, with a
natural interpretation as the blow-up $D \cong [\ol{\bbR^3}; \set 0]$ of the
radial compactification of $\bbR^3$ at the origin. Thus equipped with the
interior Euclidean coordinate $\zeta$, $\fD$ meets the boundaries of the lifted
curves $z(\ve) = \zeta_i/\ve$ transversally at the points where $\zeta =
\zeta_j$. (See Figure~\ref{F:Z_1}.) 

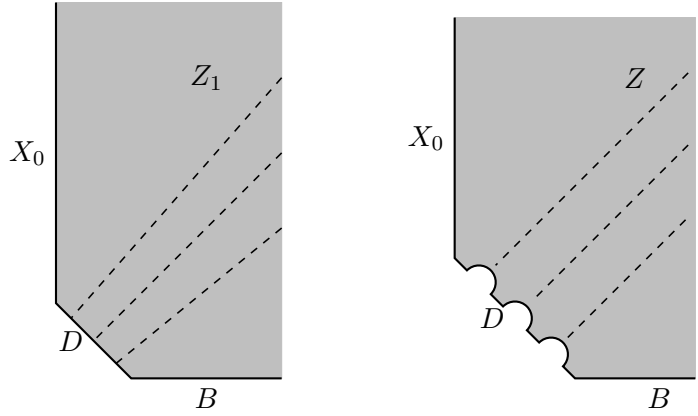
\begin{figure}[htb]
\[
\begin{tikzpicture}[>=stealth,domain=-4:4,smooth,scale=1]
%
%
\filldraw[color = lightgray]
(1,0) -- (3,0) -- (3,5) -- (0,5) -- (0,1) -- (1,0);
\draw[thick]
(3,0) -- (1,0) --(0,1) -- (0,5);
\draw (2,0) node[below] {$B$};
\draw (0,3) node[left] {$X_0$};
\draw (0.5,.5) node[left] {$D$};
\draw (2,4) node {$\fZ_1$};
\draw [semithick,dashed] (3,2) -- (0.8,.2);
\draw [semithick,dashed] (3,3) -- (0.5,.5);
\draw [semithick,dashed] (3,4) -- (0.2,.8);
\end{tikzpicture}
\qquad\qquad
\begin{tikzpicture}[>=stealth,domain=-4:4,smooth,scale=1.6]
%
%
\filldraw[color = lightgray]
(2,0) --(1,0)--(.9,.1) arc (-45:135:.14)
-- (.6,.4) arc (-45:135:.14) 
-- (.3,.7) arc (-45:135:.14)
--(0,1) --  (0,3) -- (2,3) -- (2,0);
\draw (1.7,0) node[below] {$B$};
\draw (0,2) node[left] {$X_0$};
\draw (0.5,.5) node[left] {$D$};
\draw (1.5,2.5) node {$\fZ$};
\draw [semithick,dashed]  (0.94,.34) -- (1.94, 1.34);
\draw [semithick,dashed] (1.94,1.94) -- (0.64,.64);
\draw [semithick,dashed] (1.94,2.54) -- (0.34,.94);
\draw [thick](2,0) --(1,0)--(.9,.1) arc (-45:135:.14)
-- (.6,.4) arc (-45:135:.14) 
-- (.3,.7) arc (-45:135:.14)
--(0,1) --  (0,3);
\end{tikzpicture}
\]
\caption{The space $\fZ_1$ with the lifts of the curves $z = 
\zeta_j/\ve$ hitting the new face $D$ transversely in points $\zeta_j$, and the blow-up of this to the space $\fZ$.}
\label{F:Z_1}
\end{figure}

The geometric construction is completed by blowing these points up in $\fD$, setting
\[
	\fZ = [\fZ_1; \set{\zeta_1,\ldots,\zeta_N}].
\]
The new faces are denoted by $\fX_i$, $i = 1,\ldots,N$, and the lift of the original face $\fX \times \set 0$ is
denoted by $\fX_0$ (see Figure~\ref{F:Z_1}). These admit canonical isometries $\fX_i \cong \fX$ with the radial compactification of $\bbR^3$
with respect to the lift of the original fiberwise Euclidean metric on $\fZ_0$.

We also lift the projection to the parameter interval to a map $\fb : \fZ \to
[0,\infty)_\ve$. For $\ve > 0$, the fiber $\fb^{-1}(\ve)$ is canonically
isometric to the original radially compactified $\bbR^3$, whereas $\fb^{-1}(0)$
is a manifold `with normal crossings' consisting of (the lifted) $\fD \cong
[\ol{\bbR^3}; \set{0,\zeta_1,\ldots,\zeta_N}]$ and the $\fX_i$, identified
along their boundaries. 
We denote by $\rho_\fD$, $\rho_{\fX}$, and $\rho_\fB$
choices of respective boundary defining functions for $\fD$, $\fX_0\cup \fX_1
\cup \cdots \cup \fX_N$, and $\fB$, the lift of the original boundary face $\pa
\fX \times [0,\infty)$ of $\fZ_0$. We may assume that $\ve = \rho_\fD
\rho_\fX$.

Next, we fix a principal $\SU(2)$ bundle $P$ over $\fZ$ and consider the
fiberwise Bogomolny operator
\[
	\Bogo(A,\Phi) = \star_{\fb} F_{A_\fb} - \nabla_{A_\fb} \Phi_\fb
\]
on pairs, $(A,\Phi)$, of connection and Higgs bundle on $P$, where the notation
indicates that the data are restricted to each fiber $\fb^{-1}(\ve)$ before
computing the curvature, covariant derivative, and Hodge star operator.

One checks that if $(A,\Phi)$ is smooth up to the boundaries of $\fZ$, then,
as a section of the appropriate fiberwise 1-form bundle twisted by $\adP$, 
$\Bogo(A,\Phi)$ is smooth, and in fact
\begin{equation}
	\Bogo(A,\Phi) = \cO(\rho_\fD).
	\label{E:Bogo_intro_initial_error}
\end{equation}
We ultimately embark on the process of improving the error term on the right hand side of
\eqref{E:Bogo_intro_initial_error} until it vanishes identically for small
$\ve$, thereby obtaining 1-parameter families of monopoles on the fibers
$\fb^{-1}(\ve)$ for $\ve > 0$.

The restriction of $\Bogo$ to $\fX_i$ is nothing other than the Euclidean
Bogomolny equation $\star F_{A \rst \fX_i} - \nabla_{A\rst\fX_i} (\Phi \rst
\fX_i)$, so if we choose $(A,\Phi)$ so that its restriction, $(A,\Phi)\rst
\fX_i$, solves this equation for each $i = 0,\ldots,N$, then
\eqref{E:Bogo_intro_initial_error} improves by a factor of $\rho_\fX$ and
\begin{equation}
	\Bogo(A,\Phi) = \cO(\rho_\fD\rho_\fX) = \cO(\ve).
	\label{E:Bogo_intro_error_two}
\end{equation}
Over $\fD$, the leading part of the Bogomolny operator is just $\nabla_{A\rst
\fD} (\Phi \rst \fD)$, so choosing this to vanish improves the approximation by
a factor of $\rho_\fD$. The  sub-leading order part of the Bogomolny
operator over $\fD$ imposes a condition on $A \rst \fD$. In more detail, $A
\rst \fD$ reduces to a connection on a $\UU(1)$-bundle over $\fD$ determined by
$\Phi \rst \fD$, and must solve the $\UU(1)$ monopole equations there, with
prescribed conic singularities at $\set{\zeta = \zeta_j : j = 0,\ldots,N}$.
Imposing this condition improves the error term to
\begin{equation}
	\Bogo(A,\Phi) = \cO(\rho_\fD^3\rho_\fX) = \cO(\ve\rho_\fD^2).
	\label{E:Bogo_intro_error_three}
\end{equation}

An ideal monopole is properly defined to be the restriction to $\fb^{-1}(0) = D
\cup \fX_0 \cup \cdots \cup \fX_N$ of a smooth configuration $(A,\Phi)$ on
$\fZ$ such that \eqref{E:Bogo_intro_error_three} holds. Ideal monopoles are
acted on by a gauge group given by the restriction to $\fb^{-1}(0)$ of gauge
transformations on $P \to \fZ$ which are the identity at $\fB$. In fact, though
there appears to be additional information represented by the Dirac monopole on
the `interstitial region', $\fD$, where the symmetry is broken down to
$\UU(1)$, this extra information disappears when we consider gauge equivalence
classes of ideal monopoles.

The choice of data $(A,\Phi)$ on $\fZ$, representing a given ideal monopole on
$\fb^{-1}(0)$, is the starting point for an iteration, carried out in \S\ref{S:Mgl}, in which we successively
improve the error term on the right hand side of
\eqref{E:Bogo_intro_error_three}. The output of this iteration is a section 
\[
	(a,\phi) \in \cA(\fZ; (\Lambda^1\oplus \Lambda^0)\otimes \adP)
\]
with the property that
\[
	\Bogo(A+a,\Phi+\phi) = \cO(\rho_\fD^\infty\rho_\fX^\infty\rho_\fB^\infty).
\]
Here the notation $\cA$ means that $(a,\phi)$ is smooth in the interior of
$\fZ$ and has complete asymptotic expansions in powers $\rho^j(\log \rho)^k$
for each of the boundary defining functions $\rho$, with compatibility between
these expansions at the corners. (In precise terms, $(a,\phi)$ is {\em polyhomogeneous conormal}, 
c.f.\ \cite{CCN} and \S\ref{S:functions_phg}.)

The final step---which is carried out in \S\ref{S:global} smoothly in families,
on a space fibering over $\idmon$ with fibers given by the $\fZ$---is a further
correction $(\wt a, \wt \phi) \in \rho^\infty C^\infty(\fZ; (\Lambda^1\oplus \Lambda^0)\otimes \adP)$ 
such that
\[
	\Bogo(A + a + \wt a, \Phi + \phi + \wt \phi) = 0\ \text{on}\ \fb^{-1}\big([0,\ve_0\big))
\]
for some $\ve_0 > 0$. This step follows from the construction of a
pseudodifferential parametrix for the linearization of $\Bogo$ and an implicit
function theorem argument.

Having constructed such smooth families, we then consider variations with
respect to the parameters in \S\ref{S:metric}. These represent solutions to the
linearized, fiberwise Bogomolny equations on $\fZ$, which by small perturbation
can be put into Coulomb gauge along each fiber. This gives the differential of
the gluing map \eqref{E:intro_gluing_map}, and allows the metric
$\Psi^\ast(G_k)$ to be computed by pairing two such variations and integrating
over the fibers of $\fb: \fZ \to [0,\ve_0)$. The leading order contribution in
$\ve$, \eqref{E:intro_metric_leading_order}, reduces to explicit $L^2$ pairings
of Euclidean monopole variations over the boundary faces $\fX_i$, the
computations of which already appear in \S\ref{S:bogo1}.

\subsection{Outlook: monopole compactification}

In future work with Melrose, we shall complete the compactification of
$\monM_k$ as a smooth manifold with corners and study the metric on this
compactification. (An alternative approach is being pursued by
Bielawski \cite{B_future}.)  We explain why a compactification as a manifold with
corners is natural for this problem, illustrating with the simplest
possible cases.

Let $k=3$.  According to our gluing theorem there are two asymptotic
regions of $\monM_{3}$, let's call them $\cV_{111}$ and $\cV_{21}$
corresponding respectively to the partition $3 =1+1+1$ and $3=
2+1$.   There is, however, a `transition region' between $\cV_{111}$
and $\cV_{21}$ which we can describe in terms of $\log$-smooth
$2$-parameter families $m(\ve_1,\ve_2) \in \monM_3$ with the following
properties. For fixed $\ve_1>0$, $\ve_2 \mapsto m(\ve_1,\ve_2)$ is a
smooth curve in $V_{111}$, while for $\ve_2>0$, $\ve_1 \mapsto
m(\ve_1,\ve_2)$ is a smooth curve in $V_{21}$. The parameter $\ve_2^{-1}$
is a measure of the separation of the charge-2 monopole in $\moncM_2$, so
for smaller values $\ve_2$, this charge-2 monopole is going to
infinity in $\moncM_2$.  From the other point of view,  on the curve 
$\ve_1 \mapsto m(\ve_1,\ve_2)$, $\ve_2$ is a measure of the distance
between $z_1$ and $z_2$ relative to $z_3$.  Thus for small $\ve_2$,
the configuration is becoming less and less widely separated, and the charge-1
monopoles centred at $z_1$ and $z_2$ have to be treated as a monopole
of charge 2.

The construction of such smooth $2$-parameter families will be the
work of our next paper. Implemented systematically, this will allow us
to build up a compactification of $\monM_k$ as a smooth manifold with
corners, with control of the $L^2$ metric near each corner.  The
two-parameter families described in the previous paragraph correspond
to codimension-2 corners of the moduli space.

\subsection*{Acknowledgements}
The authors would like to thank Richard Melrose
for his input into the approach we take to prove our main results. They also
would like to thank Roger Bielawski, Nicholas Manton and Karsten Fritzsch for
enlightening discussions.

%% file: bogoscat.tex
In this section we shall introduce smooth definitions of the monopole
moduli spaces.  Our approach is to pass to the radial compactification
$X = \overline{\RR^3}$ of $\RR^3$ and regard the euclidean metric as a
scattering metric on this space; we shall recall these notions
below; but $X$ is a compact manifold with boundary $\pa X$ diffeomorphic
to the 2-sphere $\bbS^2$, and instead of the decay conditions when 
$(A,\Phi)$ are regarded as fields on $\RR^3$, we assume smoothness up
to the boundary.   The whole classical theory of monopoles can be
developed in this setting and is equivalent to the `usual' approach.
The equivalence follows from the basic regularity theorems proved in
Jaffe--Taubes \cite{JT} which show that the finite-energy conditions
\eqref{e4.16.5.15} imply decay properties of the fields equivalent to smoothness
up to the boundary of $X$.

\subsection{Gauge theory on a manifold with boundary}
 \label{S:bogoscat_prelim}

Let $X$ be a manifold with boundary $\pa X$. Recall that a {\em boundary
defining function} is a smooth function $\rho : X \to [0,\infty)$, such
that $\pa X = \rho^{-1}(0)$ and $\rd \rho\neq 0$ on $\pa X$.

If $p$ is any point of $\pa X$, then there exist adapted local
coordinates $(x,y)$ defined in a neighbourhood $U$ of $p$ in $X$ such
that $x\geq 0$, $y$ is a system of local coordinates near $p$ in $\pa
X$, and such that $\rho|U = x$.

Let $P \to X$ be a smooth principal $G$-bundle, where $G$ is some Lie
group.  From here on, `smooth' will always mean `smooth up to and
including $\pa X$' unless otherwise stated, and the vector space of all
smooth functions will be denoted by $C^\infty(X)$. It may be convenient to
recall that we can also think of $C^\infty(X)$ as the space of
restrictions of smooth functions from a slightly larger manifold
$\hat{X}$ in which $X$ sits as the closed subset $\set{\rho\geq 0}$.

\begin{defn}
The space of smooth $G$-connections on $P$ will be denoted by 
$\frA(X; P)$, or just $\frA$ when there is no risk of confusion. The
gauge group $\gauG(X,P) =\gauG$  is the space of smooth sections of
$\Ad(P)$, the bundle of 
groups associated to $P$.  If $A \in  \frA$, the curvature is denoted
by $F_A$ or $F(A)$ and is a smooth section of $\Lambda^2T^*X\otimes
\ad(P)$.
\end{defn}

Similarly we define the space of smooth Yang-Mills-Higgs (YMH) configurations:
\begin{defn}  Let $P$ and $X$ be as above.  The space of smooth YMH
  (or monopole) configurations, $\cfgC(X,P)$ is defined to be
\begin{equation}
\cfgC(X,P) = \frA(X,P)\times C^\infty(X,\ad(P)).
\end{equation}
This will be abbreviated to $\cfgC(X)$ or even $\cfgC$ when there is
no risk of confusion.
\end{defn}

The gauge group $\gauG$ acts on $\cfgC$ by conjugation:
\begin{equation}
\gamma(A,\Phi) = (\gamma A,\Ad(\gamma)\Phi)
\end{equation}
and we have
\begin{equation}
F(\gamma (A) )= \Ad(\gamma)F(A),\;\nabla_{\gamma(A)}(\Ad(\gamma) \Phi) =
\Ad(\gamma)\nabla_A\Phi,
\end{equation}
where $\nabla_A$ is the covariant derivative operator
\begin{equation}\label{e1.7.10.15}
\nabla_A : C^\infty(X,\ad(P)) \to C^\infty(X,T^*X\otimes \ad(P)).
\end{equation}
\begin{rmk}
We pause to emphasise what smoothness at the boundary means here. From
the point of view of principal bundles, where connections are defined
as smooth families of horizontal subspaces of $TP$, we mean that this
family is smooth up to the boundary, and in particular admits a smooth
extension as a connection on a principal bundle $\hat{P}$ over the
larger manifold $\hat{X}$. Alternatively, for the first-order differential operator
\eqref{e1.7.10.15}, smoothness means that all coefficients are smooth
when expressed in terms of smooth (up-to-the-boundary) local
trivializations of the bundles.
\end{rmk}

\subsubsection{Restriction and extension}
\label{s2.8.10.15}
Let $E\to X$ be any smooth vector bundle, and let $\nabla_A$ be a
smooth connection, identified with a covariant derivative operator, on
$E$.   Denote by $\jmath: \pa X \to X$ the inclusion.  Denote by
$E_\pa = \jmath^*(E)$ the restriction of $E$ to the boundary.  As is
well known we have a natural exact sequence
a natural restriction map $\jmath^*$
\begin{equation}\label{e2.8.10.15}
0 \to \rho C^\infty(X,E) \to C^\infty(X,E) \stackrel{\jmath^*}{\to} C^\infty(\pa X, E_\pa)
\to 0.
\end{equation}
We note the corresponding result for connections:
\begin{prop}\label{p1.18.10.15}
Let the data be above.  Then there is a natural exact sequence 
\begin{equation}\label{e3.8.10.15}
0 \to
\rho C^\infty(X, \boT^*X\otimes \ad(P))
\to 
\frA(X,P) 
\stackrel{\jmath^*}{\to}\frA(\pa X, P_\pa) \to 0
\end{equation}
\end{prop}

Here there is a harmless abuse of notation in that the second and
third spaces are affine spaces, not vector spaces; we read from the
sequence in particular that given any connection $A_\pa$ on $P_\pa \to
\pa X$, there is a connection $A$ on $P$ with $\jmath^*A = A_\pa$, and
that $A$ is unique up to the addition of an element in
\begin{equation}\label{e3.5.8.10.15}
\rho C^\infty(X, \boT^*X\otimes \ad(P)).
\end{equation}
\begin{proof}
This is reduced to the exactness of \eqref{e2.8.10.15} by picking a
background connection.  In more detail, pick any connection $A$ on $P$
and put $\jmath^*A = A_{\pa}$. To show that $\jmath^*$ is surjective,
suppose $B$ is any other connection on $P_\pa$. Then $B - A_\pa$ is a
section of $C^\infty(\pa X, T^*Y\otimes \ad(P_\pa))$.  Now we use the
variant
\begin{equation}\label{e4.8.10.15}
0 \to \rho C^\infty(X,\boT^*X) \to C^\infty(X,T^*X) \to C^\infty(\pa X,
T^*\pa X) \to 0
\end{equation}
of \eqref{e2.8.10.15} to find an extension $\widetilde{a}$ of $B-
A_\pa$ to $X$. (See \S\ref{s1.8.10.15} below for the $\bo$-cotangent bundle
$\boT^*$.) 

Then $A + \widetilde{a}$ is an extension of $A_{\pa}$
to $X$.  Using the exactness of \eqref{e4.8.10.15} we obtain similarly
that any two extensions of $A_{\pa}$ differ by an element of \eqref{e3.5.8.10.15}.

Naturality corresponds to the statement that if $A_\pa = \jmath^*(A)$, then
\begin{equation}
\nabla_{A_\pa}\, \jmath^*(s) = \jmath^*(\nabla_A s).
\end{equation}
\end{proof}

The restriction to the boundary $A_\pa$ of $A$ is then a
connection on $E|\pa X$, and induces a covariant derivative operator
\begin{equation}\label{e1.8.10.15}
\nabla_{A_\pa}: C^\infty(\pa X, E|\pa X) \to C^\infty(\pa X,
T^*\pa X\otimes (E|\pa X)).
\end{equation}

\subsection{$\RR^3$ as a scattering manifold}

In order to write down the Bogomolny equations in this setting, we
replace $\RR^3$ by its radial compactification $X$
\cite{RBM_scat}.  This can be defined as follows: first identify
$\RR^3\setminus 0$ with $(0,\infty)_r\times \bbS^2$, where $r$ is
distance from $0$.  Now glue the half-closed cylinder
$[0,\infty)_x\times \bbS^2$ to $\RR^3\setminus 0$ by the identification
$x = r^{-1}$ to obtain the radial compactification $X$. The boundary
$x=0$ corresponds to the `sphere at infinity' in $\RR^3$.  Equip $X$
with the $C^\infty$ structure so that $x$ is a (local) boundary
defining function.

If $y$ stands for local coordinates on the boundary $\bbS^2$, the euclidean metric
takes the form
\begin{equation}
\eta = \frac{\rd x^2}{x^4} + \frac{h(x,y,dy}{x^2}
\end{equation}
near $x=0$, where for each $x$, $h(x,y,dy)$ represents a metric on $\bbS^2$
such that $h(0,y,dy)$ is the round metric.
In particular it does not extend smoothly as a metric on
$TX$.  There is, however, a replacement, the `scattering tangent
bundle' $\scT$ which we now recall.

Denote by $\cV(X)$ the space of all smooth vector fields on $X$ and by
$\cV_{\scat}(X)$ the subspace of those of finite length with respect
to $\eta$. As shown in \cite{RBM_scat}, the subspace $\cV_{\scat}(X)$ is the
full space of $C^\infty$ sections of the {\em scattering tangent
  bundle} $\scT X \to X$, equipped with a smooth bundle map
\begin{equation}
\iota : \scT X \to TX
\end{equation}
which is an isomorphism over the interior of $X$.  More precisely, the
map induced by $\iota$ on global sections
\begin{equation}
\iota: C^\infty(X,\scT X) \to C^\infty(X,TX)
\end{equation}
gives an isomorphism $C^\infty(X,\scT X) = \cV_{\scat}(X)$. 

If $p$ is a point of $\pa X$ and we choose adapted local coordinates
$(x,y_1,y_2)$ in a neighbourhood $U$ of $p$, then $\cV_{\scat}(X)$ is locally spanned by the
elements
\begin{equation}\label{e1.3.8.15}
x^2 \pa_x, x\pa_{y_1}, x \pa_{y_2}.
\end{equation}
and there is a basis $e_0,e_1,e_2$ of $\scT U$ such that 
\begin{equation}
\iota(e_0) = x^2\pa_x, \iota(e_1) = x \pa_{y_1}, \iota(e_2) = x
\pa_{y_2}.
\end{equation}
We shall often abuse notation by regarding the vector fields
\eqref{e1.3.8.15} as a local frame for $\scT X$.

We note that (tautologically) the euclidean metric $\eta$ extends from
$\mathring X$ to define a smooth metric on $\scT X$; equivalently it is a
smooth, positive definite section of $S^2( \scT^* X)$, where $\scT^*$
is the scattering cotangent bundle, dual to $\scT$. 

We remark that $\scT$ and $\cV_{\scat}(X)$ are intrinsically
associated to $X$ as a smooth manifold with boundary: this is a
question of seeing that the definitions are independent of choice of
boundary defining function.   Essentially the same observation means
that if $X$ is any manifold with boundary, we can introduce $\scT$ and
$\cV_{\scat}(X)$; from the geometric point of view, a smooth metric on
$\scT X$ corresponds to (the big end of) an asymptotically conic
metric on $\mathring X$.

The advantage of introducing $\scT$ is that quantities naturally
associated to the euclidean metric extend uniformly to the boundary.
For example, the Hodge star operator gives an isometry
\begin{equation}
\star : \Lambda^2 \,\scT^* \to \scT^*
\end{equation}
of bundles over $X$, including over the boundary.

\subsubsection{$\bo$-structure}\label{s1.8.10.15}
  We pause to mention the parallel
definitions of the algebra $\cV_{\bo}(X)$ of vector fields that are
tangent to $\pa X$, as this algebra will also play an
important role later.  We have $\cV_{\scat}(X) = x\cV_{\bo}(X)$; there
is a $\bo$-tangent bundle $\boT$ with the property that
$C^\infty(X,\boT) = \cV_{\bo}(X)$ and so on. In place of the local
description \eqref{e1.3.8.15} for $\scT$ we have a local frame
$$
x\pa_x,  \pa_{y^1}, \pa_{y^2}
$$
for $\boT$. We refer the reader to \cite{RBMgreenbook} for a complete account.

\subsection{Bogomolny equations}
\label{S:bogoscat_monopoles}
Let $X$ be the radial compactification of $\RR^3$, equipped with the
euclidean metric (regarded, as above, as a scattering metric on $X$),
and fix a $G$ principal-bundle $P \to X$.  Denote by $\rho$ any
boundary defining function for $\pa X$.

We consider here the definition of the monopole moduli spaces from the
smooth point of view (i.e. building the boundary regularity of the
data from the beginning, rather than assuming just the boundedness of
the Yang--Mills--Higgs action) and recall some standard properties of
these moduli spaces.  There is nothing really new here: the
definitions of framed and unframed moduli spaces are in \cite{AH} and
the equivalence of the smooth definition with standard works on monopoles
follows from the regularity results in \cite{JT}.

\begin{lem}
The Bogomolny operator $*F_A - \nabla_A\Phi$ extends from the interior
to define a smooth map
\begin{equation}
\Bogo(A,\Phi) : \cfgC(X,P) \to \rho C^\infty(X,\scT^*X\otimes \ad(P)).
\end{equation}
\end{lem}
\begin{proof}
The proof is very easy, but instructive.  First of all, any smooth
$1$-form extends canonically as an element of $\rho
C^\infty(X,\scT^*X)$. In particular,
\begin{equation}
(A,\Phi) \in \cfgC \Rightarrow \nabla_A\Phi \in\rho
C^\infty(X,\scT^*X\otimes\ad(P)).
\end{equation}
Next, suppose that $p$ is a point of $\pa X$, and that $(x,y_1,y_2)$ are adapted
coordinates near $p$. If we suppose that $x^{-2}\rd x$, $x^{-1}\rd
y_1$ and $x^{-1}\rd y_2$ are $\eta$-orthonormal at $p$, then
\begin{equation}
*(\rd x \wedge \rd y_1) = x^2 \rd y_2, 
\quad *(\rd y_1\wedge \rd y_2) = \rd x
\end{equation}
and it follows that $*$ maps $C^\infty(X,\Lambda^2T^*X)$ into $\rho^2
C^\infty(X,\scT^*X)$.  Hence the result.
\end{proof}

The fact that $*F(A)$ (as a section of $\scT^*X$) is an order of
magnitude smaller than $\nabla_A\Phi$ 
at $\pa X$ has the following important consequences for the boundary
behaviour of smooth monopoles.

Denote by $\ovA$ and $\ovPhi$ the restrictions of $A$ and $\Phi$
 to $\pa X$ (cf.\  \S\ref{s2.8.10.15}).

\begin{prop}\label{p2.8.10.15}
Suppose that $(A,\Phi)\in \cfgC$ satisfies
$$
\Bogo(A,\Phi)=0.
$$
Then $\nabla_{\ovA}\ovPhi=0$.
\end{prop}

\begin{proof} We have seen that $*F_A = \cO(\rho^2)$ whereas
$\nabla_A\Phi=\cO(\rho)$, as sections of $\scT^*\otimes \ad(P)$.   By
computations similar to those above, the coefficient of $\rho$ in
$\nabla_A\Phi$ is just $\nabla_{\ovA}\ovPhi$, under the map
$$
C^\infty(X,T^*X) \to \rho C^\infty(X,\scT^*X)
$$
which regards an ordinary $1$-form as a scattering $1$-form
which vanishes at the boundary.  Thus the conclusion
of the Proposition follows from the weaker assumption
$$
\Bogo(A,\Phi) \in \rho^2C^\infty(X,\scT^*X\otimes \ad(P)).
\qedhere
$$
\end{proof}

\noindent
From now on we take $G = \SU(2)$.

\begin{defn}
The {\em mass}, $m$, of the monopole $(A,\Phi)$ is the value of
$|\ovPhi|$.  If $m>0$, the
{\em charge}, or {\em monopole number}, $k$, of $(A,\Phi)$ is the
degree of the map $\ovPhi:\pa X \to \su(2)$.
\end{defn}

Proposition~\ref{p2.8.10.15} implies that $|\ovPhi|$ is constant. Hence the mass is
well-defined. By trivializing $P$ over $\pa X$, $\ovPhi$ becomes a map
into the sphere of radius $m$ in $\su(2)$. The degree of this map is
independent of the trivialization, so the charge $k$ is also well-defined.

From now on we assume $m>0$, and usually $m=1$.  The latter is no loss
of generality in euclidean space, for if $*_c$ denotes the Hodge star
operator with respect to the rescaled euclidean metric $c^2\eta$, we
have $*_c F(A) = c *F(A)$ and so if $(A,\Phi)$ is a monopole of mass
$m$, then $(A, c\Phi)$ is a monopole of mass $mc$ with respect to
$c^2\eta$.  If $m=0$, then $A$ is flat and $\Phi=0$ \cite[Theorem 10.3]{JT}.

We now come to the definition of the unframed moduli space.

\begin{defn}
Fix an $\SU(2)$-bundle $P$ over $X$, a positive mass $m$ and a
non-negative integer $k$.  Then the moduli space of monopoles of mass
$m$ and charge $k$ is defined to be the quotient
\begin{equation}\label{e2.26.7.15}
\unframe_{k,m} = \{(A,\Phi)\in \cfgC_{k,m}(X,P) : *F(A) = \nabla_A\Phi\}/\gauG.
\end{equation}
\end{defn}

We note that the boundary regularity we have built in means that our
configuration space $\cfgC$ is smaller than the usual one where
smoothness over $\RR^3$ is assumed along with the finiteness of the
action \eqref{e4.16.5.15}.   Taubes \cite[Theorem 10.5, p.\ 157]{JT} has
however proved that any finite-action solution of the Bogomolny
equations extends, in suitable gauges, to the radial compactification
$X$ of $\RR^3$.  Although the results explicitly given there give only
uniform bounds on $F_A$ and $\nabla_A\Phi$, bounds on their derivates
were obtained in Taubes's PhD thesis, and from such bounds it is
possible to prove that there are gauges in which a monopole $(A,\Phi)$
extends smoothly to $X$. Thus our `smooth' definition is equivalent to the
usual one.

It is also well known that $\unframe_{k,m}$ is a smooth manifold of real
dimension $4k-1$.  The reader is referred to \cite{AH} for more
background on $\cN_k$.  For $k=1$, every solution is, up to
translations of $\RR^3$, gauge equivalent to the spherically symmetric
solution found by  't Hooft and Polyakov \cite{tHooft,Polyakov}
Thus $\unframe_{1,m} = \RR^3$.

\subsection{The boundary connection and symmetry-breaking}

We saw in Proposition~\ref{p2.8.10.15} that the boundary values
$(\ovA, \ovPhi)$ of a smooth monopole satisfy $\nabla_{\ovA}\ovPhi=0$
on $\pa X$ and in fact this conclusion follows from the weaker
condition $\Bogo(A,\Phi) = \cO(\rho^2)$.  In this section we take this
discussion one stage further to determine $\ovA$ up to gauge.

Before we go ahead, consider the action of $\Phi$ on the
complexification $\adP := \ad(P)_\CC$ of $\ad(P)$.  At
every point this action is diagonalizable and at points where
$\Phi\neq 0$ there are precisely three distinct eigenvalues, $0$ and
$\pm i|\Phi|$. (We are using $G= \SU(2)$ here.)   Accordingly, at all
points of $X$ with $\Phi\neq 0$, and in particular in a neighbourhood
of $\pa X$, we have a decomposition $\adP = \frp_0 \oplus \frp_1$
where $\frp_0$ is the bundle annihilated by $\ad(\Phi)$, while
$\frp_1$ is its orthogonal complement.

Returning to monopoles, $\Phi\neq 0$ at infinity gives a reduction of
the symmetry group from $\SU(2)$ to the $\UU(1)$ subgroup stabilizing $\Phi$. 

Now if $(A,\Phi)$ is a solution of the Bogomolny equations, Taubes
\cite[\S IV.10, Theorem~10.5]{JT} has
proved not only the boundary regularity (in suitable gauges) but
also {\em rapid off-diagonal decay},
\begin{equation}\label{e11.8.10.15}
\Ad(\Phi)F_A = \cO(\rho^\infty),\mbox{ or equivalently }
\Ad(\Phi)\nabla_A\Phi = \cO(\rho^\infty).
\end{equation}
From this it follows that there exist local gauges in which $\Phi$ is
diagonal and $A$ is represented by a connection $1$-form which is
diagonal modulo terms of order $\rho^\infty$.

From our point of view, \eqref{e11.8.10.15} can be proved by an
iterative argument which 
combines the Bogomolny equations with the assumption that all data
are smooth up to the boundary. We shall use \eqref{e11.8.10.15} to
simplify the formal construction in \S\ref{S:Mgl}.  We shall not give
a complete proof of \eqref{e11.8.10.15} here, but we shall prove the
result to the next order.  At the same time we shall determine the
curvature of $\ovA$, which will also be important in the definition of
the framed moduli space.

Let $p\in \pa X$ be any point, and let $U$ be a product neighbourhood
of $p$, with coordinates $(x,y)$, $0\leq x < \delta$.   We may choose
`boundary radial gauge' to write $\nabla_A$ in the form
\begin{equation}\label{e21.8.10.15}
\nabla_A = \ol{\nabla} + \rd x \otimes \pa_x + xb + \cO(x^2)
\end{equation}
where $b(y,\rd y)$ is an $\ad(P)$-valued $1$-form on $\pa X$ and we've
written $\ol{\nabla}$ for $\nabla_{\ovA}$. 

In this gauge, we may expand $\Phi$ in the form
\begin{equation}\label{e22.8.10.15}
\Phi = \ovPhi + x\Phi_1(y) + \cO(x^2).
\end{equation}

\begin{prop}\label{p3.8.10.15}
Suppose that $(A,\Phi)$ have the above expressions in $U$ and satisfy
$*F(A) = \nabla_A\Phi$. Then we have
\begin{equation}\label{e11.26.7.15}
\ol{\nabla}\, \ovPhi = 0, \mbox{ and } \ol{\nabla}\,\Phi_1 = 0.
\end{equation}
We also have
\begin{equation}\label{e23.8.10.15}
\ad(\Phi_0)b = 0, \quad \ad(\Phi_0)\Phi_1 =0
\end{equation}
and
\begin{equation}\label{e24.8.10.15}
\dagger \ol{F} = \Phi_1,
\end{equation}
where $\dagger$ is the Hodge star operator with respect to the metric of
the round unit sphere $\pa X$. 
\end{prop}
\begin{proof}
It is useful to keep in mind that with respect to the euclidean
metric,
\begin{equation}
|\rd x| = \cO(x^2),\;\; |\rd y| = \cO(x).
\end{equation}
Then we compute
\begin{eqnarray}
\nabla_A\Phi &=& (\ol{\nabla} + \rd x\otimes \pa_x + xb)
(\ovPhi + x\Phi_1 +\cO(x^2)) \nonumber\\
&=& \ol{\nabla}\,\ovPhi  \label{e13.26.7.15}\\
&+& x (\ol{\nabla}\, \Phi_1 + [b,\ovPhi]) \label{e14.26.7.15} \\
&+& \rd x(\Phi_1 + [a,\ovPhi]) \label{e15.26.7.15} \\
&+& \cO(x^3).
\end{eqnarray}
We have set the terms out so that \eqref{e13.26.7.15} is $\cO(x)$, while
\eqref{e14.26.7.15} is $\cO(x^2)$ and tangential, while
\eqref{e15.26.7.15} is $\cO(x^2)$ and normal to $\pa X$.
The expansion to $\cO(x^3)$ of $F_A$ is much simpler:
\begin{equation}\label{e25.8.10.15}
F_A = \ol{F} + \cO(x^3)
\end{equation}
so 
\begin{equation}\label{e1.21.9.15}
*F_A = \dagger \ol{ F}\,\rd x + \cO(x^3).
\end{equation}
We now simply equate coefficients. At $\cO(x)$ we recover
Proposition~\ref{p2.8.10.15}, which is the first of \eqref{e11.26.7.15}.
The normal component at $\cO(x^2)$ gives
\begin{equation}
\dagger \ol{F} = \Phi_1.
\end{equation}

The tangential component at $\cO(x^2)$ gives
\begin{equation}
\ol{\nabla}\,\Phi_1 + [b,\ovPhi] = 0.
\end{equation}
But the first term is a multiple of $\ovPhi$, while the second is
orthogonal to it. Hence both terms vanish, completing the proof.
\end{proof}

\subsection{Framed moduli space}
\label{S:bogoscat_framed}
We now define the framed moduli space $\monM_k$.  This will be a
$\UU(1)$ bundle over $\unframe_k$ and is in many ways the more natural
object.  Motivated by Propositions~\ref{p2.8.10.15} and
\ref{p3.8.10.15} we make the following definition:
\begin{defn}
A pair $(\ovA,\ovPhi)$ over $\pa X$ is called {\em admissible} if
\begin{equation}\label{e1.5.8.15}
\ol{\nabla}\,\ovPhi=0,\; \ovPhi\neq 0\mbox{ and }\ol{\nabla}\,\dagger\!\ol{F} =0, 
\end{equation}
where as above we have written $\ol{\nabla}$ for $\nabla_{\ovA}$
and $\ol{F}$ for $F_{\ovA}$.
We call the pair admissible {\em charge-$k$} boundary data if
\begin{equation}\label{e31.8.10.15}
\dagger\! \ol{F} = \frac{k}{2m}\ovPhi.
\end{equation}
where $m = |\ovPhi|$.
\end{defn}

\begin{rmk}
Let $\ol{E} \to \pa X$ be the complex vector bundle associated by the
fundamental representation of $\SU(2)$. At each point $p$ of $\pa X$, we have the eigenspaces $(L_{\pm})_p$ of
$\ovPhi_p$, viewed as an endomorphism of $\ol{E}_p$.  By the first two
conditions of \eqref{e1.5.8.15}, these eigenspaces patch together to
form a pair of complex line bundles $L_{\pm}$ over $\pa X$. Then
$\Phi$ acts as multiplication by $\pm im$ on $L_{\pm}$ and $\ovA$
preserves each of $L_{\pm}$.  Thus there are $\UU(1)$ connections
$\nabla_{\pm}$ on $L_{\pm}$ such that 
$$
\ol{\nabla} = \diagl(\nabla_+, \nabla_-)
$$
with respect to the isomorphism $E = L \oplus L^{-1}$, with $\ovPhi =
\diagl(im,-im)$.  Then the curvature form of $\ol{\nabla}$ takes the
form
$$
\ol{F} = \diagl(f_+, f_-)(\dagger 1)
$$
where $f_{\pm}$ are imaginary functions on $\pa X$. In fact, $L_+$ and
$L_-$ are mutually dual, the connections $\nabla_{\pm}$ respect this
and so $f_- = - f_+$. Since $\ol{\nabla}$ acts as $\rd$ on the
diagonal components of any endomorphism (these components being
sections of a canonically trivial line bundle), the third part of
\eqref{e1.5.8.15} gives $\rd f_{\pm} = 0$, so $f_{\pm} = \pm
i\lambda$, for some real constant $\lambda$.  By Chern--Weil theory,
$\lambda = -\frac{1}{2}c_1(L_+)$, where we have identified the $c_1$
with an integer via the fundamental class of $\pa X$.   This gives
\eqref{e31.8.10.15} if $c_1(L_{\pm}) = \mp k$.

The arrangement of signs comes from a standard computation which we
recall here. We have 
$$
(F_A - *\nabla_A\Phi)\wedge *(F_A -*\nabla_A\Phi) = F_A\wedge *F_A +
\nabla_A\Phi \wedge *\nabla_A\Phi - 2F_A\wedge \nabla_A\Phi.
$$
Taking $(-1/2)\tr$ of both sides and imposing the Bogomolny equations,
$$
\|F_A\|^2 + \|\nabla_A\Phi\|^2 = -\int\rd \tr(\Phi F_A) = - \int_{\pa
  X}\tr(\ovPhi \ol{F}).
$$
If $\ovPhi = \diagl(im,-im)$, $\dagger\ol{F} = (ik/2, - ik/2)$ it
follows that
\begin{equation}\label{e1.9.10.15}
\|F_A\|^2 + \|\nabla_A\Phi\|^2 = 4\pi mk
\end{equation}
and so $k>0$ as required.
\end{rmk}

Given any choice of admissible boundary data, 
define the framed configuration space
\begin{equation}\label{e3.5.8.15}
\cfgC_0(X,P) = \{(A,\Phi)\in \cfgC(X,P): (A,\Phi)|\pa X =
(\ol{A},\ol{\Phi})\},
\end{equation}
the framed gauge group
\begin{equation}\label{e4.5.8.15}
\gauG_0(X,P) = \{ \gamma \in \gauG(X,P) : \gamma|\pa X = 1\},
\end{equation}
and the framed moduli space
\begin{equation}\label{e5.5.8.15}
\monM_{k,m} = 
\{(A,\Phi) \in \cfgC_0 : \Bogo(A,\Phi) = 0\}/\gauG_0.
\end{equation}

It is known that $\monM_k$ is non-compact smooth manifold of real
dimension $4k$.   The $L^2$ metric $G_k$ from \eqref{e2.21.9.15} gives $\monM_k$
the structure of a complete hyperk\"ahler manifold.  We shall say more
about the local structure of $\monM_k$ in the next subsection.

We note also that there is a free isometric action of $\UU(1)$ on
$\monM_k$, with $\monM_k/\UU(1) = \unframe_k$.  This action is given
explicitly by using $\Phi$ as a gauge transformation:
\begin{equation}\label{e6.5.8.15}
t\cdot(A,\Phi) = e^{t\Phi}(A,\Phi),\; t\in \RR;
\end{equation}
the point is that the boundary data $(\ovA,\ovPhi)$ are fixed by this
action.  If $0 < t< 2\pi/m$, then $e^{t\Phi}$ is not in $\gauG_0$ so
\eqref{e6.5.8.15} is a non-trivial action.  The action is periodic,
however, of period $2\pi/m$ and so \eqref{e6.5.8.15} defines an action
of the circle $\RR/(2\pi \ZZ/m)$ on $\monM_{k,m}$.  It is clear that
the quotient is $\unframe_k$.

\begin{rmk}
Although our definition of $\monM_k$ depends on a choice of admissible
boundary data, any two choices lead to the same moduli space.  To be
more precise, let $\cfgC_0'$ and $\monM_k'$ be respectively the
configuration space and the moduli space defined by replacing
$(\ovA,\ovPhi)$ by a different choice $(\ovA',\ovPhi')$ of admissible
boundary data.  Then we claim that there is a principal $\UU(1)$-set
of natural identifications $\gamma_t : \monM_k \simeq \monM'_k$,
$t\in \RR/(2\pi\ZZ/m)$, the
$\UU(1)$ action being given by composition with the above
$\UU(1)$-action on either of $\monM_k$ or $\monM_k'$.

The main point is that since $\ol{A}'$ and $\ol{A}$ have the same
curvature and the base is simply connected, $\ol{A}$ and $\ol{A}'$ are
gauge-equivalent by some gauge transformation $\ol{g}$.  Clearly
$\ol{g}$ is unique up to composition with
gauge transformations which preserve $\ol{A}$, and this group is isomorphic
to the $\UU(1)$ generated by $\exp(t\ovPhi)$.   If $g \in \gauG$ is
any extension of $\ol{g}$, then
$$
(A,\Phi) \mapsto g(A,\Phi)
$$
defines a diffeomorphism from $\cfgC_0$ to $\cfgC_0'$ and this induces
the desired identification of $\monM_k$ with $\monM_k'$.
\end{rmk}

\subsection{The local structure of $\monM_k$}
\label{S:local}
Let $(A,\Phi)$ be a solution of the Bogomolny equations representing a
point $m$ of $\monM_k$. The infinitesimal structure of $\monM_k$ is
described by the {\em deformation complex}
\begin{equation}\label{e1.10.10.15}
0 \to \Lie(\gauG_0) \stackrel{\gauD}{\to} T_{(A,\Phi)}\cfgC_0 \stackrel{D\Bogo}{\to}
C^\infty(X,\scT^*\otimes\ad(P)) \to 0.
\end{equation}
Here
\begin{eqnarray}
\Lie(\gauG_0) &=&\rho C^\infty(X,\ad(P)), \label{e2.10.10.15}\\
T_{(A,\Phi)}\cfgC_0 & = &
\rho C^\infty(X,\ad(P))\oplus
\rho^2C^\infty(X,\scT^*\otimes \ad(P))  \label{e3.10.10.15}
\end{eqnarray}
and 
\begin{equation}
\gauD_{(A,\Phi)}(\xi) = (-\rd_A u, - \ad(\Phi)u).
\end{equation}
The cohomology in the middle of \eqref{e1.10.10.15} is the
(virtual) tangent space of $\monM_k$ at $m$, and may be identified with the true tangent
space if $D\Bogo$ is surjective and the action of $\gauG_0$ is free.
Fortunately these conditions are always satisfied for monopoles.

As in the introduction, the linearization $D\Bogo$ of the Bogomolny
equation combines with the Coulomb gauge fixing operator $\gauD^*_{(A,\Phi)}$ to
give the operator $L_{(A,\Phi)}$ of \eqref{e4.21.9.15}.  One shows
\cite{kottke2015dimension} that the cohomology in the middle of \eqref{e1.10.10.15} is
  isomorphic to the null-space of $L_{(A,\Phi)}$,
\begin{equation}\label{e5.10.10.15}
L_{(A,\Phi)}:
\rho^2 C^\infty(X,\Lambda\otimes\ad(P)) \to
\rho^2 C^\infty(X,\Lambda\otimes\ad(P)),
\end{equation}
where $\Lambda = \scLam^1\oplus \scLam^0$.
\begin{rmk}
The reason why the domain is different from \eqref{e3.10.10.15} is
that the coefficient of $\rho$ in $\phi$ is fixed by
Proposition~\ref{p2.8.10.15} (by the Bogomolny equations).  In
particular any solution of the linearized Bogomolny equations in the
domain \eqref{e3.10.10.15} is actually $\cO(\rho^2)$.
\end{rmk}
By a slight abuse of notation, set
\begin{equation}\label{e6.10.10.15}
T_{(A,\Phi)}\monM_k = N(L_{A,\Phi}),
\end{equation}
the null-space of $L_{A,\Phi}$ with domain \eqref{e5.10.10.15}.

One obtains the dimension of $\monM_k$ and its smoothness by combining
an index theorem for \eqref{e5.10.10.15} with the Weitzenbock formula
\begin{equation}\label{e7.10.10.15}
L_{A,\Phi}L^*_{A,\Phi} = \nabla_A^*\nabla_A + \ad(\Phi^*)\ad(\Phi).
\end{equation}
This formula is valid whenever the base metric is Ricci-flat and $\Bogo(A,\Phi)=0$.  The
index has been studied by Taubes \cite{taubes1983stability} and Kottke \cite{kottke2015dimension} and the result is
that there is a Fredholm framework for $L_{A,\Phi}$ making it 
surjective (between spaces which are suitable completions of those in
\eqref{e5.10.10.15}) with index $4k$.

The same analytic framework allows us to prove that $\monM_k$ is a
smooth manifold.  One proves first that for all sufficiently small
$(\widetilde{a},\widetilde{\phi})\in \cfgC_0$, there is an $(a,\phi)$
gauge related to $(\widetilde{a},\widetilde{\phi})$ in Coulomb gauge
with respect to $(A,\Phi)$,
\begin{equation}
\gauD^*_{(A,\Phi)}(a,\phi) =0.
\end{equation}
Hence any nearby point $m'\in \monM_k$ is represented by
$(A+a,\Phi+\phi)$ where
\begin{equation}\label{e11.10.10.15}
L_{(A,\Phi)}(a,\phi) = - Q(a,\phi),
\end{equation}
and 
\begin{equation}\label{e12.10.10.15}
Q(a,\phi) = *[a,a] - [a,\phi]
\end{equation}
is the nonlinear part of $\Bogo$.  It follows from the surjectivity of
$L_{(A,\Phi)}$ and the implicit function theorem that the set of small
solutions to \eqref{e11.10.10.15} is the graph of a nonlinear operator
from $T_m\monM_k$ to its $L^2$-orthogonal complement in 
$\rho^2C^\infty(X,\Lambda\otimes\ad(P))$.

\subsection{The centre of a monopole}
\label{s1.29.10.15}
The unframed moduli space for monopoles of charge $1$ is diffeomorphic
to $\RR^3$; the interepretation of this is that given the 't Hooft
monopole $m_0$ centred at the origin of $\RR^3$, every other monopole
of charge $1$ is a translation of $m_0$ by some vector in $\RR^3$.

For $k>0$, one cannot generally distinguish $k$ individual charge-1
monopoles with definite centres, but every $m\in \monM_k$ does have a
well-defined centre.  To explain the definition, revert to euclidean
coordinates $z\in \RR^3$, and observe that the Higgs field has an
expansion of the form
\begin{equation}\label{e25.10.10.15}
\Phi = \begin{bmatrix} i & 0 \\ 0 & - i \end{bmatrix}
\left( m - \frac{k}{2|z|} - \frac{k}{2}\frac{v\cdot z}{|z|^3} + \cO(|z|^{-3}) \right)
\end{equation}
relative to the decomposition of $\frp = \frp_0\oplus \frp_1$.  Here $v \in \RR^3$ is some vector.

\begin{defn}
For $m\in \monM_k$ with $\Phi$ given by \eqref{e25.10.10.15}, the {\em
  centre} of $m$ is defined to be the vector $v/m$ (relative to the
  origin of the $z$ coordinates). The moduli space of monopoles with
  centre at $0$ is denoted $\monM^c_k$.
\end{defn} 

The definition can be motivated by considering the change in $v$ when
we translate the monopole.  More precisely, if $c\in \RR^3$, consider
$(A_c,\Phi_c)$, the result of pulling back by the translation
$z\mapsto z +c$. Then the expansion of $\Phi_c$ is
\begin{equation}
\Phi_c(z) = \Phi(z-c) = 
\begin{bmatrix} i & 0 \\ 0 & - i \end{bmatrix}
\left( m - \frac{k}{2|z|} - \frac{k}{2}\frac{(v+c)\cdot z}{|z|^3} + \cO(|z|^{-3}) \right)
\end{equation}
Thus our definition is consistent with the translation action on
$\RR^3$, for the centre of $m$ should be translated by $c$ if the
whole monopole is translated by $c$.

\begin{rmk}
We note without proof that the $\UU(1)$ action on $\monM_k$ is
triholomorphic. Viewing the associated hyperk\"ahler moment map $\mu$
as a map to $\RR^3$, $\mu(m)$ is equal to the centre of $m$.
\end{rmk}

\begin{rmk}
From the previous remark $\monM_k^0 := \monM_k^c/\UU(1)$ is a hyperk\"ahler
quotient of $\monM_k$, hence a hyperk\"ahler manifold of dimension
$4k-4$.  As discussed in \cite[p.\ 20]{AH}, its universal cover
$\wt{\monM}_k^0$ is a $k$-fold cover, called the moduli space of
strongly centred monopoles.  This appears in the decomposition 
\begin{equation}
\wt{\monM}_k = \RR^3\times S^1 \times \wt{\monM}_k^0
\end{equation}
as a riemannian product \cite[p.\ 21]{AH}. We shall discuss this at the infinitesimal
level in the next section.
\end{rmk}

\subsection{Structure of the tangent space}
\label{s11.25.11.15}
Although we shall not make great use of it, recall that $T_m\monM_k$,
defined as in \eqref{e6.10.10.15}, has a quaternionic structure.   For
this, pick any oriented orthonormal basis $(e_1,e_2,e_3)$ of $\RR^3$, identify
it with the corresponding frame for $T^*\RR^3$,
and write any element of $T_m\monM_k$ in the form
$$
a_0 = \phi, \ a = \sum a_je_j
$$
where $a_0,\ldots,a_3 \in \rho^2C^\infty(\bbR^3; \ad(P))$. This decomposition
identifies the domain of $L_{(A,\Phi)}$ in \eqref{e5.10.10.15} with 
$\rho^2C^\infty(\bbR^3; \ad(P))\otimes_{\RR}\RR^4$ and identifying $\RR^4$
with the quaternions we have endowed this domain with the structure of
a quaternionic vector space. One checks that $L_{(A,\Phi)}$ commutes
with this quaternionic structure, and it follows that $T_m\monM_k =
N(L_{(A,\Phi)})$ inherits a quaternionic structure.

Our next task is to describe the subspace of $T_m\monM_k$
corresponding to the action of translations or $\RR^3$ and changes of
framing. This will turn out to be a quaternionic subspace of
$T_m\monM_k$, and its orthogonal complement is identifiable with the
elements of $T_m\monM_k$ which are $\cO(\rho^3)$ at $\pa X$.  The
decomposition is the infinitesimal version of the riemannian splitting
of the universal cover of $\monM_k$ as a product of $\RR^3\times S^1$
with the moduli space of strongly centred monopoles \cite{AH}.

For any infinitesimal gauge transformation $\xi \in \Lie(\gauG)$, the
corresponding tangent vector at $(A,\Phi)$ in $\cfgC$ is
$\gauD_{A,\Phi}\xi$. Taking $\xi = \Phi$, we get the tangent vector
\begin{equation}\label{e21.10.10.15}
\tau_0 = ( - \nabla_A \Phi, 0)
\end{equation}
which is easily verified to lie in $T_m\monM_k$. (One has to check
only that $\gauD^*_{A,\Phi}\tau_0=0$ and that $\tau_0 = \cO(\rho^2)$.)
For an orthonormal basis $e_i$ of $\RR^3$, set
\begin{equation}
\tau_i = ( \iota_{e_i}F_A , \nabla_{e_i} \Phi);
\end{equation}
this is the element of $T_{(A,\Phi)}\monM_k$ corresponding to
the infinitesimal translation $z \mapsto z + te_i$ of $(A,\Phi)$.
Note that this is nothing other than the derivative of $(A,\Phi)$ with respect
to $e_i$, using the connection $A$ itself to lift $e_i$ to $P$.

We note that the $\tau_a$, $a=0,1,2,3$ span a quaternionic subspace of
$T_{A,\Phi}\monM_k$: $\tau_i$ is obtained from $\tau_0$ by applying the
$i$-th complex structure.   We leave the verification that
$L_{A,\Phi}\tau_a=0$ to the reader.

\begin{prop}\label{p2.23.11.15}
Let $(A,\Phi)$ represent an element of $\monM_k$.  There is an
orthogonal direct-sum decomposition
\begin{equation}\label{e1.6.8.15}
T_m\monM_k = \langle \tau_0,\ldots,\tau_3\rangle \oplus T'_m\monM_k
\end{equation}
where
\begin{equation}\label{e3.6.8.15}
T_m'\monM_k
= \{u \in C^\infty(X,\Lambda\otimes \ad(P)) : |u| = \cO(\rho^3)\}
\end{equation}
Moreover, if
\begin{equation}\label{e4.6.8.15}
v = \sum \lambda_a\tau_a, 
\end{equation}
we have
\begin{equation}\label{e5.6.8.15}
\|v\|^2_{L^2} = 2\pi m k(\lambda_0^2 + \cdots + \lambda_3^2).
\end{equation}
\end{prop}
\begin{rmk}
If $m \in \monM_k^c$, then $T'_m$ is identifiable with
$T_m\monM_k^0$, where we have made no distinction between $m$ and the
point it represents in $\monM_k^c/\UU(1)$.  More generally, $T'_m$ is
the subspace of infinitesimal changes in $m$ which keep its centre and
framing fixed.
\end{rmk}

\begin{proof}
We verify \eqref{e5.6.8.15} first.  For $\tau_0$, we have by \eqref{e1.9.10.15}
\begin{equation}\label{e6.6.8.15}
\|\tau_0\|^2 = \int |\nabla_A\Phi|^2 = 2\pi mk.
\end{equation}
One verifies $\|\tau_i\|^2 = 2\pi mk$ either by hand, or by noting
that $\tau_i = I_i\tau_0$, where $I_1,I_2,I_3$ are the three complex
structures on $T_{m}\monM_k$, which are isometries
  of this space.

Next, consider any solution,
\begin{equation}\label{e11.6.8.15}
L_{(A,\Phi)} u =0,\;\; |u| = \cO(\rho^2).
\end{equation}
By boundary regularity for $L_{A,\Phi}$, the off-diagonal components
of $u$ are $\cO(\rho^\infty)$, and so if $u_0$ is the diagonal $\frp_0$
component of $u$, we have
\begin{equation}\label{e12.6.8.15}
L_{A,\Phi}u_0 = \cO(\rho^\infty).
\end{equation}
The action of $L_{A,\Phi}$ on $\Lambda\otimes\frp_0$
is, up to rapidly decreasing terms at $\pa X$, that of the model operator
\begin{equation}\label{e13.6.8.15}
L = \begin{bmatrix} *\rd & - \rd \\ -\rd^* & 0 \end{bmatrix},
\end{equation}
(cf.\ Appendix~\ref{S:linear})
and so the asymptotic expansion of $u_0$ in \eqref{e12.6.8.15} is a
sum of homogeneous solutions $f$ of $Lf=0$.  Since $L^2 = \Delta$, the
Laplacian on $\Lambda$, if $f$ is $\cO(\rho^2)$, then the $4$
components of $f$ all have to be of the form $\langle
c,z\rangle/|z|^3$, for some constant $c\in \RR^3$.  Then for $Lf=0$,
we need
\begin{equation}
f_0 = \begin{bmatrix} 0 \\ z/|z|^3 \end{bmatrix},\;\;\mbox{ or }
f_c = \begin{bmatrix} \langle c,z\rangle/|z|^3
 \\ - (c\wedge z)/|z|^3 \end{bmatrix}.
\end{equation}
On the other hand, each of these solutions can be continued over $X$
as a linear combination of the $\tau_a$. It follows that we have an
exact sequence
\begin{equation}
0 \to T_{(A,\Phi)}\monM_k^c \to T_{(A,\Phi)}\monM_k \to \RR^4 \to 0
\end{equation}
where $T_{(A,\Phi)}\monM_k^c$ is as in \eqref{e3.6.8.15}.

For the orthogonality, note that we can rewrite
\begin{equation}\label{e21.6.8.15}
\sum \lambda_a\tau_a = L_{(A,\Phi)} \begin{bmatrix} \lambda_0 \Phi \\
  \ul{\lambda}\Phi\end{bmatrix}
= L^*_{(A,\Phi)}
\begin{bmatrix} \lambda_0 \Phi \\
  \ul{\lambda}\Phi\end{bmatrix}
\end{equation}
where $\ul{\lambda}$ is the euclidean vector
$(\lambda_1,\lambda_2,\lambda_3)$. 
(Similarly, the model solution $f$ can be written
\begin{equation}
f = L\begin{bmatrix} \lambda_0|z|^{-1} \\ \ul{\lambda}|z|^{-1}
\end{bmatrix}.)
\end{equation} 

Now if $v\in T_{(A,\Phi)}\monM_k^c$, we compute
\begin{equation}
(v,\sum \lambda_a\tau_a) =
\int_X \langle v, L^*_{(A,\Phi)}(\lambda \Phi)\rangle
\end{equation}
Restricting to the sphere $\rho = \delta$ 
we are left with a boundary term of the form
\begin{equation}
\int_{\rho = \delta} \langle v, \lambda\Phi\rangle
\end{equation}
because $L_{(A,\Phi)}v=0$.  Because $|v| = \cO(\rho^3)$, $|\lambda\Phi|
= \cO(1)$ but the area of $\rho = \delta$ is $4\pi \delta^{-2}$, this boundary term
goes to zero as $\delta \to 0$. This completes the proof of the proposition.
\end{proof}

%% file: dirac.tex
If we take $G=\UU(1)$, the Bogomolny equations reduce to equations
studied by Dirac in relation to ordinary magnetic monopoles.   There
are no non-trivial solutions on $\RR^3$ without singularities, but
these Dirac monopoles are also a key ingredient in the gluing theorem
that will be discussed below.  We therefore give a quick account
geared to our later applications.  In particular we give a careful
discussion of moduli spaces of Dirac monopoles with fixed
singularities.   Although the discussion can easily be extended to
other $3$-manifolds \cite{Oliveira} we shall confine ourselves to the
case of $\RR^3$.

Let $U$ be an open set of $\RR^3$. Let $Q \to U$ be a $\UU(1)$
principal bundle with connection $a$ and let $\phi$ be section of the
adjoint bundle.  Since this is canonically trivial, $\phi$ can be
identified canonically with an imaginary function on $U$.  Similarly
the curvature $F_a$ is canonically an imaginary closed $2$-form on $U$ and
$(i/2\pi)F_a$ represents $c_1(Q)$\footnote{In this section we are
  using $(a,\phi)$ for $\UU(1)$-valued monopole data, rather than
  infinitesimal deformations of $\SU(2)$-monopoles. We hope that no
  confusion will result from this}. 

Suppose that $(a,\phi)$ satisfies the Bogomolny equations
\begin{equation}\label{e1.11.10.15}
*F_a = \nabla_a\phi = \rd \phi.
\end{equation}
Since $\rd F_a = 0$ we deduce from this that $\phi$ is harmonic,
\begin{equation}\label{e2.11.10.15}
\Delta \phi = 0\mbox{ in }U.
\end{equation}

Conversely, given an imaginary harmonic function on $U$ with  {\em
  integral periods} in the sense that
\begin{equation}\label{e3.11.10.15}
\frac{i}{2\pi}\int_{\Sigma} *\rd \phi \in \ZZ\mbox{ for all }\Sigma \in H_2(U,\ZZ)
\end{equation}
there exists a pair $(Q,a)$ if a $\UU(1)$ bundle and connection with 
\begin{equation}\label{e4.11.10.15}
F_a = *\rd \phi.
\end{equation}
In this case the gauge group is $C^\infty(U, \UU(1))$ and if $\gamma$
is an element of the gauge group,
\begin{equation}
\gamma\cdot(a,\phi) = (a - (\rd \gamma)\gamma^{-1}, \phi)
\end{equation}

We shall reserve the term `Dirac monopole' for the case that $U =
\RR^3\setminus \{z_1,\ldots,z_N\}$ and $\phi$ has the simplest
possible singularities at the points $z_j$.  In keeping with the
general philosophy of this paper, we shall replace this punctured
non-compact manifold by a compact manifold with boundary
\begin{equation}\label{e6.11.10.15}
	D := [\ol{\bbR^3}; z_1,\ldots,z_N].
\end{equation}
This is the real blow-up at the points $z_j$ (assumed distinct) of the
radial compactification $X$ of $\RR^3$.
Thus $D$ has $N+1$ boundary hypersurfaces, each of which is
diffeomorphic to the $2$-sphere $\bbS^2$, and which will be denoted by $S_1,\ldots,S_N$ and
$S_\infty$.  Here $S_j$ is the boundary hypersurface introduced by
blowing up $z_j$ while $S_\infty$ is the original boundary of $X$.
For each $j$,  $S_j$ is canonically identified with the sphere bundle
of $T_{z_j}X$, $S_j = (T_{z_j}X \setminus \{0\})/\RR^+$, where $\RR^+$
is the multiplicative group of positive real numbers.  Thus $D$ is the
natural domain in which polar coordinates have been introduced at each
of the $z_j$.  Indeed, the euclidean distance $r_j$ from $z_j$ lifts
to $D$ to be a (smooth) boundary defining function for $S_j$, which we
denote without change of notation.  We continue to denote the boundary
defining function of $S_\infty$ by $\rho$; as before, this can be
taken to be the reciprocal of the euclidean distance from any given
point of $\RR^3$.
    
The Euclidean metric has the form
\begin{equation}
	\frac{\rd\rho^2}{\rho^4} + \frac{h_{\bbS^2}}{\rho^2}, 
	\quad \text{resp.} \quad dr_j^2 + r_j^2 h_{\bbS^2},
	\label{E:Dirac_metric}
\end{equation}
in a product neighborhood of $S_\infty$ or $S_j$, respectively.  Thus
it is natural to introduce a rescaled {\em conic} tangent bundle $\sccT D$,
\begin{equation}\label{e1.12.10.15}
\sccT D = (r_1\ldots r_N)^{-1} \rho \boT D.
\end{equation}
The space of all smooth sections $C^\infty(D, \sccT D)$ is the space $\sccV(D)$
of tangent vector fields of bounded length with respect to the lifted
euclidean metric; equivalently this lifted metric is a smooth and
everywhere positive-definite section of $S^2\, \sccT^* D$.  An important
point about $\sccV(D)$ is that it is not an algebra (in contrast to
$\cV_{\bo}$ and $\scV$).  It is perhaps best thought about in terms of
rescaled $\bo$-vector fields via \eqref{e1.12.10.15}.

\begin{defn} Let $D$ be as above and let $Q \to D$ be any $\UU(1)$
 principal bundle.  A {\em Dirac monopole} on $D$ is a pair
$(a,\phi)$ where $a$ is a smooth connection on $Q$, $\phi \in
C^\infty(\mathring D)$ and the Bogomolny equations \eqref{e1.11.10.15} are
satisfied in $\mathring D$. 
\end{defn}

We shall now classify Dirac monopoles in the sense of this definition.
\begin{prop}
Let $Q\to D$ and $(a,\phi)$ be a Dirac monopole configuration. Then there exists a
constant $m$ and integers $k_j$ such that
\begin{equation}\label{e3.12.10.15}
\phi = i\left(m - \sum_{j=1}^N \frac{k_j}{|z-z_j|}\right)
\end{equation}
\label{p1.13.10.15}\end{prop}
\begin{proof}
Pick any one of the $z_j$ and for simplicity write $r_j = r$, $y =
(y^1,y^2)$ for local coordinates on $S_j$.   Then $F_a$ is a smooth
linear combination of $\rd r \wedge \rd y$ and $\rd y^1\wedge \rd y^2$
and so $*F_a$ is a smooth linear combination of $\rd y$ and $r^{-2}\rd
r$.

On the other hand, $\phi$ is harmonic away from $r=0$ and so (cf.\
Appendix~\ref{S:linear})  it is a linear combination of
terms of the form $p(z-z_j)|z-z_j|^{2\nu+1}$, where $p$ is a harmonic
polynomial, homogeneous of degree $\nu$.  Given that $\pa_r \phi =
\cO(r^{-2})$, it follows that we can only have $\nu=0$, so
\begin{equation}
\phi(z) \sim \frac{\lambda_j}{|z-z_j|} + \cO(1)
\end{equation}
for some constant $\lambda_j$. For $\phi$ to have an integral period
around $S_j$, we need $\lambda_j = -i k_j/2$, $k_j$ an
integer. Repeating the argument at all of the $z_j$, we find
\begin{equation}
\phi(z) + i\sum \frac{k_j}{2|z-z_j|} = h(z)
\end{equation}
where $h$ is a bounded harmonic function on $\RR^3$.  Thus $h$ must be
a constant and $\phi$ has the given form.
\end{proof}

The $\UU(1)$-bundle $Q \to D$ is classified up to isomorphism by its
Chern class $c_1(Q) \in H^2(D; \bbZ) = \bbZ^{N}$, which is determined by the values
$k_i := c_1(Q)[S_i]$, $i = 1,\ldots,N$.  This is the significance of
the integers $k_i$ in \eqref{e3.12.10.15}. Note that we have the relation
\[
	k_\infty := c_1(Q)[S_\infty] = k_1 + \cdots + k_N.
\]
We refer to $(k_1,\ldots,k_N)$ as the {\em charge} of $Q$, and hence of the
monopole.  If $a$ is a smooth connection on $Q$ satisfying the
Bogomolny equations for some $\phi$, we call it a {\em Dirac
connection} on $Q$.  Proposition~\ref{p1.13.10.15} shows that if $a$
is a Dirac connection on $Q$, then $a$ determines a Higgs field $\phi$
uniquely up to the addition of a constant, such that $(a,\phi)$
satisfies the Bogomolny equations \eqref{e1.11.10.15}.

We now consider the framed moduli space of Dirac connections on a
given $\UU(1)$-bundle $Q$.

\begin{defn} A {\em Dirac framing} of $Q$ is a choice of boundary
connection $\ol{a}$ on $Q|\pa D$, with locally constant curvature.
Similarly if $\pa' D$ is a union of boundary components of $D$, a 
{\em partial Dirac framing} of $Q$ over $\pa'
D$ is a choice of boundary connection $\ol{a}$ with locally constant
curvature on $Q|\pa' D$.
\label{d11.25.10.15}
\end{defn}

Locally constant curvature here means that on each boundary component $S$,
the curvature $F_{\ol{a}}$ is a constant multiple of the area form on
$S$.

Denote by $\gauG_0(Q)$ the group of $\UU(1)$-gauge transformations of
$D$ which are the identity over $\pa D$.  It is convenient to
introduce the notation $\jmath: \pa D \to D$ for the boundary inclusion.

\begin{prop} Let $\ol{a}$ be a Dirac framing of $Q$.  The moduli space
\begin{equation}\label{e1.13.10.15}
\{\mbox{Dirac connections on $Q$, framed by }\ol{a}\}/\gauG_0(Q)
\end{equation}
is diffeomorphic to the principal $\UU(1)$-set
\begin{equation}\label{e2.13.10.15}
 C_1\times \cdots \times C_N
\end{equation}
where each $C_j$ is the $\UU(1)$ group of gauge transformations which
fix $\ol{a}|S_j$.
\label{p1.25.10.15}
\end{prop}
\begin{proof}  Suppose for the moment that $a$ is a Dirac connection
  with $\jmath^*(a) = \ol{a}$.  Let $b$ be any other such
  connection. By definition they have the same curvature and the same
  restriction to $\pa D$, so $\rd(b-a)=0$, $\jmath^*(a-b)=0$.  Since
  $H^1(D; i\bbR)=0$, there exists an imaginary function $u$ such that
  $b = a + \rd u$, and $\rd \jmath^*(u)=0$.  $u$ is determined up to
  the addition of a constant, and we use this to fix
  $u|S_\infty=0$.  Thus $b$ determines a collection of phases $e^{u_j}
  = e^u|S_j$; since $u$ is locally constant on $\pa D$, $e^{u_j}$ is
  constant on $S_j$.  It is clear that 
this map is a diffeomorphism between \eqref{e1.13.10.15}
and \eqref{e2.13.10.15}, given the choice of $a$.

Suppose now that $a'$ is a different basepoint in the space of Dirac
connections framed by $\ol{a}$. Then by the argument just used, there
is a function $v$ such that $a' = a + \rd v$, $v|S_\infty=0$.  Then
$b-a' = \rd (u - v)$
so the phases $e^{u_j}$ are replaced by $e^{u_j-v_j}$ where $v_j = v|S_j$.

It remains only to verify that there does exist a Dirac connection $a$
with $\jmath^*(a) = \ol{a}$.  But this is straightforward:  smooth
Dirac connections on $Q$ do exist,
by reversing the argument of Proposition~\ref{p1.13.10.15}, using the
smoothness of $*\rd \phi$.  Let $a_0$ be any such Dirac connection
and define $\ol{a}_0 = \jmath^*(a_0)$. Then $\ol{a}_0$ is a Dirac
framing of $\pa D$.  But any two Dirac framings are gauge-equivalent
because they have the same curvature and $\pa D$ is simply connected
(cf.\ \S\ref{S:bogoscat_framed}.)  Extend this gauge transformation
smoothly from the boundary, and call this $\kappa$.  Then
$\kappa^*(a_0)$ is a Dirac connection on $D$ and its restriction to
the boundary is now $\ol{a}$.
\end{proof}

If instead we have a partial Dirac framing over $\pa' D$, a union of 
$m$ connected components of the boundary, then the same argument gives
the moduli space of partially framed Dirac connections on $Q$ as being $(S^1)^m/\UU(1)$.

\begin{prop}  Let $I \subset \{1,\ldots,N\}$. Then the moduli 
space of Dirac connections framed at
$
\bigcup_{i\in I} S_i \cup S_\infty
$
is a principal $\UU(1)^{\abs{I}}$-space, consisting of the product
$\prod_{i \in I} C_i$.
\label{C:Dirac_moduli}
\end{prop}

In particular,
\begin{cor}
Let $Q \to D$ be a $\UU(1)$ bundle, with a given Dirac framing over
$S_\infty$ and set
\begin{equation}
\gauG_\infty(Q) = \{\gamma \in \gauG(Q): \gamma|S_\infty = 1\}.
\end{equation}
Then the moduli space
\begin{equation}
\{\mbox{Dirac connections $a$ on $Q$, framed over $S_\infty$}\}
/\gauG_\infty(Q)
\end{equation}
is a point.
\label{L:Dirac_moduli}
\end{cor}

\begin{rmk}  In view of this Corollary, we see that the moduli space
  of Dirac connections on $Q$
without any framing should be regarded as the space $\set \ast / \UU(1)$,
with formal dimension $-1$.  For this reason, our spaces of Dirac
monopoles will always be framed at least over $S_\infty$.
\end{rmk}
In \S\ref{S:ideal} we will let the points $(z_1,\ldots,z_N)$ vary and
consider the larger moduli space with these points as well as the framings as
moduli.

\subsubsection{Dirac $\SU(2)$ connections}
\label{s1.15.10.15}
In the next section we shall need to consider $\SU(2)$ connections
built from Dirac connections in the following trivial way.  Let $Q \to
D$ be a $\UU(1)$ bundle with a Dirac connection $a$.  Let $\imath :
\UU(1) \to \SU(2)$ be a given embedding. 
Then the $\SU(2)$ bundle
\begin{equation}
P = Q \times_{\UU(1)} \SU(2)
\end{equation}
is called a Dirac $\SU(2)$ bundle and the connection on $P$ induced by
$\imath$ from $a$ is called a Dirac $\SU(2)$ connection. 

%% file: single.tex
In this section we make a start on proving our first main
result,
Theorem~\ref{T:main_one}.   So fix a
configuration
\[
	\ul \zeta = (\zeta_1,\ldots,\zeta_N),
	\quad \zeta_i \neq \zeta_j \neq 0,
	\quad \sum_i \abs{\zeta_i}^2 = 1
\]
of nonzero points in $\bbR^3$ up to scaling, and a collection of
monopoles $(A_j,\Phi_j)$, where $(A_0,\Phi_0)$ represents a point of
$\monM_{k_0}$ and for $j=1,\ldots, N$, $(A_k,\Phi_k)$ represents a
point of $\moncM_{k_j}$.   Here $k_0\geq 0$ and $k_j\geq 1$ for
$j=1,\ldots,N$.  Given such data, we shall construct a 1-parameter
family of  $(A(\ve),\Phi(\ve)) \in \cfgC_{k}$ for $0< \ve <\ve_0$
which are very good approximate solutions to the Bogomolny equations,
\begin{equation}\label{e1.18.10.15}
\Bogo(A(\ve),\Phi(\ve)) = \cO(\ve^\infty)
\end{equation}
We shall show how to solve away the error in
\eqref{e1.18.10.15} in \S\ref{S:global}.

The construction of $(A(\ve), \Phi(\ve))$ given here takes place on a
manifold with corners, $Z$, referred to as the `gluing space',
equipped with a map $\fb : Z \to [0,\infty)$. For positive $\ve$, the fibre
$\fb^{-1}(\ve)$, is canonically a copy of $X = \ol{\RR^3}$, but
$\fb^{-1}(0)$ is a more complicated manifold with `normal
crossings'.  In our construction $(A(\ve),\Phi(\ve))$ will be the
restriction to $\fb^{-1}(\ve)$ of a monopole configuration $(A,\Phi)$ on $Z$ which is
smooth in the interior and has only conormal singularities at the
union of boundary hypersurfaces $\fb^{-1}(0)$.

\subsection{Gluing space} \label{S:gluing}

The gluing space $\fZ$, which will support $(A(\ve),\Phi(\ve))$, is constructed in two steps.
Let $\oX = \ol{\bbR^3}$ and $\fZ_0 = \oX \times [0,\infty)_\ve$. Define
\[
	\fZ_1 := [\fZ_0; \pa \oX \times \set 0],
\]
the real blow-up (cf.\ \cite{RBMgreenbook}) of $\fZ_0$ in the corner $\oX\times \set 0$.
The new boundary hypersurface, denoted by $\fD_1$, is by definition
the projectivization of the inward pointing normal bundle of $\pa \oX \times \set 0$ in $\fZ_0$,
hence diffeomorphic to $\bbS^2 \times [0,\infty]_s$, where $s = \ve/x$ is the ratio 
of $\ve$ and a fixed boundary defining function $x$ for $\oX$. Over the interior of $\fD_1$, 
there is a natural Euclidean coordinate $\zeta = \ve z : \mathring \fD_1 \cong \bbR^3 \setminus \set 0$,
where $z$ is the Euclidean coordinate on $\mathring\oX$.

The paths $z = \zeta_j/\ve$ in $\mathring \fZ_1 \cong \bbR^3 \times (0,\infty)$ have well-defined 
limits at $\fD_1$; these are simply the points $\zeta = \zeta_j \in \mathring \fD_1$.
The {\em gluing space} is defined by blowing up these points:
\[
	\fZ := [\fZ_1; \set{\zeta_1,\ldots,\zeta_N}].
\]
We denote the hypersurface which arises when $\zeta_j$ is blown up by $\fX_j$, $j = 1,\ldots,
N$, and denote 
lift of $\fD_1$ by $\fD$. Observe that $\fD_1 \cong \pa \oX \times [0,\infty]$ may be identified with 
$[\ol{\bbR^3}; 0]$, and then $\fD \cong [\ol{\bbR^3}; \zeta_0,\ldots,\zeta_N]$, where $\zeta_0 = 0$.
The lift of the original faces $\oX \times \set 0$ and $\pa \oX \times
[0,\infty)$ of $\fZ_0$ are denoted by $\fX_0$ and $\fB$, respectively. We set $\fS_j = \fX_j \cap \fD$
for $j = 0,\ldots,N$ and $\fS_\infty = \fB \cap \fD$.

Define maps
\[
	\fb : \fZ \to [0,\infty),
	\quad \piX : \fZ \to \oX,
\]
by the composite of the blow-down maps to $\fZ_0$ with the projection to $[0,\infty)$ and
$\oX$, respectively. These are easily seen to be b-fibrations \cite{CCN}.

Boundary defining functions will be denoted by $\rho_j$ for $\fX_j$, $j = 0,\ldots,N$,
and by $\rho_\fD$ and $\rho_\fB$ for the hypersurfaces $\fD$ and $\fB$. We will
sometimes write $\rho_\fX = \rho_{0}\cdots \rho_{N}$ for the product of the defining functions
for all the $\fX_j$. We confuse $\ve$ with its pull back to $\fZ$ and assume the $\rho$'s are defined
so that
\[
	\ve = \rho_\fD \rho_\fX = \rho_\fD \rho_0\cdots \rho_N.
\]
As the projectivization of the inward pointing normal bundles to the $\zeta_j$,
the faces $\fX_j$ have the structure of radially compactified affine spaces of
real dimension $3$. This can be understood geometrically as follows. 
If the path $z = \zeta_j/\ve$ in $\mathring \fZ_0$ defined above is modified by
\begin{equation}\label{e1.24.10.15}
	z = \zeta_j/\ve + v
\end{equation}
for a fixed vector in $\bbR^3$, then these two paths have
the same limits $\zeta = \zeta_j$ in $\fZ_1$ but their lifts
have distinct limits on the face $\fX_j$ in $\fZ$. 
Conversely, any two distinct points on $\fX_j$ are the endpoints of some pair
of paths which differ from one another by a fixed vector in $\bbR^3$
for $\ve > 0$.   Since a non-zero normal vector at $\zeta_j$
determines a point in $X_j$, the tangent vector to the lifted curve
$z(\ve) = \zeta_j/\ve$ singles out a point of $\fX_j$ and this gives its interior
the structure of a vector space rather than just an affine space.
Thus \eqref{e1.24.10.15} gives an identification
\begin{equation}\label{e1.22.10.15}
	\fX_j \cong \ol{\bbR^3}
\end{equation}
for each $j$.

\subsection{Metric structure} \label{S:gluing_metric}

The metric structure we consider on $\fZ$ is induced by the pullback $\piX^\ast g$
of the Euclidean metric on $\oX$. In order to interpet this correctly, we need to 
define the appropriate rescaled tangent bundle on $\fZ$.

We start with the bundle $\bT \fZ$, defined so that the space of smooth
sections $C^\infty(\fZ; \bT \fZ)$ is equal to the space of vector fields
$\bV(\fZ)$ which are tangent to all boundary hypersurfaces. The subspace of
{\em vertical vector fields}
\[
	\fbV(\fZ) \subset \bV(\fZ)
\]
consists of those vector fields which additionally tangent to the fibers of $\fb$. The
further subspace of {\em gluing vector fields} is
\[
	\gV(\fZ) = \rho_\fD\rho_\fB \fbV(\fZ).
\]
In fact 
\begin{equation}
	\gV(\fZ) \subset \fbV(\fZ) \subset \bV(\fZ) \subset \cV(\fZ)
	\label{E:vfield_filtration}
\end{equation}
are Lie subalgebras by an easy computation.

The vector bundles $\fbT \fZ$ and $\gT \fZ$ are defined by the property that
\[
	\fbV(\fZ) = C^\infty(\fZ; \fbT \fZ),
	\quad \gV(\fZ) = C^\infty(\fZ; \gT \fZ).
\]
The inclusions \eqref{E:vfield_filtration} induce bundle maps
$\gT \fZ \to \fbT \fZ$ and $\fbT \fZ \to \bT \fZ$.

\begin{prop}
There is a natural vector bundle isomorphism
\[
	\piX^\ast (\scT \oX)\cong \gT \fZ.
\]
\label{P:pullback_sc_to_g}
\end{prop}
\begin{proof}    The subset $Z' =  Z\setminus \fb^{-1}(0)$ is
canonically isomorphic to the product $X \times (0,\infty)$ and it
is clear from the definitions that if $p\in Z'$, the fibre $(\gT\fZ)_p$ is
therefore canonically isomorphic to $\scT_{\piX(p)}X$.   Thus we have
a canonical isomorphism
\begin{equation}\label{e2.18.10.15}
\piX^*(\scT \fZ)| Z' = 
(\gT \fZ)|Z'
\end{equation}
and we need to show that this isomorphism extends over $\fb^{-1}(0)$;
this will be done in local coordinates.

As a matter of notation, if $\xi$ is a
section of $\pr_1^*(\scT X)$ over a subset of $X \times (0,\infty)$, we denote
by $\piX^*(\xi)$ the {\em lift} of $\xi$ to $Z$, defined over $Z'$ by
\eqref{e2.18.10.15} and then by extension by continuity to the boundary.

Consider first an interior point of $D \subset \fZ$. Near such a
point, we have local coordinates $(s,\omega,\ve)$, the maps $\piX$ and
$\fb$ being
\[
\piX(s,\omega,\ve) = (\rho = \ve s,\omega), \quad \fb(s,\omega,\ve) = \ve,
\]
where $\rho$ is a defining function for $\pa X$ in $X$ and $\omega$
denotes some set of coordinates on $\pa \oX = \bbS^2$.  One calculates
\[
\piX^*(\rho\pa_\rho) = s\pa_s,\quad  \piX^*(\pa_\omega) = \pa_\omega
\]
from which it follows that the local frame $\{\rho^2\pa_\rho, \rho
\pa_\omega\}$ of $\scT X$ lifts to the local frame $\{\ve s^2\pa_s,
\ve s\pa_\omega\}$ of $\gT\fZ$.  Hence \eqref{e2.18.10.15} extends
smoothly over $\Int(D)$.

Next, consider an interior point $p$ of one of the $\fX_j$. As in \eqref{e1.22.10.15},
we have local coordinates $(v,\ve)$ in a neighbourhood of $p$ and by \eqref{e1.24.10.15},
\[
\piX(v,\ve) = v + \frac{\zeta_j}{\ve}, \quad \fb(v,\ve) = \ve.
\]
Thus for fixed $\zeta_j$ and $\ve$, $\piX^*\pa_z = \pa_v$.  Since (by
construction of $\scT X$) the euclidean vector fields $\pa_z$ define a
smooth frame of $\scT X$ and the $\pa_v$ define a smooth frame of
$\gT\fZ$ near $X_j$, we see that \eqref{e2.18.10.15} extends smoothly
over $\mathring X_j$.

Near the corner $D\cap X_j$, we have local coordinates $(r,x,\omega)$,
where $r$ is the restriction of $\rho_{j}$ and $x$ is the restriction
of $\rho_D$. In these coordinates,
\[
\piX(r,x,\omega) = (\rho = x,\omega), \quad \fb(r,x,\omega) = rx.
\]
Thus
\[
\piX^*(\rho\pa_\rho) = x\pa_x - r \pa_r, \quad \piX^*(\pa_\omega) =
\pa_\omega.
\]
Hence the local frame $\{\rho^2\pa_\rho,\rho \pa_\omega\}$ lifts to
give the local frame $\{x^2\pa_x - rx\pa_r, x\pa_\omega\}$ of $\gT
Z$.

Finally, near $\fD \cap \fB$, we work in coordinates $(\ve,x',\omega)$ where $x' = s^{-1} = \ve/x$. Then
\[
	\piX^\ast(\{\rho^2\pa_\rho, \rho\pa_\omega\}) = \{\ve(x')^2\pa_{x'}, \ve x'\pa_\omega\},
\]
and the identification extends also to points near $\fD\cap \fB$.

Note in particular that on global sections, we have a `lifting map'
$\pi_X^*$ 
\begin{equation}
	\piX^\ast(\scV(\oX)) \subset \gV(\fZ).
	\label{E:lift_sc}
\end{equation}
\end{proof}

Note that 
\begin{equation}
	\gT \fZ = \rho_\fD \rho_\fB \fbT \fZ = \ve \rho_\fX^{-1} \rho_\fB \fbT \fZ,
	\label{E:bundle_identities}
\end{equation}
where the bundle $\rho_{X}^{-1}\rho_{\fB}\fbT\fZ$ can be defined (as
for other rescaled versions of the tangent bundle) by its space of
sections so that
\[
\rho_X C^\infty(\fZ; 
\rho_\fX^{-1} \rho_\fB \fbT \fZ) = \rho_\fB C^\infty(\fZ; \fbT \fZ). 
\]
We observe that the restriction $\rho_X^{-1}\rho_B \fbT\fZ|D$ of this
bundle is precisely the conic tangent bundle $\sccT D$ introduced in \S\ref{S:dirac}.

The point of this is that  the scaling map $\mu : \gT \fZ \to \gT \fZ$ given by multiplication by
$\ve$, which vanishes over the boundary, extends to define a global isomorphism
\[
	\mu : \rho_\fB \rho_\fX^{-1} \fbT \fZ \stackrel \cong\to \gT \fZ.
\]

\begin{prop}
There are natural vector bundle isomorphisms 
\[
	\gT \fZ \rst_{\fX_j} \cong \scT \fX_j, 
	\quad j = 0,\ldots,N.
\]
Composition of $\mu^{-1}$ and restriction to $\fD$ determines
a ``rescaled restriction'' isomorphism
\begin{equation}
	\mu^{-1} : \gT \fZ \rst_{\fD} \cong \sccT \fD.
	\label{E:rescaled_restriction_D}
\end{equation}
\label{P:retricting_gT}
\end{prop}

\begin{prop}
Let $g$ be a scattering metric on $\oX$, and set $\wt g = \piX^\ast(g)$. Then $\wt g$ is a smooth metric on $\gT \fZ$ with the following properties:
\begin{enumerate}
[{\normalfont (a)}]
\item 
The restriction $g_0$ of $\wt g$ to $\fX_0$ is identically equal to $g$.
\item
The restriction $g_j$ of $\wt g$ to $\fX_j$ is a scattering metric, which is Euclidean with respect to 
the identification $\fX_i \cong \ol{\bbR^3}$.
\item
With respect to \eqref{E:rescaled_restriction_D}, the rescaled restriction $g_\fD
\cong g\rst_{\fD}$ defines a smooth cone metric on $\fD$, i.e., a positive definite section
of $\sccT \fD$. In fact, $g_\fD$ is the lift over $\fD \to \bbS^2\times[0,\infty]$ of the Euclidean metric on $\ol{\bbR^3}.$
\item
If $g$ is the Euclidean metric on $X$ to begin with, then
over a product neighborhood of $\fB$, $\piX^\ast g$ is independent of $\ve$.
\end{enumerate}
\label{P:lifted_metric}
\end{prop}
\begin{proof}  Parts (a), (b), and (c) hold if $X$ is a general scattering
  manifold with boundary $Y$ and we prove it in this generality.

Note first that Proposition~\ref{P:pullback_sc_to_g} implies that the
lift $\wt{g}$ is a smooth metric (uniformly up to the boundary) on
$\gT Z$.  It is thus enough to prove (a)---(d) at interior points of
the boundary hypersurfaces of $Z$.

The hypersurface $D_1$ of $\fZ_1$, being the (positive)
projectivization of the normal bundle of the corner $Y\times \{0\}$ in
$\fZ_0$, is diffeomorphic to the cylinder $Y
\times [0,\infty]$.    Here we think of $[0,\infty]$ as the set of
ratios $[\rho:\ve]$ where $\rho$ is a given defining function of the
boundary hypersurface $Y$.  

We calculate first the lift of an exact scattering metric $g$ on $X$
to $\fZ_1$.
If 
\begin{equation}
g = \frac{\rd \rho^2}{\rho^4} + \frac{h(\rho,y;\rd y)}{\rho^2}
\end{equation}
and $s = \ve/\rho$, we have, near $D$,
\begin{equation}
g =\frac{1}{\ve^2}\left(g_D + \cO(\ve)\right)
\end{equation}
where
\begin{equation}
g_D = \rd s^2 + s^2h(0,y;\rd y)
\end{equation}
and $\cO(\ve)$ denotes a section of $S^2T^*$ whose length with respect
to $g_D$ is $\cO(\ve)$, uniformly if $s$ is bounded away from $0$.  This
proves part (c), at least away from $B\cap D_1$.   Near this corner,
we should use $\wt{s} = \rho/\ve$ and then
the lift of the metric takes the form
\begin{equation}
g = \frac{1}{\ve^2}\left(\frac{\rd \wt{s}^2}{\wt{s}^4 } +
  \frac{h(\ve\wt{s},y,\rd y)}{\wt{s}^2}\right)
\end{equation}
For the euclidean metric, we have that $h$ is independent of $\rho$,
and in this case
\[
\ve^2 g = 
\frac{\rd \wt{s}^2}{\wt{s}^4 } +
  \frac{h(y,\rd y)}{\wt{s}^2}
\]
is independent of $\ve$ as required, proving part (d).  The analogue
of part (d) holds in the general setting of scattering manifolds if
$g$ is exactly conical in a collar neighbourhood of $Y$.

Now consider an interior point $p$ of $D_1$.  If we choose local
coordinates $(\zeta^\mu) = (\zeta^1,\ldots,\zeta^n)$ in $D_1$ centred at $p$, then the lifted
metric can be written
\[
\wt{g} = \frac{1}{\ve^2}\left(\wt{g}_{\mu\nu}(\zeta)\rd \zeta^\mu\rd
  \zeta^\nu + \cO(\ve)\right)
\]
Blowing up $p$ corresponds to introducing $v^\mu = \zeta^\mu/\ve$ as
new coordinates, and the metric lifts as
\[
\wt{g}(\ve v)_{\mu\nu}\rd v^\mu \rd v^\nu + \cO(\ve)
\]
so that the restriction to the new face is the euclidean metric
\[
\wt{g}(p)_{\mu\nu}\rd v^\mu \rd v^\nu.
\]
This proves (a) and (b).  
\end{proof}

\subsection{Bogomolny equation} \label{S:single_bogo}

We begin with a general discussion of smooth bundles and connections
over manifolds with corners such as $\fZ$: compare with the parallel
discussion for manifolds with boundary in \S\ref{S:bogoscat_prelim}.
In general if $Z$ is a manifold with corners, a smooth principal $G$-bundle
$P \to Z$ is defined in terms of smoothness of the data up to (and, as always,
including) the boundary. It is equivalent to assume that $Z \subset
\hat{Z}$ where $\hat{Z}$ is some manifold without boundary of the same dimension
as $Z$, and that $P$ is the restriction to $Z$ of some smooth
$G$-bundle on $\hat{Z}$.  A smooth $G$-connection $A$ on $P$ is defined
similarly in terms of smoothness of data up to the boundary (which is
again equivalent to the assumption that $A$ is the restriction of a
smooth connection on $\hat{P} \to \hat{Z}$).

Let $V$ be any bundle associated to $P$. If $H$ is any boundary
hypersurface of $Z$, then there is a restriction map on sections of
$V$; also $A$ defines a restricted connection $A_H$ on $P|H \to H$.

Conversely if $\{H_a\}$, $a=1,\ldots,p$, is an enumeration of the boundary
hypersurfaces of $Z$ and $s_a \in C^\infty(H_a, V|H_a)$, satisfying
the compatibility conditions
\begin{equation}\label{e2.28.10.15}
s_a|H_a\cap H_b = s_b|H_a\cap H_b
\end{equation}
for all pairs $a$ and $b$, then there is an {\em extension} $s \in
C^\infty(Z,V)$ of the $s_a$, i.e.\ $s|H_a = s_a$ for each $a$.  This
follows from the (much more general) results in \cite{CCN}.
The key point is that if $s_a$ vanishes at $H_a\cap H_b$, then it has
an extension $\wt{s}_a$ which vanishes to the same order at $H_b$.
To see how this implies the general extension result, suppose by
induction that we have already have a partial extension $\wt{s}$ with
the property $\wt{s}|H_a = s_a$ for $a=1,\ldots b-1$. Let $t = s_b -
\wt{s}|H_b$.  By the compatibility conditions \eqref{e2.28.10.15}, 
\[
t|H_a\cap H_b = 0\mbox{ for }a=1,\ldots, b-1.
\]
Let $\wt{t}$ be an extension of $t$ to $Z$ which also vanishes at the
$H_a$ for $a=1,\ldots, b-1$.  Then $\wt{s} + \wt{t}$ is smooth and its
restriction to $H_a$ is $s_a$ for $a=1,\ldots, b$, completing the
inductive step.

  As
in Proposition~\ref{p1.18.10.15} there is an analogous result for
connections: if $(A_a)$ is a collection of connections in $P|H_a$,
which agree over $H_a\cap H_b$, then there is a connection $A$ on
$P\to Z$ such that $A|H_a = A_a$ for each $a$.

With these preliminaries out of the way, suppose that 
$P = \pi_\oX^\ast P_X$ is a principal $\SU(2)$-bundle on our gluing space $\fZ$ with a
smooth connection $A$,  $P_X\to X$ being a principal $\SU(2)$-bundle
over $X$.   Let $V \to Z$ be any associated vector bundle. By
composing the covariant derivative operator
\[
	\wh d_{A} : C^\infty(\fZ; V) \to C^\infty(\fZ; T^\ast \fZ\otimes V)
\]
with the natural map
\begin{equation}
	T^\ast \fZ \to \gT^\ast \fZ
	\label{E:T_to_gT}
\end{equation}
(dual to the inclusions \eqref{E:vfield_filtration}) we define the
associated  $\gamma$ covariant derivative
\[
	d_{A} : C^\infty(\fZ; V) \to C^\infty(\fZ; \gT^\ast \fZ\otimes V).
\]
This operator, which differentiates along the fibres of $\fb$, will be
our main concern, which is the reason for writing $\wh d_A$ for the
`ordinary' covariant derivative.  We will be interested in the case $V = \adP$.

As a matter of notation, we write $\gLam^k$ for the bundle
$\bigwedge^k \gT^\ast\fZ$, and we note that the projection map \eqref{E:T_to_gT}
 extends to define
\[
	C^\infty(\fZ; \Lambda^k) \to (\rho_\fB \rho_\fD)^k C^\infty(\fZ; \gLam^k)
\]
since \eqref{E:T_to_gT} factors through $\fbT^\ast \fZ = (\rho_\fD \rho_\fX) \gT^\ast \fZ$. 
The following is immediate:

\begin{lem}
If $A$ is a smooth connection and $\Phi \in C^\infty(\fZ; \adP)$, then 
\begin{equation}
\begin{gathered}
	d_A \Phi \in (\rho_\fB \rho_\fD) C^\infty(\fZ; \gLam^1 \otimes \adP),
	\\ F_A \in (\rho_\fB \rho_\fD)^2 C^\infty(\fZ; \gLam^2 \otimes \adP).
\end{gathered}
	\label{E:gamma_cov_deriv_vanishing}
\end{equation}
\label{L:lifted_cov_deriv}
\end{lem}

We consider now the fiberwise Bogomolny equation on $\fZ$: 
\begin{equation}
	\Bogo(A,\Phi) = \star F_A - d_A \Phi = 0,
	\label{E:Bogo_gluing}
\end{equation}
where $A$ is a (relative) connection, $\Phi$ is a section of $\adP$, and
\[
	\star = \star_{\wt g} : \gLam^k \stackrel \cong \to \gLam^{3-k}
\]
is the relative Hodge star induced by the $\gamma$ metric $\wt g$.

As the fibers of $\fZ$ for $\ve > 0$ are canonically identified with $(\oX,g)$,
with $\gT \fZ \cong \scT \oX$ there,
a solution to \eqref{E:Bogo_gluing} represents a smooth family of Euclidean
monopoles parameterized by $\ve \in (0,\infty)$.  We shall construct
such families\footnote{Although they will only be conormal, not
  smooth} by extending monopole data initiall defined on $\fb^{-1}(0)$
to nearby fibers $\fb^{-1}(\ve)$.

From Lemma~\ref{L:lifted_cov_deriv} we obtain:
\begin{prop}\label{p11.24.10.15}
For any smooth data $(A,\Phi)$ on $Z$, we have
\[
\Bogo(A,\Phi) = \rho_B\rho_D C^\infty(Z,\gLam^1\otimes \adP).
\]
If the restriction $(A_j,\Phi_j)$ of $(A,\Phi)$ satisfies
\[
\Bogo(A_j,\Phi_j)=0\mbox{ on }X_j
\]
for each $j$, then
\[
\Bogo(A,\Phi) \in \rho_B\rho_D\rho_X C^\infty(Z,\gLam\otimes
  \adP).
\]
Finally, if in addition $\nabla_{A|D}(\Phi|D)=0$ over $D$,
\begin{equation}\label{e1.20.10.15}
\Bogo(A,\Phi) \in \rho_B^2\rho_D^2\rho_X C^\infty(Z,\gLam\otimes \adP)
\end{equation}
\end{prop}

\subsection{Ideal monopoles and pregluing configurations}
\label{S:single_ideal}
\begin{defn} An {\em ideal $\SU(2)$-monopole}, $\iota$, on $Z$, is the
  restriction to $\fb^{-1}(0)$ of a smooth $(A,\Phi)$ defined near the
  boundary of $Z$, satisfying \eqref{e1.20.10.15}, and such that $A|D$
  is an $\SU(2)$ Dirac monopole connection. In this situation
  we say that $(A,\Phi)$ {\em represents} the ideal monopole $\iota$.
\label{d1.24.10.15}\end{defn}

In the next section we shall discuss this notion further and in
particular moduli spaces of ideal monopoles.  It would be possible to
give a definition which is intrinsic to $\fb^{-1}(0)$ in terms of
bundles and monopole data over $D$ and the $X_j$ which agree at all
the corners, but \eqref{d1.24.10.15} will serve our present purposes.

\begin{prop}\label{p1.29.10.15}
Given any collection of monopoles $(A_j,\Phi_j)$, with
\[
[(A_0,\Phi_0)]\in \monM_{k_0}, \quad [(A_j,\Phi_j)]\in \monM_{k_j} \mbox{
  for }
j=1,\ldots, N,
\]
 there exists an ideal monopole whose restriction to
$X_j$ is the given data $(A_j,\Phi_j)$.
\end{prop}
\begin{proof}
Let $\cU = \{0 \leq \rho_D < \delta\}$ be a product neighbourhood of
$D$.  It is convenient to work with vector bundles, so we denote by
$E$ the complex rank-2 vector bundle associated to $P$ by the fundamental
representation of $\SU(2)$.  We can equip the restriction $E|X_j$ with
the monopole data $(A_j,\Phi_j)$ for each $j$.   We can also assume that
$\delta$ is chosen so small $\Phi_j|\cU\cap X_j$ is non-zero.  Then as
in \S\ref{S:bogo1} we have the decomposition
\begin{equation}\label{e12.25.10.15}
E|\cU\cap X_j = L_j \oplus L_j^{-1}
\end{equation}
into the eigenbundles of $\Phi_j$.  With respect to
\eqref{e12.25.10.15}, we have
\[
\Phi_j = \diagl(im - \phi_j, -im + \phi_j)
\]
and
\begin{equation}\label{e11.25.10.15}
A_j = \diagl(a_j,-a_j) + b_j
\end{equation}
where 
\begin{equation}\label{e11a.25.10.15}
\phi_j = \frac{ik_j}{2}\rho_D + \cO(\rho_D^2)
\end{equation}
and
\[
b_j \in \rho_D^\infty C^\infty(\cU\cap X_j, T^*X_j \otimes \frp_1).
\]
In \eqref{e11.25.10.15}, we have committed a slight abuse of notation;
it is to be understood that $a_j$ is a connection in $L_j$ and $-a_j$
is the dual connection in $L_j^{-1}$.  By Proposition~\ref{p3.8.10.15}
we have
\begin{equation}\label{e19.25.10.15}
\dagger\! F(a_j) = \frac{ik_j}{2}.
\end{equation}

Now let $L \to \cU$ be a complex line bundle with 
\[
L|\cU\cap X_j = L_j.
\]
By \eqref{e19.25.10.15}, the data $\{a_j|D\cap X_j\}$ comprise a Dirac
framing in the sense of Definition~\ref{d11.25.10.15}.  By
Proposition~\ref{p1.25.10.15}, there exists a Dirac connection $a_D$ on
$L$ agreeing with $a_j$ over $D\cap X_j$.  Hence there exists a
connection $\wt{a}_D$ on $L$ over $\cU$ which agrees $a_j$ over $X_j\cap \cU$
and with $a_D$ over $D$.  We now define a connection $A$ on $L\oplus
L^{-1}$ by choosing smooth extensions $\wt{b}_j$ of $b_j$, $\wt{b}_j
\in \rho_D^\infty C^\infty(Z,\gT^*Z\otimes \frp_1)$ and putting
\[
A = \diagl(\wt{a}, - \wt{a}) + \sum_{j=0}^N \wt{b}_j.
\]
Then $A$ is smooth and it restricts to $A_j$ over $X_j$ and to the Dirac
connection $A_D$ over $D$.

If we define $\Phi_D = \diagl(im,-im)$ with respect to the
decomposition $E|D = L\oplus L^{-1}$, then $\nabla_{A|D}\Phi_D = 0$
and moreover
$\Phi_D$ agrees over $D\cap X_j$ with $\Phi_j$ for each $j$. Hence
there is a smooth extension $\Phi \in C^\infty(\fZ; \ad(P))$ of these data over $Z$.
Then $(A,\Phi)$ represents the given ideal monopole configuration as required.
\end{proof}

Before starting on solving the fibrewise Bogomonly equations in
earnest, it is convenient to introduce a better choice of smooth
configuration representing a given ideal monopole. Thus we make the
following definition:

\begin{defn}\label{D:Bogo_boundary_conditions}
A {\em pregluing configuration} is a smooth configuration $(A,\Phi)$
on $\fZ$ with the following properties:
\begin{enumerate}[{\normalfont (i)}]
\item
\begin{equation}	\label{E:Bogo_pregluing_error}
\Bogo(A,\Phi) = \cO(\rho_\fX \rho_\fD^3\rho_\fB^\infty)
\end{equation}
\item
$\nabla_A$ is diagonal to infinite order with respect to the splitting
\begin{equation}
	\adP = \adPz \oplus \adPo := \sspan_\bbC \backPhi \oplus \backPhi^\perp
	\label{E:adP_splitting}
\end{equation}
in a product neighbourhood of $D$.
\end{enumerate}
\end{defn}

\begin{prop}  Let $\iota$ be an ideal monopole configuration on $Z$.  Then there exists a
  pregluing configuration on $Z$ which represents $\iota$.
\label{P:initial_data_map}
\end{prop}
\begin{proof}
The proof is a continuation of the previous proposition and we make use of the notation
established there.  Let $(A,\Phi)$ be as constructed from the boundary
data in the proof of Proposition~\ref{p1.29.10.15}.  In particular,
$\Bogo(A,\Phi) = \cO(\rho_B^2\rho_D^2\rho_X)$ and $A|D$ is an $\SU(2)$
Dirac connection.

To see how to improve the order of
vanishing at $D$, use $\ve = \rho_\fD \rho_\fX$ to write the Taylor
expansion of $\Phi$ in $\cU$ in the form
\[
	\Phi = \Phi_0 + \ve \Phi_1 + \cO(\rho_\fD^2), 
	\quad \Phi_1 \in \rho_\fX^{-1}C^\infty(\fD; \adP).
\]
Then
\[
	d_A\Phi = \underbrace{d_A \Phi_0}_{\cO(\rho_\fD)} + \underbrace{\ve d_A \Phi_1}_{\cO(\rho_\fD^2)} + \cO(\rho_\fD^3)
\]
as a section of $\gLam^1\otimes \adP$. (Note that $\ve$ commutes with the
$\gamma$ connection $d_A$ since $[\ve,\fbV] = 0$ by definition. By contrast,
$\rho_\fD$ does not commute with $d_A$.) Hence
\[
	\Bogo(A,\Phi) = - \underbrace{d_A \Phi_0}_{\cO(\rho_\fD)} + \underbrace{\star F_A - \ve d_A \Phi_1}_{\cO(\rho_\fD^2)} + \cO(\rho_\fD^3) \in C^\infty(U; \gLam^1\otimes \adP)
\]
To interpret this in terms of the geometry
of $\fD$ we make use of the rescaled restriction isomorphism 
\begin{equation}
	(\mu^{-1})^\ast : \gLam^k \fZ\rst_{\fD} \cong \sccLam^k \fD
	\label{E:rescaled_restriction_forms_D}
\end{equation}
dual to \eqref{E:rescaled_restriction_D}.
There is a commutative diagram
\begin{equation}
\begin{tikzpicture}[->,>=to,auto]
\matrix (m) [matrix of math nodes, column sep=1cm, row sep=1cm, text height=2ex, text depth=0.25ex]
{\gLam^k \fZ \rst_{\fD} & \gLam^{k+1}\fZ \rst_{\fD} \\ \sccLam^k\fD & \sccLam^{k+1}\fD \\};
\path (m-1-1) edge node {$d$} (m-1-2); 
\path (m-1-2) edge node {$(\mu^{-1})^\ast$} (m-2-2); 
\path (m-2-1) edge node {$\ve\, d$} (m-2-2); 
\path (m-1-1) edge node {$(\mu^{-1})^\ast$} (m-2-1); 
\end{tikzpicture}
	\label{E:cd_rescaled_d}
\end{equation}
and likewise $d_A$ is intertwined with $\ve\,d_{A\rst \fD}$, where $d_{A \rst \fD}$ is the 
covariant derivative of scattering/conic type on $\fD$ induced by the restriction of $A$. It follows that
\begin{equation}\label{e1.28.10.15}
	\Bogo(A,\Phi) = -\ve d_{A\rst \fD} \Phi_0 + \ve^2(\star_\fD F_{A\rst \fD} - d_{A\rst \fD} \Phi_1) + \cO(\rho_\fD^3),
\end{equation}
where we are regarding the coefficients of $\ve^k$ as sections of $\sccLam^1 \fD \otimes \adP$.
Vanishing of the first term has already been discussed. In order for the term
in parentheses to vanish, we need
\begin{equation}\label{e31.25.10.15}
\star_\fD F_{A\rst \fD} - d_{A\rst \fD} \Phi_1=0
\end{equation}
which is to say that $A|D$ is an $\SU(2)$ Dirac connection (cf.\ \S\ref{s1.15.10.15})
with abelian Higgs field $\Phi_1$.

In the previous proposition, we constructed $A$ so that $A|D$ was an
$\SU(2)$ Dirac monopole.  So it remains to show that $\Phi$ can be
constructed so that $\Phi_1$ is an abelian Higgs field.  For this,
define 
\begin{equation}\label{e21.25.10.15}
\phi_D = \frac{i}{2}\sum_{j=0}^N \frac{k_j}{|\zeta - \zeta_j|}
\end{equation}
so that $*_D \rd \phi_D$ is the curvature $F(a_D)$ of $a_D$. If we
choose $\rho_j = \rho_{X_j} = |\zeta-\zeta_j|$ near $D\cap X_j$, then clearly
\[
\psi : = \rho_0\cdots \rho_N \phi_D \in C^\infty(D)
\]
and $\psi|D\cap X_j = ik_j/2$. Similarly $\psi_j= \rho_D^{-1}\phi_j$
is smooth on $X_j$ and $\psi_j|D \cap X_j = i k_j/2$ by \eqref{e11a.25.10.15}.  Let $\wt{\psi}$
be a smooth extension of $\psi_D$ and the $\psi_j$ to $Z$, and let
$\wt{\phi} = (\rho_0\cdots \rho_N)^{-1}\wt{\psi}$.  By construction, $\ve\wt{\phi}|X_j\cap \cU
= \phi_j$, and near any interior point of $D$, 
\begin{equation}
\wt{\phi}  = \phi_D + \cO(\ve).
\end{equation}
We now define $\Phi_1 = \diagl(\wt{\phi}, - \wt{\phi})$ and by construction
$\Phi_0 + \ve \Phi_1$ is smooth and  has the correct Taylor expansion
at $\fD$.

Finally, to arrange the vanishing of $\Bogo(\backA,\backPhi)$ to
infinite order near $\fB$, note that a neighbourhood $U = \cU \cap \{0\leq
\rho_B\leq \delta'\}$ can be identified with the product of $\{0\leq
\rho_D < \delta\}$ with a neighbourhood $U'$ of $S_\infty$ in $D$, in
such a way that the fibrewise metric $\wt{g}$ on $\rho_D = \ve$ is
independent of $\ve$ in $U$ (part (d) of
Proposition~\ref{P:lifted_metric}).  In such a neighbourhood we can
choose the extensions $\wt{a}$ and $\wt{\phi}$ also to be independent
of $\ve$.  We may also suppose that all $\wt{b}_j=0$ in this
neighbourhood.   Then the argument leading to \eqref{e1.28.10.15} has
no error term here and gives
\[
	\Bogo(\backA,\backPhi) = \star F_\backA - d_\backA \backPhi \cong \ve^2 \star_\fD F_{A\rst_\fD} 
	  - \ve d_{A\rst_\fD} (\ve \phi_\fD) \equiv 0
 \mbox{ over $U$}
\]
as required.
\end{proof}

\subsection{The iteration}

Suppose $(\backA,\backPhi)$ is a pregluing configuration.
The task is now to find a perturbation $(a,\phi)$ such that $\Bogo(\backA + a,\backPhi + \phi)$
vanishes to high order in $\ve$.

For $(a,\phi) \in C^\infty(\fZ; (\gLam^1\oplus \gLam^0)\otimes \adP)$,
\[
\begin{gathered}
	\Bogo(\backA + a, \backPhi + \phi) = \Bogo(\backA,\backPhi) + D\Bogo_{\backA,\backPhi}(a,\phi) + Q(a,\phi)
	\\ D\Bogo_{\backA,\backPhi}(a,\phi) = \star d_\backA a - d_\backA \phi + [\backPhi,a],
	\\ Q(a,\phi) = \star [a,a] - [a,\phi]
\end{gathered}
\]
The linear operator $D\Bogo_{\backA,\backPhi}$ is not elliptic owing to the action
of the gauge group. To remedy this, we impose the {\em Coulomb gauge}
condition
\begin{equation}
	\rd_{\backA,\backPhi}^*(a,\phi) =  d^*_\backA  a - [\backPhi,\phi] = 0
	\in C^\infty(\fZ; \adP).
	\label{E:CG}
\end{equation}
just as in the discussion in \S\ref{S:local}, where now $d^*_A$ is the
formal adjoint of the $\gamma$-covariant derivative $d_A$ on $\adP$.
Over the $\ve > 0$ fibers of $\fZ$, this condition is known to determine 
a slice for the action of the gauge group provided $(a,\phi)\rst_{\ve}$ is sufficiently small. 

Later, when we take into account the parameters of the gluing construction, we
will show that \eqref{E:CG} determines a slice globally and uniformly in the
parameters, for sufficiently small $\ve$ (c.f.\ Theorem~\ref{T:coulomb}).

Thus we seek to solve
\begin{equation}
\begin{gathered}
	\Bogo(\backA,\backPhi) + \backL (a,\phi) + Q(a,\phi) 
	 = 0 
	\\ \backL = D\Bogo_{\backA,\backPhi} + d^*_{\backA,\backPhi},
	\\ \backL(a,\phi) = \begin{bmatrix} \star d_\backA & -d_\backA \\ -d_\backA^\ast & 0\end{bmatrix} 
	\begin{bmatrix} a \\ \phi \end{bmatrix} + \ad(\backPhi)\begin{bmatrix} a \\ \phi \end{bmatrix}
\end{gathered}
	\label{E:Bogo_expansion}
\end{equation}
where $L$ is the operator \eqref{e5.10.10.15} acting along the fibres
of $\fb$.

Over $\fX_j$, $d_\backA$ restricts to the covariant derivative $d_{A_j}$ acting
on sections of $\scLam^\ast\fX_j\otimes \adP \cong \gLam^\ast \otimes \adP \rst_{\fX_j}$,
and $\star$ is identified with the Hodge star associated to the metric $g_j$. A consequence
is the following:
\begin{prop} Denote by $L_j$ the operator $L_{(A_j,\Phi_j)}$ over
  $X_j$.  Then $\backL|\fX_j = L_j$ in the sense that
if $(a,\phi) \in C^\infty(\fZ; \gLam\otimes \adP)$, then
\[
	(\backL (a,\phi))\rst_{\fX_j} = L_j ((a,\phi)\rst_{\fX_j}) 
	  \in C^\infty(\fX_j; \scLam\otimes\adP).
\]
\label{P:L_rstn_X}
\end{prop}

At $\fD$ the situation is a bit more complicated. 
First, we decompose $\backL$ according to the splitting $\adP = \adPz\oplus \adPo$
near $\fD$.
Note that, as a section of $\adPz$, $\backPhi$ acts trivially 
on $\adPz$ and nondegenerately on $\adPo$. 
Combined with the assumption that $\backA$ is diagonal to infinite order, we may write
\begin{equation}
	\backL = \begin{pmatrix} L_0 & \cO(\rho_\fD^\infty\rho_\fB^\infty)
	  \\\cO(\rho_\fD^\infty \rho_\fB^\infty)& L_1 + \backPhi_1 \end{pmatrix}
	\label{E:backL}
\end{equation}
with respect to $\adP = \adPz\oplus \adPo,$ where $\backPhi_1$ denotes the
nondegenerate restriction to $\adPo.$ 

We focus attention on $L_0$; we will not need to consider $L_1 + \backPhi_1$ until \S\ref{S:global}.
Using the fact that $\backA$ is trivial on $\adPz = \adQ$ at $\fD$, the diagram
\eqref{E:cd_rescaled_d}, and the fact that $\star$ is interwtined with
$\star_\fD$ under the rescaled restriction isomorphism, we have the following
\begin{prop}
The rescaled restriction $\lim_{\ve \smallto 0} \ve^{-1}L_0$ of $L_0$ is the operator
\[
\begin{gathered}
	L_\fD : C^\infty(\fD; \sccLam) \to C^\infty(\fD; \sccLam),
	\\ L_\fD(a,\phi) = 
	\begin{bmatrix} \star_\fD d & -d \\ d^\ast & 0 \end{bmatrix} \begin{bmatrix} a \\ \phi \end{bmatrix}
\end{gathered}
\]
in the following sense:
if $(a,\phi) \in C^\infty(U; (\gLam^1\oplus \gLam^0)\otimes \adPz)$, then
with respect the rescaled restriction isomorphism \eqref{E:rescaled_restriction_forms_D}
\[
	(\ve^{-1}L_0 (a,\phi))\rst_{\fD} \cong  L_\fD ((a,\phi)\rst_{\fD})
	  \in C^\infty(\fD; \sccLam^1\oplus \sccLam^0).
\]
\label{P:L_rstn_D}
\end{prop}

The operators $L_\fX$ and $L_\fD$ are analyzed in Appendix~\ref{S:linear}, in
particular the regularity of solutions to $L_\bullet u = f$ is discussed. In
order to state these results, and for the formal construction below, we first
need to introduce a weakening of smoothness to allow functions with asymptotic
expansions having logs and noninteger powers of boundary defining functions.
 
\subsection{Polyhomogeneity} \label{S:functions_phg}

Briefly, a polyhomogeneous function is one having a complete asymptotic expansion
at all boundaries with terms of the form $x^s (\log x)^p$ where $x$ is a boundary defining
function. The exponents $(s,p) \in \bbR \times \bbN_0$ which appear in the expansion
are recorded by a (real) {\em index set}, which is  a discrete subset $E \subset \bbR \times \bbN_0$
satisfying the property that 
for each $\alpha \in \bbR$, there are only finitely many $(s,p) \in E$ with
$s \leq \alpha,$ and only finitely many $(s,p)$ for each fixed $s$. An index set is {\em smooth} if 
\[
	(s,p) \in E \implies (s + n, q) \in E, 
	\quad \forall\ n \in \bbN_0, \ 0 \leq q \leq p.
\]
Unless otherwise specified, all index sets will be smooth.

\begin{defn}
If $X$ is a manifold with boundary, then the polyhomogeneous space,
$\cAphg^E(X)$, is the space of smooth functions on $X\setminus \pa X$ having an
asymptotic expansion
\[
	u \sim \sum_{(s,p) \in E} u_{(s,p)} x^s (\log x)^p,
	\quad u_{(s,p)} \in C^\infty(\pa X)
\]
where $x$ is a boundary defining function for $\pa X$, and the equivalence is
modulo $x^\infty C^\infty(X)$. The finiteness condition on $E$
insures that such sums are Borel summable. If $E$ is smooth then $\cAphg^E(X)$
is independent of the choice of $x.$ 

Suppose now that $X$ is a manifold with corners with boundary hypersurfaces
$H_1,\ldots,H_N,$ and let $\cE = (E_1,\ldots,E_N)$ be an {\em index family},
meaning an $N$-tuple of smooth index sets. Then $\cAphg^\cE(X)$ is recursively
defined as the space of smooth functions on $X\setminus\pa X$ having asymptotic
expansions
\[
	u \sim \sum_{(s,p) \in E_k} u_{(s,p)}\rho_k^s (\log \rho_k)^p,
	\quad u_{(s,p)} \in \cAphg^{\cE(k)}(H_k).
\]
Here $\cE(k)$ is the index family obtained from $\cE$ which is associated to
boundary hypersurfaces of $H_k$, which consist of connected components of
$H_k \cap H_j$ for various $j$.

If $V \to X$ is a vector bundle, the spaces $\cAphg^\cE(X; V)$ are defined as
above in terms of local trivializations.
\label{D:phg}
\end{defn}

There are a number of notational conventions and operations on index sets which
we use below. First, we identify $s \in \bbR$ with the smallest smooth index
set containing $(s,0)$, namely $\set{(s + n, 0) : n \in \bbN_0}$.  In
particular, $\cAphg^s(X) \equiv x^s C^\infty(X)$ for a manifold with boundary,
and $\cAphg^{(s_1,\ldots,s_N)}(X) = \rho_1^{s_1}\cdots \rho_N^{s_N}C^\infty(X)$
for a manifold with corners. We also write $\infty$ for the empty index set,
which is consistent with the identity $\cAphg^\infty(X) = x^\infty C^\infty(X).$

We order elements of $\bbR \times \bbN_0$ lexicographically, with the opposite
order on $\bbN_0$, so that
\[
	(s,p) < (t,q) \iff s < t,\text{ or } s = t \text{ and } p > q.
\]
(This is consistent with the idea that $\cO(x^s(\log x)^p)$ has worse decay
than $\cO(x^t(\log x)^q)$ as $x \smallto 0$.) We then order index sets by
comparing their minimal elements (which exist by the finiteness conditions)
with respect to this order: thus $E < F$ (respectively $E \leq F$) if and only if
$\min E < \min F$ (respectively $\min E \leq \min F$). In particular, for $m
\in \bbR$,
\[
\begin{gathered}
	E > m \iff \min\set{ s : (s,p) \in E} > m, 
	\\ E \geq m \iff E > m, \text{ or } (m,0) \in E \text{ but } (m,p) \notin E\ \forall\ p \geq 1,
\end{gathered}
\]
and if $X$ is a manifold with boundary, $\cAphg^E(X) \subset C^0(X)$ if and
only if $E \geq 0$ and $\cAphg^E(X) \subset C^0_0(X)$ if and only if $E > 0$.

If $E$ and $F$ are (smooth) index sets, then so too are $E + F$ and $E
\cup F$. These correspond respectively to multiplication and addition of
sections of $V$, in the sense that 
\[
\begin{gathered}
	\cAphg^\cE(X; V) \times \cAphg^\cF(X; W) \ni (u,v) \mapsto  u\otimes v \in \cAphg^{\cE + \cF}(X; V\otimes W),
	\\ \cAphg^\cE(X; V) \times \cAphg^\cF(X; V) \ni (u,v) \mapsto u + v \in \cAphg^{\cE \cup \cF}(X; V).
\end{gathered}
\]

The {\em extended union} of $E$ and $F$, denoted $E \ol \cup F$, is the index set
\[
	E \ol \cup F = E \cup F \cup \set{(s,p + q + 1) : (s,p) \in E,\ (s,q) \in F}.
\]
This arises in the context of fiber integration \cite{CCN}.

Finally, we introduce one more notational convention which will be used below. 
For $n \in \bbZ$, $k \in \bbN_0$, let
\[
	\ol{(n,k)} = \set{(n+l, j) : l \in \bbN_0, 0 \leq j \leq k+l}.
\]
This is the smallest smooth index set containing $\set{(n + l, k + l ) : l \in
\bbN_0}$, and $\cAphg^{\ol{(n,l)}}(X)$ consists of functions whose asymptotic
expansions have logarithmic terms with powers growing linearly 
with the powers of $x$.

\subsection{Formal solution} \label{S:formal_formal}
The key solvability properties of $L_\fX$ and $L_\fD$, proved in Appendix~\ref{S:linear},
are summarized in the following:

\begin{thm}
[Thm.~\ref{T:LX_package}.\eqref{I:LX_package_solv}, Thm~\ref{T:LD_package}.\eqref{I:LD_package_solv}]
\mbox{}
\begin{enumerate}
[{\normalfont (a)}]
\item 
Let $f \in \cAphg^\ast(\fX; (\scLam^1\oplus \scLam^0)\otimes \adP)$ with $f = f_0\oplus f_1$
near $\pa \fX$ with $f_i \in \cAphg^{F_i}(\fX; \scLam^\ast \otimes\adPi)$, $i = 1,2$,
and suppose $F_i > \tfrac 3 2$.  Then there is a unique solution to $L_\fX u =
f$ with $u$ in $\Null(L_\fX)^\perp$ with respect to the $L^2$ pairing over $\fX$,
and $u \in \cAphg^\ast(\fX; (\scLam^1\oplus \scLam^0)\otimes \adP)$ with
$u = u_0\oplus u_1 \in \cAphg^{E_0} \oplus \cAphg^{E_1}$ near $\pa \fX$, where
\[
	E_0 = \ol{(2,0)} \ol \cup (F_0 - 1), 
	\quad E_1 = F_1.
\]
\item
let $f \in \cAphg^{F_\fX,F_\fB}(\fD; \sccLam^1\oplus \sccLam^0)$ 
with $F_\fB > \tfrac 3 2$ and $F_\fX > - \tfrac 3 2.$ Then there exists a unique solution to $L_\fD u = f$
with $u \in \cAphg^{E_\fX,E_\fB}(\fD; \sccLam^1\oplus \sccLam^0)$ where
\[
	E_\fX = \ol{(0,0)} \ol\cup (F_\fX + 1).
	\quad E_\fB = \ol{(2,0)} \ol\cup (F_\fB - 1),
\]
\end{enumerate}
\label{T:L_solvability}
\end{thm}

This leads to the following technical result, which is fundamental to the iterative step
in our construction.

\begin{lem}
Denote the bundle $(\gLam^1\oplus \gLam^0)\otimes \adP$ by $W$, and in a
neighborhood $U$ of $\fD \cup \fB$, write $W_i = (\gLam^1\oplus \gLam^0)\otimes
\adPi.$ In the following statements, both $f$ and $u$ have rapidly vanishing $W_1$
components in $U$, i.e., the $W_1$ projections of their restrictions to $U$ lie
in $\cAphg^{\infty_\fD,\infty_\fB,\ast_\fX}(U; W_1)$.
\begin{enumerate}
[{\normalfont (a)}]
\item 
Let $f \in \cAphg^{\ol{(2,j)}_\fD,0_\fX,\infty_\fB}(\fZ; W)$. Then
there exists $u \in \cAphg^{\ol{(1,j+1)}_\fD, 0_\fX, \infty_\fB}(\fZ; W)$
such that
\[
	\backL u - f \in \cAphg^{\ol{(2,j+1)}_\fD,1_\fX,\infty_\fB}(\fZ; W).
\]
\item
Let $f \in \cAphg^{0_\fD, \ol{(-1,j)}_\fX,\infty_\fB}(\fZ; W)$ be arbitrary. Then
there exists $u \in \cAphg^{0_\fD, \ol{(0,j+1)}_\fX,\ol{(2,0)}_\fB}(\fZ; W)$
supported in a neighborhood of $\fD$, such that
\[
	\backL u - \ve f \in \cAphg^{2_\fD,\ol{(0,j+1)}_\fX,\infty_\fB}(\fZ; W),
\]
and furthermore, the $W_1$ projection of $u$ vanishes identically.
\end{enumerate}
\label{L:solvability}
\end{lem}
\noindent Observe that the index set $0 = (0,0)$ means a smooth expansion up to
a given face, $1$ and $2$ means smooth and vanishing to first and second order,
respectively, and $\infty$ means rapid decay.
\begin{proof}
By Theorem~\ref{T:L_solvability}, we can solve $L_\fX u_\fX = f\rst_\fX$ for a unique
$u_\fX \in \cAphg^{E}(\fX) \cap \Null(L_\fX)^\perp$ where 
\begin{equation}
\begin{gathered}
	E_0 = \ol{(2,0)} \ol \cup \ol{(1,j)} \subset \ol{(1,j+1)},
	\\ E_1 = F_1 = \infty.
	\label{E:index_solving_at_X}
\end{gathered}
\end{equation}
In a product neighborhood of $\fX$ which does not meet $\fB$, we define
\[
	u = \chi(\rho_\fX) u_\fX \in \cAphg^{\ol{(1,j+1)}_\fD, 0_\fX, \infty_\fB}(\fZ; W)
\]
where $\chi$ is a smooth cutoff and we confuse $u_\fX$
with its pullback to a function independent of $\rho_\fX$.
Now $L$ can be expressed as $\rho_\fD\rho_\fB$ times a first order operator
of $\fb$ type; in particular we have
\[
	L : \cAphg^{F_\fD,F_\fX,F_\fB}(\fZ; W) \to \rho_\fD \rho_\fB \cAphg^{F_\fD,F_\fX,F_\fB}(\fZ; W) = 
	\cAphg^{F_\fD+1,F_\fX,F_\fB+1}(\fZ; W),
\]
so it follows that $Lu \in \cAphg^{\ol{(2,j+1)}_\fD, 0_\fX, \infty_\fB}$.
However, by construction $\backL u$ and $f$ are polyhomogeneous sections smooth
in $\rho_\fX$ which have the same restriction to $\fX$, so their difference
vanishes to first order there.

The second result is similar. We solve $L_\fD u_\fD = f\rst_\fD$ for a unique
$u_\fD \in \cAphg^{E_\fX,E_\fB}(\fD; W_0)$ with 
\[
	E_\fX = \ol{(0,0)} \ol \cup \ol{(0,j)} = \ol{(0,j+1)},
	\quad E_\fB = \ol{(2,0)}.
\]
Letting $u$ be any smooth extension of $u_\fD$ off of $\fD$ as a section of $W$
with identically vanishing $\adPo$ component, it follows that $\backL(\ve^{-1} u)
= \ve^{-1} \backL(u) \in \cAphg^{0_\fD, \ol{(-1,j+1)}_\fX,\ol{(3,0)}_\fB}$ and
$f$ have the same restriction at $\fD$, so
\[
	\backL u - \ve f \in \ve\,\cAphg^{1_\fD,\ol{(-1,j+1)}_\fX,\ol{(3,0)}_\fB}(\fZ; W)
	= \cAphg^{2_\fD,\ol{(0,j+1)}_\fX,\ol{(3,0)}_\fB}(\fZ; W).
\]
Finally, since there is a neighborhood of $\fB$ in which $\backA$ and the
metric are independent of $\ve$, Proposition~\ref{P:L_rstn_D} extends to say
that the restriction of $L_0$ to any $\ve$ fiber of this neighborhood agrees
with $\ve L_\fD$ there, so we may take the extension $u$ over this neighborhood
to satisfy $\ve L_\fD u_0(\ve) = \ve f_0(\ve)$, where $f_0$ is the $W_0$
projection of $f$, and then it follows that 
\[
	\backL u - \ve f \in \
	= \cAphg^{2_\fD,\ol{(0,j+1)}_\fX,\infty_\fB}(\fZ; W)
\]
since $\backL u_0 - \ve f_0$ vanishes identically near $\fB$ while $u_1$ and
$f_1$ are rapidly vanishing.
\end{proof}

We return now to the equation \eqref{E:Bogo_expansion}, and the main result
of this section. 

\begin{thm}
Let $(\backA,\backPhi)$ be a smooth pregluing configuration as in Definition~\ref{D:Bogo_boundary_conditions}.
Then there exists a solution $(a,\phi) \in \cAphg^{\cF}(\fZ;
(\gLam^1\oplus\gLam^0)\otimes \adP)$ to
\[
	\Bogo(\backA,\backPhi) + L(a,\phi) + Q(a,\phi) = 0 
	  \mod  \ve^\infty C^\infty(\fZ; (\gLam^1\oplus \gLam^0)\otimes \adP)
\]
where
\begin{equation}
\begin{gathered}
	\cF = (F_\fD,F_\fX,F_\fB),
	\quad F_\fD = \set{(n, 2n - 4) : 2 \leq n \in \bbN}, 
	\\ F_{\fX} = (1,0) \cup \set{(n,2n - 3) : 2 \leq n \in \bbN}, 
	\quad F_\fB = \ol{(2,0)}.
\end{gathered}
	\label{E:formal_soln_index}
\end{equation}
In other words, $\Bogo(\backA + a,\backPhi + \phi) = 0$ (along with the Coulomb
gauge condition $d^\ast_{(\backA,\backPhi)}(a,\phi) = 0$) is satisfied up to an
error which is smooth on $\fZ$ and rapidly vanishing in $\ve$.

Furthermore, the $\adPo$ components of $(a,\phi)$ are rapidly vanishing in $\rho_\fD$ and $\rho_\fB$,
and each coefficient in the expansion of $(a,\phi)$ at $\fX_j$ is $L^2$ orthogonal
to $\Null(L_\fX).$
\label{T:formal_solution}
\end{thm}

\begin{proof}
Define
\[
	N(a,\phi) = L(a,\phi) + Q(a,\phi) + \Bogo(\backA,\backPhi),
\]
so we wish to solve $N(a,\phi) = 0 \mod \ve^\infty$.  For notational
convenience, for the remainder of the section we will denote $(a,\phi)$ by a
single letter and omit the bundle $(\gLam^1\oplus \gLam^0)\otimes \adP$ from
the notation. 

Though the proof is by induction, we take the first few steps by hand to
illustrate the main idea before stating the full inductive step.  From the
construction of the initial data $(\backA,\backPhi)$, we begin with
\[
	N(0) = \Bogo(\backA,\backPhi) = \ve f,
	\quad f \in \cAphg^{2_\fD,0_\fX,\infty_\fB}(\fZ),
\]
with $\adPo$ components which are rapidly vanishing in $\rho_\fD$ and
$\rho_\fB$.  Expanding in $\rho_\fX$, We write the error as
\[
	N(0) = \ve(f' + f''), 
	\quad f' \in \cAphg^{2_\fD, 0_\fX, \infty_\fB}(\fZ),
	\quad f'' \in \cAphg^{2_\fD, 1_\fX, \infty_\fB}(\fZ).
\]

By Lemma~\ref{L:solvability}, there exists $v \in
\cAphg^{\ol{(1,0)}_\fD,0_\fX,\infty_\fB}(\fZ)$ such that $L v + f' \in
\cAphg^{\ol{(2,0)}_\fD,1_\fX,\infty_\fB}(\fZ)$ (that the index set is
$\ol{(2,0)}$ rather than $\ol{(2,1)}$ at $\fD$ follows from the fact that the
forcing term is smooth; in \eqref{E:index_solving_at_X} we have
$\ol{(2,0)}\,\ol \cup \, (1,0) \subset \ol{(1,0)}$). Since $v$ has rapidly
vanishing $\adPo$ components at $\fD \cup \fB$, and since the quadratic pairing is trivial on
$\adPz$, it follows that $Q(v) \in \cAphg^{\infty_\fD,0_\fX, \infty_\fB}(\fZ).$
Then
\[
\begin{gathered}
	\ve v \in \cAphg^\cF(\fZ),
	\\ N(\ve v) 
	= \ve( L v + f') + Q(\ve v) + \ve f'' 
	= \ve \wt h, 
	\quad \wt h \in \cAphg^{\ol{(2,0)}_\fD, 1_\fX,\infty_\fB}(\fZ).
\end{gathered}
\]
(Here we have used that $L$ commutes with $\ve$ and $Q(\ve v) = \ve^2 Q(v)$.)
Since $v$ has rapidly vanishing $\adPo$ components, so does this new error
term.

The next step is to factor out $\ve^2 = (\rho_\fD \rho_\fX)^2$ and, expanding
in $\rho_\fD$, write the new error term as
\[
	N(\ve v) = \ve^3 (h' + h''), 
	\quad h' \in \cAphg^{0_\fD,-1_\fX,\infty_\fB}(\fZ),
	\quad h'' \in \cAphg^{\ol{(1,1)}_\fD,-1_\fX, \infty_\fB}(\fZ).
\]
There exists $w \in \cAphg^{0_\fD,\ol{(0,1)}_\fX,\ol{(2,0)}_\fB}(\fZ)$
such that $Lw + \ve h' \in \cAphg^{2_\fD,\ol{(0,1)}_\fX, \infty_\fB}(\fZ)$
by Lemma~\ref{L:solvability}, and then
\[
\begin{gathered}
	\ve^2 w \in \cAphg^\cF(\fZ),
	\\ N(\ve v + \ve^2 w) 
	= \ve^2(\wt f + \wt g),
	\\ \wt f = \ve h'' \in \cAphg^{\ol{(2,1)}_\fD, 0_\fX, \infty_\fB}(\fZ)
	\\ \wt g = (Lw + \ve h') + \ve 2Q(v, w) + \ve^2 Q(w) = Lw + \ve h'
	  \in \cAphg^{2_\fD, \ol{(0,1)}_\fX, \infty_\fB}(\fZ).
\end{gathered}
\]
Here we abuse notation by confusing $Q$ and its associated bilinear form,
and we use the fact that since $w$ is supported near $\fD$ and has identically
vanishing $\adPo$ component, $Q(w,\cdot) \equiv 0.$ Again $w$, $\wt f$ and $\wt g$
have rapidly vanishing $\adPo$ components.

A key point in the iteration to follow is that we keep separate track of the
error terms which have growth in their powers of $\log \rho_\fD$ but not
$\log \rho_\fX$, such as $\wt f$ above, and those which have growth in
their powers of $\log \rho_\fX$ but not $\log \rho_\fD$, such as $\wt g$.

Finally, observe that $\wt g$ has leading order at $\fX$ given by
$\log \rho_\fX$. Using the identity $\log \rho_\fX = \log \ve - \log \rho_\fD$, this may be 
effectively removed, allowing us to write
\[
	\wt g = \log \ve g_0 + g_1,
	\quad g_j \in \cAphg^{(2,j)_\fD, \ol{(0,0)}_\fX, \infty_\fB}(\fZ),
\]
where $(n,m)$ denotes the smallest smooth index set containing $(n,m)$, namely
$\set{(k,l) : n \leq k, 0 \leq l \leq m}$.

Now we begin the induction. Suppose that we have $u_n \in \cA^\cF(\fZ)$ with 
\[
\begin{gathered}
	N(u_n) = \ve^n(\log \ve)^{2n-3}(g_0) + \ve^n(\log \ve)^{2n-4}(g_1 + f_1) + \\
	\cdots + \ve^n(\log \ve)^0(g_{2n-3} + f_{2n-3}) + \cO(\ve^{n+1}) \\
	f_j \in \cAphg^{\ol{(2,j)}_\fD,0_\fX,\infty_\fB}(\fZ), 
	\quad g_j \in \cAphg^{(2,j)_\fD,\ol{(0,0)}_\fX,\infty_\fB}(\fZ),
\end{gathered}
\]
all having rapidly vanishing $\adPo$ components near $\fD \cup \fB$,
and where $\cO(\ve^{n+1})$ denotes a finite number of terms of the form above with $n$ replaced
by $m > n$.
Furthermore, suppose that, with respect to pairing by the bilinear form,
\begin{equation}
	Q(u_n,\cdot) = \sum_{m = 1}^{n-1} \ve^m\sum_{j=0}^{2m - 3}(\log \ve)^jQ(\wt u_{m,j},\cdot),
	\quad \wt u_{m,j} \in \cAphg^{\infty_\fD, 0_\fX,\infty_\fB}(\fZ),
	\label{E:inductive_bilinear_hypothesis}
\end{equation}
i.e., with respect to multiplication, $u_n$ has an expansion in $\ve^m(\log \ve)^j$ 
with coefficients which are smooth at $\fX$ and rapidly vanishing elsewhere.
The case $n = 2$ is furnished by $u_2 = \ve v + \ve^2 w$ above,
with $f_1 = \wt f$.

Expanding in $\rho_\fX$, we write
\[
\begin{aligned}
	g_j &= g'_j + g''_j, & g'_j &\in \cAphg^{(2,j)_\fD, 0_\fX, \infty_\fB}(\fZ),
	  & g''_j &\in \cAphg^{(2,j)_\fD, \ol{(1,1)}_\fX, \infty_\fB}(\fZ),
	\\ f_j &= f'_j + f''_j, & f'_j &\in \cAphg^{\ol{(2,j)}_\fD, 0_\fX, \infty_\fB}(\fZ),
	  & f''_j &\in \cAphg^{\ol{(2,j)}_\fD, 1_\fX, \infty_\fB}(\fZ).
\end{aligned}
\]	
Invoking Lemma~\ref{L:solvability}, for each $j$ there exists
$v_j \in \cAphg^{\ol{(1,j+1)}_\fD, 0_\fX, \infty_\fB}(\fZ)$ such that 
$L v_j + (g'_j + f'_j) \in \cAphg^{\ol{(2,j+1)}_\fD, 1_\fX, \infty_\fB}(\fZ)$ for each $j$.
Then
\[
\begin{gathered}
	\ve^n v := \ve^n(\log \ve)^{2n-3} v_0 + \cdots + \ve^n(\log \ve)^0 v_{2n-3} \in \cAphg^\cF(\fZ), 
	\\ N(u_n + \ve^n v) = \ve^n(\log \ve)^{2n-3}(\wt h_1 + \wt k_0) 
	  + \ve^n(\log \ve)^{2n - 4}(\wt h_2 + \wt k_1) + 
	\\ \cdots + \ve^n(\log \ve)^0(\wt h_{2n-2} + \wt k_{2n-3}) + R_{n+1}
	\\ \wt k_j = g''_j \in \cAphg^{(2,j)_\fD, \ol{(1,1)}_\fX, \infty_\fB}(\fZ),
	\\ \wt h_j = (L v_{j-1} + g'_{j-1} + f'_{j-1}) + f''_j \in \cAphg^{\ol{(2,j)}_\fD,1_\fX,\infty_\fB}(\fZ),
	\\ R_{n+1} = 2Q(u_n,\ve^n v) + Q(\ve^n v) = \cO(\ve^{n+1})
\end{gathered}
\]

Next, before solving at $\fD$, we factor out $\ve^2 =
(\rho_\fD\rho_\fX)^2$ and use the identity $(\log \rho_\fD)^j = (\log \ve - \log
\rho_\fX)^j$ to remove the leading powers of $\log \rho_\fD$ and distribute them as
powers of $\log \ve$ and $\log \rho_\fX$. Thus we may write
\[
\begin{gathered}
	N(u_n + \ve^n v) = \ve^{n+2}(\log\ve)^{2n-2}(h_0) + \ve^{n+2}(\log \ve)^{2n-3}(h_1 + k_1) + \\
	\cdots + \ve^{n+2}(\log \ve)^0(h_{2n-2} + k_{2n-2}) + R_{n+1}, \\
	h_j \in \cAphg^{\ol{(0,0)}_\fD,(-1,j)_\fX,\infty_\fB}(\fZ), 
	  \quad k_j \in \cAphg^{0_\fD,\ol{(-1,j)}_\fX,\infty_\fB}(\fZ).
\end{gathered}
\]

Expanding in $\rho_\fD$, we write
\[
\begin{aligned}
	h_j &= h'_j + h''_j, & h'_j &\in \cAphg^{0_\fD, (-1,j)_\fX, \infty_\fB}(\fZ),
	  & h''_j &\in \cAphg^{\ol{(1,1)}_\fD, (-1,j)_\fX, \infty_\fB}(\fZ),
	\\ k_j &= k'_j + k''_j, & k'_j &\in \cAphg^{0_\fD, \ol{(-1,j)}_\fX, \infty_\fB}(\fZ),
	  & k''_j &\in \cAphg^{1_\fD, \ol{(-1,j)}_\fX, \infty_\fB}(\fZ).
\end{aligned}
\]
By Lemma~\ref{L:solvability}, there exist $w_j \in
\cAphg^{0_\fD,\ol{(0,j+1)}_\fX,\ol{(2,0)}_\fB}(\fZ)$ such that $L w_j + \ve
(h'_j + k'_j) \in \cAphg^{1_\fD, \ol{(0,j+1)}_\fX,\infty_\fB}(\fZ)$ for each
$j$. Then
\[
\begin{gathered}
	\ve^{n+1} w := \ve^{n+1}(\log \ve)^{2n-2}  w_0 + \cdots + \ve^{n+1}(\log \ve)^0  w_{2n - 2} 
	  \in \cAphg^\cF(\fZ), 
	\\ N(u_n + \ve^n  v + \ve^{n+1} w) = \ve^{n+2}(\log \ve)^{2n-2}(\wt f_0 + \wt g_1) 
	  + \ve^{n+2}(\log \ve)^{2n-3}(\wt f_1 + \wt g_2) + 
	\\ \cdots + \ve^{n+2}(\log \ve)^0(\wt f_{2n-2} + \wt g_{2n-1}) + R_{n+1}
	\\ \wt f_j = h''_j \in \cAphg^{\ol{(1,1)}_\fD,(-1,j)_\fX,\infty_\fB}(\fZ), 
	\\ \wt g_j = (Lw_{j-1} + \ve (h'_{j-1} + k'_{j-1})) + k''_j 
	  \in \cAphg^{1_\fD,\ol{(-1,j)}_\fX,\infty_\fB}(\fZ).
\end{gathered}
\]	
Here we use that $Q(w,\cdot) \equiv 0$ by the fact that $w$ is supported
near $\fD$ with vanishing $\adPo$ component.
Finally, we set $u_{n+1} = u_n + \ve^n v + \ve^{n+1}w$ and rewrite the leading $\log \rho_\fX$ terms
as $\log \ve - \log\rho_\fD$, after which we have
\[
\begin{gathered}
	N(u_{n+1}) = \ve^{n+1}(\log \ve)^{2(n+1) - 3} g_0 + \ve^{n+1}(\log \ve)^{2(n+1) - 4} (g_1 + f_1) +
	\\ \cdots + \ve^{n+1}(\log \ve)^0(g_{2(n+1)-3} + f_{2(n+1)-3}) + \cO(\ve^{n+2}),
	\\ f_j \in \cAphg^{\ol{(2,j)}_\fD,0_\fX,\infty_\fB}(\fZ), 
	\quad g_j \in \cAphg^{(2,j)_\fD,\ol{(0,0)}_\fX,\infty_\fB}(\fZ).
\end{gathered}
\]
Here the $f_j$ include the leading terms from $R_{n+1}$, which has the form
\[
\begin{gathered}
	R_{n+1} = \ve^{n+1}\sum_{j=0}^{2n-3}(\log \ve)^j Q(\wt u_{1,0},v_j) + \cO(\ve^{n+2}),
	\\ Q(\wt u_{1,0},v_j) \in \cAphg^{\infty_\fD, 0_\fX, \infty_\fB}(\fZ),
\end{gathered}
\]
and $\cO(\ve^{n+2})$ is used in the sense above. Note that $u_{n+1}$ satisfies
\eqref{E:inductive_bilinear_hypothesis} since $Q(\wt u_{m,j}, v_k) \in
\cAphg^{\infty_\fD, 0_\fX, \infty_\fB}(\fZ)$ and $Q(w_j, \cdot) = 0$. This
completes the induction.
\end{proof}

%% file: ideal.tex
In this section we determine the moduli space of ideal monopoles. More
precisely, we consider first the effect of passsng to equivalence classes with
respect to the appropriate notion of gauge transformation in \S\ref{S:1.29.10.15}, and
then, after discussing the moduli space of configurations $\ul \zeta$ in
\S\ref{S:cfg}, we consider in \S\ref{S:mon_ideal} the effect of allowing the
configuration data $\ul \zeta$ of an ideal monopole to vary.

\subsection{Moduli of ideal monopoles for a fixed configuration of points}
\label{S:1.29.10.15}

The definition of an ideal monopole was given in
\S\ref{S:single_ideal}. We shall now introduce the appropriate
notions of framing and gauge transformation and define the moduli space of ideal monopoles for
a fixed configuration $\ul \zeta = (\zeta_1,\ldots,\zeta_N)$ of points.
Having fixed $\ul \zeta$, let $Z = Z(\ul{\zeta})$ be the corresponding gluing
space.  Fix integers $k_0\geq 0$, $k_j\geq 1$ ($j=1,\ldots, N$) and
set
\begin{equation}\label{e11.29.10.15}
	k_\infty = k_0 + k_1 + \cdots + k_N.
\end{equation}

\begin{defn}
\label{D:ideal_monopole}
Let $(\ol{A}, \ol{\Phi})$ be admissible monopole boundary data over $S_\infty$
of charge $k_\infty$. Let $\iota$ be an ideal monopole 
represented by the smooth configuration $(A,\Phi)$ on $Z$.  
We denote the restriction of $(A,\Phi)$ to $\fD$ and $\fX_i$ by $(A_\fD,\Phi_\fD)$ and $(A_i,\Phi_i)$, respectively.
We say that $\iota$ is {\em framed} by $(\ol{A},\ol{\Phi})$ if $(A,\Phi)|B\cap D =
(\ol{A},\ol{\Phi})$. We say that $\iota$ is {\em centered} if for $j=1,\ldots,
N$, the $(A_j,\Phi_j)$ represent an element of the centered moduli space
$\moncM_{k_j}$ (but {\em not} $j = 0$). The group of {\em framed ideal gauge transformations} $\gauG_I$
is the group of all restrictions to $\fb^{-1}(0)$ of elements of the group,
$\gauG_B = \set{g \in C^\infty(Z(\ul \zeta); \Aut(P)) : g  \rst B = 1}$, of
gauge transformations acting by the identity on $B$.  

The moduli space of ideal monopoles (framed at $S_\infty$) for this fixed $\ul
k$ and $\ul \zeta$ is denoted by $\idmon_{\ul \zeta,\ul k}$, and sometimes abbreviated to $\cI_{\ul \zeta}$.
The moduli
space of centered ideal monopoles is denoted by $\idmonc_{\ul \zeta,\ul k}$.
\end{defn}

\begin{prop}
For fixed $\ul{\zeta}$, $\ul{k}$ and framing $(\ol{A},\ol{\Phi})$, 
there exist diffeomorphisms
\begin{equation}	\label{E:idmon_to_mon_product}
\begin{gathered}
	\idmon_{\ul \zeta} \stackrel \cong \to \monM_{k_0}\times \monM_{k_1}\times \cdots \times \monM_{k_N},
	\\ \idmonc_{\ul \zeta} \stackrel \cong \to \monM_{k_0}\times \moncM_{k_1}\times \cdots \times \moncM_{k_N}.
\end{gathered}
\end{equation}
\end{prop}

\begin{proof}
Pick framings (admissible boundary data) $(\ol{A}_j,\ol{\Phi}_j)$ over
$S_j = D\cap X_j$.  For each $j$, composing with an ideal gauge transformation
if necessary,  we may assume that $(A_j,\Phi_j)|S_j
= (\ol{A}_j,\ol{\Phi}_j)$; fixing such data reduces the gauge group to
a subgroup $\ol{\gauG}_I$ of
gauge transformations $\gamma$ over $\fb^{-1}(0)$ such that
$\gamma|S_j$ lies in the $\UU(1)$ subgroup which preserves
$(\ol{A}_j,\ol{\Phi}_j)$.  Thus we have an exact sequence
\begin{equation}\label{e23.30.10.15}
1\to \gauG_{I,0} \to \ol{\gauG}_I \to \UU(1)^{N+1} \to 1
\end{equation}
where $\gauG_{I,0}$ is the subgroup of gauge transformations which are
the identity at all corners.  Thus
\begin{equation}\label{e21.30.10.15}
\gauG_{I,0} = \gauG_0\times \cdots \times \gauG_N
\end{equation}
and if we divide by this subgroup first we get the product 
\begin{equation}\label{e22.30.10.15}
(C_1\times \cdots \times C_N) \times (\monM_{k_0}\times \monM_{k_1}\times \cdots \times \monM_{k_N})
\end{equation}
where the product of principal $\UU(1)$-spaces $C_1\times\cdots\times C_N$ is the moduli space 
\eqref{e2.13.10.15} of Dirac monopoles
over $D$, framed at all boundary faces.

The $j$-th $\UU(1)$ factor in \eqref{e23.30.10.15} acts in the obvious
way simultaneously on the $j$-th $S^1$ factor and the $j$-th moduli
space (see \eqref{e6.5.8.15}) in \eqref{e22.30.10.15}.  Dividing by $\UU(1)^{N+1}$,
we are left with the product of moduli spaces, as claimed.
\end{proof}

\begin{rmk}
Though diffeomorphisms \eqref{E:idmon_to_mon_product} exist, they are not canonical;
different choices amount to trivializations of the circles $C_j$ in \eqref{e22.30.10.15}.
\end{rmk}

We next consider the global topological behavior of the parameters in our
gluing construction, and the circle factors in \eqref{e22.30.10.15} will play
a significant role.

\subsection{Configurations of points} \label{S:cfg}
Denote by 
\[
	\cfg N = \cfg N(\bbR^3) = \set{(z_1,\ldots,z_N) \in (\bbR^3)^N : z_i \neq z_j,\ i \neq j}
\]
the configuration space of $N$ distinct points in $\bbR^3$. We denote the
components of each euclidean coordinate by $z_j = (z^1_j,z^2_j,z^3_j)$.

We write
\[
	\ncfg N = \set{(z_1,\ldots,z_N) : z_i \neq z_j \neq 0}
\]
for the configurations of distinct {\em nonzero} points in $\bbR^3$. By reindexing,
we have an identification of this space as a subset of $\cfg{N+1}$:
\begin{equation}
	\ncfg N \cong \set{(z_0,z_1,\ldots,z_N) \in \cfg{N+1} : z_0 = 0},
	\label{E:ncfg_in_cfg}
\end{equation}
which is evidently a homotopy retraction. 

This space may be decomposed by splitting off an overall scaling factor:
\begin{equation}
	\ncfg N \cong \ncfgs N \times (0,\infty)_\ve,
	\label{E:cfg_split_scaling}
\end{equation}
where
\[
	\ncfgs N = \Big\{(\zeta_1,\ldots,\zeta_N) \in \ncfg N : 
	  \textstyle\sum_i \abs{\zeta_i}^2 = 1\Big\}
	\cong \ncfg N / (0,\infty) 
\]
represents configurations of nonzero points up to scaling, with the quotient by
the scaling action
\[
	(0,\infty) \times \bbR^3 \ni (\ve, z) \mapsto z/\ve \in \bbR^3,
\]
The isomorphism \eqref{E:cfg_split_scaling} is given by
\begin{equation}
	(z_1,\ldots,z_N) = (\ve^{-1}\zeta_1,\ldots,\ve^{-1}\zeta_N),
	\label{E:z_zeta}
\end{equation}
which we will frequently make use of below. There is again a homotopy
retraction using \eqref{E:cfg_split_scaling}, so that $\ncfgs N \sim \ncfg N
\sim \cfg{N+1}$.

We partially compactify $\ncfg N$ to the space 
\[
	\bncfg N := \ncfgs N \times [0,\infty),
\]
where the set $\ncfgs N \times \set 0$ represents configurations of points
which have gone off to infinity. Note that this captures both the directions
$\zeta_i/\abs{\zeta_i} \in \bbS^2 = \pa\ol{\bbR^3}$ of the points as well as
their ``relative velocities'' 
\[
	(\abs {\zeta_1},\ldots,\abs {\zeta_N}) \in (0,1)^N, \quad 
	\abs{\zeta_1}^2 + \cdots + \abs{\zeta_N}^2 = 1.
\]

We are interested in scattering tangent vectors and vector fields on $\bncfg N$, which are
evidently generated by $\big\{\ve^2\pa_{\ve},\ve \pa_{\zeta^i_j}\big\}$, (subject to a relation
coming from the condition $\sum_j \abs{\zeta_j}^2 = 1$). However, there is another more
convenient frame. Indeed, a simple computation using \eqref{E:z_zeta} proves the following:

\begin{prop}
The tangent vectors 
\[
	\big\{\pa_{z_j^i} : j = 1,\ldots,N, \ i = 1,2,3\big\} \subset T_{\ul z} \ncfg N
\]
determine a global frame for the bundle $T \ncfg N$, and extend
by continuity to a global frame for the scattering tangent bundle $\scT \bncfg N$,
giving a trivialization
\begin{equation}
	\scT \bncfg N \cong (\bbR^3)^N\times \bncfg N.
	\label{E:scTcfg_trivn}
\end{equation}
\label{P:scTcfg_frame}
\end{prop}

\subsection{Moduli of ideal monopoles}\label{S:mon_ideal}

We now consider the moduli space of ideal monopoles, considering their
configurations as part of the moduli. To this end, we define here a provisional
global version of the gluing space from \S\ref{S:gluing} which fibers over the
partial compactification $\bncfg N$ of $\ncfg N$. (The gluing space will be
further enlarged in \S\ref{S:global} to fiber over a bigger parameter space
involving the ideal monopoles themselves.)

Let $\oX = \ol{\bbR^3}$ and begin with the product 
\[
	\Z_0 = \oX \times \bncfg N= \oX\times \ncfgs N \times [0,\infty).
\]
From \eqref{E:cfg_split_scaling}, the interior of $\Z_0$ is identified with the
space $\bbR^3\times \ncfg N$ with coordinates $(z,z_1,\ldots,z_N) =
(z,\zeta_1/\ve,\ldots,\zeta_N/\ve)$. The {\em vertical diagonals} $\set{z =
z_j}$, $j = 1,\ldots,N$ extend to $\Z_0$, where they meet the boundary in the
corner $\pa \oX \times (\ncfgs N \times \set 0)$. As before, the first step is to set
\[
	\Z_1 = [\Z_0; \pa \oX \times (\ncfgs N \times \set 0)].
\]
The front face, which we denote by $\D_1$, is diffeomorphic to
$\bbS^2\times[0,\infty]\times \ncfgs N$, the leftmost factors of which we
identify with the space $[\ol{\bbR^3}; \set 0]$. This blow-up resolves the
vertical diagonals, in the sense that they now meet the boundary transversally
over the interior of $\D_1$. Indeed, the Euclidean coordinate $\zeta = \ve z$
extends over $\ve = 0$ to a coordinate on the interior of $\D_1$, identified
with $\bbR^3 \setminus \set 0$, and then the boundary of each 
vertical diagonal $\set{z = z_j}$ is the submanifold $\cP_j = \set{\zeta = \zeta_j} \subset
\D_1$. We blow-up the boundaries of these diagonals, setting
\[
	\Z' = [\Z_1; \cP_1,\ldots,\cP_N].
\]
We denote the new front faces by $\X'_j$, $j = 1,\ldots,N$ and the lift of
$\D_1$ by $\D'$. The lift of the original faces $\oX \times (\ncfgs N \times
\set 0)$ and $\pa \oX \times (\ncfgs N \times [0,\infty))$ are denoted $\X'_0$
and $\B'$, respectively. We set $\gS'_j = \D' \cap \X'_j$ for $j = 0,\ldots, N$
and $\gS'_\infty = \D' \cap \B'$.

Note that $\D'$ fibers over $\ncfgs N$ and there is a natural identification
\[
	\D' \cong [\ol{\bbR^3}\times \ncfgs N; 
	  \set{\zeta = 0},\set{\zeta = \zeta_1},\ldots,\set{\zeta = \zeta_N}].
\]
The fact that $\set{\zeta = 0}$ is blown up here (in fact already in $\D_1$) is
consistent with the identification \eqref{E:ncfg_in_cfg}.

For each fixed $\ul \zeta = (\zeta_1,\ldots,\zeta_N) \in \ncfgs N$, the fiber of $\Z'$ over 
$\ul \zeta$ is equal to the gluing space $\fZ = \fZ(\ul \zeta)$ constructed in \S\ref{S:gluing}. 
We consider now the problem of obtaining families of ideal monopoles over the $\ve = 0$ boundaries
of $\Z'$ parameterized by their configurations $\ul\zeta$.

\begin{lem}
For each $j = 0,\ldots,N$, there is a canonical diffeomorphism
\begin{equation}
	\X'_j \cong \ol{\bbR^3} \times \ncfgs N.
	\label{E:triv_Xj}
\end{equation}
\label{L:triv_Xj}
\end{lem}
\begin{proof}
The difference function
\[
	\mathring \Z' \cong \bbR^3\times \ncfg N \ni (z,z_1,\ldots,z_N) 
	\mapsto w = z - z_j \in \bbR^3
\]
extends to a smooth, bounded function on the interior of $\X'_j$. Indeed, $\zeta
= \ve z$ serves as an interior coordinate on $\D_1$, after which $\X'_j$ arises
from the blow-up of $\zeta - \zeta_j = \ve(z - z_j)$ at $\ve = 0$, recovering
the Euclidean coordinate $w = z - z_j$ on the interior. Along with the
projection $\X'_j \to \ncfgs N$, this gives a diffeomorphism $\mathring \X'_j
\cong \bbR^3\times\ncfgs N$. As for the boundary of $\X'_j$, the function $\zeta
- \zeta_j : \D_1 \to \bbR^3$ lifts to a map from $\D'$ to $[\bbR^3; \set 0]$
sending $\gS_j$ to the front face of the blow-up of $\set 0$, which along with
$\D' \to \ncfgs N$ gives a diffeomorphism $\gS'_j \cong \bbS^2 \times \ncfgs N$
which is consistent with the radial compactification of $\bbR^3 \times \ncfgs
N \cong \mathring \X'_j$.
\end{proof}

\begin{lem}
The cohomology group $H^2(\D'; \bbZ) = H^2(\mathring \D'; \bbZ)$ splits as a direct sum
\[
	H^2(\D'; \bbZ) \cong \bbZ^{N+1} \oplus H^2(\ncfgs N; \bbZ),
	\quad H^2(\ncfgs N; \bbZ) = \bbZ^{N(N+1)/2}
\]
with the first factor representing the cohomology of the fiber $\bbR^3
\setminus \set{0,\zeta_1,\ldots,\zeta_N}$ of $\mathring \D'$.  
In terms of the generator $\omega$ of $H^2(\bbR^3\setminus \set0; \bbZ) =
\bbZ$, generators of $H^2(\D'; \bbZ)$ are given by
\begin{equation}
\begin{gathered}
	\xi_j = g_j^\ast \omega, 
	\quad j = 0,\ldots,N,
	\\ g_j : \mathring \D' \ni (\zeta,\zeta_1,\ldots,\zeta_N) 
	  \mapsto \zeta - \zeta_j \in \bbR^3 \setminus \set 0
\end{gathered}
	\label{E:cohom_D_gj}
\end{equation}
for the first factor, and 
\begin{equation}
\begin{gathered}
	\eta_{ij} = f_{ij}^\ast \omega,
	\quad 0 \leq i < j \leq N,
	\\ f_{ij}: \mathring \D' \ni (\zeta,\zeta_1,\ldots,\zeta_N) 
	  \mapsto \zeta_i - \zeta_j \in \bbR^3 \setminus \set 0
\end{gathered}
	\label{E:cohom_D_fij}
\end{equation}
for the second, where we define $\zeta_0 = 0.$

There are canonical de Rham representatives of $\xi_j$ and $\eta_{ij}$ given by
the pullbacks $g_j^\ast (\tau)$ and $f_{ij}^\ast (\tau)$, 
where
\begin{equation}
	\tau = \star d \big( \tfrac 1 r\big) \in C^\infty(\bbR^3 \setminus \set 0; \Lambda^2),
	\label{E:dR_gen_sphere}
\end{equation}
is the generator of $H^2(\bbR^3 \setminus \set 0; \bbR)$.
\label{L:cohom_D}
\end{lem}
\begin{rmk}
It is convenient to allow $i > j$ as well for the $\eta_{ij}$, which are then 
defined via $\eta_{ij} = - \eta_{ji}$; this follows from the action by $-1$ 
of the antipodal map in $\bbR^3 \setminus \set 0$ on $H^2(\bbR^3\setminus \set 0; \bbZ)$.
\end{rmk}
\begin{proof}
Consider the configuration spaces $\cfg m = (\bbR^3)^m \setminus \Delta$ for
arbitrary $m$, where here $\Delta$ denotes the union of all pairwise diagonals.
These admit fiber bundle structures
\begin{equation}
	\cfg m \to \cfg {m-1}
	\label{E:V_fibn}
\end{equation}
with fiber $\bbR^3 \setminus \set{z_1,\ldots,z_m}$,
and it follows by induction on $m$ and the spectral sequence for
\eqref{E:V_fibn} that $H^1(\cfg m; \bbZ) = H^3(\cfg m; \bbZ) = 0$ for all $m$ and 
\[
\begin{aligned}
	H^2(\cfg {m+1}; \bbZ) &= H^2(\bbR^3 \setminus \set{z_1,\ldots,z_m}; \bbZ)\oplus H^2(\cfg {m}; \bbZ)
	\\&= \bbZ^{m}\oplus \bbZ^{m(m-1)/2} = \bbZ^{m(m+1)/2}.
\end{aligned}
\]
Moreover, $H^2(\cfg m; \bbZ)$ is generated by the pullbacks of the generator
$\omega$ of $H^2(\bbR^3 \setminus \set 0; \bbZ)$ via the maps
\[
	\cfg m \to \bbR^3 \setminus 0,
	\quad (z_1,\ldots,z_m) \mapsto z_i - z_j, \quad 1 \leq i < j \leq m.
\]

As a manifold with boundary, $\D'$ has the same cohomology as its interior,
$\mathring \D'$, and the fiber bundle $\mathring \D' \to \ncfgs N$ is
equivalent to the restriction of $\cfg {N+2} \to \cfg {N+1}$ over the subspace
$\ncfgs N \subset \cfg {N+1}$ (which is a homotopy retraction), via the relabelling
$(\zeta, 0, \zeta_1,\ldots,\zeta_N) = (z_1,\ldots,z_{N+2})$.
\end{proof}

A family of ideal monopoles represented by data on the fibers of
$\Z' \to \bncfg N$ requires, as part of its definition, a family of
$\SU(2)$-Dirac connections on $\D'$. We take a moment to consider the existence
of such connections.

In the first place, we require a $\UU(1)$ bundle $Q' \to \D'$ whose restriction
to each fiber over $\ncfgs N$ has class $(k_0,\ldots,k_N) \in H^2([\ol{\bbR^3};
\set{0,\zeta_1,\ldots,\zeta_N}]; \bbZ)$; that such bundles exist follows from
Lemma~\ref{L:cohom_D}.  We call a bundle $Q' \to \D'$ with
\[
	[Q'] = \sum_{j=0}^N k_j\,\xi_j \in H^2(\D'; \bbZ),
\]
a {\em universal Dirac bundle}.

Next, we require a connection on $Q'$ restricting to a Dirac connection
fiberwise, which is to say that $a$ solves the fiberwise abelian Bogomolny
equation with respect to some Higgs field $\phi$. Set
\[
	\phi_{\D'} = \sum_{i=0}^N k_i\, g_i^\ast \big(\tfrac 1 r\big)
	=\sum_{i=0}^n \frac{k_i}{\abs{\zeta - \zeta_i}},
\]
and consider the 2-form
\begin{equation}
	F_a = \sum_{j=0}^N k_j\, g_j^\ast (\tau).
	\label{E:global_Dirac_curv}
\end{equation}
Restricted to a fiber, $\fD'$, of $\D' \to \ncfgs N$, this satisfies $F_a \rst
\fD' = \star d\phi_{\D'} \rst \fD'$ by \eqref{E:dR_gen_sphere}. Evidently
$[F_a] = [Q'] \in H^2(\D'; \bbR)$, so there exists a connection $a$ with
curvature $F_a$, and since $H^1(\D'; \bbZ) = 0$, the pair $(Q',a)$ is unique up
to isomorphism. In fact, pulling back a Dirac framing $\ol a$ on $\bbS^2$ to
$\gS'_\infty \cong \bbS^2 \times \ncfgs N$, we may assume that $a \rst
\gS'_\infty = \ol a$, and two such connections are uniquely gauge equivalent.
We call such an $a$ a {\em universal Dirac connection}, and the association of
$a$ to a connection $A$ on $Q' \times_{\UU(1)} \SU(2)$ will be called a {\em
universal $\SU(2)$-Dirac connection}.

For each $j = 0,\ldots, N$ we likewise define a (unique up to isomorphism)
circle bundle with connection 
\[
	(L_j, a_j) \to \ncfgs N,
	\quad [L_j] = \sum_{i \neq j} k_i \eta_{ji} \in H^2(\ncfgs N; \bbZ),
\]
with curvature
\[
	F_{a_j} = \sum_{i \neq j} k_i f_{ji}^\ast (\tau)
\]
over $\ncfgs N$. We call $L_j$ a {\em Gibbons-Manton circle factor}, and make the following

\begin{defn}
The (weighted) {\em Gibbons-Manton torus bundle} corresponding to $\ul k$ is
\begin{equation}
	\TGM = \bigoplus_{j=0}^N L_j \to \ncfgs N.
	\label{E:TGM}
\end{equation}
We equip $\TGM$ with the connection $\bigoplus_{j=0}^N a_j$.
\label{D:TGM}
\end{defn}

\begin{prop}
Fix a degree $k_j$ bundle $Q_{k_j} \to \bbS^2$ with any Dirac framing $\ol a_{k_j}$.
Then with respect to the trivialization $\gS'_j \cong \bbS^2\times \ncfgs N$, the restriction of a universal
Dirac bundle $(Q',a)$
to $\gS'_j$ admits an isomorphism
\begin{equation}
\begin{tikzpicture}[->,>=to,auto,baseline=(current bounding box.center)]
\matrix (m) [matrix of math nodes, column sep=1cm, row sep=1cm, text height=2ex, text depth=0.25ex]
{ (Q',a) \rst_{\gS'_j} & \pr_1^\ast (Q_{k_j},\ol a_{k_j})\otimes \pr_2^\ast (L_j,a_j) \\ \gS'_j & \bbS^2\times \ncfgs N \\};
\path (m-1-1) edge node {$\cong$} (m-1-2); 
\path (m-1-2) edge (m-2-2); 
\path (m-2-1) edge node {$\cong$} (m-2-2); 
\path (m-1-1) edge (m-2-1); 
\end{tikzpicture}
	\label{E:Q_framing_downstairs}
\end{equation}
\label{P:GM_framing}
\end{prop}

\begin{proof}
This is a computation in cohomology using the generators $\xi_j$ and
$\eta_{ij}$. As noted in the proof of Lemma~\ref{L:triv_Xj}, the isomorphism
$\gS'_j \cong \bbS^2 \times \ncfgs N$ is obtained by what amounts to the change
of variables $(\zeta,\zeta_1,\ldots,\zeta_N) \to (\zeta' = \zeta - \zeta_j,
\zeta_1,\ldots,\zeta_N)$ on $\mathring \D_1\cong (\bbR^3 \setminus \set
0)\times \ncfgs N$, after which $\gS'_j$ is identified with the front face of
the blow-up of $\zeta' = 0$.

In terms of this change of coordinates, the map $g_j$ from \eqref{E:cohom_D_gj}
becomes simply
\[
	g_j(\zeta',\zeta_1,\ldots,\zeta_N) = \zeta' \in \bbR^3\setminus \set 0,
\]
so its extension, $\wt g_j : \D' \to [\ol{\bbR^3}; \set 0]$, is identified over
$\gS'_j$ with the projection map
\[
	\wt g_j \cong \pr_1 : \gS'_j \cong \bbS^2\times \ncfgs N \to \bbS^2.
\]
On the other hand, for $i \neq j$, the map $g_i$ can be written as 
\[
	g_i = g_j - f_{ij} : (\zeta',\zeta_1,\ldots,\zeta_N) = \zeta' - (\zeta_i - \zeta_j).
\]
The extension, $\wt g_i$, to $\D'$ never vanishes on $\gS'_j$ (it vanishes
precisely over $\gS'_i$, which is disjoint), while $\zeta' \equiv 0$ over
$\gS'_j$ as an $\bbR^3$-valued function, so we obtain
\[
	\wt g_i\rst_{\gS'_j} \cong - f_{ij} \circ \pr_2 : \gS'_j \cong \bbS^2 \times \ncfgs N \to \bbR^3 \setminus \set0.
\]
It follows that the cohomology class $[Q'] = \sum_i k_i\,\xi_i$ restricts over
$\gS'_j$ to
\[
	[Q'] =  k_j \xi_j - \sum_{i \neq j}k_i \eta_{ij} \cong k_j\omega \oplus [L_j]
	\in H^2(\bbS^2\times \ncfgs N) = H^2(\bbS^2)\oplus H^2(\ncfgs N).
\]

The curvature form $F_a$ behaves similarly. Indeed,the restriction of the form
\eqref{E:dR_gen_sphere} to the front face of the blow-up of the origin in
$\bbR^3$ is the standard volume form on $\bbS^2$, so it follows that
\[
	F_a \cong \pr_1^\ast (F_{\ol a_{k_j}}) +  \pr_2^\ast(F_{a_j}),
\]
and since $H^1(\gS_j'; \bbR) = 0$, the respective bundles with connection $(Q',a)$ and
$\pr_1^\ast(Q_{k_j}, \ol a_{k_j})\otimes \pr_2^\ast(L_j,a_j)$ are intertwined by an
isomorphism which is unique up to a constant gauge transformation.
\end{proof}

Proposition~\ref{P:GM_framing} shows that it is not possible
to choose global framings for a Dirac monopole at the $\gS'_j$, $j =
0,\ldots,N$. Indeed, such a framing over $\gS'_j$  would amount to the existence of an
isomorphism $(Q,a)\rst_{\gS'_j} \cong \pr_1^\ast(Q_{k_j},\ol a_{k_j})$, which
is obstructed by the nontriviality of $L_j$.

Proceeding instead locally, we restrict to the preimage, $U'$, in $\D'$ of an
open set $U \subset \ncfgs N$ over which $\TGM$ is trivial, and then we may
assume that $(a,\phi_{\D'})\rst_{U'}$ is framed at the faces $\gS'_j \cap U'$.
Any two such Dirac monopoles are identified by a gauge transformation $g
\in C^\infty(U'; \Ad Q')$ such that $g \rst_{\gS'_\infty \cap U'} = 1$ and $g
\rst_{\gS'_j \cap U'}$ is constant on the fibers of $\gS'_j \to \ncfgs N$; such
$g$ is unique modulo the subgroup $\gauG_0(U'; Q)$ of gauge transformations
which are the identity over each $\gS'_j$. Dividing by this subgroup, we are left
with the action of the quotient group, identified with the fiberwise constant gauge transformations
on each $\gS'_j \to U$, or equivalently the group $C^\infty(U; \UU(1)^{N+1})$.

It follows that the moduli space of framed Dirac monopoles with configurations
$\ul \zeta \in U \subset \ncfgs N$ is a product 
\begin{equation}
	C_0\oplus \cdots \oplus C_N \to U
	\label{E:trivial_circles}
\end{equation}
of trivial $\UU(1)$ bundles $C_j$ over $U$. 
The isomorphism \eqref{E:Q_framing_downstairs} intertwines the action of the
fiberwise constant gauge transformations on $\gS'_j \to \ncfgs N$ with the
fiberwise $\UU(1)$ action on $L_j \to \ncfgs N$, and thus determines a natural
isomorphism
\[
	C_j \cong L_j \rst_U.
\]
We conclude the following result; compare Proposition~\ref{p1.25.10.15}.

\begin{prop}
The moduli space of framed Dirac monopoles of charge $\ul k$ with point configurations in $\ncfgs N$
is isomorphic to the Gibbons-Manton torus bundle \eqref{E:TGM}.
\label{P:global_moduli_framed_Dirac}
\end{prop}

Similarly, we have

\begin{thm}
For fixed $\ul k = (k_0,\ldots,k_N)$, the full moduli space, $\idmon_{\ul k}$, of ideal monopoles is a
fiber bundle associated to $\TGM$ with respect to the product action of
$\UU(1)^{N+1}$ on $\monM_{k_0}\times \cdots \times \monM_{k_N}$:
\begin{equation}
	\idmon_{\ul k} = \TGM \times_{\UU(1)^{N+1}} (\monM_{k_0}\times\cdots \times \monM_{k_N})
	\to \ncfgs N.
	\label{E:idmon}
\end{equation}
Here the notation means we take the quotient of $\TGM\times \monM_{k_0}\times
\cdots \times \monM_{k_N}$ by the diagonal action of $\UU(1)^{N+1}$ acting on
the right of $\TGM$ and on the left on the monopole moduli spaces.

Similarly, 
\begin{equation}
	\idmonc_{\ul k} = \TGM \times_{\UU(1)^{N+1}} (\monM_{k_0}\times \moncM_{k_1} \times \cdots \times \moncM_{k_N})
	\label{E:idmonc}
\end{equation}
\label{T:idmon_as_moduli}
\end{thm}
\begin{proof}
We prove the result for $\idmon$; the argument for $\idmonc$ is similar.
Fix a sufficiently small open subset $U \subset \ncfgs N$ and a trivialization
of $\TGM \rst_U$.  We restrict consideration to the preimage of
$U\times[0,\infty)$ in $\Z'$ without change of notation.  Suppose $(A,\Phi)$
represents a smooth family of ideal monopoles on $\Z' \to U
\times [0,\infty)$. In particular, $A \rst \D'$ is associated to some Dirac connection $a$
by an identification $P \cong Q \times_{\UU(1)} \SU(2)$.

Fix charge $k_j$ admissible monopole boundary data $(\ol A_j,\ol \Phi_j)$ on $\bbS^2$ for $j =
0,\ldots,N$, and pull these back to $\gS'_j$ using the diffeomorphism $\gS'_j
\cong \bbS^2 \times U$ and left projection.  By the trivialization of
the $L_j \rst_U$, we may suppose that $(A,\Phi) \rst \gS'_j = (\ol A_j,\ol
\Phi_j)$, composing with a gauge transformation if necessary. This reduces the
gauge freedom to the subgroup $\ol \gauG_I$ of gauge transformations $\gamma$
over $\fb^{-1}(U \times \set 0)$ such that $\gamma \rst \gS'_j$ reduces to a
$\UU(1)$ gauge transformation which is constant on each fiber over $U$. This determines an exact sequence
\[
	1 \to \gauG_{I,0} \to \ol \gauG_I \to \cH \to 1
\]
where $\gauG_{I,0} = \gauG_0(\D') \times \gauG_0(\X'_0)\times \cdots \times
\gauG_0(\X'_N)$ and $\cH = C^\infty(U; \UU(1)^{N+1})$ may be regarded as the
group of gauge transformations on a trivial $\UU(1)^{N+1}$ bundle over $U$.

Thus, we have a well-defined map
\[
\begin{gathered}
	\idmon\rst_U \to (C_0\oplus \cdots \oplus C_N)
	  \times_{\UU(1)^{N+1}}(\monM_{k_0}\times \cdots \times \monM_{k_N}),
	\\ [(A,\Phi)] \mapsto [(A_{\D'},\Phi_{\D'}), (A_0,\Phi_0),\ldots, (A_N,\Phi_N)],
\end{gathered}
\]
with $C_0\oplus \cdots \oplus C_N$ as in \eqref{E:trivial_circles}.
Here we take the quotient by gauge transformations on $\Z'$ restricting to the identity at $\B'$ on the left, and
on the right we take the quotient by $\ol \gauG_I$ on the right, first by $\gauG_{I,0}$ and then by $\cH$.

Conversely, if we fix a universal framed $\SU(2)$-Dirac monopole 
$(A'_{\D'},\Phi'_{\D'})$ on $\D'$ over $U$, then
\[
\begin{gathered}
	\monM_{k_0}\times \cdots \times\monM_{k_N}\times U \to \idmon\rst_U,
	\\ ([A_0,\Phi_0],\ldots,[A_N,\Phi_N]) \mapsto [(A'_{\D'},\Phi'_{\D'}),(A_0,\Phi_0),\ldots,(A_N,\Phi_N)]
\end{gathered}
\]
gives an inverse map.

Using the fact that $C_j \cong L_j \rst_U$ and patching together various local
trivializations for $\TGM$ over $\ncfgs  N$ gives the global result.
\end{proof}

\begin{rmk}
We proceed to define the gluing map \eqref{E:intro_gluing_map} on the moduli
space $\cI_{\ul k}$ of all ideal monopoles, as our construction works perfectly
well in this setting. However, in order to obtain a local diffeomorphism onto
moduli space, it will eventually be necessary to restrict to centered ideal
monopoles.

In addition, the correct spaces to consider should really be the quotients
$\idmonc_{\ul k} / \Sigma_{\ul k}$, where $\Sigma_{\ul{k}}$ is the the subgroup
of the symmetric group on $N$ letters which preserves the sequence
$(k_1,\ldots,k_N)$ (so $\Sigma_{\ul{k}} = \Sigma_N$ if all the $k_j$ are equal
and is equal to $\set{1}$ if they are all distinct), acting by permutation on
the configurations $\ul \zeta$ and by the obvious factor exchange on the fibers
\eqref{E:idmon_to_mon_product}. Indeed, one expects these quotient spaces, not
the  $\idmonc_{\ul k}$ themselves, to form the boundary hypersurfaces of the
compactification of $\monM_k$.  However, since our gluing map is local, and we
only aim to prove that it is a local diffeomorphism onto its image, it suffices
to work with $\idmon_{\ul k}$ for simplicity.

Henceforth we will suppress the dependence on $\ul k$ from the notation, writing
simply $\idmon$ or $\idmonc$.
\end{rmk}

Taking the product with $[0,\infty)$, we obtain the fibration 
\begin{equation}
	\varphi : \base \to \ncfgs N\times [0,\infty) = \bncfg N.
	\label{E:ideal_moduli_extension}
\end{equation}
With respect to the scattering structure on the base which was considered
above, $\base$ inherits a natural {\em fibered boundary structure}
\cite{MM_FB}. These
vector fields, denoted by $\phiV(\base)$, are defined as those vector fields
$V$ on $\base$ for which $V\ve \in \ve^2C^\infty(\base)$ and which are tangent
to the fibers of $\varphi : \idmon \to \ncfgs N$ over $\set{\ve = 0}$.
Essentially this means that $V$ behaves as a scattering vector field along the
base $\bncfg N$ and an ordinary vector field along the fibers
$\monM_{k_0}\times \cdots \times \moncM_{k_N}$. The associated tangent bundle
will be denoted $\phiT(\base)$, and we note that we have an exact sequence
\begin{equation}
	0 \smallto T\monM_{k_0}\times T\moncM_{k_1}\times \cdots \times T\moncM_{k_N}
	\smallto \phiT (\base)
	\smallto \scT \bncfg N \cong (\bbR^3)^N \smallto 0,
	\label{E:phiT_sequence}
\end{equation}
where we use Proposition~\ref{P:scTcfg_frame} to trivialize the latter space.

In fact we have a natural splitting of \eqref{E:phiT_sequence} since
\eqref{E:ideal_moduli_extension} is an associated fiber bundle to the extension of $\TGM$
to $\bncfg N$, which comes equipped with a canonical connection. Such
a connection induces a splitting of the tangent bundle sequence \eqref{E:phiT_sequence} for any associated bundle.

%% file: global.tex
We now define the universal version of the gluing space from \S\ref{S:gluing}, 
The {\em universal gluing space} is 
\[
	\Z = \varphi^\ast \Z' \to \base.
\]
This pullback factors through the lift of $\Z_0$ 
giving a map
\begin{equation}
	\Z \to \varphi^\ast \Z_0 = \oX\times \base.
	\label{E:Z_to_oXbase}
\end{equation}
We denote the lifts of $\D'$, $\X'_j$, $\gS'_j$ and $\B'$ simply by $\D$, $\X_j$, $\gS_j$ and
$\B$, respectively. As before, we fix boundary defining functions $\rho_\D$,
$\rho_\B$ and  $\rho_j = \rho_{\X_j}$ such that $\ve = \rho_\D \rho_\X$, with
$\rho_\X := \rho_0\ldots,\rho_N$.

Composing \eqref{E:Z_to_oXbase} with projections, we obtain the three fundamental
maps
\begin{align}
	\fu &: \Z \to \idmon,
	\label{E:fu}
	\\ \fb &: \Z \to \base, 
	\label{E:fb}
	\\ \piX &: \Z \to \oX.
	\label{E:pi_oX}
\end{align}
The maps $\fb$ and $\piX$ are b-fibrations, while $\fu$ is a smooth fiber bundle whose fibers are
manifolds with corners.
Indeed, each fiber of the map $\fu$ is a single parameter gluing space as
defined in \S\ref{S:Mgl}, and then $\fb$ and $\piX$ restrict over each such fiber
to the maps of the same names in \S\ref{S:gluing}.

To conform to the notation in \S\ref{S:gluing},
we observe a notational convetion whereby fibers of $\fu$ are denoted by non-calligraphic versions of
the global spaces; thus a typical fiber of $\Z$ is denoted by $\fZ$, and fibers
of $\fu \equiv \fb : \D \to \idmon$ and $\fu \equiv \fb : \X_i \to \idmon$ are denoted by
$\fD$ and $\fX_i$, respectively. 

The relevant geometric structures on $\Z$ are generated by Lie subalgebras of
$\cV(\Z)$ as in \S\ref{S:gluing}.
Within the algebra $\bV(\Z)$ of vector fields tangent to
the boundary faces of $\Z$, we let $\fuV(\Z)$ and $\fbV(\Z)$ denote the vector
fields which are additionally tangent to the fibers of $\fu$ and $\fb$,
respectively. 

Note that $\bV(\Z)$ includes vector fields in the parameter directions, while
$\fuV(\Z)$ consists solely of the b vector fields along the fibers $\fZ$, i.e.,
$\fuV(\Z)\rst_{\fZ} \equiv \bV(\fZ)$. Likewise, $\fbV(\Z)$ consists of the $\fb$ 
vector fields along the fibers $\fZ$.

The algebra $\gV(\Z)$ is defined by $\gV(\Z) = \rho_\D\rho_\B \fbV(\Z)$ and
consists of the fiberwise $\gl$ vector fields (as defined in \S\ref{S:gluing})
fiberwise. We have a filtration
\[
	\gV(\Z) \subset \fbV(\Z) \subset \fuV(\Z) \subset \bV(\Z) \subset \cV(\Z).
\]

The associated tangent bundles are defined as before, via $C^\infty(\Z; {}^\bullet T \Z) = \cV_\bullet(\Z)$,
for $\bullet \in \set{\mathrm{b},\fu,\fb,\gamma}$, and the results of \S\ref{S:gluing} carry over; namely,
$\piX^\ast(\scT \oX)$ is naturally isomorphic to $\gT \Z$, which in turn restricts over $\X_j$ to the fiberwise
scattering tangent bundle with respect to the fibration $\fX_j \to \X_j \to \idmon$, and rescaled restriction
gives an $\ve$-dependent isomorphism of $\gT \Z \rst_\D$ with the fiberwise conic tangent bundle with respect
to $\fD \to \D \to \idmon$.

\subsection{Bogomolny equation on $\Z$} \label{S:bogo_global}
As in \S\ref{S:gluing} we define the $\gamma$ metric 
\begin{equation}
	\wt g = \piX^\ast g 
	\label{E:wtg_global}
\end{equation}
and consider the Bogomolny operator
\begin{equation}
	(A,\Phi) \mapsto \Bogo(A,\Phi) = \star F_A - d_A \Phi \in \cAphg^\ast(\Z; \gLam^1\otimes \adP)
	\label{E:Bogo_on_Z}
\end{equation}
for $P = \piX^\ast P$. 

We now wish to globalize the formal construction of \S\ref{S:formal_formal}, to
the extent possible, over $\base$. The first step is the choice of a universal ``pregluing configuration.''

\begin{prop}
Fix framed monopole data $(A_j,\Phi_j)$ 
with $[(A_j,\Phi_j)] \in \in \monM_{k_j}$, $j = 0,\ldots,N$.
Then there exist neighborhoods $U_j$ of the $[(A_j,\Phi_j)]$, closed with respect to
the $\UU(1)$ actions on $\monM_{k_j}$, and a smooth {\em pregluing
configuration} $(\backA,\backPhi)$ on $\Z \rst_{\U\times [0,\infty)}$ where $\U$ is the
set
\begin{equation}
	\U = \TGM \times_{\UU(1)^{N+1}}(U_0\times \cdots \times U_N) \subset \idmon.
	\label{E:U_in_idmon}
\end{equation}
The configuration satisfies the following properties:
\begin{enumerate}
[{\normalfont (a)}]
\item 
$(A,\Phi)$ is an approximate solution to the Bogomolny equation \eqref{E:Bogo_on_Z} with error
\[
	\Bogo(\backA,\backPhi) = \cO(\rho_\X \rho_\D^3\rho_\B^\infty).
\]
\item \label{I:pregluing_tautological}
For every ideal monopole $\iota \in \U$, the ideal monopole represented by the
restriction of $(\backA,\backPhi)$ to $\ve = 0$ in the fiber $\fZ = \Z_\iota$
is precisely $\iota$.
\item
The $\gl$ covariant derivative $d_A$ on $\gLam^\ast\otimes \adP$ is diagonal to
infinite order at $\D \cup \B$ with respect to the splitting
\[
	\adP = \adPz \oplus \adPo := \sspan_\bbC \backPhi \oplus \backPhi^\perp
\]
\end{enumerate}
\label{P:global_pregluing}
\end{prop}
\begin{rmk}
The need to represent monpoles by smooth families of data $(A,\Phi)$
necessitates the restriction to the neighborhoods $U_i \subset \monM_{k_i}$;
however we are able to work globally in the base directions. Note that
$(\backA,\backPhi)$, and hence the map into moduli space $\monM_{k}$ that we
later construct, is dependent on the choices of initial representatives
$(A_i,\Phi_i)$.
\end{rmk}

\begin{rmk}
We refer to any configuration satisfying property
\eqref{I:pregluing_tautological} as a {\em tautological configuration}.
\end{rmk}

\begin{proof}
First, we take $\wt U_i \subset \monM_{k_i}$ to be the set of classes
represented by solutions of the form $(A_i + a,\Phi_i + \phi)$ on $\ol{\bbR^3}$
for sufficiently small $(a,\phi)$, where $(a,\phi)$ satisfy the Coulomb gauge
condition $\gauD^\ast_{A_i,\Phi_i}(a,\phi) = 0$ and in addition $\phi$ is $L^2$
orthogonal to $\nabla_{A_i} \Phi_i$ (see \eqref{e21.10.10.15}). Thus $(A_i +
a,\Phi_i + \phi)$ gives a slice for the semidirect product of the gauge group
and the $\UU(1)$ action on $\monM_{k_i}$. We then let $U_i \cong \wt U_i \times
\UU(1)$ consist of  the $\UU(1)$ orbits of the elements in $\wt U_i$. Defining
$\cU$ by \eqref{E:U_in_idmon}, it follows that we have a trivialization
\begin{equation}
	\U \cong \TGM \times (\wt U_0\times \cdots \times \wt U_N).
	\label{E:trivn_U}
\end{equation}

Denote by $L_j \to \cU$ the pull back of the Gibbons-Manton torus factor from
$\ncfgs N$ to $\cU$.  In light of \eqref{E:trivn_U}, this is equivalent to the
product of pull back of $L_j$ to $\TGM$, which is canonically trivial, with
$\wt U_0\times \cdots \times \wt U_N$. Thus each $L_j \to \cU$ is trivialized by our
choices of $(A_i,\Phi_i)$.

This allows us to consider global framings for ideal monopoles on
$\Z \rst \cU$. Indeed, Lemma~\ref{L:triv_Xj} gives a diffeomorphism $\X_j \cong
\ol{\bbR^3}\times \U$ with respect to which any universal Dirac bundle $Q \to
\D$, obtained by pull back from $\D'$, admits an isomorphism
\begin{equation}
\begin{tikzpicture}[->,>=to,auto]
\matrix (m) [matrix of math nodes, column sep=1cm, row sep=1cm, text height=2ex, text depth=0.25ex]
{ Q \rst_{\gS_j} & \pr_1^\ast Q_{k_j}\otimes \pr_2^\ast L_j\\ \gS_j & \bbS^2\times \U \\};
\path (m-1-1) edge node {$\cong$} (m-1-2); 
\path (m-1-2) edge (m-2-2); 
\path (m-2-1) edge node {$\cong$} (m-2-2); 
\path (m-1-1) edge (m-2-1); 
\end{tikzpicture}
	\label{E:triv_Q_globally}
\end{equation}
By triviality of $L_j \to \cU$, we may consider monopole framings $(\ol A_j,\ol
\Phi_j)$ on $\gS_j \rst \U$ which are pulled back from $\bbS^2$.

We equip each $\X_j$
with a canonical smooth family of monopoles; regarding $\X_j$ as the space
\[
	\X_j \cong \ol{\bbR^3} \times \U \cong \ol{\bbR^3}\times \TGM\times (\wt U_0\times \cdots \times \wt U_N),
\]
with the corresponding projection $\X_j \to \wt U_j$,
we endow the $\ol{\bbR^3}$ factors with the smooth family $(A_j+a,\Phi_j + \phi)$
determined by the projection to $\wt U_j$. The framings
$P\rst_{\gS_j} \cong \pr_1^\ast Q_{k_j}\times_i \SU(2)$ determined by this
family then identify $P\rst_{\gS_j}$ with a universal Dirac bundle $Q
\rst_{\gS_j}$. 

Next, we may pull back a universal Dirac monopole $(a,\phi_{\D'}$
from $\D'$ to $\D$, and by associating $Q$ to an $\SU(2)$ bundle, we obtain a
familiy of $\SU(2)$-Dirac connections $A_\D$, and an $\SU(2)$ Higgs field
$\Phi_\D$ satisfying $\nabla_{A_\D} \Phi_\D = 0$. Composing with a gauge
transformation if necessary, we may assume that $(A_\D,\Phi_\D)$ agrees with
the framing of $(A_i,\Phi_i)$ (and therefore of $(A_i+a,\Phi_i +\phi)$ and the
$\UU(1)$ orbits of these) at $\gS_i$.
%

Finally, proceeding as in Proposition~\ref{P:initial_data_map}, we may produce
a smooth pair $(\backA,\backPhi)$ extending $(A_j + a,\Phi_j + \phi)$ on $\X_j$
and $(A_D, 1 + \ve \Phi_D)$ on $\D$, with the required properties.
\end{proof}

Before proceeding with the construction of a solution from this background
configuration $(\backA,\backPhi)$ we need to discuss two topics related to
analysis on $\Z$: normal operators and Sobolev spaces.


\subsection{Differential and normal operators} \label{S:global_normal}

The algebras of vector fields $\fbV(\Z)$ and $\gV(\Z)$ give rise to algebras
of differential operators: $\fbDiff^\ast(\Z)$ and $\gDiff^\ast(\Z)$ are
essentially the respective universal enveloping algebras of $\fbV$ and $\gV$, generated by 
composition with respect to the action on $C^\infty(\Z)$. If $E$ and
$F$ are vector bundles over $\Z$, we have similar spaces of differential
operators $\fbDiff^\ast(\Z; E,F)$ and $\gDiff^\ast(\Z; E,F)$ acting from
$C^\infty(\Z; E)$ to $C^\infty(\Z; F).$

As $\gV = \rho_\D \rho_\B \fbV$, we have inclusions
\[
	(\rho_\D\rho_\B)^k \fbDiff^k \subset \gDiff^k,
\]
though equality does not hold, the difference being in the lower order terms. For example,
\[
	\gDiff^1 = (\rho_\D \rho_\B) \fbDiff^1 + \fbDiff^0
\]
and so on.

For elements of $\fbV(\Z)$, restriction to $\D$ makes sense as this face lies in a
fiber of $\fb$; this restriction can be identified with the quotient
\[
	\fbV(\Z) \to \fbV(\Z)/\rho_\D \fbV(\Z)
\]
where $\rho_\D\fbV(\Z) \subset \fbV(\Z)$ is easily seen to be an ideal. The restriction to $\X_j$
is similar. Along with restriction of smooth functions, this generates maps from
$\fbDiff^\ast(\Z; E,F)$ to differential operators on $\D$ or $\X_j$, which we
call the {\em normal operator homomorphisms} $\noD^\fb$ and $\noX^\fb$.

\begin{prop}
The normal operator homomorphisms define short exact sequences
\[
\begin{gathered}
	0 \to \rho_\D \fbDiff^k(\Z; E,F) \to \fbDiff^k(\Z; E,F) \stackrel{\noD^\fb}
	\to \bDiff^k(\D/\idmon; E,F) \to 0
	\\ 0 \to \rho_j\fbDiff^k(\Z; E,F) \to \fbDiff^k(\Z; E,F) \stackrel{\noX^\fb}
	\to \bDiff^k(\X_j/\idmon; E,F) \to 0
\end{gathered}
\]
The latter spaces denote fiberwise b differential operators with respect
to the fibrations 
\begin{equation}
	\fX_j \to \X_j \to \idmon \text{ and } \fD \to \D \to \idmon.
	\label{E:XD_fibrations}
\end{equation}
\label{P:fb_normal_seqs}
\end{prop}
\begin{proof}
This follows immediately from the basic claim that the quotient algebras
$\fbV(\Z)/\rho_\D \fbV(\Z)$ and $\fbV(\Z) / \rho_j\fbV(\Z)$ may be identified
with $\fbV(\D)$ and $\fbV(\X_j)$, respectively, which are precisely the fiberwise
b vector fields with respect to the fibrations \eqref{E:XD_fibrations}. This is
an easy exercise in local coordinates; for instance, near $\D \cap \X$, a
general element of $\fbV(\Z)$ is $a(x,r,y,m)(x\pa_x - r\pa_r) + \sum_{j}
b_j(x,r,y,m)\pa_{y_j}$ with local coordinates $(x,r,y)$ on the fiber $\fZ$ and
coordinate $m$ on the base $\idmon$, and the quotient by $\set{x(x\pa_x -
r\pa_r),x\pa_y}$ amounts to expanding the coefficients in Taylor series about
$x = 0$ and throwing out terms of order $\cO(x)$ and identifying $x\pa_x -
r\pa_r$ with $-r\pa_r$, giving $- a(0,r,y,m)r\pa_r + \sum_j b_j(0,r,y,m)\pa_{y_j}$.
\end{proof}

For the algebra $\gDiff^\ast$, the normal operators at $\X_j$ and $\D$ are
quite different from one another. On the one hand, the quotient map $\gV(\Z) \to
\gV(\Z)/\rho_j \gV(\Z)$ is easily identified with the ordinary restriction of
vector fields to $\X_j$; in fact the quotient can be identified with the algebra
$\scV(\X_j/\idmon)$ of fiberwise scattering vector fields, and we have:
\begin{prop}
The sequence
\[
	0 \to \rho_j \gDiff^k(\Z; E,F) \to \gDiff^k(\Z; E,F) \stackrel{\noX^\gamma}\to \scDiff^k(\X_j/\idmon; E,F) \to 0
\]
is exact.
\label{P:g_normal_seq_X}
\end{prop}

On the other hand, $[\gV(\Z),\gV(\Z)] \subset \rho_\D \gV(\Z)$, so in fact
the quotient $\gV(\Z)/\rho_\D \gV(\Z)$ is an {\em abelian} Lie algebra,
i.e., the bracket is trivial. Elements of this quotient can be regarded as
families of first order, constant coefficient differential operators along the
fibers of the vector bundle $\gT \Z \to \D$, and in general the normal operator
homomorphism is a map
\[
	\noD^\gamma : \gDiff^k(\Z; E,F) \to \Diff_{I,\fib}^k(\gT \Z \rst_\D; E,F)
\]
where $\Diff^k_{I,\fib}(\gT \Z \rst_\D)$ denotes fiberwise constant
coefficient differential operators. It is then convenient to use the fiberwise
Fourier transform to identify such operators with polynomials in the fibers of
$\gT^\ast \Z$, and identify composition of those differential operators with
multiplication of polynomials. We denote the Fourier transform of $\noD^\gamma$ by
$\sigma_\D$.

\begin{prop}
The sequence
\begin{multline*}
	0 \smallto \rho_\D \gDiff^k(\Z; E,F) \smallto \gDiff^k(\Z; E,F) 
	\\ \stackrel{\sigma_\D} \smallto C^\infty(\D; P^k(\gT \Z)\otimes \Hom(E,F)) \smallto 0
\end{multline*}
is exact, where $P^k(\gT \Z) = \bigoplus_{l \leq k} S^l (\gT \Z)$ is a sum of
symmetric products of $\gT \Z$, whose sections are 
polynomials on the fibers of $\gT^\ast \Z$.
\label{P:g_normal_seq_D}
\end{prop}
\begin{proof}
We again give an indication of how this works in local coordinates. The abelian
lie algebra generated by $\set{x^2\pa_x - xr\pa_r, x\pa_y}$ may be identified
with the one generated by the translation invariant vector fields
$\set{\pa_\xi, \pa_{\eta}}$ on $\gT \Z$, where $(\xi,\eta)$ are linear
coordinates on $\gT \Z$ associated to the basis $\set{x^2\pa_x -
xr\pa_r,x\pa_y}$. Using ordinary restriction for smooth functions,
the general local vector field $a(x,r,y,m)(x^2\pa_x - xr\pa_r) + \sum_j
b_j(x,r,y,m)x\pa_{y_j}$ is then identified with $a(0,r,y,m)\pa_\xi +
\sum_j b_j(0,r,y,m)\pa_{\eta_j}$, and under the fiberwise Fourier transform 
this becomes $a(0,r,y,m)\hat \xi + \sum_j b_j(0,r,y,m)\hat \eta_j$, 
where $(\hat\xi,\hat \eta)$ are the dual coordinates on $\gT^\ast \Z$.
\end{proof}

The content of Propositions~\ref{P:L_rstn_X} and \ref{P:L_rstn_D} in this
setting is the following. Let $(\backA,\backPhi)$ be the pregluing
configuration from Proposition~\ref{P:global_pregluing}, and let
\begin{equation}
	\backL = D\Bogo_{\backA,\backPhi} + \gauD^\ast_{\backA,\backPhi}
	\label{E:backL_global}
\end{equation}
denote the linearized Bogomolny operator at $(\backA,\backPhi)$, augmented by
the gauge fixing operator, which we decompose relative to the
splitting $\adP = \adPz \oplus \adPo$ as
\begin{equation}
	\backL = \begin{pmatrix} L_0 & \cO(\rho_\D^\infty\rho_\B^\infty)
	  \\\cO(\rho_\D^\infty \rho_\B^\infty)& L_1 + \backPhi_1 \end{pmatrix}
	\label{E:backL_again}
\end{equation}
near $\D\cup \B$.

\begin{prop}
The $\gamma$ normal operator at $\X_j$ of $\backL$ is the family of linearized gauge fixed operators
\[
	\noX^\gamma(\backL)
	= L_\fX
	= \begin{bmatrix} \star d_\backA + \ad (\backPhi) & -d_\backA \\ -d_\backA^\ast & \ad (\backPhi)\end{bmatrix}
	\in \scDiff^1(\X_j/\idmon; (\scLam^1\oplus \scLam^0)\otimes \adP),
\]
and the $\fb$ normal operator of $(\rho_\D\rho_\B)^{-1}L_0$ is identified with
the family of operators
\[
\begin{gathered}
	\noD^\fb((\rho_\D\rho_\B)^{-1}L_0) = 
	\rho_{\X}\rho_{\B}^{-1} L_D \in \bDiff^1(\D/\idmon; \sccLam^1\oplus \sccLam^0),
	\\ L_D = \begin{bmatrix} \star_D d & -d \\ -d^\ast & 0\end{bmatrix}
\in \sccDiff^1(\D/\idmon; \sccLam^1\oplus \sccLam^0),
\end{gathered}
\]
where we use the trivialization of $\adPz$ over $\D$ and the rescaled
restriction isomorphisms $\gLam^k\Z \cong \sccLam^k \D$, and $d^\ast$ denotes
the $L^2$ adjoint of $d$ with respect to the family of fiberwise conic metrics
$g_\D$ on the fibration $\D \to \idmon$.
\label{P:global_normal_ops}
\end{prop}

As for the operator $L_1 + \backPhi_1$, we note that $L_1$ is a twisting of a Dirac
operator on $\gLam^1\oplus \gLam^0$ by the bundle $\adPo$, hence its normal
symbol in the sense of Proposition~\ref{P:g_normal_seq_D} is the corresponding Clifford multiplication.

\begin{prop}
The normal symbol of $L_1 + \Phi$ 
is invertible, and is given at $(x,\xi) \in \gT\Z)$ by
\[
	\sigma_\D(L_1 + \Phi)(x,\xi) = i\cl(\xi)\otimes 1 + 1\otimes\ad\Phi_x \in \End((\gLam^1\oplus \gLam^0)\otimes \adPo),
\]
where $\cl$ is a skew adjoint Clifford action on $\gLam^1\oplus \gLam^0$.
\label{P:normal_symbol_D}
\end{prop}
\begin{proof}
Invertibility follows from the fact that $i\cl_\odd(\xi)$ and $\ad \Phi_x$ commute and are
self-adjoint and skew-adjoint, respectively, with the latter nondegenerate. 
\end{proof}

\subsection{Sobolev spaces} \label{Ss:sobolev}
We first define fiberwise $L^2$-based Sobolev spaces with respect to $\fu : \Z \to \idmon$.
The bilinear form $\wt g$ in \eqref{E:wtg_global} is a fiberwise metric with respect to $\fb : \Z \to \base$, 
so to obtain a metric on the fibers of $\fu$ we set
\begin{equation}
	\olg = \wt g + \pi_{[0,\infty)}^\ast \tfrac{d\ve^2}{\ve^2}.
	\label{E:fZ_metric}
\end{equation}
This is a complete metric on the interior of any fiber $\fZ$ of $\fu$, and
\[
	L^2(\fZ; \olg) = (\rho_\D\rho_\B)^{3/2} \bL^2(\fZ),
\]
where the latter space is the $L^2$ space defined by any $b$-metric on $\fZ$.
For what follows we fix a fiber $\fZ$ and a Hermitian vector bundle $V \to \M$.

For $k,l,m \in \bbN$, let
\begin{multline}
	\bpgH^{k,l,m}(\fZ;V) \ni v 
	\iff \gV^{m'}\cdot \fbV^{l'}\cdot \fuV^{k'} v \in L^2(\fZ; V; \olg),
	\\ \forall\ m' \leq m,\ k' \leq k,\ l' \leq l.
	\label{E:basic_sobolev}
\end{multline}
Here the vector fields are lifted to act on sections of $V$ by a choice of
$\gl$, $\fb$, and $\fu$ connections $\nabla_\gl$, $\nabla_\fb$ and $\nabla_\fu$, respectively,
on which choices \eqref{E:basic_sobolev} does not depend.
The subspaces \eqref{E:basic_sobolev} may then be equipped with inner products 
associated to the norms
\[
	\norm{u}^2_{\bpgH^{k,l,m}} = \sum_{\substack{0\leq m' \leq m,\\0 \leq k' \leq k,\\ 0 \leq l'\leq l}}
	\norm{\nabla_\gl^{m'}\nabla_\fb^{k'}\nabla_\fu^{l'}u}^2_{L^2(\fZ; V; \olg)},
\]
with respect to which the \eqref{E:basic_sobolev} are Hilbert spaces, whose
topology is independent of the choice of connections. For brevity, we write
$\bpH^{k,l}(\fZ; V) = \bpgH^{k,l,0}(\fZ; V)$ and $\bgH^{k,l}(\fZ; V) =
\bpgH^{k,0,l}(\fZ; V)$.
Some properties of
these spaces, including multiplicativity results, are proved in
Appendix~\ref{S:sobolev}.

\begin{rmk}
An alternate (and in many ways more convenient) definition of $\pH^\ast(\fZ)$
and $\gH^\ast(\fZ)$ in terms of pseudodifferential operators is given in
Appendix~\ref{S:double} which permits the order to take any real value.
However, nonnegative integer orders will suffice for our purposes.
\end{rmk}

Due to the nature of the operator \eqref{E:backL}, we will need to measure reguarity
differently near $\D \cup \B$ according to the splitting $\adP = \adPz \oplus \adPo$.

Thus, for a fixed smooth $\Phi \in C^\infty(\Z; \adP)$ such that $\Phi \neq 0$ on $\D \cup \B$,
the {\em split Sobolev spaces} are defined on a fiber $\fZ$ via the norm
\[
	\threeH^{k,l}(\fZ; \adP) 
	\ni v \iff
	\norm{\rho^{l-1}\chi v_0}_{\bpH^{k,l}} 
	  + \norm{\chi v_1}_{\bgH^{k,l}} 
	  + \norm{(1 - \chi) v}_{H^{k+l}} < \infty,
\]
where $\rho = \rho_\D \rho_\B$, $\chi$ is a cutoff function near $\D \cup \B$
with support in the neighborhood where the splitting $\adP = \adPz \oplus
\adPo = \bbC\pair \Phi \oplus \Phi^\perp$ is defined, and $\chi v = \chi v_0 + \chi v_1$ denotes the decomposition
with respect to the splitting. Thus, $v_1$ supports up to $l$ derivatives of
gluing type with $k$ additional $\fu$-derivatives in $L^2$, while $v_0$
supports up to $l$ derivatives of $\fb$ type, with up to $k$ additional
$\fu$-derivatives in $\rho^{1-l} L^2.$ In other words, near $\fD \cup \fB$,
\[
	\threeH^{k,l}(\fZ; \adP)
	\simeq \rho^{1-l} \bpH^{k,l}(\fZ; \adPz) \oplus \bgH^{k,l}(\fZ; \adPo).
\]
The spaces $\threeH^{k,l}(\fZ; \adP)$ are independent of the choice of $\chi$, and
the definition extends naturally to the spaces $\threeH^{k,l}(\fZ; \gLam^\ast\otimes \adP)$.

For fixed $k > 2$,
these are the basic
(fiberwise) Sobolev spaces with which we work, where
\[
	\threeH^{k,2}(\fZ; \gLam^\ast\otimes \adP)
	  \simeq \rho^{-1}\bpH^{k,2}(\fZ; \gLam^\ast \otimes \adPz)
	  \oplus \bgH^{k,2}(\fZ; \gLam^\ast\otimes \adPo)
\]
supports the infinitesimal gauge transformations,
\[
	\threeH^{k,1}(\fZ; \gLam^\ast\otimes\adP)
	  \simeq \bpH^{k,1}(\fZ; \gLam^\ast \otimes \adPz)
	  \oplus \bgH^{k,1}(\fZ; \gLam^\ast\otimes \adPo)
\]
supports the infinitesimal monopole data, and
\[
	\threeH^{k,0}(\fZ; \gLam^\ast\otimes\adP)
	  \simeq \rho^{1}\fuH^{k}(\fZ; \gLam^\ast \otimes \adPz)
	  \oplus \fuH^{k}(\fZ; \gLam^\ast\otimes \adPo)
\]
is the range of the Bogomolny map. 
By increasing $k$ we increase the overall regularity, and we note that the space
$\threeH^{\infty,l} = \bigcap_k \threeH^{k,l}$ includes polyhomogeneous sections with 
appropriate decay. 

Letting the fiber $\fZ$ vary,
we obtain Hilbert space bundles over $\idmon$ with fibers given by the
$\threeH^{k,l}(\fZ; \gLam^\ast \otimes \adP)$. When working globally over any open set $\U \subset
\idmon$, we use the Fr\'echet spaces
\begin{equation}
	\sH^{k,l}(\Z\rst_\U; \gLam^\ast\otimes \adP) := 
	C^\infty(\U; \threeH^{k,l}(\fZ; \gLam^\ast\otimes \adP)),
	\label{E:global_sobolev}
\end{equation}
consisting of smooth sections of these Hilbert bundles. 

We shall also need to restrict to a smaller range in $\ve$; thus we denote
\[
	\sH^{k,l}(\Z \rst_{\U\times[0,\ve_0]}; \gLam^\ast\otimes \adP)
	= C^\infty(\U; \threeH^{k,l}(\fZ \rst_{\set{0\leq \ve \leq \ve_0}}; \gLam^\ast \otimes \adP)),
\]
which is to say the smooth sections over $\U$ with values in the Sobolev space
of sections on fibers $\fZ$ restricted over $[0,\ve_0]$ which admit extensions
to \eqref{E:global_sobolev}. (Note that the fibers $\fZ \cap \set{0 \leq \ve
\leq \ve_0}$ are not complete with respect to $\olg$, though this will not
cause any problems.)

The next result is proved in Appendix~\ref{S:sobolev}.

\begin{thm}
For each $k > 2$, there is a well-defined
gauge group 
\[
	\gauG^{k}(\Z) = \sH^{k,2}(\Z; \Ad P) 
\]
with Lie algebra consisting of
$\sH^{k,2}(\Z; \adP))$. This group acts on the spaces
$\sH^{k,1}(\Z; \gLam^\ast \otimes \adP))$ and
$\sH^{k,0}(\Z; \gLam^\ast \otimes \adP))$.
Additionally, the product
on $\gLam^\ast\otimes \adP$ extends to a continuous bilinear map
\begin{equation}
	\sH^{k,1}(\Z; \gLam^\ast \otimes \adP)
	\times \sH^{k,1}(\Z; \gLam^\ast \otimes \adP) 
	\to \sH^{k,0}(\Z; \gLam^\ast\otimes \adP) 
	\label{E:Frechet_sobolev_mult}
\end{equation}
\fixthmeq
\label{T:gauge_spaces}
\end{thm}

\subsection{Gauge fixing} \label{S:gauge}
Having just defined the Sobolev versions of the gauge group we will consider
over $\Z$, we digress briefly to discuss the issue of gauge fixing. Fixing a
sufficiently smooth configuration $(A,\Phi)$, the infinitesimal action of the
gauge group at $(A,\Phi)$ is given by the operator
\begin{equation}
\begin{gathered}
	\gauD_{A,\Phi} : C^\infty(\Z; \adP) \to C^\infty(\Z; (\gLam^1\oplus \gLam^0)\otimes \adP),
	\\ \gauD_{A,\Phi} \eta = (-d_A \eta, - [\Phi,\eta]).
\end{gathered}
	\label{E:gauge_action}
\end{equation}
Note that if $(A',\Phi') = (A + a',\Phi + \phi')$, then 
\[
	\gauD_{A',\Phi'} = \gauD_{A,\Phi} - (a',\phi')
\]
where the latter multiplication operator acts by $(a',\phi') \cdot \eta =
([a',\eta],[\phi',\eta])$.

For smooth $(A,\Phi)$, $\gauD_{A,\Phi}$ admits a bounded extension 
\[
	\gauD_{(A,\Phi)} : 
	\sH^{k,2}(\Z; \adP) \to \sH^{k,1}(\Z; (\gLam^1\oplus \gLam^0)\otimes \adP)
\]
for all $k$, and if $(A',\Phi') = (A+ a',\Phi+\phi')$ is a perturbation with
$(a',\phi') \in \sH^{k,1}(\Z; (\gLam^1\oplus \gLam^0)\otimes \adP)$
then $\gauD_{A',\Phi'}$ admits a similar extension for $k > 2$ by Theorem~\ref{T:gauge_spaces}.

We recall that a section $(a,\phi)$ of $(\gLam^1\oplus \gLam^0)\otimes \adP)$ is in Coulomb gauge
with respect to $(A,\Phi)$ if
\begin{equation}
	\gauD^\ast_{A,\Phi}(a,\phi) = -d_A^\ast a + [\Phi,\phi] = 0,
	\label{E:coulomb_gauge}
\end{equation}
where the adjoints are taken with respect to the formal fiberwise $L^2$ pairing
on sections $\gLam^\ast\otimes\adP$ using the volume form from the $\gamma$
metric $\wt g$. In particular, $d_A^\ast = - \star d_A \star$. 

This is naturally an infinitesimal condition, where $(a,\phi)$ are considered
as elements in the tangent space to the space of configurations at $(A,\Phi)$,
and then $\Null(\gauD^\ast_{(A,\Phi)})$ determines a subspace complementary to
the action of the gauge group in this tangent space. However, it is also known to 
give local slices for the gauge action on the configuration space itself. In the present
setting, this takes the form of the following result, proved in Appendix~\ref{S:coulomb}.

\begin{thm}
Suppose $A$ is a smooth (true) connection on $\adP$ over $\Z$, and $\Phi \in
C^\infty(\Z; \adP)$, where both are diagonal to infinte order with respect to the
splitting $\adP = \adPz\oplus \adPo$ near $\D \cup \B$.  Fix $l > 2$ in $\bbN$. 

Then for any compact set $\K \subset \idmon$, there exists $\ve_\K > 0$
such that for all sufficiently small
\[
	(a,\phi) \in \sH^{l,1}(\Z\rst_{\K\times[0,\ve_\K]}; (\gLam^1\oplus \gLam^0)\otimes \adP),
\]
there exists a unique gauge transformation
\[
	\gamma \in \gauG^{l}(\Z \rst_{\K\times [0,\ve_\K]})
\]
such that
\[
	\gamma\cdot(A + a,\Phi +\phi)
\]
is in Coulomb gauge with respect to $(A,\Phi)$ on $\Z$ over $\K \times
[0,\ve_\K]$. Furthermore, if $(a,\phi)$ is additionally in $\sH^{l+n,1}(\Z\rst_{\K\times [0,\ve_\K]};(\gLam^1\oplus \gLam^0)\otimes \adP)$ for any
$n \geq 0$, then in fact $\gamma \in \gauG^{l+n,2}(\Z\rst_{\K\times[0,\ve_\K]})$ as well.
\label{T:coulomb}
\end{thm}

\noindent This justifies the addition of the Coulomb gauge operator to $\backL$ in \eqref{E:backL_global}.

The proof of Theorem~\ref{T:coulomb} makes use of the following invertibility result for the associated linear
operator, which will also be used in \S\ref{S:metric}.

\begin{prop}
Let $(\backA,\backPhi)$ satisfy the hypotheses of Theorem~\ref{T:coulomb}. Then
for any $k \geq 0$, the linear operator
\[
	\gauD_{A,\Phi}^\ast \gauD_{A,\Phi} = \Delta_A + (\ad\Phi)^\ast(\ad \Phi)
	: \sH^{k,2}(\Z; \adP) \to \sH^{k,0}(\Z; \adP)
\]
is invertible over sets of the form $\K \times [0,\ve_\K]\subset \base$, with inverse independent of $k$. 
\label{P:Coulomb_linear}
\end{prop}

We also record another important fact concerning the background configuration $(\backA,\backPhi)$.
\begin{prop}
Let $(\backA,\backPhi)$ be the pregluing configuration on $\Z \rst_\U$ of
Proposition~\ref{P:global_pregluing} and let $\K \subset \U$ be any compact
set. Then there exists $\ve_\K > 0$ such that the map
\[
	\K\times (0,\ve_\K) \ni (m,\ve) \mapsto (A,\Phi) \rst_{(m,\ve)} \in \cfgC(\ol{\bbR^3})
\]
is smooth and transverse to the orbits of the gauge group.
\label{P:transverse_to_gauge}
\end{prop}
\begin{proof}
It suffices to show that the derivative of this map has nontrivial projections
onto a complementary subspace to the infinitesimal gauge action. Recall that
the infinitesimal gauge action at $(A,\Phi)$ is given by the image of the map
\[
	\gauD_{A,\Phi} : \eta \mapsto (-d_A \eta, -[\Phi,\eta])
\]
and that a natural complementary subspace is the nullspace of
$\gauD^\ast_{A,\Phi}$, which is the Coulomb gauge slice. Projection onto $\Null(\gauD^\ast_{A,\Phi})$
is given by the operator
\[
	\Pi = I - \gauD_{(A,\Phi)} G \gauD_{(A,\Phi)}^\ast,
	\quad G = (\gauD_{(A,\Phi)}^\ast \gauD_{(A,\Phi)})^{-1}.
\]

By Proposition~\ref{P:Coulomb_linear},
we have a bounded operator
\[
	\Pi : \sH^{\infty,1}(\Z\rst_{\K\times [0,\ve_\K]}; (\gLam^1\oplus \gLam^0)\otimes \adP)
	\to \sH^{\infty,1}(\Z\rst_{\K\times [0,\ve_\K]}; (\gLam^1\oplus \gLam^0)\otimes \adP),
\]
representing projection onto the Coulomb gauge slice of $(A,\Phi)$ over each
fiber $(m,\ve) \in \K\times [0,\ve_\K].$

Let $V \in \phiT_{m,\ve}(\K\times [0,\ve_\K])$ be any fibered boundary vector, as
defined at the end of \S\ref{S:ideal}.
For $\ve > 0$ this is any tangent vector. We may extend $V$ as a vector field over
a neighborhood of $(m,\ve)$ and apply it to $(A,\Phi)$ to get 
\[
	(a,\phi) := V\cdot (A,\Phi) \in \cAphg^\ast(\Z\rst_{\K \times [0,\ve_\K]}; (\gLam^1\oplus \gLam^0)\otimes \adP).
\]

By construction of the pregluing configuration $(a,\phi)$ is in the Coulomb slice over $\ve = 0$
so 
\[
	\Pi(a,\phi) \in \sH^{\infty,1}(\Z\rst_{\K\times [0,\ve_\K]}; (\gLam^1\oplus \gLam^0)\otimes \adP)
\]
is nonvanishing at $\ve = 0$, and therefore also for $\ve \leq \ve_V$ for some $\ve_V > 0$. Minimizing $\ve_V$ over the unit
sphere bundle of $\phiT (\K\times [0,\ve_\K])$, we obtain $\ve_\K$, and the result is proved.
\end{proof}

\subsection{True solution} \label{S:true}

Given the pregluing configuration $(\backA,\backPhi)$, the formal solution
procedure in \S\ref{S:formal_formal}, applied fiber by fiber over $\U$,
produces an asymptotic series for a correction $(a,\phi) \in \cAphg^\ast(\Z\rst_{\U\times[0,\infty)};
(\gLam^1\oplus\gLam^0)\otimes \adP)$ at the faces $\D$ and $\X = \cup_j \X_j$.
To sum such a series, which has smooth coefficients in the parameters $\idmon$,
it is necessary to restrict to a compact set. Thus let $\K \subset \U$
be any compact subset of the space \eqref{E:U_in_idmon}.

\begin{prop}
Given the pregluing configuration $(\backA,\backPhi)$ from
Proposition~\ref{P:global_pregluing} and a compact set $\K \subset
\U$, there exists $(a,\phi) \in \cAphg^\cF(\Z\rst_{\K\times[0,\infty)}; (\gLam^1\oplus
\gLam^0)\otimes \adP)$ such that $\Bogo(\backA + a, \backPhi + \phi) =
\cO(\ve^\infty)$, with $\cF$ given as in \eqref{E:formal_soln_index}.
\label{P:formal_global}
\end{prop}

Next we remove the $\cO(\ve^\infty)$ error of our solution to the Bogomolny
equation which remains after the formal construction, over a region where $\ve$
is sufficiently small. The first step is the existence of a good right
parametrix for the linear operator $L$. The following is proved in
Appendix~\ref{S:parametrices} using the pseudodifferential operator calculus
developed in Appendix~\ref{S:double}. The main ingredients are the
invertibility of the normal operators/symbols of $\backL$ as in
Propositions~\ref{P:global_normal_ops} and \ref{P:normal_symbol_D}.

\begin{prop}
There exists a right parametrix, $R$ to $\backL$ with bounded extensions
\[
	R: \sH^{k,0}(\Z\rst_\U; (\gLam^1\oplus\gLam^0)\otimes \adP)
	\to \sH^{k,1}(\Z\rst_\U; (\gLam^1\oplus \gLam^0)\otimes \adP),
	\quad k > 2,
\]
such that, for some $0 < \delta < \tfrac 1 2$, 
\[
	\backL R = I - \ve^\delta E,
\]
where $E$ extends to a map
\begin{equation}
	E : \sH^{k,0}(\Z\rst_\U; (\gLam^1\oplus \gLam^0) \otimes \adP) 
	  \to \sH^{k,0}(\Z\rst_\U; (\gLam^1\oplus \gLam^0)\otimes \adP),
	\quad k > 2,
	\label{E:pmtx_error_mapping}
\end{equation}
i.e., is a smooth section over $\U$ of bounded linear
maps on the Hilbert bundle fibers $\threeH^{k,0}(\fZ; (\gLam^1\oplus\gLam^0)\otimes \adP)$.
\label{P:bogo_pmtx}
\end{prop}

Next we set up an inverse function theorem type fixed point argument in the
range space $\sH^{k,0}(\Z\rst_\U; (\gLam^1\oplus\gLam^0) \otimes \adP))$.

\begin{prop}
For any compact set $\K \subset \U$, there exists $\ve_\K > 0$ and $(\wt a ,
\wt \phi) \in \ve^\infty \sH^{\infty,1}(\Z\rst_{\K\times[0,\ve_\K]};
(\gLam^1\oplus\gLam^0)\otimes \adP))$ such that $\cB(\backA + a + \wt a,
\backPhi + \phi + \wt \phi) = 0$ on $\Z$ over $\K\times[0,\ve_\K].$
\label{P:true_soln}
\end{prop}

\begin{proof}

For notational convenience for the remainder of this proof, we omit the
decoration $\gLam^\ast$ from the bundle, understand $\Z$ to be restricted over
$\K \times[0,\infty)$, and let $u_0 = (\backA,\backPhi)$ and $u_1 = (a,\phi)$,
so that $U = u_0 + u_1$ is the the formal gauge-fixed solution of the previous
section:
\[
	\cB(U) \in \ve^\infty \sH^{\infty,0,0}(\Z; \adP).
\]
We seek $u \in \ve^\infty \sH^{\infty,0}(\Z; \adP)$ such that $\cB(U +
u) = 0$, for sufficiently small $\ve$. Expanding the gauge-fixed Bogomolny
equation into background, linear and quadratic parts, we have
\[
\begin{aligned}
	\cB(U + u) &= \cB(u_0) + L_{u_0}(u_1 + u) + Q(u_1 + u)
	\\& = \cB(U) + L_{u_0}u + Q(u_1 + u) - Q(u_1).
\end{aligned}
\]
Here $\cB(U) = \cO(\ve^\infty)$, $Q$ is the zeroth order quadratic term, and
$L_{u_0} = L$ is the linear operator to which we constructed a right parametrix
above. It follows from the multiplicativity result in Theorem~\ref{T:gauge_spaces}
that $u \mapsto Q(u_1 + u) - Q(u_1)$ is a bounded map from $\ve^N\sH^{k,1}(\Z; \adP)$
to $\ve^N\sH^{k,0}(\Z; \adP)$.

We initially seek a solution of the form $u = Rv$, $v \in \ve^N
\sH^{k,0}(\Z; \adP)$ for fixed $N$ and $k > 2$,
which should satisfy
\[
\begin{gathered}
	0 = \cB(U) + v - \ve^\delta E v + Q(u_1 + Rv) - Q(u_1),
	\\ \iff v = Tv := \ve^\delta Ev - \cB(U) - Q(u_1 + Rv) + Q(u_1).
\end{gathered}
\]
Restricting consideration to a single fiber $\fZ$ over $m \in \K$ for the moment,
we claim that, for $\ve_m$ sufficiently small, $T$ is a contraction mapping on
a ball of sufficiently small radius in the Hilbert space
$\ve^N\threeH^{k,0}(\fZ\rst_{[0,\ve_m]}; \adP)$. Indeed, as a
bounded operator on the latter space, $\norm{\ve^\delta E} \leq \ve_0^\delta C$ for
some $C > 0$, and $G(v) := Q(u_1 + Rv) - Q(u_1)$ vanishes at $0$ along with its
derivative, hence by the mean value theorem, there exists $R_{m} > 0$ such that
\[
	\norm{v} \leq R_{m} \implies 
	\norm{G(v)} \leq \frac 1 3\norm{v}, \quad \norm{G'_v v'} \leq \frac 1 3\norm{v'}.
\]
Then taking $\ve_m$ small enough that $\ve_m^\delta C < \tfrac 1 3$ and $\norm{\cB(U)} < \tfrac 1 3$,
it follows that 
\begin{multline*}
	T : B(0,R_{m}) \subset \ve^N \threeH^{k,0}(\fZ\rst_{[0,\ve_m]}; \adP)
	\to  B(0,R_{m}) \subset \ve^N \threeH^{k,0}(\fZ\rst_{[0,\ve_m]}; \adP)
\end{multline*}
is a contraction mapping, hence has a unique fixed point $v$. By uniqueness of
$v$, along with the fact that $\ve$ commutes with the linear operators $E$ and
$R$, it follows that in fact we may take $N \smallto \infty$ without altering the
size of $\ve_m$ or $R_m.$

To see that we may also take $k \smallto \infty$, we proceed as follows. We have
shown thus far that there exists a unique $v \in B(0,R_m) \subset
\ve^N\threeH^{k,0}$ such that $v - T(v) = 0$; in particular, 
\begin{equation}
	I - dT_v : \ve^N\threeH^{k,0} \to \ve^N\threeH^{k,0},
	\quad dT_v = \ve^\delta E - Q'(u_1 + Rv) \,R
	\label{E:lin_fixed_point_iso_k}
\end{equation}
is an isomorphism. Applying the above argument in the space $\ve^N
\threeH^{k+1,0}$ and shrinking $\ve_m$ and $R_m$ once if necessary, we
likewise conclude that $v \in \ve^N \threeH^{k+1,0}$, and that
\begin{equation}
	I - dT_v : \ve^N\threeH^{k+1,0} \to \ve^N \threeH^{k+1,0}
	\label{E:lin_fixed_point_iso_k1}
\end{equation}
is an isomorphism which coincides with \eqref{E:lin_fixed_point_iso_k} where
defined.

Now let $V \in \bV(\fZ)$ and apply it to $v - T(v) = 0$ to obtain
\[
	0 = V\cdot v - V\cdot (T(v)) = (I - dT_v)(V\cdot v) + (V\cdot T)(v),
\]
where 
\begin{multline*}
	(V\cdot T)(v) = [V,\ve^\delta E]\,v - Q'(u_1 + Rv)\,[V,R]\,v 
	\\+ (Q'(u_1) - Q'(u_1 + Rv))(V\cdot u_1) - V\cdot \cB(U) \in \ve^N \threeH^{k+1,0}.
\end{multline*}
Indeed, $[V,\ve^\delta E]$ is bounded on $\ve^N \threeH^{k+1,0}$ and
$[V,R]$ is bounded from $\ve^N \threeH^{k+1,0}$ to $\ve^N
\threeH^{k+2,0}$ as shown in Appedix~\ref{S:double}, and $u_1$ and
$\cB(U)$ are polyhomogeneous. We conclude that $V\cdot v \in \ve^N \threeH^{k+1,0}$, and since $V$ was
arbitrary, $v \in \ve^N \threeH^{k+2,0}$ and the linear map
\[
	I - dT_v : \ve^N \threeH^{k+2,0} \to \ve^N \threeH^{k+2,0}
\]
is an isomorphism coinciding with \eqref{E:lin_fixed_point_iso_k} and
\eqref{E:lin_fixed_point_iso_k1} where defined. Proceeding inductively, we
conclude that $v \in \ve^N \threeH^{\infty,0}$ without any further
reductions on $\ve_m$.

Applying the above argument fiber by fiber over $\K$ and taking $\ve_\K =
\min\set{\ve_m : m \in \K}$ gives $v \in
\ve^\infty\sH^{k,0}(\Z\rst_{\K\times[0,\ve_\K]}; \adP)$.
\end{proof}

Combining the previous results, we have proved the following.

\begin{thm}
Let $\U \subset \idmon$ be the set \eqref{E:U_in_idmon}, determined by choices
$(A_j,\Phi_j)$ of framed monopole solutions in $\monM_{k_i}$
$i = 0,\ldots,N$. Then for every compact set $\K \subset \U$, there exists
$\ve_\K > 0$ and a solution to $\Bogo(A,\Phi) = 0$ on $\Z \rst_{\K
\times[0,\ve_\K]}$ which is tautological over $\ve = 0$. The solution has the
form $(A_0 + a,\Phi_0 + \phi)$ where $(A_0,\Phi_0)$ is smooth and $(a,\phi) \in
\cAphg^\cF(\Z; \gLam^\ast\otimes \adP) \subset \sH^{\infty,1}(\Z;
\gLam^\ast\otimes \adP)$, with $\cF$ given as in
Theorem~\ref{T:formal_solution}.

In particular, $(A,\Phi)$ is smooth on $\Z\rst_{\K\times (0,\ve_\K)} =
\ol{\bbR^3}\times \K\times(0,\ve_\K)$ and therefore determines a smooth map
\begin{equation}
	\Psi : \K\times(0,\ve_\K) \to \monM_{k},
	\quad (\iota,\ve) \mapsto [(A,\Phi)\rst_{\iota,\ve}]
	\label{E:smooth_map_to_moduli}
\end{equation}

We denote the restriction of $\Psi$ to the sets $\U^c = \U \cap \idmonc$ and $\K^c = \K \cap \idmonc$ 
of initial data representing centered ideal monopoles by
\begin{equation}
	\Psi^c : \K^c\times(0,\ve_\K) \to \monM_{k},
	\quad (\iota,\ve) \mapsto [(A,\Phi)\rst_{\iota,\ve}]
	\label{E:smooth_map_from_centered}
\end{equation}
\label{T:smooth_map}
\end{thm}

We show below that $\Psi^c$ is a local diffeomorphism onto its image (for
possibly smaller $\ve_\K$), and then compute the metric asymptotics to leading order
in $\ve$. However, it is convenient for the metric computation to allow
variation in the centers of the ideal monopole data, hence our defining $\Psi$
as we have.

\begin{thm}
For possibly smaller $\ve_\K$, the map \eqref{E:smooth_map_from_centered} is a local diffeomorphism
onto its image.
\label{T:local_diffeo}
\end{thm}
\begin{proof}
Since $\K\times(0,\ve_\K)$ and $\monM_{k}$ are both smooth manifolds of
dimension $4 k$, it suffices to verify that $\Phi^c$ is an immersion. Letting $V
\in T_p(\K\times(0,\ve_\K))$ be a nonzero tangent vector, we may regard it 
as an element of $\phiT_p(\K\times[0,\ve_\K))$ and extend it locally to a vector field $\wt V$.
The derivative is given by
\[
	D\Psi^c_p V = [(\iota_{\wt V} F_A,\nabla_{\wt V}\Phi) \rst_{p}] \in T_{[A,\Phi]} \monM_k,
\]
where have used the connection $A$ to lift $\wt V$ to act on $(A,\Phi)$.
As $(A,\Phi)$ are a smooth family of solutions to $\Bogo(A,\Phi) = 0$, it
follows that $(a,\phi) := \wt V\cdot (A,\Phi)$ satisfies
$D\Bogo_{A,\Phi}(a,\phi) = 0$. To see that $[(a,\phi)] \neq 0 \in T_{[A,\Phi]}
\monM_k$, it suffices to verify that $(a,\phi)$ are transverse to the gauge
orbit. However, shrinking $\ve_\K$ if necessary, this follows from
Proposition~\ref{P:transverse_to_gauge} since $(A,\Phi)$ satisfies the
tautological property. Thus, for some $\ve_\K > 0$, 
\[
	V \neq 0 \implies d\Phi_p V \neq 0, \quad p \in \K \times (0,\ve_\K),
\]
so that $\Phi^c$ is an immersion.
\end{proof}

%% file: metric.tex
In this section we show how to compute the monopole metric to leading
order in $\ve$ at a monopole $m(\ve)$, say, which is in the image of
our `gluing map' $\Psi$ from Theorem~\ref{T:smooth_map}.  In
particular we shall prove Theorem~\ref{T:main_two}. 

Let $Z$ be the gluing space of \S\ref{S:Mgl} used in the construction
of a 1-parameter family of monopoles with the configuration data
$\ul \zeta = (\zeta_1,\ldots, \zeta_N)$ fixed.  The formal solution constructed in
that section (and then improved to a genuine solution in
\S\ref{S:true}) is of the general form
\begin{equation}\label{e2.8.11.15}
(A,\Phi) = (\ol{A},\ol{\Phi}) + (a,\phi)
\end{equation}
where $(\ol{A},\ol{\Phi})$ is a smooth pre-gluing
configuration\footnote{Elsewhere $(\ol{A},\ol{\Phi})$ has typically been used for
framings. We hope the reader will forgive us for the present change of usage.}
\begin{equation}\label{e3.8.11.15}
(a,\phi) \in \cA_s(Z),
\end{equation}
this space $\cA_s(Z)$ being the conormal space 
given in Theorem~\ref{T:formal_solution} (cf.\ \eqref{E:formal_soln_index}):
\begin{eqnarray}
F_B &=& \{(2,0), (3,1),\ldots,\}, \nonumber  \\
F_D &=& \{(2,0), (3,1),\ldots\}, \label{e10.8.11.15} \\
F_X &=& \{(1,0), (2,0),  (3,1),\ldots\}. \nonumber
\end{eqnarray}
(The index sets $F_B$ and $F_D$ start at order $2$ because of the
definition of pregluing configuration.)     In this section we
generally suppress the coefficient bundles $\gLam = \gLam^0\oplus
\gLam^1$ as well as $\frp$. We recall that the component in $\frp_1$
of $(a,\phi)$ is rapidly decreasing at $B\cup D$.

It will also be convenient to denote the restriction of $(A,\Phi)$ to
the fibre $\fb^{-1}(\ve)$ by $(A(\ve),\Phi(\ve))$.  We apply this
convention similarly to other data, so for example we shall write
$\monM_k(\ve)$ for the framed moduli space of monopoles of charge $k$
on $\fb^{-1}(\ve)$ and shall denote by $m(\ve)$ the point of
$\monM_k(\ve)$ represented by $(A(\ve),\Phi(\ve))$.  Recall that for
positive $\ve$ all fibres $\fb^{-1}(\ve)$ are canonically identified
with the original $\ol{\RR}^3$, so we can equally regard $m(\ve)$ as a
1-parameter family in $\monM_k$.

For the purposes of this section, it is convenient to allow 
uncentred ideal monopoles as initial data in our
construction, which is to say we consider the map \eqref{E:smooth_map_to_moduli};
instead, we restrict the configuration 
data to $\ul \zeta$. For clarity, we denote this restricted gluing map by
\begin{equation}\label{e1.20.11.15}
\psi = \Psi \rst_{\ul \zeta} : U_0\times U_1\times \cdots \times U_N \times(0,\ve_0)
 \longrightarrow \monM_k
\end{equation}
for bounded open neighhbourhoods $U_j$ of $m_j$ in $\monM_{k_j}$.  

Now consider a smooth one-parameter family of ideal monopoles
$\iota_t$.
As in Prop.~\ref{P:global_pregluing} we
may also assume that $\iota_t|D$ is independent of $t$.
We assume $\iota_0 = \iota$ is a centred ideal monopole and
that $\iota_t \in U_0\times U_1\times\cdots \times U_N$ is in the domain of \eqref{e1.20.11.15}.
Let
\begin{equation}
u_j = \left.\frac{\rd}{\rd t}\left( \iota_t|X_j\right)\right|_{t=0}.
\end{equation}
This is the tangent vector to $(A_j,\Phi_j) = \iota|X_j$ in the
$1$-parameter family $\iota_t$. We may assume that $u_j$ is in Coulomb
Gauge with respect to $(A_j,\Phi_j)$ so that 
\begin{equation}
L_j u_j =0 \mbox{ on }\mathring{X}_j
\end{equation}
where as before, $L_j$ is the linearization/gauge-fixing operator
associated to $(A_j,\Phi_j)$ (cf.\ \S\ref{S:local}).

Our 1-parameter family $\iota_t$ gives
rise to a smooth $1$-parameter family $(\ol{A}_t,\ol{\Phi}_t)$ of pregluing
configurations and hence a smooth $1$-parameter family of solutions
\begin{equation}\label{e12.8.11.15}
(A_t, \Phi_t) = (\ol{A}_t,\ol{\Phi}_t) + (a_t,\phi_t).
\end{equation}
Then the derivative $D\psi$ of $\psi$  assigns to $(u_0,\ldots,u_n)$
the field
\begin{equation}\label{e4.8.11.15}
u(\ve) = \left.\left(\left.\frac{d}{d t}(A_t,\Phi_t)\right|_{t=0}\right)\right|_{\fb^{-1}(\ve)}.
\end{equation}
on $\fb^{-1}(\ve)$ which represents  a tangent vector $[u(\ve)]$ to $\monM_k(\ve)$.

\begin{lem}
If $u$ is as defined in \eqref{e4.8.11.15}, then 
\begin{equation}\label{e5.8.11.15}
u \in \cA'_s(Z),\; u|X_j = u_j,
\end{equation}
where the index sets of $\cA'_s(Z)$ are
\begin{eqnarray}
F'_B &=& \{(2,0), (3,1),\ldots,\}, \nonumber  \\
F'_D &=& \{(2,0), (3,1),\ldots\}, \label{e1a.8.11.15} \\
F'_X &=& \{(0,0), (1,0), (2,0),  (3,1),\ldots\}. \nonumber
\end{eqnarray}
\end{lem}
\begin{proof}
Because the framings at the corners are independent of $t$, we can take the
restriction to $D$ of the pregluing configuration also to be
independent of $t$, as in the proof of Prop.~\ref{P:global_pregluing}.
Then we have
\begin{equation}\label{e6.8.11.15}
\left.\frac{d}{d t}(A_t,\Phi_t)\right|_{t=0} \in
\rho_B^2\rho_D^2C^\infty(Z) \subset \cA'_s(Z)
\end{equation}
and its restriction to $X_j$ is just the variation
$u_j=(\dot{A}_j,\dot{\Phi}_j)$ of $\iota_0|X_j$.  The variation
in $(a,\phi)$ lies in $\cA_s(Z) \subset \cA'_s(Z)$ and so the result follows.
\end{proof}

To compute the length-squared of $[u(\ve)]$ with respect to the
monopole metric $G(\ve)$ on $\monM_k(\ve)$, we need to replace
$u$ by a representative of the same element of
$T_{m(\ve)}\monM_k(\ve)$ but which is
in Coulomb gauge with respect to $(A(\ve),\Phi(\ve))$.  This is
accomplished in the following:

\begin{lem}\label{l1.23.11.15}
Given the above data, there exists an infinitesimal gauge
transformation $\xi$ over $Z$ such that
\begin{equation}\label{e1.8.11.15}
d^*_{A,\Phi}\left(u - d_{A,\Phi}\xi\right) = 0,
\end{equation}
where
$\xi$ is smooth over the fibers $\fb^{-1}(\ve)$ for $\ve > 0$ and
vanishing over $\fb^{-1}(0)$.
\end{lem}
\begin{proof}
(Cf.\ also Proposition~\ref{P:transverse_to_gauge}.)  It is convenient
in this discussion to use the notation $f =
O(\rho_B^a\rho_D^b\rho_X^c)$ to mean that $f$ has a polyhomogeneous
conormal expansion on $Z$ with smooth index sets of the kinds that
have appeared throughout this paper, with the lower bounds
$(a,0)_B$, $(b,0)_D$ and $(c,0)_X$.  Thus
$u = O(\rho_B^2\rho_D^2)$ by virtue \eqref{e5.8.11.15}.  Recall also
the notation
$$
d_{A,\Phi}\xi = (d_A\xi, [\Phi,\xi])
$$
for the infinitesimal action of the gauge group on monopole configurations.

By construction, on each fibre $u$ is in Coulomb gauge with respect to
$(\ol{A},\ol{\Phi})$,
\begin{equation}\label{e51.8.11.15}
d^*_{\ol{A},\ol{\Phi}}u|\fb^{-1}(\ve) = 0
\end{equation}
and 
\begin{equation}\label{e52.8.11.15}
(d^*_{A,\Phi}u  -  d^*_{\ol{A},\ol{\Phi}}u)|\fb^{-1}(\ve)  = \{(a,\phi),u\}|\fb^{-1}(0)
\end{equation}
by \eqref{e2.8.11.15}, where $\{\cdot,\cdot\}$ on the RHS is some
bilinear operation with smooth coefficients. Hence
\begin{equation}\label{e53.8.11.15}
d^*_{A,\Phi}u  = O(\ve \rho_B^4\rho_D^3).
\end{equation}
The equation
\begin{equation}\label{e54.8.11.15}
d^*_{A,\Phi}(u - d_{A,\Phi}\xi) = 0
\end{equation}
can be solved with $\xi = O(\ve \rho_B^2\rho_D)$ and $d_{A,\Phi}\xi$ therefore
$O(\ve \rho_B^3\rho_D^2)$. Indeed, we may first construct a formal solution,
proceeding as in \S\ref{S:formal_formal}, giving the asymptotic estimate, and
then remove the rapidly vanishing error using
Proposition~\ref{P:Coulomb_linear}.    The above estimates give that
$\xi$ vanishes on $\fb^{-1}(0)$, completing the proof.

\end{proof}

With $\xi$ from the Lemma, define
\begin{equation}
\wt{u} = u - d_{A,\Phi}\xi
\end{equation}
Then for each positive $\ve$,
\begin{equation}
[\wt{u}(\ve)] = [u(\ve)] \in T_{m(\ve)}\monM_k(\ve)
\end{equation}

In this way we define a family of mappings
\begin{equation} \label{E:tangent_mappings}
\wt{f}_\ve : 
T_{m_0}\monM_{k_0} \times
T_{m_1}\monM_{k_1} \times
T_{m_N}\monM_{k_N}  \to T_{m(\ve)}\monM_{k}(\ve)
\end{equation}
where 
\begin{equation}
\wt{f}_\ve(u_0,u_1,\ldots,u_N) = \wt{u}(\ve)
\end{equation}
which represents the differential, $D\psi$, of \eqref{e1.20.11.15},
the advantage being that
\begin{equation}\label{e1.23.11.15}
\wt{u}(\ve)\neq 0 \Leftrightarrow [\wt{u}(\ve)] \neq 0.
\end{equation}

We note
\begin{lem} The map $\wt{f}_\ve$ is an isomorphism for sufficiently
  small $\ve>0$.  In particular $\Psi$ is a local diffeomorphism for
  sufficiently small $\ve > 0$.
\end{lem}
\begin{proof}  If $u_1$, say, is non-zero, then from the estimates in
Lemma~\ref{l1.23.11.15}, it follows that $\wt{u}(\ve)\neq 0$ near
$X_1$. By \eqref{e1.23.11.15}, this means that $[\wt{u}(\ve)]\neq 0$
and so the map is injective.  Since $\wt{f}_\ve$ is a linear map
between vector spaces of the same dimension, it is an isomorphism.
\end{proof}

We can now prove Theorem~\ref{T:main_two}.  We have seen in
Prop.~\ref{p2.23.11.15} that for a centred monopole $m$,  the metric
on $T_m\monM_k$ decomposes 
canonically as an orthogonal direct sum of $\RR^3$, with $2\pi k$
times the euclidean metric, and $T_m\monM_k^c$. It is therefore enough
to prove that the metric $G_\ve$ on $T_{m(\ve)}\monM_k(\ve)$  is
approximately equal to the product metric on the product of
$T_{m_j}\monM_{k_j}$. 

\begin{thm}
Let $f_\ve$ be as above.  Then $f_{\ve}$ is an approximate isometry.
\end{thm}
\begin{proof}
By the previous lemmas,
\begin{equation}
\wt{u} = O(\rho_B^2\rho_D^2) 
\end{equation}
and
\begin{equation}
\wt{u}|X_j = u_j.
\end{equation}

Hence the pointwise length-squared $|\wt{u}(\ve)|^2$ on $\fb^{-1}(0)$
is $O(\rho_B^4\rho_D^4)$ and its restriction to each $X_j$ is $|u_j|^2$. 

In order to do the integration, we have to multiply by the lift of the
euclidean density $\rd \mu_e$ to $Z$. Since $\gT = \rho_B\rho_D \fbT$, the
lift of the euclidean density has the form
$\rho_B^{-3}\rho_{D}^{-3}\mu_{\bo}$ where $\mu_{\bo}$ is a smooth
positive section of $\Lambda^3\, \fbT^*$, whose restriction to each $X_j$
is the euclidean volume element $\rd \mu_j$.  Hence the density
\begin{equation}
|\wt{u}|^2 \rd \mu_e  = O(\rho_B\rho_D)
\end{equation}
and its restriction to each of the $X_j$ is the density $|u_j|^2\rd
\mu_j$.  Performing the integration, we see that
\begin{equation}
G_{\ve}(\wt{u}(\ve),\wt{u}(\ve)) =
 \sum_{j=0}^N G_j(u_j,u_j) + O(\ve)
\end{equation}
as required.
\end{proof}

\subsection{Infinitesimal translations} \label{S:metric_transl}

We now consider the differential of the gluing map $\Psi$ or $\Psi^c$ with respect to variation
in the base parameters, $\bncfg N$. In light of Proposition~4.3, it suffices to consider
translations.

Thus, fix $j \in \set{1,\ldots,N}$, let $V = \xi \cdot \pa_{z_j} = \xi_1
\pa_{z^1_j} + \xi_2 \pa_{z^2_j} + \xi_3 \pa_{z^3_j} \in \scV(\bncfg N)$ be an
infinitesimal translation along the $j$th Euclidean factor. In particular, as
an ordinary vector field, $V$ vanishes at $\ncfgs N = \pa \bncfg N$. 

In order to lift this to $\Z$, we first lift $V$ to $\Z'$ using the product
structure on the interior, $\mathring \Z' = \bbR^3 \times \ncfg N$, and
extension by continuity.  As discussed at the end of \S\ref{S:mon_ideal}, we
may then use the connection on $\base$ induced by the canonical connection on
the Gibbons-Manton torus bundle to obtain a horizontal lift $\wt V \in
\cV(\Z)$.

\begin{lem}
$\wt V$ vanishes at $\D$ and $\X_i$, $i \neq j$, and with respect to the identification 
of $\X_j$ with $\ol{\bbR^3} \times \idmon$ induced by Lemma~\ref{L:triv_Xj}, 
we have
\begin{equation}
	\wt V \rst \cX_j = (- \xi \cdot \pa_z, 0), \quad \cX_j \cong \ol{\bbR^3} \times \cI.
	\label{E:lift_translation_to_Z}
\end{equation}
\label{L:lift_translation_to_Z}
\end{lem}
\begin{proof}
As discussed in the proof of Lemma~\ref{L:triv_Xj}, 
coordinates on $\Z'$ near
the interior of $\X'_i$ are given by $(\ve, w_i, \zeta_1,\ldots, \zeta_N)$,
where $w_i = z - z_i$. The lift, $\wt V'$,  of $V$ to $\X_i'$ is therefore given by
\[
	\wt V' = \ve \xi \cdot \pa_{\zeta_j} - \delta_{ij}\xi \cdot \pa_{w_i},
\]
which vanishes at $\ve = 0$ if $i \neq j$ and gives a vector field like
\eqref{E:lift_translation_to_Z} otherwise, with respect to $\X'_j \cong
\ol{\bbR^3} \times \ncfgs N$. It is similarly easy to verify that $\wt V' =
\cO(\ve)$ near $\D'$ and $\B'$. 

Denoting the further lift of $\wt V'$ to $\Z$ by $\wt V$ using the
Gibbons-Manton connection, we note that any component of $\wt V$ in the
parameter directions vanishes over $\ve = 0$, since this is the lift with
respect to a smooth connection of $V$, which vanishes there; and
\eqref{E:lift_translation_to_Z} follows at once.
\end{proof}

Now let $(A,\Phi)$ represent a solution to $\Bogo(A,\Phi) = 0$ on $\Z \rst_{\K
\times [0,\ve_\K]}$ as obtained in the previous section.  To compute the
variation in $(A,\Phi)$ with respect to $\wt V$, we may use the connection $A$ 
itself to differentiate (recall that, while the subsequent modifications of the pregluing
connection were in the fiber directions, i.e., sections of $\gLam^1\otimes \adP$, $A$
is nevertheless a full connection on $P \to \Z$), which yeilds
\[
	\wt V \cdot (A,\Phi) = (\iota_{\wt V} F_A, \nabla_{\wt V} \Phi).
\]

In light of Lemma~\ref{L:lift_translation_to_Z}, we obtain the following result:
\begin{cor}
The variation $\wt V \cdot (A,\Phi)$ vanishes at $\D$ and $\X_i$ for $i \neq j$, while
\[
	\wt V \cdot (A,\Phi) \rst \X_j = (- \iota_{\xi} F_A, -\nabla_\xi \Phi) = -\tau_\xi,
\]
where $\tau_\xi$ was introduced in \S\ref{s11.25.11.15}.

\label{C:translation_on_solution}
\end{cor}

In particular, this is equivalent modulo $\cO(\ve)$ to a variation of the ideal monopole family
by the infinitesismal translation $-\xi$ in the $j$th factor, the metric evaluation of which 
was considered in the previous section.

%% file: sobolev.tex
In this section we prove the fundamental multiplicativity results for the
Sobolev spaces introduced in \S\ref{Ss:sobolev}.  It will be
sufficient to work fiberwise over $\idmon$, so for the remainder of the section we
consider a fixed fiber, $\fZ$, of $\fu : \Z \to \idmon$.

\begin{lem}
Let $\rho := \rho_\fD\rho_\fB.$ For $k > 2$, $l \geq 0$, multiplication of smooth functions
extends to bilinear maps on the Sobolev spaces of Definition~\ref{E:basic_sobolev}:
\begin{subequations}
\begin{gather}
	\rho^\alpha\bpH^{k,l}(\fZ;\olg) \times \rho^\alpha\bpH^{k,l}(\fZ; \olg) \to \rho^{2\alpha+3/2}\bpH^{k,l}(\fZ;\olg), 
	\label{E:mult_ppp}
	\\ \rho^\alpha \bpH^{k,l}(\fZ;\olg) \times \rho^\beta \bgH^{k,l}(\fZ;\olg) \to \rho^{\alpha + \beta + 3/2} \bgH^{k,l}(\fZ; \olg),
	\label{E:mult_pgg}
	\\ \rho^\beta\bgH^{k,l}(\fZ; \olg) \times \rho^\beta\bgH^{k,l}(\fZ; \olg) \to \rho^{2\beta - l + 3/2}\bpH^{k,l}(\fZ; \olg).
	\label{E:mult_ggp}
\end{gather}
\end{subequations}
\label{L:mult_in_H}
\end{lem}
\fixthmeq
\begin{proof}
Let $\bg = \rho^{-2} \olg$ be the associated $b$-metric on $\fZ$. Then
$(\mathring \fZ, \bg)$ is a complete Riemannian 4-manifold, which enjoys the
same Sobolev embedding results as $\bbR^4$ with respect to derivatives which
are bounded with respect to $\bg$, i.e., with respect to $b$-derivatives
$\bV(\fZ) \equiv \pV(\fZ)$. In particular, $\fuH^k(\fZ; \bg)$ is an algebra for
$k > 2$, and distributing derivatives via the Liebnitz formula, it follows
that, for $l \geq 0$, the spaces $\bpH^{k,l}(\fZ; \bg)$ and $\bgH^{k,l}(\fZ;
\bg)$ are algebras, and
\[
\begin{gathered}
	\rho^{\alpha'}\bpH^{k,l}(\fZ; \bg)\times \rho^{\beta'} \bpH^{k,l}(\fZ; \bg) \to \rho^{\alpha'+\beta'}\bpH^{k,l}(\fZ; \bg),
	\\\rho^{\alpha'}\bgH^{k,l}(\fZ; \bg)\times \rho^{\beta'} \bgH^{k,l}(\fZ; \bg) \to \rho^{\alpha' + \beta'}\bgH^{k,l}(\fZ; \bg).
\end{gathered}
\]
The results above then follow from the identity $\rho^{\alpha'} L^2(\fZ; \bg) =
\rho^\alpha L^2(\fZ; \olg)$ where $\alpha' = \alpha + \tfrac 3 2$, and the inclusions
\[
	\rho^\alpha \bpH^{k,l}(\fZ; \olg) \subset \rho^\alpha \bgH^{k,l}(\fZ; \olg), 
	\quad \rho^\beta \bgH^{k,l}(\fZ; \olg) \subset \rho^{\beta - l}\bpH^{k,l}(\fZ; \olg),
\]
which in turn follow from the fact that $\gV(\fZ) \ni X = \rho \wt X$ for $\wt X
\in \pV(\fZ).$
\end{proof}

\begin{proof}[Proof of Theorem~\ref{T:gauge_spaces}]
With respect to the splitting $\adP = \adPz\oplus \adPo$, the product on
$\Lambda^\ast \otimes \adP$ decomposes as
\[
	[u,v]_0 = [u_1,v_1], 
	\quad [u,v]_1 = [u_0,v_1] + [u_1,v_0].
\]
Boundedness of the products  $\threeH^{k,l' ,\beta}\times \threeH^{k,l,\beta}
\to \threeH^{k,l,\beta}$ for $l' \geq l$ then follows from
\eqref{E:mult_pgg} and \eqref{E:mult_ggp}

For the gauge group, we work in the universal enveloping algebra $\envU(\adP)$,
here identifiable with $2\times 2$ complex matrices locally. Near $\fD \cup \fB$, we
have a splitting
\[
	\envU(\adP) = \envU(\adP)_0 \oplus \envU(\adP)_1
\]
consistent with the splitting of $\adP \subset \envU(\adP)$; indeed, we may
take $\envU(\adP)_0$ and $\envU(\adP)_1$ to be the diagonal and anti-diagonal
matrices, respectively. The product in $\envU(\adP)$ then decomposes as
\[
	(uv)_0 = u_0v_0 + u_1 v_1,
	\quad (uv)_1 = u_0v_1 + u_0 v_1,
\]
and it follows from Lemma~\ref{L:mult_in_H},
\eqref{E:mult_ppp}--\eqref{E:mult_ggp} that $\threeH^{k,2,\beta}(\fZ;
\envU(\adP))$ is an algebra. Adjoining a unit and exponentiating in the algebra
$1 + \threeH^{k,2,\beta}(\fZ; \envU(\adP))$, we obtain the gauge group
$\threeH^{k,2,\beta}(\fZ; \Ad P)$, with Lie algebra $\threeH^{k,2,\beta}(\fZ;
\adP)$ as claimed. The action of this group on the spaces
$\threeH^{k,l,\beta}(\fZ ; \Lambda^\ast\otimes \adP)$ follows from boundedness
of the infinitesimal action.
\end{proof}

\subsection{Sobolev spaces for $\fX_i$} \label{S:sobolev_Eucl}

In appendix~\ref{S:linear} we require hybird b/scattering Sobolev spaces on the fibers
$\fX_i$. Thus let
\[
	\bscH^{k,l}(\fX_i;V) \ni v \iff \bV^{k'}\cdot \scV^{l'} v \in L^2(\fX_i;V; g),
	\quad \forall\ k' \leq k,\ l' \leq l.
\]
Here $L^2(\fX_i;V;g)$ is defined with respect to the induced metric on $\fX_i$;
from Proposition~\ref{P:lifted_metric} this is the Euclidean
metric with respect to the identification $\fX_i \cong \ol{\bbR^3}.$ The spaces
$\bscH^{k,l}(\fX_i;V)$ are Hilbert spaces with respect to inner products
constructed from any choices of b and scattering connections on $V$. We consider
also weighted versions $\rho^\alpha \bscH^{k,l}(\fX_i;V)$, where $\rho = \rho_\fD$.
The next result is proved in \cite{kottke2015callias}.

\begin{prop}
If $\alpha' \geq \alpha$, $k' \geq k$ and $l' \geq l$, then
\[
	\rho^{\alpha'} \bscH^{k',l'}(\fX_i; V) \subset \rho^\alpha \bscH^{k,l}(\fX_i; V).
\]
Furthermore, if $\alpha' > \alpha$ and either $k' > k$ or $l' > l$, then the inclusion
is compact. If $\alpha \geq \beta + l$, then
\[
	\rho^\alpha \bscH^{k,l}(\fX_i; V) \subset \rho^\beta \bH^{k+l}(\fX_i; V).
\]
\label{P:bsc_spaces}
\end{prop}

For $V = \adP$, with the associated splitting $\adP = \adPz\oplus \adPo$ near
$\pa \fX_i$, we define the {\em split Sobolev spaces} (cf.\ \S\ref{Ss:sobolev})
\[
	\bhscH^{k,l}(\fX_i;\adP) \simeq \bH^{k+l}(\fX_i; \adPz)\oplus \bscH^{k,l}(\fX_i; \adPo)
\]
via the norm
\[
	\bhscH^{k,l}(\fX_i; \adP) 
	\ni v \iff
	\norm{\chi v_0}_{\bH^{k+l}} 
	  + \norm{\chi v_1}_{\bscH^{k,l}} 
	  + \norm{(1 - \chi) v}_{H^{k+l}} < \infty,
\]
where $\chi$ is a smooth cutoff supported near $\pa \fX_i = \fD$. We denote
weighted versions of these spaces by 
\[
	\rho^{\alpha,\beta}\bhscH^{k,l}(\fX_i; \adP) 
	\simeq \rho^\alpha \bH^{k+l}(\fX_i;\adPz)
	\oplus \rho^\beta\bscH^{k,l}(\fX_i,\adPo),
\]
where $\rho = \rho_\fD$.

%% file: coulomb.tex
This section is devoted to a proof of Theorem~\ref{T:coulomb}.
Combining \eqref{E:gauge_action} and \eqref{E:coulomb_gauge}, the condition
that $\gamma\cdot(A + a,\Phi+\phi)$ be in Coulomb gauge with respect to $(A,\Phi)$
amounts to the condition $G(a,\phi,\gamma) = 0$, where
\begin{equation}
	G(a,\phi,\gamma) = d_A^\ast(a - (d_{A + a}\gamma)\gamma^{-1}) - \ad \Phi(\gamma(\Phi + \phi)\gamma^{-1}).
	\label{E:gauge_fixing_nonlin}
\end{equation}

\begin{lem}
For $k \geq 3$, 
\eqref{E:gauge_fixing_nonlin} extends to a differentiable map
\begin{equation}
	 G : \sH^{k,1}(\Z; \Lambda^1\otimes \adP) 
	  \times \sH^{k,1}(\Z; \adP)
	  \times \sH^{k,2}(\Z; \Ad P) 
	  \to \sH^{k,0}(\Z; \adP).
	\label{E:gauge_fixing_nonlin_extn}
\end{equation}
with derivative
\begin{equation}
	F := \pa_\gamma G(0,0,1) = d_A^\ast d_A - \ad \Phi^2 : \sH^{k,2}(\Z; \adP) 
	\to \sH^{k,0}(\Z; \adP).
	\label{E:gauge_fixing_lin_zero}
\end{equation}
\label{L:gauge_mult_lin}
\end{lem}
\begin{proof}
Fix a fiber
$\fZ$ of $\fu : \Z \to \idmon$.
That \eqref{E:gauge_fixing_nonlin_extn} is a bounded map over $\fZ$ follows from Theorem~\ref{T:gauge_spaces},
and the diagonality assumption on $\Phi$, from which it follows that 
$\ad \Phi$ maps $\threeH^{k,l}$ to $\threeH^{k,l} \cap \threeH^{k,0}$
for $l = 1,2$ (while $\threeH^{k,2} \not \subset \threeH^{k,1} \not \subset \threeH^{k,0}$,
$\ad \Phi$ kills the $\adPz$ components to infinite order). Since all the nonlinear
terms are simple products, \eqref{E:gauge_fixing_nonlin_extn} is an analytic map.

Setting $\gamma = 1 + \eta$, where $\eta \in \threeH^{k,2}(\fZ; \adP)$
and discarding terms of quadratic and higher order in $\eta$, we obtain the
linearization
\begin{equation}
\begin{gathered}
	\pa_\gamma G(a,\phi,1) \eta = d_A^\ast d_{A + a} \eta - \ad \Phi \ad (\Phi + \phi) \eta.
	\label{E:gauge_fixing_lin}
\end{gathered}
\end{equation}
Setting $(a,\phi) = (0,0)$ gives \eqref{E:gauge_fixing_lin_zero}.  Letting
$\fZ$ vary, it is clear that, as bounded operators, $G$ and $F$ vary smoothly
over $\idmon$.
\end{proof}

We will show that \eqref{E:gauge_fixing_lin_zero} is invertible for
sufficiently small $\ve$ and appeal to the implicit function theorem.

The restrictions of $F$ in \eqref{E:gauge_fixing_lin_zero} to the boundary faces $\fX$ and
$\fD$ of $\fZ$ are analyzed in \S\ref{S:linear_coulomb} where they are shown
to be invertible, and in \S\ref{S:parametrices} we construct a smooth family of
fiberwise parametrices for $F$ on $\fu : \Z \to \idmon.$ There we prove

\begin{prop}
There exist right and left parametrices $Q^R$ and $Q^L$ for $F$ such that, for some
$0 < \delta < \tfrac 1 2$,
\begin{equation}
	FQ^R = I - \ve^\delta E^R,
	\quad Q^LF = I - \ve^\delta E^L,
	\label{E:coulomb_parametrix}
\end{equation}
where $E^R$ and $E^L$ extend to (fiberwise bounded) linear maps on
$\sH^{k,0}(\Z; \adP))$ and
$\sH^{k,2}(\Z; \adP))$, respectively.
\label{P:coulomb_parametrix}
\end{prop}

\begin{proof}[Proof of Theorem~\ref{T:coulomb}]
Fixing a compact set $\K \subset \idmon$, and $\ve_\K > 0$, the error terms in
\eqref{E:coulomb_parametrix} satisfy bounds of the form $C\ve_\K^\delta$ on
$\sH^{k,\ast}(\Z\rst_{\K\times[0,\ve_\K]}; \adP)$, and,
making $\ve_\K$ sufficiently small, can be inverted by Neumann series. 

It follows that 
\[
	F : \sH^{k,2}(\Z\rst_{\K\times[0,\ve_\K]}; \adP))
	\to \sH^{k,0}(\Z\rst_{\K\times[0,\ve_\K]}; \adP))
\]
is invertible map of Hilbert bundles over $\K$, and then the existence of a 
unique $\gamma$ satisfying $G(a,\phi,\gamma) = 0$ for $(a,\phi) \in \sH^{k,1}(\Z\rst_{\K\times[0,\ve_K]})$
sufficiently small (with respect to $\sup_\K \norm {\cdot}_{\threeH^{k,1}}$) is 
a consequence of the implicit function theorem.

For the regularity statement it suffices to work on a fixed fiber $\fZ$. We
proceed by induction on $l$, showing that there are unique solution maps
\begin{equation}
\begin{aligned}
	U_{l-1} \subset &\threeH^{k+l-1,1}(\fZ; \adP) \to \threeH^{k+l-1,2}(\fZ; \Ad P)
	\\ U_{l} = i^{-1}(U_{l-1}) \subset &\threeH^{k+l,1}(\fZ; \adP) \to \threeH^{k+l,2}(\fZ; \Ad P)
	\\&(a,\phi) \mapsto \gamma \quad \text{s.t.}\ G(a,\phi,\gamma) = 0,
	\label{E:coulomb_unique_pair}
\end{aligned}
\end{equation}
where the $U_l$ are convex open neighborhoods of the origin and $i :
\threeH^{k+l,1} \to \threeH^{k+l-1,1}$ denotes the natural
inclusion. In particular, the domains don't decrease with $l$.

The above construction, applied with $k$ and $k+1$, furnishes the base case;
shrinking $U_0$ if necessary, we may assume that $U_1 = i^{-1}(U_0)$.  For the
inductive step, suppose $(a,\phi) \in i^{-1}(U_{l}) \subset \threeH^{k+l+1,1}$ with
the solution $\gamma = \exp(\eta)$, $\eta \in \threeH^{k+l,2}(\fZ; \adP)$.
Let $V \in \bV(\fZ)$ be an arbitrary b vector field, and consider 
\begin{multline}
	0 = V \cdot \big(G(a,\phi,\exp(\eta))\big)
	\\= (V \cdot G)(a,\phi,\exp(\eta)) + G_1(a,\phi,\exp(\eta)) V(a,\phi)
	+ G_2(a,\phi,\exp(\eta)) V\eta.
	\label{E:coulomb_implicit_differentiation}
\end{multline}
Here $G_1(a,\phi,\exp(\eta)) = \pa_{(a,\phi)} G(a,\phi,\exp(\eta))$ is the
linearization of $G$ with respect to the $(a,\phi)$ variables,
$G_2(a,\phi,\exp(\eta)) = \pa_{\eta} G(a,\phi,\exp(\eta))$ is the linearization
with respect to $\eta$, and $(V \cdot G)$ denotes all terms where $V$
differentiates the coefficients of $G$, i.e., where $V$ differentiates
a term in the background configuration $(A,\Phi)$.

By the smoothness assumption on $(A,\Phi)$, the proof of
Lemma~\ref{L:gauge_mult_lin} applies to $(V\cdot G)$, and we conclude that 
\[
	(V\cdot G) : \threeH^{k+l,1}(\fZ; \adP) \times \threeH^{k+l,2}(\fZ; \Ad P)
	\to \threeH^{k+l,0}(\adP)
\]
is a $C^1$ map. Likewise, the linear map
\[
	G_1(a,\phi,\exp(\eta)) : \threeH^{k+l,1}(\fZ; \adP) \to \threeH^{k+l,0}(\fZ; \adP)
\]
is bounded, and $V(a,\phi)$ is in $\threeH^{k+l,1}$ by assumption.
Finally, as a result of the inductive hypothesis \eqref{E:coulomb_unique_pair},
it follows that
\[
\begin{aligned}
	G_2(a,\phi,\exp(\eta)) &: \threeH^{k+l-1,2}(\fZ; \adP) \to \threeH^{k+l-1,0}(\fZ; \adP),
	\\G_2(a,\phi,\exp(\eta)) &: \threeH^{k+l,2}(\fZ; \adP) \to \threeH^{k+l,0}(\fZ; \adP)
\end{aligned}
\]
are isomorphisms, with inverses which coincide where defined. Rearranging
\eqref{E:coulomb_implicit_differentiation}, we conclude that
\begin{multline*}
	V\eta = - G_2^{-1}(a,\phi,\exp(\eta))
		\big((V\cdot G)(a,\phi,\exp(\eta)) + G_1(a,\phi,\exp(\eta)) V(a,\phi)\big)
	\\ \in \threeH^{k+l,2}(\fZ; \adP).
\end{multline*}
Since $V$ was arbitrary, in fact $\eta \in \threeH^{k+l+1,2}$. Letting
$(a,\phi)$ vary in $i^{-1}(U_l)$ and appealing to the uniqueness of the
solution $\gamma$, we conclude that $(a,\phi) \mapsto \gamma$ such that
$G(a,\phi,\gamma) = 0$ defines a $C^1$ map
\[
	U_{l+1} := i^{-1}(U_l) \subset \threeH^{k+l+1,1}(\fZ; \adP) 
	  \to \threeH^{k+l+1,2}(\fZ; \Ad P),
\]
which completes the induction.
\end{proof}

%% file: linear.tex
\subsection{Linear analysis of $L_{\fX_i}$} \label{S:linear_X}

In this section we give the analysis of the operator $L_{(A,\Phi)}$
\eqref{E:Bogo_expansion} as 
needed for the construction of the formal solution in \S\ref{S:formal_formal} over
the Euclidean boundary hypersurfaces $\X_i$.  Thus let $X$ be the
radial compactification of $\RR^3$, with boundary defining function
$\rho$ and let $(A,\Phi)$ be a smooth solution of the Bogomolny
equations.   This operator has been considered previously by Taubes
and more systematically by the first author.  We need to refine the
parametrix found by Kottke in \cite{kottke2015dimension} for applications in this paper;
on the other hand there are some simplifications that result from $X$
being the radial compactification of $\RR^3$ rather than a general
scattering manifold.  Denote by $\Lambda$ the bundle $\scLam^1\oplus \scLam^0$ 
over $\RR^3$ and by $\frp$ the complexification of the adjoint bundle $\ad(P)$. Then
\begin{equation}
L_{\fX} = L_{(A,\Phi)}\begin{bmatrix} a \\ \phi\end{bmatrix}
= \begin{bmatrix} *\rd_A & -\rd_A \\ -\rd_A^* & 0 \end{bmatrix}
\begin{bmatrix} a \\ \phi\end{bmatrix}
+ \ad(\Phi)\begin{bmatrix} a\\ \phi \end{bmatrix}.
\end{equation}
where $(a,\phi)\in \Lambda\otimes \frp$.

We begin with the consequences of the decomposition $\frp = \frp_0
\oplus \frp_1$ near $\pa X$.  Let $\hat{\Phi} = \Phi/|\Phi|$.  

\begin{lem}
Let $C = \nabla_A\hat{\Phi}$, defined over a collar neighbourhood
$\cU$ of $\pa X$. Then 
\begin{equation}
C \in \rho^\infty C^\infty(\cU,\frp_1).
\end{equation}
\end{lem}
\begin{proof}
By differentiation of $|\hat{\Phi}|^2=1$, we get $\langle
  \hat{\Phi},C\rangle =0$, showing that $C \in
  C^\infty(\cU,\frp_1\otimes \scT^*)$.   We also have
\begin{equation}
C = \nabla_A(|\Phi|^{-1}\Phi) = \rd(|\Phi|^{-1})\Phi +
|\Phi|^{-1}\nabla_A\Phi
\end{equation}
Since also $\ad\hat{\Phi} C = C$ by definition of $\frp_1$, it follows
that
\begin{equation}
C = |\Phi|^{-1}\ad(\hat{\Phi})\nabla_A\Phi
\end{equation}
and so the rapid decay of $C$ follows from that of
$\ad(\Phi)\nabla_A\Phi$ discussed in \S\ref{S:bogo1}.
\end{proof}

Given any section $u$ of $\frp$ define $u_0 \in C^\infty(\cU,\CC)$ by
\begin{equation}
u_0 = \langle \hat{\Phi},u\rangle.
\end{equation}
Then 
\begin{equation}
u_1 = u - u_0\hat{\Phi}
\end{equation}
is a section of $\frp_1$ and thus satisfies
\begin{equation}
\ad\hat{\Phi}u_1 = u_1.
\end{equation}
Denote by $\nabla_1$ the connection on $\frp_1$ induced by projection
of $\nabla_A$ on $\frp_1$.
\begin{prop}
Under the identification $\frp = \frp_0 \oplus \frp_1$, $\frp_0 \simeq
\CC$ just described,
\begin{equation}
\nabla_A\begin{bmatrix} u_0 \\ u_1\end{bmatrix}
=
\begin{bmatrix} \rd u_0 \\ \nabla_1 u_1\end{bmatrix}
+
\begin{bmatrix}  0 & \ad(C) \\ C & 0\end{bmatrix}
\begin{bmatrix} u_0 \\ u_1\end{bmatrix}
\label{e1.11.8.15}
\end{equation}
\end{prop}
\begin{proof}
If $u_0\in C^\infty(\cU,\CC)$, then we calculate
\begin{equation}
\nabla(u_0\hat{\Phi}) = \rd u_0\otimes \hat{\Phi} + u_0 C.
\end{equation}
Thus the $\frp_0$ and $\frp_1$ components are precisely $\rd u_0$ and
$u_0 C$, proving the first line of \eqref{e1.11.8.15}. If $u_1 \in
C^\infty(\cU,\frp_1$, then by definition
\begin{equation}
\nabla_1 u_1 = \ad(\hat{\Phi})\nabla_A u_1.
\end{equation}
Differentiating the equation $u_1 = \ad(\hat{\Phi})u_1$, we obtain
\begin{equation}
\nabla_A u_1 = \ad(\hat{\Phi})u_1 + \ad(C)u_1= \nabla_1u_1 + \ad(C)
u_1.
\end{equation}
The result is proved.
\end{proof}
As a consequence, we have (c.f.\ \eqref{E:backL})
\begin{prop}
Over $\cU$, relative to the decomposition $\frp = \frp_0\oplus
\frp_1$, 
\begin{equation}
L_{\fX} = \begin{bmatrix} L & C_1 \\ C_1^* & L_1 +
  \Phi_1 \end{bmatrix}
\end{equation}
where
\begin{equation}
L = \begin{bmatrix} *\rd & -\rd \\ -\rd^* & 0 \end{bmatrix},
\end{equation}
is the euclidean Hodge-de Rham operator on $\RR^3$, 
$L_1$ is the same operator coupled to $\frp_1$, 
$\Phi_1$ denotes $\ad(\Phi)$
acting on $\frp_1$ and $C_1$ is a zeroth order $O(\rho^\infty)$ term.
\end{prop}

Thus, near $\cU$, $L_{\fX}$ behaves like a sum of the uncoupled euclidean
Hodge-de Rham operator $L$ and the fully elliptic scattering operator
$L_1 + \Phi_1$. 

As is well known, Fredholm extensions, solvability and boundary
regularity properties of the $\frp_0$ component of $L_{\fX}$ are
therefore governed by the {\em homogeneous solutions} of $Lu=0$.

\begin{prop}
Suppose that $Lu=0$ over $\RR^3\setminus 0$ and $u$ is homogeneous of
degree $\alpha$.  Then if $\alpha \geq 0$, it follows that $\alpha =
n$ is a non-negative integer, and there is a homogenous harmonic
polynomial $h$ of degree $n+1$, such that
\begin{equation}\label{e2.11.8.15}
u = L h
\end{equation}
If $\alpha<0$ then $\alpha = -2 - n$, where $n$ is a non-negative
integer, and there is a harmonic polynomial, homogeneous of degree
$n$, such that
\begin{equation}\label{e3.11.8.15}
u = L(|z|^{-2n-1}h)
\end{equation}
\label{p1.11.8.15}\end{prop}

\begin{proof}
It is clear that \eqref{e2.11.8.15} and \eqref{e3.11.8.15} do give
homogenous solutions of the given degree, because $L^2 =
\Delta$.  Let us consider \eqref{e3.11.8.15}. If $Lu=0$ and $u$ is
homogeneous of some negative degree $\alpha$, then $Lu$ extends to
$\RR^3$ uniquely as a homogeneous distribution supported at $0$. The
only possibility is a linear combination of derivatives of the Dirac
distribution $\delta_0$.  Since this distribution is of degree $-3$,
it follows that $s=Lu$ can only be homogeneous of degree $-3-n$, for
some $n\geq 0$.  Applying $L$, $\Delta u =Ls$ and $u = Lf$, where $f$
is the unique homogeneous solution of $\Delta f = s$. Thus $f$ has the
form $|z|^{-2n-1}h$ where $h$ is an $\RR^4$-valued harmonic
polynomial, homogeneous of degree $n$, proving the result.

The proof of \eqref{e3.11.8.15} is similar; by elliptic regularity, if
the homogeneity is non-negative, then $Lu=0$ on $\RR^3$. So $u$ is an $\RR^4$-valued
harmonic polynomial, homogeneous of degree $n\geq 0$. However any such
polynomial can be written $u = \Delta v$ (take $v$ a multiple of
$|z|^2u$). Then $u = L(Lv)$ and $Lv$ is an $\RR^4$-valued function,
homogeneous of degree $n+1$. 
\end{proof}

The remainder of the section is devoted to a proof of the following 
\begin{thm}
\mbox{}
\begin{enumerate}
[{\normalfont (a)}]
\item \label{I:LX_package_fred}
For $m \geq 0$, $\beta \in \bbR$, and $\alpha \in (-1,1)$ the bounded extension 
\begin{equation}
	L_{\fX} : \rho^{\alpha-1/2,\beta}\bhscH^{m,1}(\fX; \Lambda\otimes \adP) 
	\to \rho^{\alpha+1/2,\beta}\bhscH^{m,0}(\fX; \Lambda\otimes \adP)
	\label{E:LX_extn}
\end{equation}
is Fredholm and surjective. The spaces here are defined in \S\ref{S:sobolev_Eucl}.
The null-space $N$ is of complex dimension $4k$ and if $u\in N$ is
decomposed as $u = u_0 + u_1$   relative to $\frp = \frp_0 \oplus
\frp_1$ near the boundary, 
\begin{equation}
u_0 \in \rho^2C^\infty(\cU), u_1 \in \rho^\infty C^\infty(\cU,\frp_1).
\end{equation}

\item \label{I:LX_package_solv} There is a right-inverse $G$ of
  \eqref{E:LX_extn} with range equal to the $L^2$ orthogonal
  complement of $N$ with the following roperty. 
If $f \in \cAphg^\ast(\fX; \Lambda\otimes \adP)$ with $f = f_0\oplus f_1$
near $\pa \fX$,
\begin{equation}
f_i \in \cAphg^{F_i}(\fX; \Lambda \otimes\adPi),
\end{equation} 
then $u= Gf$ solves $Lu =f$ with $u \in \cAphg^\ast(\fX;
\Lambda\otimes \adP)$.  Moreover, decomposing $u=u_0+u_1$ near $\pa
X$, we have $u_i \in \cAphg^{E_i}$, where
\[
\begin{aligned}
	E_0 &= 
	  \wh {2} \ol \cup (F_0 - 1), \\ 
	E_1 &= F_1.
\end{aligned}
\]
\item \label{I:LX_package_G}  Let $\cU$ be a product neighbourhood of
$\pa X$ as before and consider the Schwarz kernel of $G$ on
restricted to $\cU\times \cU$. Then $G|\cU\times \cU$ decomposes with
respect 
to $\Hom(\pi_R^\ast (\adPz\oplus \adPo), \pi_L^\ast (\adPz \oplus
\adPo))$ 
\begin{equation}
\begin{gathered}
	G = \begin{pmatrix} \rho \wt G_0\rho^{-2} & G_{01}\\G_{10} & G_1\end{pmatrix},
	\\ \wt G_0 \in \bPsi^{-1,\cF}(\fX; \Lambda\otimes\adPz),
	\quad \cF = (F_L,F_R,F_F), \quad F_L,\ F_R,\ F_F \geq 0.
	\\ G_1 \in \scPsi^{-1,0}(\fX; \Lambda\otimes\adPo),
	\quad G_{ij} \in \rho^\infty \Psi^{-1}(\fX; \Lambda\otimes\adP_j,\Lambda\otimes\adP_i).
\end{gathered}
	\label{E:linear_bogo_X_inverse}
\end{equation}
Here $\rho^\infty \Psi^{-1}$ is well-defined in either calculus as
$\rho^\infty \Psi^{-1} = \bPsi^{-1,(\infty,\infty,\infty)} = \scPsi^{,\infty}$.
\end{enumerate}
\label{T:LX_package}
\end{thm}

The proof of this is an elaboration of work of the first author in \cite{kottke2015dimension} and
starts from a parametrix construction using both the $\bo$ and $\scat$ calculi
as well as the pseudo-differential operators in $\rho^{\infty}
\Psi^s(\fX)$.  Note that 
this forms a bi-ideal in either calculus, in the sense that
\begin{equation}
\begin{gathered}
	\rho^\infty\Psi^k \circ \bPsi^{l,\cE} \subset \rho^\infty \Psi^{k+l}, 
	\quad \bPsi^{l,\cE} \circ \rho^\infty\Psi^k \subset \rho^\infty \Psi^{k+l},
	\\ \rho^\infty\Psi^k \circ \scPsi^{l,e} \subset \rho^\infty \Psi^{k+l}, 
	\quad \scPsi^{l,e} \circ \rho^\infty\Psi^k \subset \rho^\infty \Psi^{k+l}.
\end{gathered}
	\label{E:sc_b_ideal}
\end{equation}

It is convenient to replace $L$ by $\wt{L} = \rho^{-2}L\rho$ in what
follows.  We also omit mention of $\Lambda$ since this is a passenger
and the important thing is the splitting $\frp = \frp_0\oplus
\frp_1$.  In order to construct a parametrix near the boundary for the
$\frp_0$ part $\wt L$, we need to know its indicial roots:
\begin{lem}
$\bspec(\wt L) = \pm \bbN_1 = \bbZ \setminus \set{0}$.
\label{P:LX_ind_roots}
\end{lem}
\begin{proof} This follows at once from Proposition~\ref{p1.11.8.15}
  and the definition $\wt L = \rho^{-2} L \rho$.  The calculation is
  done in the more general setting that $X$ is an 
arbitrary scattering $3$-manifold in  \cite{kottke2015dimension}, by identifying
$L$ with the odd-signature operator on $X$.
\end{proof}

As an initial step, let $Q$ be a distribution on $\fX^2$ conormal to the
diagonal on the interior, with principal symbol inverting that of $L_{\fX}$,
and decomposing near $\cU^2$ as 
\[
	Q = \begin{pmatrix} \rho \wt Q_0\rho^{-2} & 0\\0 & Q_1\end{pmatrix}.
\]
Here we assume that $Q_1 \in \scPsi^{-1,0}(\fX; \adPo)$ has scattering symbol
inverting that of $L_1 + \Phi_1$ (which is invertible by the fact that $L_1$
is self-adjoint, $\Phi_1$ is skew-adjoint and nondegenerate on $\frp_1$, and
these commute to leading order at $\pa \fX$), and we assume that $\wt Q_0 \in
\bPsi^{-1,(\wh 1_L, 1_R, 0_F)}(\fX; \adPz)$ satisfies $I - \wt L \wt Q_0 \in
\bPsi^{-1,(\infty_L,1_R,1_F)}(\fX;\adPz)$, following the first few standard steps
in the construction of parametrices in the b-calculus \cite{RBMgreenbook}.  More precisely we
assume that that $\wt Q_0$ has interior principal symbol inverting that of $\wt
L$, that the indicial operator $I(\wt Q_0)$ is obtained by taking the inverse
Mellin transform of $I(\wt L, \lambda)$ along $\re (\lambda) = \alpha \in
(-1,1)$ (which is free of indicial roots by Proposition~\ref{P:LX_ind_roots};
the indicial roots to the right and left of this line contribute the index set
$1$ at the right and left faces for $\wt Q_0$), and that the Schwartz kernel of
$\wt Q_0$ is in the formal nullspace of the lift of $\wt L$ to
$\fX^2_\mathrm{b}$ at the left face, contributing the index set $\wh 1$ for
$\wt Q_0$ and the rapid vanishing of $I - \wt L \wt Q_0$ there. These
assumptions are all consistent with the choice of interior conormal symbol of
$Q$.

It follows that the initial error term $E = I - L_{\fX} Q$ has the form
\[
\begin{gathered}
	E = \begin{pmatrix} \rho^{2} \wt E_0 \rho^{-2} & E_{01} \\ 
	 E_{10} & E_1\end{pmatrix},
	\\ \wt E_0 \in \bPsi^{-1,(\infty,1,1)}(\fX; \adPz),
	\quad E_1 \in \scPsi^{-1,1}(\fX; \adPo),
	\\ E_{01} = - C_1 Q_1 \in \rho^\infty\Psi^{-1}(\fX; \adPo, \adPz)
	\\ E_{10} = - C_1^* \rho \wt Q_0 \rho^{-2} \in \rho^\infty\Psi^{-1}(\fX; \adPz, \adPo),
\end{gathered}
\]
near $(\pa \fX)^2$, with interior conormal singularity of order $-1$. This term
may now be removed by Neumann series. Indeed, absorbing terms $\rho^\infty
\Psi^{-N}$ into $\bPsi^{-N,(\infty,\infty,\infty)}$ and $\scPsi^{-N,\infty}$,
it follows that
\[
\begin{gathered}
	E^N = \begin{pmatrix} \rho^{2} \wt E_0^\pns N \rho^{-2} & E_{01}^\pns N\\E_{10}^\pns N & E_1^\pns N\end{pmatrix},
	\\ \wt E_0^\pns N \in \bPsi^{-N,(\infty,\wh 1_N, N)}(\fX; \adPz),
	\quad E_1^\pns N \in \scPsi^{-N,N}(\fX; \adPo),
	\\ E_{ij}^\pns N \in \rho^\infty \Psi^{-N}
\end{gathered}
\]
where $\wh 1_N = 1 \ol \cup 2 \ol\cup \cdots \ol \cup N.$ The series
$\sum_{k=0}^\infty E^k$ may be summed asymptotically, resulting in an operator
$I - S$, where $S$ has an expression similar to $E$, except that $\wt S_0 \in
\bPsi^{-1,(\infty, \wh 1, 1)}(\fX; \adPz).$ Denoting an improved right parametrix by
$Q^R = Q(I - S)$, it follows that the new error term $E^R = I - L_\fX Q^R$ has the form
\[
\begin{gathered}
	E^R = \begin{pmatrix} \rho^{2} \wt E_{0}^R \rho^{-2} & E_{01}^R \\E_{10}^R & E_1^R\end{pmatrix},
	\\ \wt E_0^R \in \bPsi^{-\infty,(\infty,\wh 1, \infty)}(\fX; \adPz),
	\quad E_1^R \in \scPsi^{-\infty,\infty}(\fX; \adPo),
	\\ E_{ij}^R \in \rho^\infty \Psi^{-\infty}
\end{gathered}
\]
The parametrix $Q^R$ has a block form similar to the block form of $Q$, with $\wt Q^R_0 \in
\bPsi^{-1,(\wh 1, \wh 1 \ol \cup 2, \wh 2 \ol \cup 1 \cup 0)}(\fX; \adPz)$, and
off-diagonal terms in $\rho^\infty\Psi^{-1}$. The precise form of the index
sets is not as important as the statement that $\wt Q^R_0$ lies in the
b-calculus.

A similar construction gives a left parametrix $Q^L$ with error term $E^L  = I
- Q^L L_\fX$ of the form
\[
\begin{gathered}
	E^L = \begin{pmatrix} \rho \wt E_{0}^L \rho^{-1} & E_{01}^L \\E_{10}^L & E_1^L\end{pmatrix},
	\\ \wt E_0^L \in \bPsi^{-\infty,(\wh 1,\infty, \infty)}(\fX; \adPz),
	\quad E_1^L \in \scPsi^{-\infty,\infty}(\fX; \adPo),
	\\ E_{ij}^L \in \rho^\infty \Psi^{-\infty}
\end{gathered}
\]

With these parametrices in hand, we may prove the following:
\begin{prop}
\mbox{}
\begin{enumerate}
[{\normalfont (a)}]
\item \label{I:LX_fred_null_fred}
The extension \eqref{E:LX_extn} is Fredholm and surjective for $\alpha \in (-1,1)$, $m \geq 0$ and $\beta \in \bbR$.
\item \label{I:LX_fred_null_null}
The nullspace of such an extension consists of polyhomogeneous sections, with
\[
\begin{gathered}
	\Null(L_\fX) \ni u = u_0 \oplus u_1, \quad \text{near $\pa \fX$},
	\\ u_0 \in \rho^2C^\infty(\fX; \adPz), \quad u_1 \in \rho^\infty C^\infty(\fX; \adPo).
\end{gathered}
\]
\end{enumerate}
\label{P:LX_fred_null}
\end{prop}

\begin{proof}
To see that the extension is Fredholm, it suffices to verify
that $E^L$ and $E^R$ are compact on the appropriate spaces. Over the interior of $\fX$
this is clear, and near the boundary we have
\[
\begin{aligned}
	\rho\wt E^L_0 \rho^{-1} &: \rho^{\alpha-1/2}\bH^{m+1} 
	  \to \rho^{\alpha+1/2}\bH^{\infty} \subset \rho^{\alpha-1/2} \bH^{m+1},
	\\ E^L_1 &: \rho^\beta \bscH^{m,1} 
	  \to \rho^\infty \bscH^{\infty,1} \subset \rho^\beta \bscH^{m,1}
	\\ E_{10} &: \rho^{\alpha-1/2} \bH^{m+1} 
	  \to \rho^\infty \bH^\infty \subset \rho^{\beta} \bscH^{m,1}
	\\ E_{01} &: \rho^\beta \bscH^{m,1} 
	  \to \rho^\infty \bscH^{\infty,1} \subset \rho^{\alpha-1/2} \bH^{m+1}
\end{aligned}
\]
where all inclusions are compact by Proposition~\ref{P:bsc_spaces}. The 
argument for $E^R$ is similar.

For \eqref{I:LX_fred_null_null}, observe that $E^L$ maps
$\rho^{\alpha-1/2,\beta}\bhscH^{k,1}$ into $\rho^{0,\infty}\cAphg^{\wh 2}$,
meaning polyhomogeneous sections whose $\adPo$ components are rapidly
vanishing and whose $\adPz$ components have index set $\wh 2 = \wh 1 + 1$.  Thus,
supposing $L_\fX u = 0$ and applying the left parametrix, it follows that $u =
E^L u \in \rho^{0,\infty}\cAphg^{\wh 2}$.

In fact, since $\fX$ is Euclidean, $\wt L \equiv I(\wt L)$ agrees identically
with its indicial operator in a neighborhood of $\pa \fX$. It then follows from
$L_\fX u = 0$ that $I(\wt L) \rho^{-1} u_0 = 0 \mod \rho^\infty C^\infty$, and
so $u_0 \in \rho^2C^\infty$.

Surjectivity will follow from injectivity for the adjoint operator 
\begin{equation}\label{e11.11.8.15}
	L_{\fX}^\ast : \rho^{-\alpha-1/2,-\beta} \bhscH^{m,1} \to \rho^{-\alpha+1/2,-\beta}\bhscH^{m,0}.
\end{equation}
Here we are considering the adjoint determined by the $L^2$ pairing between
weighted spaces $\rho^{\alpha,\beta} L^2$ and $\rho^{-\alpha,-\beta}L^2$, with
$\bhscH^{m,l}$ functioning as domains for the unbounded operators $L_{\fX}$ and
$L_{\fX}^\ast$. Note that $-\alpha \in (-1,1)$ since $\alpha \in (-1,1)$ by
assumption.  Since $L_{\fX}^\ast = L_A - \ad(\Phi)$ differs from
$L_{\fX}$ only in the sign of the zeroth order term, we may apply
the foregoing analysis to it, and in 
particular deduce that if $u$ is in the null space of
\eqref{e11.11.8.15}, then it is $O(\rho^2)$ (and conormal).

Now we use the basic identity 
\begin{equation}\label{E:LX_bochner}
L_{\fX}L_{\fX}^* = \nabla_A^*\nabla_A -(\ad(\Phi))^2
\end{equation}
which follows from the Bogomolny equations and the fact that $\RR^3$
is (Ricci) flat.  If $L_{\fX}^*u=0$, $u=O(\rho^2)$, we have
\begin{equation}
(u,L_{\fX}L_{\fX}^* u) = (u,\nabla_A^*\nabla_A u) +
\|\ad(\Phi)u\|^2
\end{equation}
as $u$ and its derivatives are all in $L^2$.  The decay at $\pa X$ is
sufficient to integrate by parts with no boundary term, we conclude
\begin{equation}
0=\|L_{\fX}^* u\|^2 = \|\nabla_A u\|^2 + \|\ad(\Phi)u\|^2
\end{equation}
from which $u$ is covariant constant hence zero because vanishing at
the boundary.
\end{proof}

It follows from the previous result that there exists a bounded right inverse
to $L_{\fX}$ on the Sobolev spaces in consideration:
\[
\begin{gathered}
	G : \rho^{\alpha+1/2,\beta} \bhscH^{m,0}(\fX; \adP) 
	\to \rho^{\alpha-1/2,\beta} \bhscH^{m,1}(\fX; \adP),
	\\ L_{\fX} G = I,
	\quad G L_{\fX} = I - \Pi_N
\end{gathered}
\]
where $\Pi_N$ is projection on the null space $N$ of $L_{\fX}$.
From the identities
\[
\begin{gathered}
	Q^L = Q^L L_{\fX} G = G - E^L G,
	\\ G - GE^R = G L_{\fX} Q^R = Q^R - \Pi_N Q^R,
\end{gathered}
\]
it follows that $G$ satisfies
\[
	G 
	= Q^R + Q^L E^R - \Pi_NQ^R + E^L GE^R.
\]
Because of the bi-ideal properties of $\rho^\infty \Psi^{-\infty} =
\scPsi^{-\infty,\infty}$ and left (resp.\ right) ideal properties of
$\bPsi^{-\infty,(\infty,\ast,\infty)}$ (resp.\ $\bPsi^{-\infty,(\ast,\infty,\infty)}$),
the last term has the form
\[
	E^L G E^R \in \begin{pmatrix} \rho \bPsi^{-\infty,(\wh 1, \wh 1, \infty}) \rho^{-2}  & \rho^\infty \Psi^{-\infty}
	\\\rho^\infty \Psi^{-\infty} & \scPsi^{-\infty,\infty}\end{pmatrix}
\]
near $(\pa \fX)^2$. Computing the compositions of the other terms leads to a
proof of Theorem~\ref{T:LX_package}.\eqref{I:LX_package_G}. Note that
the precise index sets are not of critical importance; for the present purpose
it suffices to only keep track of the leading orders.

As a consequence, $G$ maps polyhomogeneous sections to polyhomogeneous
sections:
\begin{equation}
	G : \cAphg^\ast \cap \rho^{\alpha+1/2,\beta} \bhscH^{k,0}(\fX; \adP) \to \cAphg^\ast(\fX; \adP).
	\label{E:LX_G_to_phg}
\end{equation}
The precise behavior of the index sets, as in
Theorem~\ref{T:LX_package}.\eqref{I:LX_package_solv}, is then determined {\em a
posteriori}, rather than from the estimates on the index sets for $G$.

The following is a general result for b differential operators:
\begin{prop}[\cite{RBMgreenbook}, Prop.~5.61, p.\ 205]
Let $P \in \bDiff^k(\fX; V)$, and suppose $u \in \cAphg^G(\fX; V) \cap
\rho^\alpha\bH^k(\fX; V)$ satisfies $P u = f \in \cAphg^F(\fX; V) \cap x^\alpha
L^2(\fX; V),$ where $\alpha \notin \bspec(P)$. Then in fact $u \in \cAphg^{E}(\fX;
V)$, where 
\begin{equation}
	E 
	= 
	\wh E^+(\alpha) \ol \cup F,
	\label{E:b_soln_index}
\end{equation}
where $E^+(\alpha) = \set{(z,k) \in \bSpec(P) : \re(z) > \alpha}$.
\label{P:a_posteriori_phg}
\end{prop}

\begin{proof}[Proof of Theorem~\ref{T:LX_package}]
Parts \eqref{I:LX_package_fred} and \eqref{I:LX_package_G} were proved above;
it remains to prove part \eqref{I:LX_package_solv}. Suppose $f = f_0 \oplus
f_1$ near $\pa \fX$ with $f_0 \in \cAphg^{F_0}(\fX; \adPz)$, $f_1 \in
\cAphg^{F_1}(\fX; \adPo)$ and let $u = G f$, which is polyhomogeneous by
\eqref{E:LX_G_to_phg}.  Thus $L_\fX u = f$ and  $u = u_0 \oplus u_1$ near $\pa
\fX$ with $u_i \in \cAphg^{G_i}$, for some index sets $G_0$ and
$G_1.$ We now use the fact that $L_\fX$ has the form
\[
	L_\fX = \begin{pmatrix} \rho^2 \wt D_0 \rho^{-1} & 0 \\0 & D_1 + \ad \Phi\end{pmatrix} + \cO(\rho^\infty).
\]
with $\wt D_0 \in \bDiff^1(\fX; \adPz).$ Since the off-diagonal terms coupling
between $\adPz$ and $\adPo$ are vanishing rapidly, they have no effect on
the terms which actually appear in the index sets and will be ignored. 

It suffices therefore to suppose that, near $\pa \fX$, we have
\[
\begin{gathered}
	\rho^2 \wt D_0 \rho^{-1} u_0 = f_0 \in \cAphg^{F_0}(\fX; \adPz),
	\\fD_1 u_1 + \ad \Phi u_1 = f_1 \in \cAphg^{F_1}(\fX; \adPo).
\end{gathered}
\]
Rewriting the first equation as $\wt D_0 (\rho^{-1} u_0) = \rho^{-2}f_0 \in \cAphg^{F_0 - 2}$ 
and invoking Proposition~\ref{P:a_posteriori_phg} with $E^+(\alpha) = 1$, we obtain
\[
\begin{gathered}
	u_0 \in \cAphg^{E_0}(\fX; \adPz),
	\\ E_0 = \wh 1 \ol \cup (F_0 - 2) + 1
	= \wh 2 \ol \cup (F_0 - 1)
\end{gathered}
\]
as claimed.

In the second equation, $\ad \Phi$ is the dominant term in the operator as far
is polyhomogeneity is concerned; thus we may write
\[
	\ad \Phi u_1 = f_1 - \rho \wt D_1 u_1
	\quad \wt D_1 \in \bDiff^1(\fX; \adPo).
\]
Proceeding inductively over the leading terms in the index set $F_1$, we
conclude that $u_1 \in \cAphg^{F_1}(\fX; \adPo).$
\end{proof}

\subsection{Linear analysis of $L_\fD$} \label{S:linear_D}

Next we consider the linear operator 
\[
	L_\fD = \begin{bmatrix} \star d  & - d \\-d^\ast & 0\end{bmatrix}
\]
on $\Lambda = \Lambda^1\oplus \Lambda^0$ the
face $\fD$.

\begin{thm}
\mbox{}
\begin{enumerate}
[{\normalfont (a)}]
\item  \label{I:LD_package_inv}
For $\alpha,\beta \in (-1,1)$ and any $k \geq 0$, 
\begin{equation}
	L_\fD : \rho_\fB^{\alpha-1/2}\rho_\fX^{\beta + 1/2} \bH^{k}(\fD; \Lambda)
	 \to \rho_\fB^{\alpha+1/2}\rho_\fX^{\beta - 1/2} \bH^{k-1}(\fD; \Lambda)
	\label{E:LD_op}
\end{equation}
is invertible. Here the Sobolev spaces are based on $L^2(\fD)$ with respect to
the rescaled metric $\ve^{-2} g \rst_\fD$ of conic/scattering type and $\rho_\fX
= \prod_i \rho_{\fX_i}.$
\item \label{I:LD_package_solv}
Given 
\[
\begin{gathered}
	f \in \cAphg^{\cF}(\fD; \Lambda) 
	  \cap \rho_\fB^{\alpha + 1/2}\rho_\fX^{\beta - 1/2}\bH^{k-1}(\fD; \Lambda),
	\\ \cF = (F_\fB, F_\fX),
\end{gathered}
\]
there exists a unique $u \in \cAphg^{\cE}(\fD; \Lambda)$
such that $L_\fD u = f$, where
\[
\begin{gathered}
	\cE  = (E_\fB,E_\fX),
	\\ E_\fB = 
	 \wh 2 \ol\cup (F_\fB - 1),
	\\ E_\fX = 
	  \wh 0 \ol\cup (F_\fX + 1).
\end{gathered}
\]
\item \label{I:LD_package_G}
The inverse $L_\fD^{-1}$ to \eqref{E:LD_op} has the form
\begin{equation}
\begin{gathered}
	L_\fD^{-1} = \rho_\fB \rho_\fX^{-1}\wt G\rho_\fX^2 \rho_\fB^{-2},
	\quad \wt G \in \bPsi^{-1,\cF}(\fD; \Lambda),
	\\ \cF = (F_L,F_R,F_F), 
	\quad F_L, F_R \geq 1, \quad F_F \geq 0.
\end{gathered}
	\label{E:linear_bogo_D_inverse}
\end{equation}
\end{enumerate}
\label{T:LD_package}
\end{thm}

The proof is similar to the proof of Theorem~\ref{T:LX_package}, so we shall be
brief.

Recall that, as a compact manifold
with boundary, $\fD = [\oX;\set{\zeta_j}]$ is the radial compactification of
euclidean $\bbR^3$, blown up at a finite number of points. 

Thus, as $L_\fD$ is a homogeneous operator of order 1, of conic
type near $\fD \cap \fX_i$, and scattering type near $\fD \cap \fB,$ the first step
in the proof of Theorem~\ref{T:LD_package} is to define the related
b operator
\begin{equation}
	\wt L = \rho_\fB^{-2}\rho_\fX^2 L_\fD \rho_\fX^{-1} \rho_B \in \bDiff^1(\fD; \Lambda).
	\label{E:LD_b_op_defn}
\end{equation}
Observe that the mapping properties (boundedness, invertibility, Fredholmness,
self-adjointness, etc.) of
\begin{equation}
	\wt L: \rho_\fB^\alpha \rho_\fX^\beta \bH^k(\fD; \Lambda; \bg)
	\to \rho_\fB^\alpha \rho_\fX^\beta \bH^{k-1}(\fD; \Lambda; \bg)
	\label{E:LD_b_op}
\end{equation}
are the same as the mapping properties of \eqref{E:LD_op}, where
$\bg = \rho_\fB^2 \rho_\fX^{-2}\ve^{-2} g$ is the conformally related b-metric.

\begin{prop}
For $\alpha, \beta \in (-1,1)$, the extension \eqref{E:LD_op} is invertible,
and 
$\Null(L_\fD) \in \rho_\fB^2C^\infty(\fD; \Lambda)$
\label{P:LD_fred}
\end{prop}
\begin{proof}
Proceeding as in Proposition~\ref{P:LX_ind_roots}, we see that the
indicial roots at all boundary faces of $\fD$ are given by
\[
	\bspec(\wt L) = \bbZ \setminus \set{0}.
\]
As the extension \eqref{E:LD_op} of $L_\fD$ and hence the extension
\eqref{E:LD_b_op} $\wt L$ are self-adjoint when $\alpha = \beta = 0$, it
follows from standard results for b differential operators that \eqref{E:LD_op}
is Fredholm with index 0 for $\alpha,\beta \in (-1,1)$,
and it follows from the fact that $\wt L$ agrees identically with its indicial
operator near $\fX_i$ and $\fB$ that $\Null(L_\fD) \in \rho_\fB^2C^\infty(\fD;
\Lambda)$.

To see that \eqref{E:LD_op} is invertible, we use the Bochner formula
\[
	L_\fD^\ast L_\fD = \Delta = \nabla^\ast\nabla + \Ric
	= \nabla^\ast\nabla,
\]
where the adjoints are computed with respect to $L^2(\fD; \Lambda; \ve^{-2}
g).$ This formula follows as in the proof of Proposition~\ref{P:LX_fred_null}.
Thus for $u \in \Null(L_\fD)$ with respect to any
extension \eqref{E:LD_op} with $\alpha,\beta \in (-1,1),$
\[
	\norm{\nabla u}^2_{L^2} = \pair{\nabla^\ast \nabla u, u} = \pair{L_\fD^\ast L_\fD u, u} = 0
\]
and since $u \rst_\fB = 0$, $u$ must vanish identically. The integration by
parts is justified by comparing the decay rate $\cO(\rho_\fB^2\rho_\fX^0)$
of the nullspace with the volume element near $\fB$ and $\fX$.
It
follows that \eqref{E:LD_op} is injective, and therefore surjective since it
has index 0.
\end{proof}

\begin{proof}[Proof of Theorem~\ref{T:LD_package}]
Part \eqref{I:LD_package_inv} has been shown, and part
\eqref{I:LD_package_solv} then follows immediately from
Proposition~\ref{P:a_posteriori_phg} and \eqref{E:LD_b_op_defn}. Finally, part
\eqref{I:LD_package_G} follows by expressing the inverse, $\wt G = \wt L^{-1}$, of
\eqref{E:LD_b_op} as
\[
	\wt G = \wt Q_R + \wt Q_L \wt E_R + \wt E_L \wt G \wt E_R \in \bPsi^{-1,\cF}(\fD; \Lambda),
\]
where $\wt Q_R \in \bPsi^{-1,(\wh 1 \ol \cup 2, \wh 1, \wh 2 \ol \cup 1\cup
0)}(\fD; \Lambda)$ is a right parametrix for $\wt L$ with $\wt E_R = I - \wt
Q_R \wt P_0 \in \bPsi^{-\infty,(\wh 1,\infty,\infty)}(\fD; \Lambda)$, and
likewise $\wt Q_L \in \bPsi^{-1,(\wh 1,\wh 1 \ol \cup 2, \wh 2 \ol \cup 1\cup
0)}(\fD; \Lambda)$ is a left parametrix with $\wt E_L = I - \wt P_0 \wt Q_L \in
\bPsi^{-\infty,(\infty,\wh 1,\infty)}(\fD; \Lambda)$. These may be constructed
as in the previous section.
\end{proof}

\subsection{Linear analysis for Coulomb gauge} \label{S:linear_coulomb}
Here we analyze the linearized operator 
\[
	F = \Delta_A - \ad \Phi^2 = d_A^\ast d_A - \ad \Phi^2
\]
from the Coulomb gauge fixing problem. Here $(A,\Phi)$ are assumed to be smooth
and diagonal to infinite order near $\D \cap \B$, and we restrict our analysis
to boundary faces $\fX$ and $\fD$ of a fiber $\fZ$ of $\Z$ over $\idmon.$

Over $\fX$, we may assume $A$ is in radial gauge with respect to the boundary
defining function $\rho$ for $\fX \cap \fD$. It follows that
\[
	F_\fX := \noX(F) \rst_{\fX} = \begin{pmatrix} \Delta_A & F_{01}\\F_{10} & \Delta_A - \ad \Phi^2 \end{pmatrix}
	\quad F_{ij} \in \rho^\infty\Psi^{1}(\fX; \adP)
\]
On $\adPo$, the term $- \ad \Phi^2 = (\ad \Phi)^\ast (\ad \Phi)$ is positive
and nondegenerate so that $\Delta_A - \ad \Phi^2 \in \scDiff^1(\fX; \adPo)$
is fully elliptic as a scattering operator \cite{RBM_scat}. On $\adPz$ we write
\[
\begin{gathered}
	\Delta_A = \rho^{1 + 3/2} \wt \Delta_A \rho^{1 - 3/2}
	= \rho^{5/2}\wt \Delta_A \rho^{-1/2}
	\\ \wt \Delta_A = -(\rho\pa_\rho)^2 + \tfrac 1 4 + \Delta_{A,\pa \fX} \in \bDiff^1(\fX; \adPz).
\end{gathered}
\]
Here $\Delta_{A,\pa \fX}$ denotes the scalar Laplacian on $\pa \fX = \bbS^2$
determined by the restriction of $A$ (which is well-defined as $A$ is a true
connection) to $\pa \fX$ and the rescaled metric $\rho^2 g \rst_{\pa \fX} = \hS$.
This explicit form of $\wt \Delta_A$ follows by a direct computation from the
assumption that $A$ is in radial gauge. As a consequence,
\[
	\bspec(\wt \Delta_A) = \set{\pm \sqrt{\tfrac 1 4 + \nu} : \nu \in \spec(\Delta_{A,\pa \fX})}.
\]
In particular the interval $(-\tfrac 1 2, \tfrac 1 2)$ is disjoint from
$\bspec(\wt \Delta_A)$.

\begin{thm}
For $\alpha \in (-\tfrac 1 2, \tfrac 1 2)$, and any $\beta \in \bbR$, $k \in \bbN$, 
the bounded operator
\begin{equation}
	F_{\fX} : \rho^{\alpha-1,\beta}\bhscH^{k,2}(\fX; \adP)
	\to \rho^{\alpha+1,\beta}\bhscH^{k,0}(\fX; \adP)
	\label{E:linear_coulomb_X_extn}
\end{equation}
is invertible, with inverse $G$ represented as a conormal distribution on
$\fX^2$, decomposing with respect to $\Hom(\pi_R^\ast(\adPz\oplus
\adPo),\pi_L^\ast(\adPz\oplus \adPo))$ as
\begin{equation}
\begin{gathered}
	G = \begin{pmatrix} \rho^{1/2}\wt G_0 \rho^{-5/2} & G_{01}\\G_{10} & G_1\end{pmatrix}
	\\ \wt G_0 \in \bPsi^{-2,\cF}(\fX; \adPz)
	\quad \cF = (F_L,F_R, F_F), \quad F_L,F_R \geq \tfrac 1 2, \quad F_F \geq 0.
	\\ G_1 \in \scPsi^{-2,0}(\fX; \adPo),
	\quad G_{ij} \in \rho^\infty \Psi^{-2}(\fX; \adP_j, \adP_i)
	\label{E:linear_coulomb_X_inverse}
\end{gathered}
\end{equation}
\label{T:linear_coulomb_X}
\end{thm}
\begin{proof}
The construction is similar to the one in Theorem~\ref{T:LX_package}, so we shall be brief.
Fixing the parameters $\alpha \in (-\tfrac 1 2,\tfrac 1 2)$, $\beta \in \bbR$, and $k \in \bbN$,
we proceed as in \S\ref{S:linear_X}, beginning with an initial right parametrix
$Q$ of the form
\[
\begin{gathered}
	Q = \begin{pmatrix} \rho^{1/2}\wt Q_0 \rho^{-5/2} & 0\\0 & Q_1\end{pmatrix}
	\\ \wt Q_0 \in \bPsi^{-2,(\wh S,S,0)}(\fX; \adPz),
	\quad Q_1 \in \scPsi^{-2,0}(\fX; \adPo),
\end{gathered}
\]
where $S = \set{(s,k) \in \bSpec(\wt \Delta_A) : s \geq 1/2}$ and $\wh
S = \ol \bigcup_{n \in \bbN}(S + n)$, for which the error $E = I -
F_\fX Q$ has the form
\[
\begin{gathered}
	E = \begin{pmatrix} \rho^{5/2}\wt E_0 \rho^{-5/2} & E_{01}\\E_{10} & E_1\end{pmatrix}
	\\ \wt E_0 \in \bPsi^{-1,(\infty,S,1)}(\fX; \adPz),
	\quad E_1 \in \scPsi^{-1,1}(\fX; \adPo),
	\\ E_{ij} \in \rho^\infty\Psi^{-1}
\end{gathered}
\]
Summing the Neumann series for $E$ leads to the improved right parametrix $Q^R
= Q(\sum_{N = 0}^\infty E^N)$ with $E^R = I - F_\fX Q^R$ of the form
\[
\begin{gathered}
	E^R = \begin{pmatrix} \rho^{5/2}\wt E^R_0 \rho^{-5/2} & E^R_{01}\\E^R_{10} & E^R_1\end{pmatrix}
	\\ \wt E^R_0 \in \bPsi^{-\infty,(\infty,\wh S,\infty)}(\fX; \adPz),
	\quad E^R_1 \in \scPsi^{-\infty,\infty}(\fX; \adPo),
	\\ E^R_{ij} \in \rho^\infty\Psi^{-\infty}
\end{gathered}
\]
and $Q^R$ having a similar decomposition to $Q$, but with off-diagonal terms in
$\rho^{\infty}\Psi^{-2}$ and $\wt Q^R_0 \in \bPsi^{-2,(\wh S, S \ol
\cup (\wh S + 1), 1 \ol \cup (S+ \wh S)\cup 0)}(\fX; \adPz)$.

A similar procedure gives a left parametrix $Q^L$ with $E^L = I - Q^L F_\fX$ of the
form
\[
\begin{gathered}
	E^L = \begin{pmatrix} \rho^{1/2}\wt E^L_0 \rho^{-1/2} & E^L_{01}\\E^L_{10} & E^L_1\end{pmatrix}
	\\ \wt E^L_0 \in \bPsi^{-\infty,(\wh S,\infty,\infty)}(\fX; \adPz),
	\quad E^L_1 \in \scPsi^{-\infty,\infty}(\fX; \adPo),
	\\ E^L_{ij} \in \rho^\infty\Psi^{-\infty}
\end{gathered}
\]

As operators on $\rho^{\alpha+1,\beta}\bhscH^{k,0}(\fX; \adP)$ and
$\rho^{\alpha-1,\beta}\bhscH^{k,2}(\fX; \adP)$, $E^R$ and $E^L$ are compact,
so the extension \eqref{E:linear_coulomb_X_extn} is Fredholm. From the fact
that $E^L$ maps $\rho^{\alpha-1,\beta}\bhscH^{k,2}$ into
$\rho^{0,\infty}\cAphg^{\wh S + 1/2}$ and $u \in \Null(F_\fX) \iff u = E^L u$, it
follows that 
\[
	\Null(F_\fX ) \subset \rho^{0,\infty}\cAphg^{\wh S+1/2}(\fX; \adP).
\]
In particular $u \in \Null(F_\fX )$ has leading order $\cO(\rho^{1+\epsilon})$ since $S \geq \tfrac 1 2.$
This is enough decay to justify the integration by parts in
the identity $\pair{\Delta_Au,u}_{L^2} = \norm{d_A u}^2_{L^2}$, from which it
follows that
\[
	u \in \Null(F_\fX ) \implies \norm{d_A u}_{L^2} = 0,
\]
so that $u$ is covariant constant, and therefore vanishing since $u \rst_{\pa
\fX} = 0.$ Thus \eqref{E:linear_coulomb_X_extn} is injective, and applying 
the same reasoning to the $L^2$ adjoint 
\[
	\Delta_A - \ad \Phi^2 : \rho^{-\alpha-1,-\beta}\bhscH^{k,2}(\fX; \adP) 
	\to \rho^{-\alpha+1,-\beta}\bhscH^{k,0}(\fX; \adP)
\]
proves that \eqref{E:linear_coulomb_X_extn} is surjective as well.  Letting $G$
denote the inverse of \eqref{E:linear_coulomb_X_extn} and writing 
\[
	GF_\fX Q^R = G - GE^R = Q^R,
	\quad Q^L F_{\fX} G = G - E^L G = Q^L,
\]
it follows that $G$ satisfies the identity
\[
	G = Q^R + Q^L E^R + E^L G E^R.
\]
\eqref{E:linear_coulomb_X_inverse} is then a consequence of the bi-ideal
properties of $\rho^\infty \Psi^{-\infty}$ and the left (resp.\ right) ideal
properties of $\bPsi^{-\infty,(\infty,\ast,\infty)}$ (resp.\
$\bPsi^{-\infty,(\ast,\infty,\infty)}$).
\end{proof}

\bigskip

Next we consider the normal operators of $F$ at $\D$.
Again decomposing with repsect to $\adPz \oplus \adPo$, and writing
\[
	F = \begin{pmatrix} F_0 & F_{01}\\F_{10} & F_1\end{pmatrix},
\]
with $F_{ij} = \cO(\rho^\infty)$, we have
\[
	F_\fD := \noD(F_0)\rst_{\fD} = (\lim_{\ve \smallto 0} \ve^2 F_0) \rst_{\fD}
	= \Delta_{A\rst \fD}
\]
where $A\rst \fD$ is the restriction of $A$ as a smooth connection to $\fD$.

\begin{thm}
For $\alpha, \beta \in (-\tfrac 1 2, \tfrac 1 2)$, and $k \in \bbN$,
the extension
\begin{equation}
	\Delta_{A\rst \fD} : \rho_\fB^{\alpha - 1}\rho_\fB^{\beta + 1} \bH^{k+2}(\fD; \adPz)
	\to \rho_\fB^{\alpha+1}\rho_\fB^{\beta - 1} \bH^{k} (\fD; \adPz)
	\label{E:linear_coulomb_D_extn}
\end{equation}
is invertible, with inverse of the form
\begin{equation}
\begin{gathered}
	\Delta_{A\rst \fD}^{-1} = \rho_\fB^{1/2}\rho_\fX^{-1/2}\, \wt G\, \rho_\fX^{5/2} \rho_\fB^{-5/2},
	\quad \wt G \in \bPsi^{-2, \cF}(\fD; \adPz),
	\\ \cF = (F_L,F_R, F_F), \quad F_L, F_R \geq \tfrac 1 2, \quad F_F \geq 0.
	\label{E:linear_coulomb_D_inverse}
\end{gathered}
\end{equation}
\label{T:linear_coulomb_D_b}
\end{thm}
\begin{proof}
For the remainder of the proof, we denote $A\rst \fD$ simply by $A$.
We consider the operator
\[
	\wt \Delta_A = \rho_\fB^{-5/2}\rho_\fX^{5/2}\, \Delta_A \, \rho_\fX^{-1/2}\rho_\fB^{1/2}
	\in \bDiff^{2}(\fD; \adPz).
\]
Taking $A$ to be in radial gauge near the ends of $\fD$, it follows
that near $\fD \cap \fX$, 
\[
	\wt \Delta_A = - (r\pa_r)^2 + \tfrac 1 4 + \Delta_{A,\fD \cap \fX},
\]
with $\Delta_{A,\fD \cap \fX}$ denoting the Laplacian on $\fD \cap \fX = \bbS^2$ 
induced by $A$ and the metric $\rho_\fX^{-2}\ve^{-2}g = \hS$, 
and likewise near $\fD \cap \fB$, 
\[
	\wt \Delta_A = - (x\pa_x)^2 + \tfrac 1 4 + \Delta_{A,\fD \cap \fB}.
\]
In either case, the indicial roots near an end of $\fD$ have the form
\[
	\bspec(\wt \Delta_A) = \set{\pm \sqrt{\tfrac 1 4 + \nu} : \nu \in \spec(\Delta_{A,\pa \fD})}
\]
and in particular are always disjoint from the interval $(-\tfrac 1 2, \tfrac 1
2).$ 

Proceeding with the standard steps in the b-calculus, we may construct
right and left parametrices
\[
\begin{gathered}
	\wt Q^R \in \bPsi^{-2,(\wh S, S \ol \cup (\wh S + 1), 1 \ol \cup (S + \wh S) \cup 0)}(\fD; \adPz),
	\\ \wt Q^L \in \bPsi^{-2,(S \ol \cup (\wh S + 1), \wh S, 1 \ol \cup (S + \wh S) \cup 0)}(\fD; \adPz)
\end{gathered}
\]
for $\wt \Delta_A$, where $S = \set{(s,k) \in \bSpec(\wt \Delta_A) : s > 1/2}$ and $\wh S = \ol
\bigcup_n (S + n)$ as before. The error terms $\wt E^R = I - \wt \Delta_A \wt Q^R$ and 
$\wt E^L = I - \wt Q^L \wt \Delta_A$ have the form
\[
	\wt E^L \in \bPsi^{-\infty,(\wh S, \infty,\infty)}(\fD; \adPz),
	\quad \wt E^R \in \bPsi^{-\infty,(\infty, \wh S, \infty)}(\fD; \adPz).
\]
Then $Q^{L/R} := \rho_\fB^{1/2}\rho_\fX^{-1/2}\, \wt Q^{L/R}
\,\rho_\fX^{5/2}\rho_\fB^{-5/2}$ are left/right parametrices for $
\Delta_A$, and the error terms $E^L = I - Q^L \ol \Delta_A$ and $E^R = I - \ol
\Delta_A Q^R$ extend to compact operators on the domain and range of
\eqref{E:linear_coulomb_D_extn}, respectively, with $E^L$ mapping the domain
into $\rho_\fB^{1/2}\rho_\fX^{-1/2}\cAphg^{\wh S}(\fD; \adPz)$. 

It follows that \eqref{E:linear_coulomb_D_extn} is Fredholm, with nullspace in
$\rho_\fB^{1/2}\rho_\fX^{-1/2} \cAphg^{\wh S}$; in particular, the leading
order of $u \in \Null(\Delta_A)$ has order
$\cO(\rho_\fB^{1},\rho_\fX^{0}).$ Due to the homogeneity in the
conic/scattering volume form of $\fD$, this is enough decay to justify the
integration by parts in the identity $\pair{\Delta_A u , u}_{L^2} =
\norm{d_A u}^2_{L^2}$, from which it follows that $u \in \Null(\Delta_A)$
is covariant constant, hence vanishing since it vanishes at $\fD \cap \fB.$
Thus \eqref{E:linear_coulomb_D_extn} is injective, and by self-adjointness is
also surjective, hence invertible. \eqref{E:linear_coulomb_D_inverse} follows
by writing
\[
	\Delta_A^{-1} = Q^R + Q^L E^R + E^L \Delta_A^{-1} E^R.
\qedhere
\]
\end{proof}

As for $F_1$, we recall the normal operator homomorphism of
Proposition~\ref{P:g_normal_seq_D}. Regarding $\sigma_\D(F_1)$ as a smooth
fiberwise polynomial on $\gT \D \to \D$ with values in $\End(\adPo)$, the
observation that $\Delta_A$ is a laplacian on $\adPo$ lead to the following
Theorem.

\begin{thm}
The symbol $\sigma_\D(F_1) \in S^2(\gT \D; \adPo)$ is given at $(x,\xi) \in \gT \D$ by
\[
	\sigma_\D(F_1)(x,\xi) = \abs{\xi}^2 - \ad \Phi_x^2,
\]
and is invertible for all $(x,\xi)$.
\label{T:linear_coulomb_D_sc}
\end{thm}

%% file: double.tex
Here we construct the calculi of pseudodifferential operators which `microlocalize'
the algebras $\pDiff^\ast(\Z)$ and $\gDiff^\ast(\Z)$ of differential operators
associated to the vector fields $\pV(\Z)$ and $\gV(\Z)$, respectively. To keep the notational
complexity at a minimum we consider operators on scalar functions only; the extension to operators
acting between sections of vector bundles over $\Z$ is a straightforward matter. 

The kernels of these pseudodifferential operators are defined as distributions
on appropriate geometric resolutions of the fiber product $\Z\times_\fb \Z$;
these ``double spaces'' are discussed in \S\ref{S:double_double}. The composition of
two such operators is defined via the associated ``triple
spaces''---resolutions of the triple fiber product $\Z\times_\fb \Z \times_\fb
\Z$---which are discussed next in \S\ref{S:triple}. The $\fb$ and $\gl$
pseudodifferential operators are defined and some of their essential properties
are investigated in \S\ref{S:pseudo} and their mapping properties with respect
to Sobolev spaces are proved in \S\ref{S:pseudo_sobolev}. 

Note that, in our constructions involving monopoles, we only compose
pseudodifferential operators with differential ones.  Such compositions can be
defined directly on the double spaces, avoiding the technical complexity of the
triple spaces. However, the composition of pseudodifferential operators is used
to establish their mapping properties with respect to $L^2$ via a standard
argument due to H\"ormander (see the proof of Lemma~\ref{L:dpM_mapping}); for
this reason we have developed the general composition results in
Theorems~\ref{T:dpM_package} and \ref{T:dgM_package} below.
 
\subsection{The double spaces} \label{S:double_double}

In the first place, $\dpM$ is meant to be a resolution of the fiber product $\Z^{[2]} := \M
\times_\fb \Z$ of the single space with itself with respect to the b-fibration
$\fb : \Z \to \base$.  
There are two ways to achieve this. The most direct is to use the theory
developed in \cite{kottke2015generalized} and \cite{kottke2015blow} regarding resolutions of fiber products. Alternatively,
since many readers will not be familiar with this theory, $\dpM$ may be
constructed directly via a sequence of blow-ups from the space $(\oX)^2 \times
\base$. We shall describe both approaches.

Let $Y$ be a manifold with corners and denote by $\cF(Y) = \bigsqcup_d\cF_d(Y)$ the set of boundary faces
of $Y$, where $d$ is codimension.
Recall from \cite{kottke2015generalized,kottke2015blow} that $Y$ has an associated {\em
monoidal complex} $\cP_Y$, which is a collection of monoids $\sigma_G \cong
\bbN^d$ for each boundary face $G \in \cF_d(Y)$, with canonical injective maps
$i_{FG} : \sigma_G \hookrightarrow \sigma_F$ whenever $F \subset G$ identifying
$\sigma_G$ with a face of $\sigma_F.$ These monoids $\sigma_G$ are all {\em
smooth}, meaning that they are freely generated by independent elements 
in the associated vector space $\sigma_G \otimes_\bbN \bbR \cong \bbR^d$; indeed
the generators for a given $\sigma_G$ may be identified with the faces $i_{GH}
(\sigma_H) \cong \bbN$ for the hypersurfaces $H$ meeting $G$. For any b-map $f
: Y \to Z$, there is an associated morphism $f_\natural : \cP_Y \to \cP_Z$,
mapping each $\sigma_G$ into $\sigma_{f_\#(G)}$ where $f_\#(G)$ is the boundary
face of $Z$ of maximal codimension into which $G$ maps. Expressed in a basis
with respect to the generators of $\sigma_G$ and $\sigma_{f_\#(G)},$ these maps
of monoids are represented by matrices with nonnegative integer entries given by the exponents
$e(H,H') \in \bbN$, where
\[
	f^\ast(\rho_H) = a\prod_{H \in \cF_1(Y)} \rho_H^{e(H,H')}, 
	\quad H \in \cF_1(Z), \quad 0 < a \in C^\infty.
\]
In fact, $f_\natural : \sigma_G \to \sigma_{f_\#(G)}$ may be identified with the
b-differential
\[
	\bd f_\ast : \bN G \to \bN f_\#(G).
\]
In particular, $f$ is b-normal if and only if no generator of any $\sigma_G$ is
mapped into the interior of any monoid $\sigma_F$ of $\cP_Z.$

If $f_1 : Y_1 \to Z$ and $f_2 : Y_2 \to Z$ are b-maps which are {\em
b-transversal}, a condition which is automatically satisfied if at least one is
a b-fibration, then the fiber product
\begin{equation}
	Y_1 \times_Z Y_2 \subset Y_1 \times Y_2
	\label{E:fib_product}
\end{equation}
is what is called an {\em interior binomial variety} of the manifold $Y_1
\times Y_2$.  Such a space, though smooth in its interior, is generally not a
manifold with corners; it belongs to a category of ``manifolds with generalized
corners,'' which has been developed by Joyce \cite{joyce2015generalization}, though subspaces such as
\eqref{E:fib_product} and their resolution theory is described in
\cite{kottke2015generalized} and \cite{kottke2015blow}. Indeed, a complex $\cP_{Y_1\times_Z Y_2}$ may also be associated to
the fiber product, with each monoid of dimension $l$ corresponding to a
boundary face (which is again an interior binomial variety) of codimension $l$,
and the face relations of the complex corresponding to the meeting of boundary
faces. The failure of the fiber product to be smooth is measured by the failure
of the monoids of $\cP_{Y_1\times_Z Y_2}$, which have the form $\sigma_{G_1}
\times_{\sigma_F} \sigma_{G_2}$, $F = (f_1)_\#(G_1) \cap (f_2)_\#(G_2)$, to be
smooth (meaning freely generated). To any {\em resolution} of $\cP_{Y_1\times_Z
Y_2}$, meaning a consistent way of subdividing the non-smooth
monoids into smooth ones, there corresponds an {\em generalized blow-up} of
$Y_1\times_Z Y_2$ resolving it to a smooth manifold with corners, the
combinatorial structure of whose boundary faces (i.e., codimension and meeting
of boundary faces) is again encoded by the resolving monoidal complex. Below we
will also be concerned with the problem of determining the induced resolution
of a boundary hypersurface $H \subset Y_1 \times_Z Y_2$ from a resolution of
$\cP_{Y_1 \times_Z Y_2}$. The monoidal complex of $H$ itself
is obtained from $\cP_{Y_1\times_Z Y_2}$ as the complex of quotient
monoids $\set{\sigma/\sigma_H : \sigma_H \subset \sigma \in \cP_{Y_1\times_Z
Y_2}}$, which we may denote by $\cP_{Y_1 \times_Z Y_2} / \sigma_H$. A smooth
resolution of $\cP_{Y_1\times_Z Y_2}$ induces a smooth resolution of the
quotient since the quotient of a freely generated monoid by a generator is
again freely generated.

In the present case of the b-fibration $\fb : \Z \to \base$, the space $\Z$
has boundary faces of codimension at most $2$, so $\cP_\Z$ consists of the
1-dimensional monoids $\sigma_{\X_i}, \sigma_\D, \sigma_\B \cong \bbN$ and the
2-dimensional monoids $\sigma_{\D\cap \B}$, $\sigma_{\D\cap \X_i}$. The
monoidal complex of the base $\base$ consists of the single nontrivial
monoid $\tau \cong \bbN$. The associated morphism $\fb_\natural$ sends
$\sigma_\B$ to $\set{0}$ while $\sigma_\D$ and $\sigma_{\X_i}$ are mapped
isomorphically onto $\tau.$ On the monoids of dimension 2 we have
\[
\begin{aligned}
	\fb_\natural &: \sigma_{\D \cap \X_i} \cong \bbN^2 \to \tau \cong \bbN,
	& &(m,n) \mapsto m + n,
	\\ \fb_\natural  &: \sigma_{\D \cap \B} \cong \bbN^2 \to \tau \cong \bbN,
	& &(m,n) \mapsto m.
\end{aligned}
\]
It follows that the only singular monoids in $\cP_{\Z\times_\base \Z}$ are 
\[
	\sigma_{\D \cap \X_i} \times_{\tau} \sigma_{\D\cap \X_j}, 
	\quad 0 \leq i,j \leq N.
\]
Each of these is a 3-dimensional monoid of the form 
\begin{equation}
	\nu := \set{(m_1,n_1,m_2,n_2) \in \bbN^4 : m_1 + n_1 = m_2 + n_2},
	\label{E:square_monoid}
\end{equation}
with dependent generators $(1,0,1,0)$, $(0,1,0,1)$, $(1,0,0,1)$ and $(0,1,1,0)$.
These generators may be identified respectively with the four hypersurfaces
\begin{equation}
	\D\times_\idmon\D,\ \X_i\times_\idmon \X_j,\ \D\times_\idmon \X_j,\  \X_i \times_\idmon \D
	\in \cF_1(\Z^{[2]})
	\label{E:fib_prod_bd_hyp}
\end{equation}
of the fiber product, which all meet at a face of codimension 3 (hence the 
singularity). Indeed, since $\D$ and $\X$ are both mapped to $\U
\times \set{0} \subset \base$, it follows that $\D \times_\base \D \equiv
\D\times_\idmon\D$, which is a fibration over $\idmon$ with fibers $\fD^2$. Likewise
$\D \times_\base \X_i \equiv \D \times_\idmon \X_i$ (with fibers $\fD\times \fX_i$ over
$\idmon$) and so on, and these form boundary hypersurfaces of $\Z^{[2]}$
with associated monoids $\sigma_{\D} \times_{\tau} \sigma_\D \cong
\bbN$, $\sigma_{\D} \times_\tau \sigma_{\X_i} \cong \bbN$, etc. The latter are
identified with the generating 1-dimensional submonoids of $\sigma_{\D \cap \X_i} \times_{\tau}
\sigma_{\D \cap \X_j}$ in the complex. For brevity of notation, 
we will denote the boundary hypersurfaces \eqref{E:fib_prod_bd_hyp}
as $\set{\DD,\XiXj,\DX_j,\XiD}$.

The task is to resolve the singular monoids above by smooth ones in a way which
preserves the b-fibrations to the single space $\Z$; these b-fibrations are
represented by the monoid homomorphisms
\begin{equation}
	\nu \to \bbN^2,
	\quad (n_1,m_1,n_2,m_2) \mapsto (n_i,m_i), \quad i = 1,2.
	\label{E:monoid_b_fibn}
\end{equation}
There are two inequivalent minimal resolutions, though in this case there is a
canonical choice which to which the fiber diagonal lifts to be transversal to
the boundary.  This is to replace each monoid $\sigma_{\D\cap \X_i}
\times_{\tau} \sigma_{\D\cap \X_j}$ by the two smooth monoids 
\begin{equation}
	\bbN\pair{\DD, \DX_j, \XiXj},
	\quad \bbN\pair{{\DD}, {\XiXj}, {\XiD}},
	\label{E:double_mon_resolve}
\end{equation}
which is to say
\[
	\bbN\pair{(1,0,1,0),(1,0,0,1),(0,1,0,1)},
	\quad \bbN\pair{(1,0,1,0),(0,1,0,1),(0,1,1,0)}),
\]
respectively. In doing so we
replace each singular codimension 3 face by two smooth codimension 3 faces
which are joined by a new codimension two face (corresponding to the monoid
generated by $\set{\sigma_{\DD}, \sigma_{\XiXj}}$). (See Figure~\ref{F:refinement}.) The
result is a smooth manifold with corners we shall provisionally call
$\Z^2_\mathrm{b}$ with two b-fibrations to $\Z$ lifting the right and left
projection maps from the fiber product, and a b-fibration to the parameter
space $\base$.  This is not yet the space we want; the final step is to blow-up
the codimension 2 face represented by $\B \times_\base \B$:
\begin{equation}
	\dpM := [\Z^2_\mathrm{b}; \B\times_\base \B] \to \Z^2_\mathrm{b}.
	\label{E:dpM_final_blow}
\end{equation}
At the level of monoids, this corresponds to subdividing $\sigma_{\B^2} \cong
\bbN^2$ into the two submonoids $\bbN\pair{(1,0),(1,1)}$ and
$\bbN\pair{(1,1),(0,1)}$, with the new hypersurface represented by the common
face $\bbN\pair{(1,1)}$.

\begin{figure}[htb]
\begin{center}
\begin{tikzpicture}[thick,scale=2.0]
\draw (1,0) node[right] {$\sigma_{\XD}$} -- (0,1) node[above] {$\sigma_\DD$} -- (-1,0) node[left] {$\sigma_\DX$} 
	-- (0,-1) node[below] {$\sigma_\XX$} -- cycle;
\draw[dashed] (0,1) -- (0,-1);
\end{tikzpicture}
\end{center}
\caption{Smooth refinement of the monoid $\sigma_{\DX} \times_\tau \sigma_{\DX}$}
\label{F:refinement}
\end{figure}
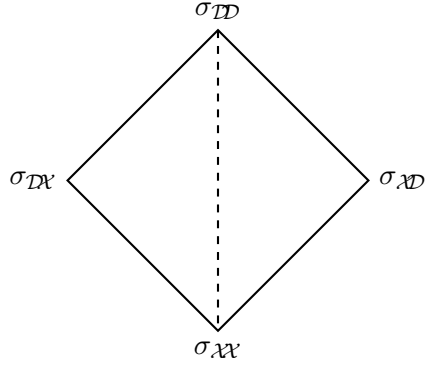

We denote the boundary hypersurfaces of $\dpM$ by their factors in the original
face of $\dpM$ lifting the product $\D\times_\base \D \equiv \D\times_\idmon \D$ in
$\Z^{[2]}$; below we show that it is diffeomorphic to the families b double
space $\dbD$.
Similarly, $\DXi$ denotes the lift of $\D\times_\base \X_i$ and so on. We
write $\BB$ to denote the front face of the blow-up \eqref{E:dpM_final_blow},
and $\BM$, etc. to denote the hypersurface lifting the original face
$\B\times_{\base} \Z$ of the fiber product. The complete list of boundary
hypersurfaces of $\dpM$ is as follows (see Figure~\ref{F:p_double}):
\[
	\DD, 
	\quad \DXi, 
	\quad \XiD,
	\quad \XiXj,
	\quad \BB,
	\quad \MB,
	\quad \BM,
\]
\begin{figure}[htb]
\begin{center}
\begin{tikzpicture}[thick,scale=1.0]
\draw (2,0) -- (1,1.73) -- (-1,1.73) -- (-2,0) -- (-1,-1.73) -- (1,-1.73) -- cycle;
\draw (-3.5,0) -- (-2,0) (2,0) -- (3.5,0);
\draw (-1,1.73) -- (-1,3.5) (1,1.73) -- (1,3.5);
\draw (-1,-1.73) -- (-2.5,-3) (1,-1.73) -- (2.5,-3);
\draw[dashed]  (0,3.5) -- (0,1.73) -- (0,-1.73) -- (0.5,-3);
\path (0,0) node[left] {$\DD$};
\path (0,2.5) node[left] {$\BB$};
\path (0,-2.5) node[fill=white] {$\XX$};
\path (2.5,-1.5) node {$\XD$};
\path (-2.5,-1.5) node {$\DX$};
\path (2.5,1.5) node {$\MB$};
\path (-2.5,1.5) node {$\BM$};
\end{tikzpicture}
\end{center}
\caption{The double space $\dpM$.}
\label{F:p_double}
\end{figure}

\begin{lem}
\begin{enumerate}
[{\normalfont (a)}]
\item \label{I:dpM_bfib}
The double space $\dpM$ has b-fibrations $\pi_R : \dpM \to \Z$ and $\pi_L :
\dpM \to \Z$ and $\fb : \dpM \to \base$ lifting the corresponding maps from the
fiber product $\Z^{[2]} = \Z\times_\fb \Z$. Composing with the projection $\base \to \idmon$
defines a fiber bundle $\fu: \dpM \to \idmon$ whose fibers, denoted generically by $\dpfM$,
are manifolds with corners.
\item \label{I:dpM_Delta}
The lifted fiber diagonal $\Delta \subset \dpM$ meets all boundary faces
transversally.
\item \label{I:dpM_doubles}
The boundary faces meeting $\Delta$ may be identified with
the family (over $\idmon$) b-double spaces $\dbD$, $\dbXi$, 
and the front face of the blow-up
\eqref{E:dpM_final_blow}, which may be identified with the product of
the b-front face of $(\ol{\bbR^3})^2_\mathrm{b}$ with $\base.$
\item \label{I:dpM_fiber}
For $\ve > 0$, the fibers $\fb^{-1}(\set p\times \set \ve) \subset \dpM$, $p \in \idmon$ are diffeomorphic to the
b double space ${\oX}^2_\mathrm{b}$.
\end{enumerate}
\label{L:dpM}
\end{lem}
\noindent 
\begin{rmk}
$\dbD$ is a fiber bundle over $\idmon$ with fibers given by the b-double space $\dbfD$, obtained by blowing
up the coimension 2 corner in the fiber product $\D \times_\idmon \D$. It supports Schwartz kernels
of families of b-pseudodifferential operators on the fiberr of $\D \to \idmon$, which we denote by
$\bPsi^\ast(\D/\idmon).$ Similar statements hold for $\dbXi$ and $\dbXve$.
\end{rmk}

\begin{proof}
In the first place, the space $\Z^2_\mathrm{b}$ has b-fibrations to $\Z$ and
$\base$, lifting the assocaited maps from the fiber product. Indeed, since only
boundary faces have been blown up in passing from $\Z^{[2]}$ to
$\Z^2_\mathrm{b}$, the only issue to check is b-normality, and this can be
verified at the level of the monoids. The homomorphisms \eqref{E:monoid_b_fibn}
from singular monoids $\sigma_{\D \cap \X_i} \times_{\tau} \sigma_{\D \cap
\X_j}$ have the b-normality property, which is that no 1-dimensional face of a
monoid is mapped into the interior of a monoid in the range complex $\cP_\Z$,
and the resolution described above retains this property. The subsequent
blow-up of $\B\times_\base \B$ to define $\dpM$ introduces a new boundary
hypersurface, but this is mapped via the right and left projections to the
hypersurface $\B \subset \Z$, so the lifted maps from $\dpM$ are b-fibrations
as well. Since b-fibrations are closed under composition and $\base \to \idmon$ is
a b-fibration to a manifold with no boundary, \eqref{I:dpM_bfib} follows.

For \eqref{I:dpM_Delta} and \eqref{I:dpM_doubles}
consider the boundary faces of the singular space $\Z^{[2]}$
which meet $\Delta$. These are evidently the faces $\D^{[2]} \equiv
\D\times_\idmon \D$, 
$\X_i^{[2]} \equiv \X_i\times_\idmon \X_i$ (which are hypersurfaces as
remarked above), $\B^{[2]}$ (which has codimension 2), the singular
corners $(\D\cap \X_i)^{[2]}$ of codimension 3, and the
smooth codimension 3 corner $(\D \cap \B)^{[2]}$. Upon
resolution, $\B^{[2]}$ is replaced by a hypersurface, with the lifted
fiber diagonal meeting it transverally, and the codimension 3 faces above are
replaced by codimension 2 corners. 

By considering the quotient monoids $\sigma_{\D \cap \X_i} \times_{\tau}
\sigma_{\D \cap \X_j} / \sigma_{\DD} \cong \bbN^2$ and $\sigma_{\D \cap \B} \times_\tau \sigma_{\D \cap \B} / \sigma_\D \cong \bbN^2$, and the corresponding quotients
of the resolution, which induce the ordinary blow-up (i.e., star subdivision) of $\bbN^2$,
it follows that the lift of $\D^{[2]}$ to $\dpM$ 
corresponds to the the blow-up $[\D^{[2]}; (\D\cap \B)^{[2]}, (\D\cap
\X_i)^{[2]}]$, which is precisely the b-double space $\dbD$. A similar argument
shows that the hypersurfaces $\X_i^{[2]} \subset \cM^{[2]}$ lift to 
the b-double spaces $\dbXi$.  The lifted fiber diagonal passes through the 
front faces of $\dbD$ and $\dbXi$, and is therefore a p-submanifold. This can be verified
directly in the monoidal theory; the monoidal subcomplex $\cP_\Delta \subset
\cP_{\Z\times_\base \Z}$ meets each singular monoid of the form $\nu$ in the
submonoid generated by $\set{(1,0,1,0),(0,1,0,1)}$ and meets
$\sigma_{\B\times_\base \B} \cong \bbN^2$ in the submonoid generated by
$(1,1)$. In the above resolution, we introduced each of these into the complex,
and by a result in \cite{kottke2015generalized}, it follows that $\Delta$ lifts to a p-submanifold
in the resolution.

Finally, to see \eqref{I:dpM_fiber}, note that the fiber of $\Z^{[2]}$
over $\ve >0$ is smooth, and coincides with the product $\fb^{-1}(\ve) \times_\U
\fb^{-1}(\ve) \cong \oX^2\times\idmon$. This is unchanged in passing to $\Z^2_\mathrm{b}$,
since all blow-ups take place at $\ve = 0$. The blow-up of $\B\times_\base \B$
then restricts to the blow up $[\oX^2\times\idmon; (\pa \oX)^2\times\idmon]$ in the fiber over $\ve$,
giving the b-double space as claimed.
\end{proof}

\bigskip Alternatively, $\dpM$ may be constructed as follows. For simplicity we
restrict consideration to a single fiber $\fZ$ of $\fb :\Z \to \base$.  First,
note that $\fZ$ itself may be obtained as a blow up of the product $\fD \times
[0,1)$, where $\fD = \ol{\bbR^3}$:
\[
	\fZ  = [\fD \times [0,1); \set{p_0,p_1,\ldots,p_N} \times \set{0}],
\]
where the collection of points $\set{p_1,\ldots,p_N}$ at $\ve = 0$ is
determined by the appropriate configuration in $\ncfgs N$, and $p_0 = 0$ is the
origin in $\fD$. That this is is equivalent to the definition of $\Z$ in
\S\ref{S:Mgl} can be checked in local coordinates. The intermediate space
$\Z^2_\mathrm{b}$ is obtained by iterated blow-up; fiberwise
\[
	\fZ^2_\mathrm{b} = [\fD^2 \times [0,1); \set{p_i\times p_j}\times \set{0}, \set{p_i\times \fD} \times \set{0}, \set{\fD\times p_j} \times \set{0}].
\]
First the pairs $p_i\times p_j$ are blown up at $\ve  = 0$ (they are separated
so the order is not important); once this is done the lifts of the subspaces
$p_i \times \fD\times \set{0}$ and $\fD \times p_j\times \set{0}$ are separated and
may be blown up in any order. The final step is the blow-up
\eqref{E:dpM_final_blow}. That the resulting space satisfies the properties in
Lemma~\ref{L:dpM} may be verified by straightforward but tedious
computations in local coordinates.

\medskip

Since the interiors of $\Z$ and $\dpM$ may be identified with the simple
products $\bbR^3 \times \intbase$ and $(\bbR^3)^2 \times \intbase$,
respectively, the lift of vector fields in $\pV(\Z)$ to $\dpM$ from the left or
right is well-defined by continuous extension from the interior, and we have
the following result:

\begin{lem}
Let $V \in \pV(\Z)$. The left and right lifts $\pi_L^\ast(V)$ and
$\pi_R^\ast(V)$ are tangent to all boundary faces of $\dpM$, and 
differentiate transversally to $\Delta$. In particular, the restriction of the
lift $\pi_R^\ast(V)$ to the boundary faces $\DD$, $\XiXi$, or to the fiber
$\X_\ve^2 := \pi_{[0,1)}^{-1}(\ve)$ for $\ve > 0$ may be respectively identified
with the following:
\begin{subequations}
\begin{align}
	\pi_R^\ast(V) \rst_{\DD} &\cong \pi_R^\ast (V \rst_{\D}), 
	& \pi_R &: \dbD \to \D \subset \Z,
	\label{E:dpM_lift_V_D}
	\\\pi_R^\ast(V) \rst_{\XiXi} &\cong \pi_R^\ast(V \rst_{\X_i}),
	&\pi_R &: \dbXi \to \X_i \subset \Z,
	\label{E:dpM_lift_V_X}
	\\\pi_R^\ast(V) \rst_{\X_\ve^2} &\cong \pi_R^\ast(V \rst_{\X_\ve}),
	&\pi_R &: \dbXve \to \X_\ve \subset \Z,
	\label{E:dpM_lift_V_fib}
\end{align}
\end{subequations}
and similarly for the restriction of $\pi_L^\ast(V).$
\label{L:dpM_lift_V}
\end{lem}
\begin{proof}
Restricted to the interior, the left and right lifts of $\pV(\Z)$ to $\dpM$ 
are readily seen to differentiate transverally to $\Delta$, so it remains to
determine their behavior at the boundary faces.

The restriction of $\pi_R : \dpM \to \Z$ to $\DD \cong \dbD$ factors through
the corresponding face of the fiber product $\Z^{[2]}$, which as noted
above is simply the product $\D\times_\idmon \D$, followwed by projection onto the
right factor, realized as the boundary face of $\Z$. It follows that 
\[
	\pi_R \rst_{\DD} : \DD \cong \dbD \to \D
\]
is identified with the conventional right projection from the b-double space.
The corresponding statements for the left projection, and the restriction of
$\pi_R$ and $\pi_L$ to the other boundary faces and the fiber $\dbXve$ are
similar. 

The identifications \eqref{E:dpM_lift_V_D}--\eqref{E:dpM_lift_V_fib} follow
immediately, and the the result that $\pi_R^\ast(V)$ and $\pi_L^\ast(V)$ are
tangent to the boundary faces away from $\Delta$ follows from factoring the
restriction of $\pi_R$ or $\pi_L$ through similar product faces of
$\Z^{[2]}$.
\end{proof}

The lift of vector fields from the left or right extends to differential
operators.  Thus for $D \in \pDiff^\ast(\Z; V)$, the pull-backs $\pi_L^\ast(D)$
and $\pi_R^\ast(D)$ are well-defined.

\bigskip
Next we consider the gluing double space $\dgM$. This is obtained from $\dpM$
by a sequence of two blow-ups:
\[
	\dgM = [\dpM; \Delta \cap \BB; \Delta \cap \DD].
\]
We denote the new boundary hypersurfaces by $\scB$ and $\scD$, respectively.
(See Figure \ref{F:g_double}.)
\begin{figure}[htb]
\begin{center}
\begin{tikzpicture}[thick,scale=1.0]
\draw (2,0) -- (1,1.73) -- (0.5,1.73) -- (0.5, 3.5);
\draw (-0.5,3.5) -- (-0.5,1.73) -- (-1,1.73) -- (-2,0) -- (-1,-1.73) -- (-0.3,-1.73) (0.3,-1.73) -- (1,-1.73) -- (2,0);
\draw (0.5,1.73) 
	.. controls (0.5,1.63) and (0.4,1.52) .. (0.3,1.52) 
	.. controls (0.3,1.73) and (-0.3, 1.73)  .. (-0.3,1.52) 
	.. controls (-0.4,1.52) and (-0.5, 1.63) .. (-0.5,1.73);
\draw (0.3,1.52) -- (0.3,-1.73) 
	.. controls (0.3,-1.94) and (-0.3, -1.94) .. (-0.3,-1.73) -- (-0.3,1.52);
\draw (-4,0) -- (-2,0) (2,0) -- (3.5,0);
\draw (-1,1.73) -- (-1,3.5) (1,1.73) -- (1,3.5);
\draw (-1,-1.73) -- (-2.5,-3) (1,-1.73) -- (2.5,-3);
\draw[dashed]  (0,3.5) -- (0,1.73) -- (0,-1.85) -- (0.5,-3);
\path (0,0) node[fill=white,inner sep=0pt] {$\scD$};
\path (0,2.5) node[fill=white] {$\scB$};
\path (-0.5,0) node[left] {$\DD$};
\path (0,-2.5) node[fill=white] {$\XX$};
\path (2.5,-1.5) node {$\XD$};
\path (-2.5,-1.5) node {$\DX$};
\path (2.5,1.5) node {$\MB$};
\path (-2.5,1.5) node {$\BM$};
\end{tikzpicture}
\end{center}
\caption{The double space $\dgM$.}
\label{F:g_double}
\end{figure}

\begin{lem}
The b-fibrations $\pi_{R},\,\pi_L : \dpM \to \Z$ lift to b-fibrations of $\dgM$
to $\Z$. The boundary hypersurface $\scD$ is diffeomorphic to the 
fiberwise radial compactification
$\ol{\gT \D} \to \D$.  The lift of $\XiXi$ to $\dgM$ is diffeomorphic to
the scattering double space $\dscXi$ (with fibers $(\fX_i)^2_\msc$ over $\idmon$),
as is the lift of any $\ve > 0$ fiber.

The scattering face of $\dscXi$ is identified with the fiber
$\ol{\gT \D} \rst_{\D \cap \X_i}$.
\label{L:dgM_props}
\end{lem}
\begin{proof}
To show that $\pi_R$ and $\pi_L$ lift to b-fibrations it suffices to show that
no boundary hypersurface of $\dgM$ is mapped into a boundary face of $\Z$ of
codimension more than 1 under $\pi_R$ or $\pi_L$. For the lifts of boundary
hypersurfaces from $\dpM$, this follows from the fact that the original maps on
$\dpM$ are b-fibrations. This leaves only $\scD$ and $\scB$, which are mapped
into the hypersurfaces $\D$ and $\B$, respectively.

As noted previously, $\XiXi \subset \dpM$ is diffeomorphic to the b-double
space $\dbXi$.  The only blow-up which affects this boundary face is the blow
up of $\Delta \cap \DD$, which meets $\dbXi$ at the intersection $\Delta \cap
\mathrm{bf}\subset \dbXi$ of the diagonal with the b-front face. It follows
that that lift of $\XiXi$ to $\dgM$ is diffeomorphic to $\dscXi$. A similar
argument applies to $\X_\ve^2$ which likewise meets the blow-up locus $\Delta
\cap \BB$ at $\Delta \cap \mathrm{bf}$ under the idenfication $\X_\ve^2 \cong
\dbXve$.

For the identification of $\scD$, we work with local coordinates. Local
coordinates on $\dpM$ near $\DD \cap \XiXi$ are furnished by $(x,r',s,
y,y',q)$, where $s = \tfrac{x'}{x} = \tfrac{r}{r'}$ and $x$ is boundary
defining for $\DD$. Passing to the blow-up, these give local coordinates
\[
	(x,r', \sigma, \eta, y', q),
	\quad \sigma = \tfrac{s - 1}{x} \in \bbR,
	\quad \eta = \tfrac{y'-y}{x} \in \bbR^2,
\]
on $\scD$ (with boundary defining coordinate $x$) near $\XiXi$ but away from
its intersection with $\DD$. This front face has the structure of a (radially
compactified) vector bundle over $\D$ where $(\sigma,\eta) \in \bbR^3$ are the
fiber coordinates and the projection is given by $(0,r',\sigma,\eta,y',q)
\mapsto (r',y',q).$ To identify this with $\gT \D$, we consider 
the generating vector fields $\set{\tfrac 1 2(xr\pa_r - x^2\pa_x), x\pa_y}$,
which lift to
\[
	\set{\pa_\sigma, \pa_\eta} + \cO(x)
\]
Thus the coordinate basis $\set{\pa_\sigma,\pa_\eta}$ for the vector space
constituting a fiber of $\scD$ over is naturally associated to the basis
of a fiber of $\gT \D.$ For later use, we record
the fact that the lift of $\set{\tfrac 1 2 (x'r' \pa_{r'} - (x')^2\pa_{x'}, x' \pa_{y'}}$ (i.e., the lift
of gluing vector fields from the left) is given by $\set{- \pa_\sigma, -\pa_\eta} + \cO(x)$.

At the other end, where $\scD$ meets $\scB$, we start with the local
coordinates $(x,\ve, s, y, y',q)$ on $\dpM$ near $\DD \cap \BB$, where here
$s = \tfrac{x'}{x}$ and now $x$ and $x'$ are the right and left lifts of the
boundary defining function $\rho_\B$ for the big end $\D \cap \B$ of $\D.$ 
The first blow-up, of $\Delta \cap \BB$ is represented by coordinates
\[
	(x,\ve, \sigma,\eta,y',q),
	\quad \sigma = \tfrac{s - 1}{x},
	\quad \eta = \tfrac{y'-y}{x},
\]
where now $(\sigma,\eta)$ represent coordinates on the scattering tangent
bundle $\scT \D$. The subsequent blow-up of $\Delta \cap \DD$ (given locally
by $\set{\ve = \sigma = \eta = 0}$) is then represented by
\[
	(x,\ve, \varsigma, \xi, y',q),
	\quad \varsigma = \tfrac{\sigma}{\ve},
	\quad \xi = \tfrac{\eta}{\ve},
\]
where here $(\varsigma,\xi) \in \bbR^3$ are fiber coordinates on $\scD \to \D$
and $\ve$ is boundary defining for $\scD$.  The 
generating
vector
fields $\set{\ve x^2\pa_x, \ve x\pa_y}$ lift to
\[
	\set{\pa_{\varsigma}, \pa_{\xi}} + \cO(\ve)
\]
from the right, and to $\set{-\pa_{\varsigma}, -\pa_\xi} + \cO(x)$ from the
left.
\end{proof}

As with the $\fb$ double space, we may lift vector fields in $\gV(\Z)$ to
$\dgM$ from the left or right; these lifts are well-defined by continuity from
the interior using the product decomposition $\mathring \dgM \cong \mathring
(\bbR^3)^2 \times \intbase$, and we have

\begin{lem}
The left/right lifts of $V  \in \gV(\Z)$ are tangent to all boundary faces of
$\dgM$ and are differentially transversal to $\Delta$. The restriction of
$\pi_R^\ast(V)$ to $\scD$, $\XiXi$ or $\X_\ve^2$ may be identified with 
the lift of the restriction to the corresponding double space:
\begin{subequations}
\begin{align}
	\pi_R^\ast(V) \rst_{\scD} &\cong \pi_R^\ast (V \rst_{\D}), 
	& \pi_R &: \ol {\gT \D} \to \D \subset \Z,
	\label{E:dgM_lift_V_D}
	\\\pi_R^\ast(V) \rst_{\XiXi} &\cong \pi_R^\ast(V \rst_{\X_i}),
	&\pi_R &: \dscXi \to \X_i \subset \Z,
	\label{E:dgM_lift_V_X}
	\\\pi_R^\ast(V) \rst_{\X_\ve^2} &\cong \pi_R^\ast(V \rst_{\X_\ve}),
	&\pi_R &: \dscXve \to \X_\ve \subset \Z,
	\label{E:dgM_lift_V_fib}
\end{align}
\end{subequations}
and similarly for $\pi_L^\ast(V)$.
\label{L:dgM_lift_V}
\end{lem}
\begin{proof}
For $\XiXi$ and $\X_\ve^2$, the result follows from Lemma~\ref{L:dpM_lift_V},
and Lemma~\ref{L:dgM_props}.
The result for $\scD$ follows from the computations done in the proof of
Lemma~\ref{L:dgM_props}.
\end{proof}

\subsection{The triple spaces} \label{S:triple}

To obtain the triple space $\tpM$, 
we start with the binomial subvariety (manifold with generalized corners)
\[
	\Z^{[3]} := \Z\times_\fb \Z\times_\fb \Z \subset \Z^3,
\]
and its monoidal complex, consisting of monoids of the form $\sigma_\bullet
\times_\tau \sigma'_\bullet\times_\tau \sigma''_\bullet$, where
$\sigma_{\bullet}, \sigma'_\bullet,\sigma''_\bullet \in
\set{\sigma_\D,\sigma_{\X_i}, \sigma_\B}$. Of these, the most singular are
$\sigma_{\D\cap \X_i}\times_\tau \times\sigma_{\D\cap \X_j} \times_\tau
\sigma_{\D \cap \X_k}$, which are each isomorphic to the 4-dimensional monoid
\begin{equation}
	\mu := \set{(n_1,m_1,n_2,m_2, n_3,m_3) : n_1 + m_1 = n_2 + m_2 = n_3 + m_3} \subset \bbN^6
	\label{E:cube_monoid}
\end{equation}
with the 8 generators
\[
\begin{gathered}
	(1,0,1,0,1,0),
	\quad (1,0,1,0,0,1),
	\\ (1,0,0,1,1,0),
	\quad (1,0,0,1,0,1),
	\\ (0,1,1,0,1,0),
	\quad (0,1,1,0,0,1),
	\\ (0,1,0,1,1,0),
	\quad (0,1,0,1,0,1).
\end{gathered}
\]
These may be identified respectively with the hypersurfaces
(introducing the obvious notation) $\DDD, \DDX_k,
\DXD_j, \ldots, \XXX_{ijk}.$ To resolve these, we subdivide each such monoid into the 6
smooth submonoids:
\begin{equation}
\begin{gathered}
	\bbN\pair{\DDD, \DDX_k, \DXX_{jk},\XXX_{ijk}},
	\quad \bbN\pair{{\DDD}, {\DDX_k}, {\XDX_{ik}},{\XXX_{ijk}}},
	\\ \bbN\pair{{\DDD}, {\DXD_j}, {\DXX_{jk}},{\XXX_{ijk}}},
	\quad \bbN\pair{{\DDD}, {\DXD_j}, {\XXD_{ij}},{\XXX_{ijk}}},
	\\ \bbN\pair{{\DDD}, {\XDD_i}, {\XDX_{ik}},{\XXX_{ijk}}},
	\quad \bbN\pair{{\DDD}, {\XDD_i}, {\XXD_{ij}},{\XXX_{ijk}}}.
\end{gathered}
	\label{E:resolve_cube}
\end{equation}

The remaining singular monoids are of the form $\sigma_{\D \cap \X_i}
\times_\tau \times \sigma_{\D \cap X_j} \times_\tau \sigma_\B$ (and various
permutations of the factors), which are products of the form $\nu \times \bbN$,
where $\nu$ is the monoid in \eqref{E:square_monoid} and are resolved by taking
the product of the resolution of $\nu$ discussed above with $\bbN$. The result of all this
is a smooth manifold $\Z^3_\mathrm{b}$, and the final step is the blow-up
\[
	\tpM := [\Z^3_\mathrm{b}; \BBB, \BBM, \BMB, \MBB].
\]
This corresponds to the subdivision of $\sigma_{\BBB} \cong \bbN^3$ (generated by 
the 1-dimensional submonoids $\sigma_{\BMM} = \bbN\pair{(1,0,0)}$, $\sigma_{\MBM} = \bbN\pair{(0,1,0)}$, 
and $\sigma_{\MMB} = \bbN\pair{(0,0,1)}$) into the following 6 submonoids:
\begin{equation}
\begin{gathered}
	\bbN\pair{(1,0,0),(1,0,1), (1,1,1)},
	\quad \bbN\pair{(1,0,0),(1,1,0), (1,1,1)},
	\\ \bbN\pair{(0,1,0),(0,1,1), (1,1,1)},
	\quad \bbN\pair{(0,1,0),(1,1,0), (1,1,1)},
	\\ \bbN\pair{(0,0,1),(1,0,1), (1,1,1)},
	\quad \bbN\pair{(0,0,1),(0,1,1), (1,1,1)}.
\end{gathered}
	\label{E:tbb_three}
\end{equation}
(This is equivalent to the so-called ``total boundary blow-up'' of the
codimension 3 corner $\BBB$.)

As with $\dpM$, we label the boundary hypersurfaces of $\tpM$ by the boundary
faces of $\Z^{[3]}$ which they lift, which are labeled in turn by the
corresponding products of boundary hypersurfaces of $\Z$, omitting the symbol $\times_\idmon$. The complete
list of these is:
\[
\begin{gathered}
	\DDD, \DDX_k, \DXD_j, \XDD_i, \DXX_{jk}, \XXD_{ij}, 
	\XDX_{ik}, \XXX_{ijk}, 
	\\\BBB, \BBM, \BMB, \MBB, \BMM, \MMB, \MBM.
\end{gathered}
\]
\begin{lem}
\mbox{}
\begin{enumerate}
[{\normalfont (a)}]
\item \label{I:tpM_bfibn}
$\tpM$ admits three b-fibrations $\tpM \to \dpM$ lifting the 3 projections $\Z^{[3]} \to \Z^{[2]}$.
\item \label{I:tpM_psub}
The total and partial fiber diagonals in $\Z^{[3]}$ lift to p-submanifolds of $\tpM$.
\item \label{I:tpM_boundary}
The boundary hypersurface $\DDD$ is diffeomorphic to the b-triple space $\tbD$ (with fiber $\tbfD$ over $\idmon$)
and likewise $\XXX_{iii} \cong \tbX$ (with fiber $\tbfX$). The lifted projections restrict over
these hypersurfaces to the corresponding lifted projections to the b-double
spaces.
\item \label{I:tpM_fiber}
For $\ve > 0$, the fiber $\XXX_\ve := \fb^{-1}(\idmon\times \set \ve)$ of $\fb : \tpM \to \base$
is diffeomorphic to the b triple space $(\ol{\bbR^3})^3_\mathrm{b}\times \idmon.$
\end{enumerate}
\label{L:tpM}
\end{lem}
\begin{proof}
In the first place, observe that, under the 3 projections $\bbN^6 =(\bbN^2)^3
\to \bbN^4 = (\bbN^2)^2$, the resolution \eqref{E:resolve_cube} projects to the
resolution \eqref{E:double_mon_resolve} of \eqref{E:square_monoid} described
above. By results in \cite{kottke2015generalized,kottke2015blow}, it follows
that the maps $\Z^{[3]} \to \Z^{[2]}$ lift to well-defined b-maps $\tpM \to
\dpM$. That they are b-fibrations follows from the fact that, under the
corresponding morphisms of monoids, no 1-dimensional monoid of $\cP_{\tpM}$ is
mapped into the interior of a monoid in $\cP_{\dpM}$, (i.e., the map is
b-normal; b-surjectivity is automatic here). This proves \eqref{I:tpM_bfibn}.

Part \eqref{I:tpM_psub} also follows in part from a result in \cite{kottke2015generalized}. Indeed,
the partial fiber diagonals are also binomial subvarieties of $\Z^{[3]}$, and
near the corners formed by $\D$ and $\X_i$ are associated to further submonoids
of \eqref{E:cube_monoid} where $(n_i,m_i) = (n_j,m_j)$. Near corners formed by
$\B$, they are associated to submonoids of $\sigma_\BBB \cong \bbN^3 \ni
(n_1,n_2,n_3$ where $n_i = n_j$.  The resolutions \eqref{E:resolve_cube} and
\eqref{E:tbb_three} are {\em compatible} with these submonoids, and by \cite{kottke2015generalized}
Proposition 10.3 it follows that the diagonals lift to p-submanifolds
of $\tpM$.

For \eqref{I:tpM_boundary}, we consider the effect of the resolution $\tpM \to
\Z^{[3]}$ on the hypersurfaces of $\Z^{[3]}$. For example, $\Z^{[3]}$ has a
boundary hypersurface given by the product $\D \times_\idmon \D \times_\idmon \D$.
Considered as a binomial variety in its own right, this has smooth corners of
codimension at most 3, and its monoidal complex is isomorphic to the quotient
complex $\cP_{\Z^{[3]}} / \sigma_\DDD$. For instance, it is straightforward to
verify that the quotient of \eqref{E:cube_monoid} by $\sigma_{\DDD} =
\bbN\pair{(1,0,1,0,1,0)}$ is a monoid freely generated by the images of
$\XDD_i$, $\DXD_j$ and $\DDX_k$. The image of the resolution
\eqref{E:resolve_cube} under this quotient is likewise easily seen to coincide
with the ``total boundary blow-up'' resolution \eqref{E:tbb_three} of $\bbN^3$,
which corresponds in turn to the sequence of blow-ups realizing the b-triple
space $\tbD$; thus $\DDD \cong \tbD$ in $\tpM$. A similar argument applies to
$\XXX_i$ and to \eqref{I:tpM_boundary}.
\end{proof}

\bigskip 

To define the triple space $\tgM$, let $\Delta_{123}$ and $\Delta_{ij}$, $1
\leq i < j \leq 3$ denote the maximal and partial diagonals, respectively, as
submanifolds of $\tpM$, and for a boundary face $F$ of $\tpM$ write
\[
	\Delta_{\ast}^{F} = \Delta_{\ast} \cap F
\]
for the intersection of the face with one of these diagonals. Then
\[
	\tgM := [\tpM; \Delta_{123}^\DDD,\Delta_{123}^\BBB, \Delta_{ij}^\DDD, \Delta_{ij}^\BBB,
	\Delta_{12}^\DDX, \Delta_{13}^\DXD, \Delta_{23}^\XDD, \Delta_{12}^\BBM, \Delta_{13}^\BMB,
	\Delta_{23}^\MBB]
\]
In addition to the lifts of the boundary faces of $\tpM$, for which we use the same notation,
$\tgM$ has as additional boundary faces the various front faces of the blow up, which we denote
by
\[
\begin{gathered}
	(\DDD_\msc)_{123}, (\BBB_\msc)_{123}, (\DDD_\msc)_{ij}, (\BBB_\msc)_{ij},
	\\ \DDX_\msc^k, \DXD_\msc^j, \XDD_\msc^i, \BBM_\msc, \BMB_\msc, \MBB_\msc
\end{gathered}
\]
where the superscripts indicate the corresponding face $\X_i$ while the
subscripts are used to indicate the diagonals.

\begin{prop}
\begin{enumerate}
[{\normalfont (a)}]
\item 
$\tgM$ admits three b-fibrations $\pi_{ij} : \tgM \to \dgM$ lifting the fiber projections
$\pi_{ij} : \Z^{[3]} \to \Z^{[2]}$.
\item
The partial diagonals meet all boundary hypersurfaces of $\tgM$ transversally.
\item
The boundary face $\XXX_{iii}$ of $\tgM$ is isomorphic to the families
scattering triple space $(\X\times_\idmon\X \times_\idmon \X)_\msc$ (with fiber $\fX^3_\msc$),
and the $\pi_{ij}$ restrict over this face to the corresponding b-fibrations to 
$\XX_{ii} \cong (\X\times_U \X)_\msc$.
\item
For $\ve > 0$, the fiber $\XXX_\ve := \fb^{-1}(\idmon\times \set{\ve})$ of $\tgM$ is isomorphic
to the scattering triple space $(\ol \bbR^3)^3_\msc \times \idmon$.
\item
The face $\DDD_\msc$ is diffeomorphic to a compactification of the fiber
product $\gT \D \times_\D \gT \D$.  The lifted projections $\pi_{ij}$ map
$\DDD_\msc$ to $\scD \cong \ol {\gT \D}$, and on the interior are identified
with the linear maps $\pi_{12}(p,\xi,\xi') = (p,\xi)$, $\pi_{23}(p,\xi,\xi') = (p,\xi')$
and $\pi_{13}(p,\xi,\xi') = (p,\xi + \xi')$.
\end{enumerate}
\label{P:tgM}
\end{prop}
\noindent The proof, which involves tedious local coordinate computations similar to
those in the proof of Lemma~\ref{L:dgM_props}, is left as an exercise to the reader.

\subsection{Pseudodifferential operators} \label{S:pseudo}

Define $\fb$-pseudodifferential operators by
\[
\begin{gathered}
	\pPsi^{s,\cE}(\Z) = \cAphg^{\cE} \cI^s(\dpM; \Delta),
	\\ \cE = (E_{\DD},E_{\XX},E_{\BB},E_{\DX},E_{\XD},E_{\MB},E_{\BM}, \infty_{\XiXj})
\end{gathered}
\]
Here the index sets are indexed by the hypersurfaces of $\dpM$; $E_{\XX}$
means a fixed index set for the uniion of hypersurfaces of the form $\XiXi$,
and likewise $E_{\DX}$ means a fixed index sets for the union of hypersurfaces
of the form $\DXi$. We allow nontrivial asymptotics at all faces
meeting $\Delta$, as well as those spaces one step removed; however our kernels
will be rapidly decreasing at the hypersurfaces two steps removed from
$\Delta$; i.e., at $\XiXj$ for $i \neq j$. The extension to operators
acting on sections of bundles $V_i \to \Z$, $i = 1,2$ is achieved by 
considering coefficients in $\Hom(\pi_L^\ast V_1,\pi_R^\ast V_2)$ on $\dpM.$

By Lemma~\ref{L:dpM}.\eqref{I:dpM_Delta} and \eqref{I:dpM_fiber}, The
restriction of $Q \in \pPsi^{s,\cE}(\Z)$ to a boundary face $\DD \cong \dbD$
or $\XiXi \cong \dbXi$ (where restriction of a polyhomogeneous section to a
boundary face is defined in general to be the restriction of the leading order
term), or to a fiber $\X_\ve^2 \cong \dbXve$ may be identified with the
Schwartz kernel of a b-pseudodifferential operator on the associated space. Thus we define
the {\em normal operators} of $Q \in \pPsi^{s,\cE}(\Z)$ by
\begin{subequations}
\begin{align}
	N_\D(Q)	&:= Q\rst_{\DD} \in \bPsi^{s,\cF_\D}(\D/\idmon),
	\label{E:dpM_normal_D}
	\\ N_{\X_i}(Q) &:= Q\rst_{\XiXi} \in \bPsi^{s,\cF_{\X_i}}(\X_i/\idmon),
	\label{E:dpM_normal_X}
	\\ N_{\X_\ve}(Q) &:= Q \rst_{\X_\ve^2} \in \bPsi^{s,\cF_{\X_\ve}}(\X_\ve/\idmon), \quad \ve > 0
	\label{E:dpM_normal_fib}
\end{align}
\end{subequations}
where the index sets $\cF_\ast$ are determined from $\cE$:
\[
\begin{aligned}
	\cF_\D &= (F_{\mathrm{lf},\X}, F_{\mathrm{rf},\X},F_{\mathrm{lf},\B},F_{\mathrm{rf},\B},F_{\mathrm{bf},\X}, F_{\mathrm{bf},\B}) 
	\\&\hfil= (E_{\XD},E_{\DX}, E_{\BM}, E_{\MB} ,E_{\XX},E_{\BB})
	\\\cF_{\X_i} &= (F'_\mathrm{lf}, F'_\mathrm{rf}, F'_\mathrm{bf}) 
	= (E_{\DX}, E_{\XD}, E_{\DD})
	\\ \cF_{\X_\ve} &= (F''_\mathrm{lf},F''_\mathrm{rf},F''_\mathrm{bf})
	= (E_{\MB},E_{\BM},E_{\BB})
\end{aligned}
\]

To define the action of $\pPsi^{s,\cE}(\Z)$ on $\dot C^\infty(\Z)$ we need to
be able to pushforward with respect to $\pi_L : \dpM \to \Z$, which is only
defined for fiber densities. Since it is most straightforward to prove mapping
properties with respect to fiber b-densities, we define the action by
\[
	Qu := (\pi_L)_\ast \bpns{Q \pi_R^\ast u \pi_R^\ast \pnu}
	\quad Q \in \pPsi^{s,\cE}(\Z),
	\quad u \in \dot C^\infty(\Z),
\]
where $\pnu$ is the trivializing section of the $\fu$ density bundle on $\Z$
obtained from the volume form of $\pg = (\rho_\D \rho_\B)^2 g$. In particular,
$\pi_L^\ast\pnu \otimes
\pi_R^\ast\pnu\otimes\pi_I^\ast(\abs{\tfrac{d\ve}{\ve}})$ is a trivializing
fiber b-density on $\dpM \to \idmon$. 

Likewise, the composition of $A \in \pPsi^{s,\cE}(\Z)$ with $B \in \pPsi^{t,\cF}(\Z)$ is defined on the triple
space by
\begin{equation}
	A \circ B := (\pi_{13})_\ast (\pi_{12}^\ast A \, \pi_{23}^\ast B\,\otimes \pi_{2}^\ast \nu)
	\label{E:dpM_compn}
\end{equation}
where $\pi_{ij} : \tpM \to \dpM$ denote the lifted projections to the double
space and $\pi_{i} : \tpM \to \Z$ denote the lifted projections to the single
space $\Z$. The conditions under which this composition is well-defined are
discussed in the Theorem below.

The properties of the $\fb$-pseudodifferential oeprators that we shall need are
summarized in Theorems~\ref{T:dpM_package} and \ref{T:dpM_mapping}

\begin{thm}
\mbox{}
\begin{enumerate}
[{\normalfont (a)}]
\item \label{I:dpM_package_indicial}
Let $Q \in \pPsi^{s,\cE}(\Z)$.
At the common boundary face $\D \cap \X_i$, the indicial operators of $N_\D(Q)$
and $N_{\X_i}(Q)$ are related by
\[
	I(N_\D(Q), \lambda) = I(N_{\X_i}(Q), -\lambda).
\]
\item \label{I:dpM_package_comp_psi}
If $A \in \pPsi^{s,\cE}(\Z)$, $B \in \pPsi^{t,\cF}(\Z)$, and
$E_\MB + F_\MB > 0$, then the composition $A \circ B \in \pPsi^{s + t,\cG}(\Z)$ is well-defined, with
\[
\begin{gathered}
	G_\DD = (E_\DD + F_\DD) \ol \cup (E_\DX + F_\XD), 
	\quad G_\XD = (E_\XX + F_\XD)\ol \cup (E_\XD + F_\DD),
	\\ G_\DX = (E_\DX + F_\XX)\ol \cup (E_\DD + F_\DX),
	\quad G_\XX = (E_\XX + F_\XX)\ol \cup (E_\XD + F_\DX),
	\\ G_\BM = (E_\BB + F_\BM)\ol \cup E_\BM,
	\quad G_\MB = F_\MB\ol \cup (E_\MB + F_\BB),
	\\ G_\BB = (E_\BM + F_\MB) \ol \cup (E_\BB + F_\BB).
\end{gathered}
\]
We always have
\[
	N_{\X_\ve}(A \circ B) = N_{\X_\ve}(A) \circ N_{\X_\ve}(B) \in \bPsi^{s + t,\ast}(\X_\ve/\idmon),
\]
and, provided that $E_\DX + F_\XD > E_\DD + F_\DD \geq 0$ and $E_\XD + F_\DX >
E_\XX + F_\XX \geq 0$, respectively, then
\begin{equation}
\begin{gathered}
	N_\D(A \circ B) = N_\D(A) \circ N_\D(B) \in \bPsi^{s + t, \ast}(\D/\idmon),
	\\ N_{\X_i}(A \circ B) = N_{\X_i}(A) \circ N_{\X_i}(B) \in \bPsi^{s + t, \ast}(\X_i/\idmon),
\end{gathered}
	\label{E:dpM_normal_operator_comp}
\end{equation}
\item \label{I:dpM_package_diff}
There is an injective homomorphism of graded algebras
\[
	\pDiff^\ast(\Z) \hookrightarrow \pPsi^\ast(\Z),
\]
with respect to which $\pDiff^k(\Z) \subset \pPsi^{k,(0_\DD,0_\XX,0_\BB,\infty_\ast)}(\Z),$
and the normal operators of $P \in \pDiff^\ast(\Z)$ are identified with
restriction of $P$ to the corresponding boundary face of $\Z$:
\[
	N_\D(P) \cong P \rst_\D,
	\quad N_{\X_i}(P) \cong P \rst_{\X_i},
	\quad N_{\X_\ve}(P) \cong P \rst_{\X_\ve}.
\]
\end{enumerate}
\label{T:dpM_package}
\end{thm}

\begin{proof}
To prove \eqref{I:dpM_package_indicial} we emply a local coordinate
description. Near the interior of $\DD \cap \XiXi$ in $\dpM$, the coordinates 
\[
	(x,r', s, y, y', q),
	\quad s = \tfrac{x'}{x} = \tfrac{r}{r'}
\]
are valid, with $x$ boundary defining for $\DD$ and $r'$ boundary defining for
$\XiXi$. (Alternatively, we may use $(x',r,s^{-1},y,y',q)$). Here $(x,r,y,q)$
denote local coordinates on $\Z$ pulled back from the right and $(x',r',y',q)$
denote coordinates pulled back from the left. In any case, $I(N_\D(Q),
\lambda)$ may be expressed as the Mellin transform of the restriction of $Q$ to
$\set{x = r' = 0}$ with respect to $s$, while $I(N_\D(Q),\lambda)$ is the
Mellin transform of the same with respect to $s^{-1}$, from which
\eqref{I:dpM_package_indicial} follows.

The composition of pseudodifferential operators in
\eqref{I:dpM_package_comp_psi} is defined by \eqref{E:dpM_compn}. By wavefront
considerations, it is clear that the only interior conormal singularities of $A
\circ B \in \dpM$ can occur along the fiber diagonal, and the fact that $A
\circ B$ has interior conormal order $s + t$ (with the principal symbolic
composition formula $\sigma(A\circ B) = \sigma(A) \sigma(B)$) follows from the
usual local considerations. To verify the index set formulae, we may assume $s
= t = - \infty$. The result is then a consequence of the pullback and
pushforward theorems for polyhomogeneous functions with respect to
b-fibrations \cite{CCN}. To invoke these theorems, it is only necessary to determine the
mapping properties of boundary hypersurfaces with respect to the lifted
projections $\pi_{ij}$ (since the boundary exponents occuring in the
b-fibrations $\pi_{ij} : \tpM \to \dpM$ are either $0$ or $1$), but these have
been made obvious from the notation.  For instance, the formula for $G_\DX$ is a
consequence of the following:
\[
\begin{gathered}
	\pi_{13}^{-1}(\DX) = \DDX \cup \DXX,
	\\(\pi_{12})_\#(\DDX) = \DD, \quad (\pi_{23})_\#(\DDX) = \DX,
	\\(\pi_{12})_\#(\DXX) = \DX, \quad (\pi_{23})_\#(\DXX) = \XX.
\end{gathered}
\]
The others are similar.

To see \eqref{E:dpM_normal_operator_comp}, observe that if $E_\DD, F_\DD \geq
0$ and $E_\DX + F_\XD > 0$, then the leading order contribution of $A \circ B$
at $\DD$ comes from $\DDD$. In other words, this leading order term is given by
\[
\begin{gathered}
	(\pi_{13})_\ast (\pi_{12}^\ast A\, \pi_{23}^\ast B\, \pi_2^\ast \nu_\D),
	\\ \pi_{ij} : \DDD \to \DD \subset \dpM,
	\quad \pi_2 : \DDD \to \D \subset \Z,
\end{gathered}
\]
(Note that the fiberwise b-density $\pi_2^\ast \nu$ on $\tpM$ restricts
canonically to a fiberwise b-density on the hypersurface $\DDD$, which is
interwtined with the pullback of the restriction of $\nu$ to $\D$.) That this
may be identified with the usual composition of b-pseudodifferential operators
then follows from Lemma~\ref{L:tpM}. A similar argument applies to the
composition of normal operators with respect to $\X_i$ and $\X_\ve.$

For part \eqref{I:dpM_package_diff}, let $P \in \pDiff^k(\Z).$ By
Lemma~\ref{L:dpM_lift_V}, the lifts $\pi_L^\ast(P)$ and $\pi_R^\ast(P)$
differentiate transversally to $\Delta$ and tangentially to all boundary faces
of $\dpM$, and the image of $P$ in $\pPsi^{k,\ast}(\Z)$ is given by applying
$\pi_L^\ast(P)$ to the kernel of $\id \in \pPsi^{0,(0_\DD,0_\XX,0_\BB)}(\Z)$.
Composing two such images of $P_1,P_2 \in \pDiff^\ast(\Z)$ in $\pPsi^\ast$ is
readily seen to be equivalent to the image of $P_1\circ P_2$. The
identification of the normal operators follows from Lemma~\ref{L:dpM_lift_V}.
\end{proof}

\bigskip
Define $\gl$-pseudodifferential operators by
\[
\begin{gathered}
	\gPsi^{s,\cE}(\Z) = \cAphg^{\cE'} \cI^s(\dgM; \Delta),
	\\ \cE = (E_{\scD}, E_{\XX}, E_{\scB}, \infty_\ast).
	\\ \cE' = (E_{\scD} - 3, E_{\XX}, E_{\scB} - 3, \infty_\ast).
\end{gathered}
\]
These kernels are required to have rapid decay at all boundary hypersurfaces
besides those which meet the lift of $\Delta$. The shift in index by $-3$ at
$\scD$ and $\scB$ may be regarded as a normalization convention; in particular
$\id \in \gPsi^{0,(0,0,0)}(\Z)$ with this convention. (The lift of the delta
function---which is homogeneous of degree $-1$ as a distrubiton---of the fiber
diagonal from $\dpM$ to $\dgM$ has polyhomogeneous order $-3$ at $\scD$ and
$\scB$.) Likewise, the pullback of the fiber b-density bundle $\bO(\dpM)$ to
$\dgM$ is isomorphic to $\rho_{\scD}^3\rho_{\scB}^3 \bO(\dgM)$, so the
shift is nullified when these Schwartz kernels are multiplied by the volume
form $\nu$ from $\Z$ prior to pushing forward.

By Lemma~\ref{L:dgM_props}, the restriction of $Q \in \gPsi^{s,\cE}(\Z)$
to a boundary face $\XiXi$ may be identified with the kernel of a family of scattering
pseudodifferential operators on $\X_i$, and likewise the restriction to the
fiber $\X_\ve^2 = \fb^{-1}(\idmon \times \set \ve)$, $\ve > 0$ may be identified with a family of scattering
pseudodifferential operators on $\X_\ve.$ The restriction of $Q$  to $\scD$
defines a conormal distribution on $\ol{\gT \D}$ conormal to the $0$-section,
whose fiberwise Fourier transform may be identified with a symbol. Thus
we define
\begin{subequations}
\begin{align}
	\sigma_\D(Q) &= \sF_\fib(Q \rst_{\scD}) \in \cAphg^{\cF_\D} S^{s}_{1,0}(\gT^\ast \D),
	\label{E:dgM_normal_D}
	\\N_{\X_i}(Q) &= Q \rst_{\XiXi} \in \scPsi^{s,f}(\X_i/\idmon)
	\label{E:dgM_normal_X}
	\\N_{\X_\ve}(Q) &= Q \rst_{\X_\ve^2} \in \scPsi^{s,f'}(\X_\ve/\idmon), \quad \ve > 0
	\label{E:dgM_normal_fib}
\end{align}
\end{subequations}
where 
\[
	\cF_\D = (E_{\scB}, E_{\XX}),
	\quad f = E_{\scD},
	\quad f' = E_{\scB}
\]

\begin{thm}
\mbox{}
\begin{enumerate}
[{\normalfont (a)}]
\item \label{I:dgM_package_indicial}
Let $Q \in \gPsi^{s,\cE}(\Z)$.
At the common boundary face $\D \cap \X_i$, the symbol $\sigma_\D(Q)$ and the 
scattering symbol $\sigma_\msc(N_{\X_i}(Q))$ agree.
\item \label{I:dgM_package_comp_psi}
If $A \in \gPsi^{s,\cE}(\Z)$, $B \in \gPsi^{t,\cF}(\Z)$ then the composition $A
\circ B \in \pPsi^{s + t,\cG}(\Z)$ is well-defined, with
\[
\begin{gathered}
	G_{\scD} = E_{\scD} + F_{\scD}
	\quad G_{\scB} = E_{\scB} + F_{\scB}
	\quad G_\XX = E_\XX + F_\XX.
\end{gathered}
\]
and the normal operator maps are homomorphisms:
\begin{equation}
\begin{gathered}
	N_{\X_\ve}(A \circ B) = N_{\X_\ve}(A) \circ N_{\X_\ve}(B) \in \scPsi^{s + t,\ast}(\X_\ve/\idmon),
	\\ N_{\X_i}(A \circ B) = N_{\X_i}(A) \circ N_{\X_i}(B) \in \scPsi^{s + t, \ast}(\X_i/\idmon),
	\\ \sigma_\D(A \circ B) = \sigma_\D(A)\sigma_\D(B) \in \cAphg^{\ast} S^{s+t}_{1,0}(\gT^\ast \D).
\end{gathered}
	\label{E:dgM_normal_operator_comp}
\end{equation}
\item \label{I:dgM_package_diff}
There is an injective homomorphism of graded algebras
\[
	\gDiff^\ast(\Z) \hookrightarrow \gPsi^\ast(\Z),
\]
with respect to which $\gDiff^k(\Z) \subset \gPsi^{k,(0_\scD,0_\XX,0_\scB)}(\Z),$
and the normal operators of $P \in \gDiff^\ast(\Z)$ are identified with
restriction of $P$ to the corresponding boundary face of $\Z$:
\[
	\quad N_{\X_i}(P) \cong P \rst_{\X_i},
	\quad N_{\X_\ve}(P) \cong P \rst_{\X_\ve}.
\]
Likewise, the semiclassical symbol of $P$ is the same, considered as a differential
or pseudodifferential operator.
\end{enumerate}
\label{T:dgM_package}
\end{thm}
\noindent
The proof is similar to the proof of Theorem~\ref{T:dpM_package}, and is left to the reader. It is
worth remarking that, although the triple space $\tgM$ is more complex, the composition result here
is much simpler owing to the rapid decay of elements of $\gPsi^\ast(\Z)$ away from all boundary faces
except $\scD$, $\scB$ and $\XX$.

\subsection{Sobolev spaces and mapping properties} \label{S:pseudo_sobolev}

Consider the set of {\em small} $\fb$-pseudodifferential operators:
\[
	\psmPsi^{s}(\Z) := \pPsi^{s,(\geq 0_\DD,\geq 0_\XX,\geq 0_\BB, \infty_\ast)}(\Z),
\]
where the notation indicates that the index sets of the operators have leading
order $(0,0)$ at $\DD$, $\XX$ and $\BB$ and are empty everywhere else.  By
Theorem~\ref{T:dpM_package} this set is closed with respect to composition.
Furthermore, if $A \in \psmPsi^s(\Z)$ is elliptic (meaning its principal symbol
as a conormal distribution is uniformly invertible off the diagonal), then there exists a
{\em small parametrix}, meaning $B \in \psmPsi^{-s}(\Z)$ such that that $I -
AB, I - BA \in \psmPsi^{-\infty}(\Z).$ 

The (fiberwise) $\fb$ Sobolev spaces are most efficiently defined as follows.
Fix a fiber $\fZ$ of $\fu : \Z \to \idmon$ and set
\[
	\pH^s(\fZ) = \set{u \in C^{-\infty}(\fZ) : A u \in L^2(\fZ; \olg), \ \forall A \in \psmPsi^{s}(\fZ)}
\]
where $\psmPsi^{s}(\fZ)$ denotes the restriction of $\psmPsi^{s}(\Z)$ to the
corresponding fiber of $\fu : \dpM \to \idmon$.  It will follow from the results below
that if $s \geq 0$, this can be taken to be a domain in $L^2(\fZ)$.
Moreover, if $s \in \bbN$ then the $A$ can be taken
in $\pDiff^s(\fZ).$ 

\begin{lem}
\mbox{}
\begin{enumerate}
[{\normalfont (a)}]
\item \label{I:dpM_mapping_L2}
Every $A \in \psmPsi^{0}(\fZ)$ extends to a bounded operator $A : L^2(\fZ; \olg) \to L^2(\fZ;\olg)$. In particular, $\pH^0(\fZ;\olg) \equiv L^2(\fZ;\olg).$
\item \label{I:dpM_mapping_smoothing}
If $A \in \psmPsi^{-\infty}(\fZ)$, then
\[
	A : L^2(\fZ) \to \pH^{s}(\fZ), \quad \forall s \in \bbR.
\]
\item \label{I:dpM_mapping_elliptic}
Fix any elliptic element $P \in \psmPsi^{s}(\fZ)$. Then
\[
	\pH^s(\fZ) = \set{u \in C^{-\infty}(\fZ) : Pu \in L^2(\fZ)}.
\]
\end{enumerate}
\label{L:dpM_mapping}
\end{lem}
\begin{proof}
The first result follows by a standard trick due to H\"ormander. Namely, let
$c > 0$ such that $c^2 \geq \sup\abs{\sigma(A)^\ast \sigma(A)}$. Then by an
iterative symbolic procedure there exists a formally self-adjoint $B \in
\pPsi^{s}(\fZ)$ and $R \in \psmPsi^{-\infty}(\fZ)$ such that
\[
	B^2 = c^2\id - A^\ast A + R
\]
with composition here defined as operators on distributions. Then for $u \in
C^\infty_c(\fZ)$,
\[
	\norm{Au}^2_{L^2} = \pair{c^2u,u} + \pair{Ru,u} - \norm{Bu}^2
	\leq c\norm{u}^2 + \pair{Ru,u}.
\]
Boundedness then follows from part \eqref{I:dpM_mapping_smoothing} with $s =
0$, which follows in turn from Schur's Lemma. 

Indeed, any $R \in \psmPsi^{-\infty}(\fZ)$ is represented by a kernel on
$\dpfM$ which is uniformly bounded (by the hypothesis that its index sets are
$\geq 0$) and smooth on the interior. Schur's Lemma states that this extends to
a bounded operator from $L^2(\fZ; \olg)$ to $L^2(\fZ; \olg)$  provided
its left and right projections $(\pi_L)_\ast(\abs R), (\pi_R)_\ast (\abs R) \in
\cAphg^\ast(\fZ)$ are uniformly bounded. Here the pushforward is with respect
to the volume form associated to $\ol g$; to convert to the natural b-volume
form on $\fZ$ we consider instead the conjugated operator $\wt R = \rho^{-3/2}
R \rho^{3/2}$, which has the same index index sets as $R$, namely $\wt R \in
\cAphg^{(0_\DD,0_\XX,0_\BB,\infty_\ast)}(\dpfM)$. By
the pushforward and pullback theorems in \cite{CCN},
this has left and right projections in
$\cAphg^0(\fZ)$, which are indeed uniformly bounded. This proves part
\eqref{I:dpM_mapping_smoothing} in the case $s = 0$ (and hence part
\eqref{I:dpM_mapping_L2}). The general case follows by composing with any $A
\in \psmPsi^{s}(\fZ)$ and using the fact that $\psmPsi^{-\infty}$ is an
ideal.

From part \eqref{I:dpM_mapping_L2} and composition, it follows that
$\psmPsi^{-s} \ni Q : L^2 \to \pH^s$, and then part
\eqref{I:dpM_mapping_elliptic} follows from the existence of a parametrix and
the identity $QPu = u - Ru$, $R \in \psmPsi^{-\infty}$.
\end{proof}

For any choice of elliptic operator $P_s \in \psmPsi^{s/2}(\fZ)$, 
\[
	P_s^\ast P_s + 1 : \pH^s(\fZ) \to L^2(\fZ)
\]
is bounded, self-adjoint, and easily seen to have no nullspace. It is therefore
an isomorphism, in terms of which $\pH^s(\fZ)$ may be given the structure of a
complete Hilbert space, with topology independent of the choice of $P_s$. The following
result follows from this observation and composition.

\begin{cor}
Every $A \in \psmPsi^{s}(\fZ)$ extends to a bounded operator
\[
	A : \pH^{m}(\fZ) \to \pH^{m-s}(\fZ), 
	\quad \forall m \in \bbR.
\]
\label{C:dpM_mapping_cor}
\end{cor}

By definition, an operator $A \in \pPsi^\ast(\Z)$ is smoothly parameterized by $\idmon$, i.e.,
$A$ can be viewed as a family of pseudodifferential operators with respect to the fibration
$\fu: \Z \to \idmon$. It follows that the space of smooth sections of the Hilbert bundle over $\idmon$,
with fibers $\pH^s(\fZ)$, is characterized as follows:
\[
	C^\infty(\idmon; \pH^s(\fZ)) 
	= \set{u : A u \in C^\infty(\idmon; L^2(\fZ)), \text{ for all } A \in \psmPsi^s(\fZ)}
\]
where ``for all'' may be replaced by ``for some elliptic''. 

Finally, we consider the mapping properties of the full calculus $\pPsi^\ast(\Z)$ with respect to the 
spaces $C^\infty(\idmon; \bpH^{m,s}(\fZ))$, where we restrict to $m \in \bbN$ for convenience (to avoid
discussion of the associated pseudodifferential operators).

\begin{thm}
If $Q \in \pPsi^{s,\cE}(\Z)$ and 
$\cE$ satisfies
\[
\begin{aligned}
	E_{\XX} &\geq 0,
	& E_{\DD}, E_{\BB} &\geq \alpha' - \alpha, 
	\\ E_{\DX}, E_{\BM} &> \alpha' + \tfrac 3 2,
	& E_{\XD}, E_{\MB} &> -\alpha - \tfrac 3 2,
\end{aligned}
\]
then $Q$ extends to a bounded operator
\begin{equation}
	Q : C^\infty(\idmon; \rho^\alpha \bpH^{k,l}(\fZ)) \to C^\infty(\idmon; \rho^{\alpha'}\bpH^{k,l-s}(\fZ)).
	\label{E:dpM_mapping}
\end{equation}
for any $k \in \bbN$ and $l \in \bbR$, where $\rho = \rho_\D\rho_\B$. 
\label{T:dpM_mapping}
\end{thm}

\begin{proof}
We first restate the result in terms of
b-metrics; since $L^2(\fZ; \olg) = \rho^{3/2} L^2(\fZ; \bg)$, the result to prove
is equivalent to the boundedness of
\[
	Q : C^\infty(\idmon; \rho^\beta \bpH^{k,l}(\fZ; \bg)) \to C^\infty(\idmon; \rho^{\beta'}\bpH^{k,l-s}(\fZ; \bg)),
\]
under the assumptions that
\[
\begin{aligned}
	E_{\XX} &\geq 0,
	& E_{\DD}, E_{\BB} &\geq \beta' - \beta, 
	\\ E_{\DX}, E_{\BM} &> \beta',
	& E_{\XD}, E_{\MB} &> -\beta,
\end{aligned}
\]
which is in turn equivalent to boundedness of
\[
\begin{gathered}
	\rho^{-\beta'} Q \rho^\beta : C^\infty(\idmon; \bpH^{k,l}(\Z; \bg)) \to C^\infty(\idmon; \bpH^{k,l-s}(\Z; \bg)),
	\\ \rho^{-\beta'} Q \rho^\beta \in \pPsi^{s,\cE'}(\Z),
	\\
\begin{aligned}
	E'_{\DD} &= E_{\DD} + (\beta - \beta'), 
	& E'_{\BB} &= E_{\BB} + (\beta - \beta'),
	\\ E'_{\DX} &= E_{\DX} - \beta',
	& E'_{\XD} &= E_{\XD} + \beta,
	\\ E'_{\BM} &= E_{\BM} - \beta',
	& E'_{\MB} &= E_{\MB}  + \beta,
	\\ E'_{\XX} &= E_{\XX},
\end{aligned}
\end{gathered}
\]
so it suffices to consider the case where $\beta = \beta' = 0$. 

Likewise, since everything is smoothly parameterized by $\idmon$, it
suffices to restrict $Q$ to a fixed fiber $\dpfM = \fu^{-1}(q)$. 

We first consider the case $k = 0$. 
$Q$ may be decomposed as
\[
	Q = Q_{\mathrm{sm}} + Q_\infty,
	\quad Q_{\mathrm{sm}} \in \psmPsi^{s}(\fZ) 
	\ Q_\infty \in \pPsi^{-\infty,\cE}(\fZ),
\]
into an element of the small calculus and a smoothing element, which may be considered seperately.
Then $Q_{\mathrm{sm}}$ was shown above to be bounded, and
$Q_\infty$ is seen to be bounded by an application of Schur's Lemma. Indeed,
\[
	(\pi_R)_\ast(\abs{Q} \pi_L^\ast \pnu) \in \cAphg^{\cF_R}(\fZ), 
	\quad (\pi_L)_\ast(\abs{Q} \pi_R^\ast \pnu) \in \cAphg^{\cF_L}(\fZ)
\]
are well-defined provided $E_{\BM} > 0$ (respectively $E_{\MB} > 0$), and 
\[
\begin{gathered}
	\cF_R = (F_{\D},F_{\X},F_{\B})
	= (E_{\DD} \ol \cup E_{\XD}, E_{\XX} \ol \cup E_{\DX}, E_{\BB} \ol \cup E_{\MB}),
	\\ \cF_L = (F'_{\D},F'_{\X},F'_{\B})
	= (E_{\DD} \ol \cup E_{\DX}, E_{\XX} \ol \cup E_{\XD}, E_{\BB} \ol \cup E_{\BM}).
\end{gathered}
\]
The hypotheses on $\cE$ guarantee that these index sets are $\geq 0$, hence
the pushforwards are uniformly bounded.

In the case $k > 0$, observe that a general vector field
in $\fuV(\fZ) \equiv \bV(\fZ)$ may be decomposed into an element of $\pV(\fZ)$
and a multiple of $\ve \pa_\ve$, where the latter denotes (by abuse of notation) a choice of lift of the canonical
b-vector field $\ve \pa_\ve$ on $[0,1)$. 
Since $\ve \pa_\ve$ differentiates in the fiber direction, $\ve \pa_\ve (Qu) =
(\ve\pa_\ve Q) u + Q (\ve \pa_\ve u)$, but $\ve \pa_\ve Q \in
\pPsi^{s,\cE}(\fZ)$ again since $\ve \pa_\ve$ is tangent to all boundary
faces of $\dpfM$. The general result then follows by commutation and induction.
\end{proof}

\bigskip
Proceeding in a similar manner, we may characterize the fiberwise $\gl$ Sobolev spaces by
\[
	\gH^s(\fZ) = \set{u \in C^{-\infty}(\fZ) : Au \in L^2(\fZ; \olg), 
	\text{ $\forall$ /($\exists$ elliptic) } 
	A \in \gsmPsi^{s}(\fZ)},
\]
where $\gsmPsi^{s}(\fZ)$ denotes the restriction to a fiber $\dgfM$ of $\fu :
\dgM \to \idmon$ of the set $\gsmPsi^s = \gPsi^{s,\cE}(\Z)$ where $E_\scD, E_\scB,
E_\XX \geq 0.$. A nearly identical proof leads to the obvious analogue of Corollary~\ref{C:dpM_mapping_cor},
and we have

\begin{thm}
If $Q \in \gPsi^{s,\cE}(\Z)$ and 
$\cE$ satisfies
\[
\begin{aligned}
	E_{\XX} &\geq 0,
	& E_{\scD}, E_{\scB} &\geq \alpha' - \alpha, 
\end{aligned}
\]
then $Q$ extends to a bounded operator
\begin{equation}
	Q : C^\infty(\idmon; \rho^\alpha \bpgH^{k,l,m}(\fZ)) \to C^\infty(\idmon; \rho^{\alpha'}\bpgH^{k,l,m-s}(\fZ)).
	\label{E:dgM_mapping}
\end{equation}
for any $k,l \in \bbN$ and $m \in \bbR$.
\label{T:dgM_mapping}
\end{thm}

\begin{proof}
The proof is almost entirely similar to the proof of
Theorem~\ref{T:dpM_mapping} above. The only additional wrinkle is that, for
nonzero $k,l$, it is neccessary to consider the lift to $\dgM$ of vector fields
in $\pV(\Z)$. As shown above, these are tangent to the boundary faces of
$\dpM$, so it follows that they lift to be singular at $\scD$ and $\scB$ in
$\dgM.$ More precisely, the lifts of $V \in \pV(\Z)$ have the form
\[
	\pi_R^\ast(V) = (\rho_\scD)^{-1}(\rho_{\scB})^{-1} \wt V,
	\quad \pi_L^\ast(V) = (\rho_\scD)^{-1}(\rho_{\scB})^{-1} \wt V',
\]
where $\wt V$ and $\wt V'$ restrict to the {\em opposite} fiberwise constant
vector field on $\scD$ and $\scB$, as follows from the computations in the
proof of Lemma~\ref{L:dgM_props}. In particular it follows that the order
$(\rho_\scD\rho_\scB)^{-1}$ term of the commutator $[Q,V] = (\pi_L^\ast
(V)  + \pi_R^\ast(V))Q$ (the sign change on $\pi_R^\ast(V)$ is due to integration
by parts) vanishes, so that 
\[
	[Q,V] \in \gPsi^{s,\cE}(\fZ), \quad V \in \pV(\Z).
\]
This may be iterated to show that, for $P \in \pDiff^l(\Z)$,  $PQ - QP' \in
\gPsi^{s,\cE}$ for some $P' \in \pDiff^{l}(\Z)$, from which the result
follows.

Finally, for $k \neq 0$, it suffices to note that $\ve \pa_\ve Q \in
\gPsi^{s,\cE}(\Z)$. Then the lift any $V \in \bV(\Z)$ can be locally
decomposed into $\ve\pa_\ve$ and the lift of an element in $\pV(\Z)$. Since 
$\ve \pa_\ve (Qu) = (\ve \pa_\ve Q) u + Q(\ve\pa_\ve u)$, 
\[
	[Q,P] \in \gPsi^{-\infty,\cE}(\Z), \quad P \in \bDiff^k(\Z),
\]
and the result for general $k$ and $l$ follows.
\end{proof}

\subsection{Residual ideals} \label{S:pseudo_resid}
Consider the subset $\pPsi^{s,(\infty_\ast)}(\Z) \subset \pPsi^{s,\ast}(\Z)$.  These {\em
residual operators} have kernels which are conormal of order $s$ at the
diagonal and vanish rapidly at all boundary faces of $\dpM$. 
They lift to similarly conormal kernels with rapid vanishing on the space $\dgM$,
which is to say the subset $\gPsi^{s,(\infty_\ast)}(\Z)\subset
\gPsi^{s,\ast}(\Z)$, and conversely $\gPsi^{s,(\infty_\ast)}(\Z)$ pushes
forward under the blow-down $\dgM \to \dpM$ to
$\pPsi^{s,(\infty_\ast)}(\Z)$. 

We identify these two subspaces and denote them simply by
\[
	\rho^\infty\Psi^{s}(\Z) := \pPsi^{s,(\infty_\ast)}(\Z)
	\equiv \gPsi^{s,(\infty_\ast)}(\Z),
	\quad s \in \bbR.
\]
It follows from Theorems~\ref{T:dpM_package} and \ref{T:dgM_package} that these
subsets form graded ideals with respect to composition, i.e.,
\[
\begin{gathered}
	\rho^\infty \Psi^{s}(\Z) \circ \pPsi^{t,\cE}(\Z),
	\ \pPsi^{t,\cE}(\Z)\circ \rho^\infty \Psi^s(\Z) \subset \rho^\infty \Psi^{s + t}(\Z),
	\\\rho^\infty \Psi^{s}(\Z) \circ \gPsi^{t,\cE}(\Z),
	\ \gPsi^{t,\cE}(\Z)\circ \rho^\infty \Psi^s(\Z) \subset \rho^\infty \Psi^{s + t}(\Z).
\end{gathered}
\]
(c.f.\ \eqref{E:sc_b_ideal}).
We refer to $\bigcup_s \rho^\infty \Psi^s(\Z)$ as the {\em residual ideal}.

\subsection{Parametrices} \label{S:parametrices}
Finally, we emply the pseudodifferential operator calculus developed above
to construct parametrices for the linearized Coulomb gauge fixing operator and Bogomolny operator,
respectively.

\begin{proof}[Proof of Proposition~\ref{P:coulomb_parametrix}]
We construct $Q^R$ as a conormal distribution on the double space $\dpM$,
decomposing near $\DD \cap \BB \cap \Delta$ according to the splitting
$\adP = \adPz \oplus \adPo$ as
\begin{equation}
\begin{gathered}
	Q^R = \begin{pmatrix} \pi_L^\ast(\rho^{1/2})\,\wt Q_0 \pi_R^\ast(\rho^{-5/2}) & Q_{01}
	  \\ Q_{10}& Q_1\end{pmatrix}
	\\ \wt Q_0 \in \pPsi^{-2,\cF^0}(\Z; \adPz),
	\quad Q_1 \in \gPsi^{-2,\cF^1}(\Z; \adPo),
	\\ Q_{ij} \in \rho^\infty\Psi^{-2}
\end{gathered}
	\label{E:coulomb_right_parametrix}
\end{equation}
for some index sets $\cF^i$, $i = 1,2$, determined below.
$Q_1 \in \gPsi^{-2,\cF^1}(\Z; \adPo)$ is the pushforward to $\dpM$ of an operator
defined on the gluing double space $\dgM$. 

Working fiberwise, from Theorem~\ref{T:linear_coulomb_X}, the inverse of
\eqref{E:linear_coulomb_X_extn} for $\alpha = 0$ may be represented as a
conormal distribution on $\dbfX$ with a decomposition
\eqref{E:linear_coulomb_X_inverse}. By the smoothness of $F\rst_\fX$ over $\idmon$,
these inverses patch together smoothly as a family of distributions on $\dbX
\cong \XX \subset \dpM$ over $\idmon$.  Likewise, the fiberwise Fourier transforms of the symbolic
inverses from Theorem~\ref{T:linear_coulomb_D_sc} form a smoothly varying
family of conormal distributions on $\ol \gT \D \cong \scD \subset \dpM$, and
the inverses of \eqref{E:linear_coulomb_D_extn} for $\alpha = \beta = 0$ form a
smoothly varying family of distributions on $\dbD \cong \DD \subset \dpM$.
Moreover, the leading order terms in the expansions of these inverses at $\X
\cap \D$ are compatible; the scattering symbol of $G_1$ in
\eqref{E:linear_coulomb_X_inverse} agrees with the inverse of the symbol 
in Theorem~\ref{T:linear_coulomb_D_sc} there, and after accounting for
the various boundary defining factors, the indicial operators of $\wt G_0$ in
\eqref{E:linear_coulomb_X_inverse} and $\wt G$ in
\eqref{E:linear_coulomb_D_inverse} agree.

Consequently, there exists a distribution $Q$ on $\dpM$ whose restriction to $\XX$ agrees
with \eqref{E:linear_coulomb_X_inverse}, and which decomposes as
\eqref{E:coulomb_right_parametrix} where
\[
\begin{gathered}
	\sigma_\fD(\wt Q_0) = \sigma_\fD(F_1)^{-1}, \quad \text{from \eqref{E:linear_coulomb_D_inverse}},
	\\ N_\fX(\wt Q_0) = \wt G_0, \quad \text{near $\fX \cap \fD$, $\wt G_0$ from \eqref{E:linear_coulomb_X_inverse}},
	\\ I(N_\fX(\wt Q_0), \lambda) = I(N_\fD(\wt Q_0), - \lambda),
\end{gathered}
\]
and $Q_1 \in \gPsi^{-2,\cF^1}(\Z; \adPz)$ satifies
\[
\begin{gathered}
	N_{\scfD}(Q_1) = G_1, \quad \text{$G_1$ from Theorem~\ref{T:linear_coulomb_D_sc}}
	\\ N_{\fX_i}(Q_1) = G_1, \quad \text{$G_1$ from \eqref{E:linear_coulomb_X_inverse}},
	\\ \scsigma(N_{\fX_i}(Q_1)) \cong \sigma_\D(N_{\scfD}(Q_1)).
\end{gathered}
\]
Furthermore, by the usual iterative argument using principal symbols, we can
arrange for the interior conormal singularity of $Q^R$ to invert that of $F$ to
all orders.

The index sets $\cF^i$, $i = 1,2$ for $Q^R$ satisfy
\[
\begin{gathered}
	F^0_{\DD},\ F^0_{\XX},\ F^0_{\BB} \geq 0,
	\quad F^0_{\XD},\ F^0_{\DX},\ F^0_{\BM},\ F^0_{\MB} \geq \tfrac 1 2, 
	\\ F^1_{\scD},\ F^1_{\scB},\ F^1_{\XX} \geq 0.
\end{gathered}
\]

From Theorems~\ref{T:dpM_package} and \ref{T:dgM_package}, the error term $E'
= I - Q^RF$ has (interior conormal) order $-\infty$, and admits a similar
decomposition, with $E'_0 = \pi_L^\ast(\rho^{5/2}) \wt E'_0
\pi_R^\ast(\rho^{-5/2})$, where $\wt E'_0 \in \pPsi^{-\infty,\wt \cG^0}(\Z;
\adPz)$ and $E'_1 \in \gPsi^{-\infty,\cG^1}(\Z; \adPo)$ have the same
estimates on their index sets as $\wt Q_0$ and $Q_1$, but with extra vanishing
at the boundary faces meeting the diagonal over $\ve = 0$ since they invert $F$
there (for instance, $\wt G^0_\DD = F_\DD \setminus \min F_\DD$, etc.). Thus
\[
\begin{gathered}
	\wt G^0_{\DD},\ \wt G^0_{\XX},\ G^1_{\scD}, G^1_{\XX} > 0,
	\quad \wt G^0_{\BB},\ G^1_{\scB} \geq 0,
	\\ \wt G^0_{\XD},\ \wt G^0_{\DX},\ \wt G^0_{\BM},\ \wt G^0_{\MB} \geq \tfrac 1 2, 
\end{gathered}
\]
Once we account for the left and right factors of $\rho$, it follows that $E'_0 =
\pi_L^\ast(\rho^{5/2}) \wt E'_0 \pi_R^\ast(\rho^{-5/2})$ is in
$\pPsi^{-\infty,\cG^0}(\Z; \adPz)$ where
\[
\begin{gathered}
	G^0_{\DD},\ G^0_{\XX} > 0,
	\quad G^0_{\BB} \geq 0,
	\\G^0_{\DX},\ G^0_{\BM} \geq \tfrac 1 2 + \tfrac 5 2 = 3,
	\quad G^0_{\XD},\ G^0_{\MB} \geq \tfrac 1 2 - \tfrac 5 2 = -2.
\end{gathered}
\]

Now, on $\dpM$ (respectively $\dgM$), $\ve$ is a product of defining functions
\[
	\ve = \rho_{\DD}\rho_{\XX}\rho_{\DX}\rho_{\XD},
	\quad \text{(resp. $\ve = \rho_{\DD}\rho_{\XX}\rho_{\DX}\rho_{\XD}\rho_{\scD}$)}.
\]
Taking 
\[
	\delta < \min(G^0_{\DD},G^0_{\XX},G^1_{\scD},G^1_{\XX},\tfrac 1 2),
\]
it follows that we can write $E' = \ve^\delta E^R$,
where $E^R = \begin{pmatrix} E^R_0 & E^R_{01}\\E^R_{10} & E^R_1\end{pmatrix}$ satisfies
$E^R_0 \in \pPsi^{-\infty,\cI^0}(\Z; \adPz)$, $E^R_1 \in \gPsi^{-\infty,\cI^1}(\Z; \adPo)$,
and
\[
\begin{gathered}
	I^0_{\DD},\ I^0_{\XX}, I^1_{\scD}, I^1_{\XX} > 0,
	\quad I^0_{\BB}, I^1_{\scB} \geq 0,
	\\ I^0_{\DX},\ I^0_{\BM} > \tfrac 5 2,
	\quad I^0_{\XD},\ I^0_{\MB} > - \tfrac 5 2.
\end{gathered}
\]
By Theorems~\ref{T:dpM_mapping} and \ref{T:dgM_mapping}, it follows that
\[
\begin{aligned}
	E^R_0 &: C^\infty(\idmon;\rho^1\bpgH^{k,0,0}(\fZ;\adPz)) 
	  \to C^\infty(\idmon;\rho^1\bpgH^{k,0,0}(\fZ; \adPz))
	\\E^R_1 &: C^\infty(\idmon;\rho^\beta\bpgH^{k,0,0}(\fZ;\adPo)) 
	  \to C^\infty(\idmon;\rho^\beta\bpgH^{k,0,0}(\fZ; \adPo))
	\\E^R_{01} &: C^\infty(\idmon;\rho^\beta\bpgH^{k,0,0}(\fZ;\adPo)) 
	  \to C^\infty(\idmon;\rho^1\bpgH^{k,0,0}(\fZ; \adPz))
	\\E^R_{10} &: C^\infty(\idmon;\rho^\beta\bpgH^{k,0,0}(\fZ;\adPz)) 
	  \to C^\infty(\idmon;\rho^1\bpgH^{k,0,0}(\fZ; \adPo))
\end{aligned}
\]
are bounded (where defined), and hence $E^R$ is a smooth bundle map of
$\threeH^{k,0,\beta}(\fZ; \adP)$ over $\idmon$ as claimed.

The construction of $Q^L$ and estimates on $E^L$ follow similarly.
\end{proof}


\begin{proof}[Proof of Proposition~\ref{P:bogo_pmtx}]
The construction is similar to the one in the previous proof so we shall be somewhat brief.
We define $R$ on $\dpM$, such that,
with respect to the splitting $\adP = \adPz \oplus \adPo$ in a neighborhood of
$\DD\cap \BB\cap \Delta$, 
\[
	R = \begin{pmatrix} \pi_L^\ast \rho \wt R_0\pi_R^\ast \rho^{-2} & R_{01}\\R_{10} & R_1\end{pmatrix}
\]
where $\rho := \rho_\D \rho_\B$, and
$\wt R_0 \in \pPsi^{-1,\wt \cF^0}(\Z; \gLam^\ast\otimes \adPz)$, $R_1
\in \gPsi^{-1,\cF^1}(\Z; \gLam^\ast \otimes \adPo)$, and $R_{ij} \in
\rho^\infty \Psi^{-\infty}(\fZ;\gLam^\ast \otimes \adP_j,\gLam^\ast\otimes \adP_i)$, where
fiberwise over $\idmon$,
\[
\begin{gathered}
	N_\fD(\wt R_0) = \wt G, \quad \text{$\wt G$ from \eqref{E:linear_bogo_D_inverse}},
	\\ N_\fX(\wt R_0) = \wt G_0, \quad \text{near $\fX \cap \fD$, $\wt G_0$ from \eqref{E:linear_bogo_X_inverse}},
	\\ I(N_\fX(\wt R_0), \lambda) = I(N_\fD(\wt R_0), - \lambda),
\end{gathered}
\]
and $R_1 \in \gPsi^{-1,\cF^1}(\Z; \adPo)$ satifies
\[
\begin{gathered}
	\sigma_\fD(R_1) = \sigma_\fD(L_1 + \Phi)^{-1}, \quad \text{from Proposition~\ref{P:normal_symbol_D}}
	\\ N_{\fX_i}(R_1) = G_1, \quad \text{$G_1$ from \eqref{E:linear_bogo_X_inverse}},
	\\ \scsigma(N_{\fX_i}(R_1)) \cong \sigma_\D(N_{\scfD}(R_1)).
\end{gathered}
\]
That such a $R$ exists follows from compatibility of the inverses for
$L_{\fX_i}$ and $L_{\fD}$, which are smoothly parameterized over $\idmon$, and the
fact that such distributions may be extended smoothly off the relevant boundary
faces of $\dpM$ and $\dgM$. The index sets $\wt \cF^0$ and $\cF^1$ satisfy
\[
\begin{gathered}
	F^0_{\DD},\ F^0_{\XX},\ F^0_{\BB} \geq 0,
	\quad F^0_{\XD},\ F^0_{\DX},\ F^0_{\BM},\ F^0_{\MB} \geq 1
	\\ F^1_{\scD},\ F^1_{\scB},\ F^1_{\XX} \geq 0.
\end{gathered}
\]
We may futher suppose that the interior conormal singularity of $R$ inverts
that of $L$ to all orders.

It follows from Theorems~\ref{T:dpM_package} and \ref{T:dgM_package} that
\[
	LR = I - E', 
	\quad E' = \begin{pmatrix} E'_0 & E'_{01}
	  \\E'_{10} & E'_1 \end{pmatrix}
\]
where $E'_0 = \pi_L^\ast \rho^2\wt E'_0\pi_R^\ast \rho^{-2} \in
\pPsi^{-\infty,\cG^0}(\Z: \gLam^\ast \otimes \adP)$, $E'_1 \in
\gPsi^{-\infty,\cG^1}(\Z; \gLam^\ast \otimes \adP)$ and $E'_{ij} \in
\rho^\infty \Psi^{-\infty}(\Z; \gLam^\ast\otimes \adP_j,\gLam^\ast \otimes\adP_i)$. Here
\[
\begin{gathered}
	G^0_{\DD},\ G^0_{\XX},\ G^1_{\scD}, G^1_{\XX} > 0,
	\quad G^0_{\BB}, G^1_{\scB} \geq 0,
	\\G^0_{\DX},\ G^0_{\BM} \geq 1 + 2 = 3,
	\quad G^0_{\XD},\ G^0_{\MB} \geq 1  - 2 = -1.
\end{gathered}
\]
If we choose
\[
	\delta < \min(G^0_{\DD},G^0_{\XX},G^1_{\scD},G^1_{\XX},\tfrac 1 2),
\]
Then it follows that $E' =: \ve^\delta E$, where $E = \begin{pmatrix} E_0 & E_{01}\\E_{10} & E_1\end{pmatrix}$
has index sets
\[
\begin{gathered}
	E_0 \in \pPsi^{-\infty,\cI^0}, \quad E_1 \in \gPsi^{-\infty,\cI^1},
	\\ I^0_{\DD},\ I^0_{\XX}, I^1_{\scD}, I^1_{\XX} > 0,
	\quad I^0_{\BB}, I^1_{\scB} \geq 0,
	\\ I^0_{\DX},\ I^0_{\BM} > \tfrac 5 2,
	\quad I^0_{\XD},\ I^0_{\MB} > -\tfrac 3 2 > - \tfrac 5 2.
\end{gathered}
\]
Boundedness of $E$ as an operator \eqref{E:pmtx_error_mapping} follows from
Theorems~\ref{T:dpM_package} and \ref{T:dgM_package}.
\end{proof}